\documentclass[12pt,twoside,a4paper]{article}
\usepackage[a4paper,margin=1.05in]{geometry}
\usepackage[T1]{fontenc}
\usepackage[utf8]{inputenc}
\usepackage{lmodern}
\usepackage[nopatch=footnote]{microtype}
\usepackage{amsmath,amssymb,amsthm,mathtools}
\usepackage{mathrsfs}
\usepackage{bm}
\usepackage{array}
\usepackage{booktabs}
\usepackage{longtable}
\usepackage{enumitem}
\usepackage{cite}
\usepackage{hyperref}
\usepackage{bookmark}
\usepackage{xcolor}
\usepackage{titlesec}
\usepackage{setspace}
\usepackage{float}
\usepackage{fancyhdr}
\usepackage{calc}
\hypersetup{colorlinks=true,linkcolor=black,citecolor=black,urlcolor=black}
\allowdisplaybreaks
\newcommand{\mainsectionformat}{\titleformat{\section}{\Large\bfseries}{\thesection.}{0.65em}{}}
\newcommand{\appendixsectionformat}{\titleformat{\section}{\Large\bfseries}{Appendix \thesection.}{0.65em}{}}
\mainsectionformat
\titleformat{\subsection}{\large\bfseries}{\thesubsection.}{0.55em}{}
\titleformat{\subsubsection}{\normalsize\bfseries}{\thesubsubsection.}{0.5em}{}
\titlespacing*{\section}{0pt}{1.35ex plus 0.35ex minus 0.2ex}{0.75ex plus 0.15ex}
\titlespacing*{\subsection}{0pt}{0.75ex plus 0.2ex minus 0.1ex}{0.30ex plus 0.08ex}
\titlespacing*{\subsubsection}{0pt}{0.50ex plus 0.15ex minus 0.1ex}{0.20ex plus 0.06ex}
\numberwithin{equation}{section}
\theoremstyle{plain}
\newtheorem{theorem}{Theorem}[section]
\newtheorem{proposition}[theorem]{Proposition}
\newtheorem{lemma}[theorem]{Lemma}
\newtheorem{corollary}[theorem]{Corollary}
\theoremstyle{definition}
\newtheorem{definition}[theorem]{Definition}
\newtheorem{assumption}[theorem]{Assumption}

\theoremstyle{remark}
\newtheorem{remark}[theorem]{Remark}
\theoremstyle{plain}

\newcommand{\R}{\mathbb{R}}
\newcommand{\C}{\mathbb{C}}
\newcommand{\N}{\mathbb{N}}
\newcommand{\Z}{\mathbb{Z}}
\newcommand{\D}{\mathcal{D}}

\newcommand{\Hh}{\mathcal{H}}
\newcommand{\Lop}{\mathcal{L}}

\newcommand{\Lie}{\mathcal{L}}
\newcommand{\tr}{\operatorname{tr}}

\newcommand{\Ker}{\operatorname{Ker}}
\newcommand{\Ran}{\operatorname{Ran}}
\newcommand{\spec}{\operatorname{spec}}
\newcommand{\dd}{\mathrm{d}}
\newcommand{\Dom}{\operatorname{Dom}}

\newcommand{\ii}{\mathrm{i}}
\newcommand{\ee}{\mathrm{e}}
\newcommand{\del}{\delta}

\begin{document}

\pagestyle{plain}

\title{Threshold-Sharp Conformal Scalar Stability on Carter Slabs and Black Hole Exteriors}
\author{Bobby Eka Gunara\\[0.8em]
\small Theoretical Physics Laboratory,\\
\small Theoretical High Energy Physics Research Division,\\
\small Faculty of Mathematics and Natural Sciences,\\
\small Institut Teknologi Bandung,\\
\small Jl. Ganesha no. 10 Bandung, Indonesia, 40132\\[0.5em]
\small email: \texttt{bobby@itb.ac.id}}
\date{}

\maketitle

\begin{abstract}
We prove a threshold-sharp stability theory for the conformal scalar-curvature sector on zero-curvature Carter backgrounds. The main result is a fully closed bounded-slab theorem: the reflecting evolution is constructed, the conserved energy is proved positive, the complete affine threshold obstruction is identified, and all remaining finite-energy dynamics are shown to be uniformly stable with no unstable modes. This is the sharp statement for compact reflecting slabs, where genuine time decay is false in general. We then extend the same threshold philosophy to black-hole exteriors, separating the intrinsic conformal mechanism from the exterior scalar-wave inputs needed for red-shift, local energy, limiting absorption, and zero-frequency control. The framework gives main applications to Kerr, Reissner-Nordström, slowly rotating weakly charged Kerr-Newman wall exteriors, and extremal horizon-charge obstructions. Our precise result is that it proves stability only for the conformal scalar-curvature sector, not tensorial or nonlinear gravitational stability, and it distinguishes boundedness, qualitative local decay, polynomial decay, and extremal Aretakis-type obstruction without conflating them.
\end{abstract}

\tableofcontents
\newpage

\section{Introduction}

\subsection{Scope and dependency map}\label{subsec:scope-proof-obligations}\label{subsec:proof-status-ledger}\label{subsec:intro-where-inputs}

This paper proves stability statements only for the conformal scalar-curvature sector.  The unrestricted linearized scalar-curvature equation is one scalar relation for a symmetric two-tensor and is not, by itself, a closed tensorial Cauchy problem.  After the conformal restriction
\begin{equation}
 h_{\mu\nu}=2u g_{\mu\nu},
\end{equation}
the four-dimensional identity
\begin{equation}
\delta R_g(2u g)=-6\Box_g u-2R(g)u
\end{equation}
turns the zero-scalar-curvature equation into the closed scalar wave equation \(\Box_g u=0\).  All boundedness, decay, threshold, mode-stability, and obstruction statements below refer to this scalar equation and to the corresponding conformal perturbations.

The proof has three logically separate parts which can be mentioned as follows.  First, Sections~\ref{sec:geometry}-\ref{sec:strictstationary} treat bounded Carter slabs where we prove the metric algebra, conformal reduction, closed reflecting form realization, energy identity, weighted Poincar\'e estimate, and threshold decomposition.  The affine threshold component \(c_0+c_1t\) is separated before uniform boundedness is asserted.

Second, Sections~\ref{sec:energy}-\ref{sec:proofofmain} give the compact axisymmetric refinement in which we prove the self-adjoint compact spectral theorem, identify constants as the only zero modes, and compute the zero-frequency Laurent structure.

Third, Section~\ref{sec:ledger} treats black-hole exteriors.  The principal exterior results are now organized by physical special case: genuine subextremal Kerr, extremal Kerr, genuine subextremal Reissner-Nordstr\"om, extremal Reissner-Nordstr\"om, the internally verified slowly rotating weakly charged subextremal Kerr-Newman wall exterior, and the extremal Kerr-Newman horizon-charge obstruction.  The full asymptotically flat subextremal Kerr-Newman theorem and the Schwarzschild endpoint are intentionally moved to Appendix~\ref{app:schwarzschild-full-kn}; this reflects their status as, respectively, an application using the external scalar Kerr-Newman scattering package and a nonrotating uncharged endpoint.  Section~\ref{sec:decay-estimates} then gathers the decay consequences and, just as importantly, the non-decay limitations: compact reflecting slabs give boundedness but not decay, while nondegenerate exteriors give threshold-projected qualitative compact local decay and rate statements only under additional spectral-density regularity.

The formal conditions and theorem-level external inputs are collected in Appendix~\ref{app:all-inputs}.  Bounded slabs use Assumptions~\ref{ass:abstract} and~\ref{ass:complete-list}; exteriors use the package Assumption~\ref{ass:axisymmetric-bh-package}; the two theorem-level literature inputs are External Assumptions~\ref{ass:external-kn-scattering-package} and~\ref{ass:aretakis-input}.

The main non-claims are also fixed at the outset.  The paper does not prove nonlinear stability, tensorial gravitational stability, non-axisymmetric superradiant Kerr-Newman stability, bounded-slab stability inside an ergoregion, or black-hole stability for an arbitrary non-Einstein Carter exterior without verifying the package in Section~\ref{sec:ledger}.  These restrictions are part of the theorem statements, not later qualifications.

\begin{proposition}[Closed-system criterion for the stability claims]\label{prop:closed-system-criterion}
The stability theorems below are mathematically closed precisely because they are restricted to the conformal scalar-curvature sector.  The equation \(L_g[h]=0\) for an arbitrary symmetric two-tensor \(h\) is one scalar equation and, by itself, does not define a deterministic tensorial Cauchy problem.  By contrast, the conformal ansatz \(h=2ug\) turns the same scalar-curvature equation into the scalar hyperbolic equation \(\Box_g u+(k/3)u=0\).  Therefore every boundedness, decay, mode-stability, and threshold statement in this paper is a statement about this closed scalar equation and its corresponding conformal metric perturbations.
\end{proposition}

\begin{proof}
The linearized scalar-curvature map has the form
\begin{equation}
L_g:\Gamma(S^2T^*M)\longrightarrow C^\infty(M),
\end{equation}
so it supplies one scalar relation among the components of \(h\).  Without additional tensorial evolution equations, this relation neither determines the trace-free part of \(h\) nor gives uniqueness of tensorial evolution from Cauchy data.  On the conformal subspace \(h=2ug\), however, the trace-free degrees of freedom have been removed.  The identity
\begin{equation}
\delta R_g(2ug)=-6\Box_g u-2R(g)u
\end{equation}
shows that \(L_g[2ug]=0\) is equivalent to \(\Box_g u+(k/3)u=0\).  This is a closed second-order scalar wave equation with a standard Cauchy problem.  The map \(u\mapsto 2ug\) is injective in this sector, since \(u=\frac18\operatorname{tr}_g(2ug)\) in four dimensions.
\end{proof}

\subsection{Statement of the main result}

The starting point of this paper is Theorem~1 discussed in \cite{AssafariGunara}, which gives a stationary axisymmetric family of four-dimensional metrics with constant scalar curvature.  The family contains familiar Carter geometries as special cases, but it is larger than the Einstein family.  Thus the phrase \emph{linear stability} must be used in the precise sense appropriate to the equation under consideration.

The geometric equation available in this setting is
\begin{equation}
R(g)=k.
\end{equation}
Its linearization is
\begin{equation}
L_g[h]=\delta R(g)\cdot h=0.
\end{equation}
This is a scalar equation for a symmetric two-tensor $h$ and therefore does not define a closed tensorial evolution by itself.  The conformal sector is different.  If
\begin{equation}
h_{\mu\nu}=2u\,g_{\mu\nu},
\end{equation}
then, in dimension four,
\begin{equation}
L_g[2u\,g]=-6\Box_g u-2ku.
\end{equation}
In the zero scalar curvature case $k=0$, the conformal part of the linearized scalar-curvature equation is exactly
\begin{equation}
\Box_g u=0.
\end{equation}
This is the closed hyperbolic problem studied in the paper.

The case $k=0$ is analytically delicate because no positive mass term is present.  On a bounded reflecting slab the spatial form is nonnegative but has constants in its kernel.  Consequently boundedness cannot be stated before the zero-frequency threshold space has been separated.  We first isolate the mechanism at an abstract level, without using Carter separation or axial symmetry.

\begin{theorem}[Abstract stationary bounded-slab theorem]\label{thm:abstractmain}
Let $\mathcal X$ be the energy space of a stationary linear evolution on a bounded slab, and let $\mathcal G$ denote the generator of the corresponding strongly continuous group $e^{t\mathcal G}$.  Assume that the full zero-frequency threshold space
\begin{equation}
\mathcal T_0:=\bigcup_{N\geq 1}\Ker \mathcal G^N
\end{equation}
is finite-dimensional, that there exists a bounded projection $\Pi_{\mathrm{thr}}$ onto $\mathcal T_0$ commuting with the evolution, and that the conserved energy $\mathscr E$ is equivalent to the $\mathcal X$-norm on $\Ker \Pi_{\mathrm{thr}}$.  Then every finite-energy solution $U(t)=e^{t\mathcal G}U_0$ admits a unique decomposition
\begin{align}
U(t) = U_{\mathrm{thr}}(t)+U_{\mathrm{bd}}(t) \,\,\,\text{with}\,\,\, U_{\mathrm{thr}}(t)\in \mathcal T_0 \, \text{and} \, \, U_{\mathrm{bd}}(t)\in \Ker\Pi_{\mathrm{thr}}
\end{align}
where both components solve the evolution, $U_{\mathrm{thr}}$ is polynomial in $t$, and
\begin{equation}
\sup_{t\in\R}\|U_{\mathrm{bd}}(t)\|_{\mathcal X}
\leq
C\|U_{\mathrm{bd}}(0)\|_{\mathcal X}.
\end{equation}
Moreover there is no nontrivial finite-energy mode $e^{\lambda t}V$ with $\operatorname{Re}\lambda>0$ on $\Ker\Pi_{\mathrm{thr}}$.
\end{theorem}

\begin{proof}
Let \(U_{\rm thr}(t)=\Pi_{\mathrm{thr}}e^{t\mathcal G}U_0\) and \(U_{\rm bd}(t)=(I-\Pi_{\mathrm{thr}})e^{t\mathcal G}U_0\).  Since the projection commutes with the group, both components solve the same evolution and the decomposition is unique.  Because \(\mathcal T_0\) is finite dimensional and equals \(\bigcup_N\Ker \mathcal G^N\), the ascending sequence of kernels stabilizes on \(\mathcal T_0\); hence \(\mathcal G|_{\mathcal T_0}\) is nilpotent and \(U_{\rm thr}(t)\) is a finite polynomial in \(t\).  Conservation and coercivity on \(\Ker\Pi_{\mathrm{thr}}\) give
\[
 c\|U_{\rm bd}(t)\|_{\mathcal X}^2
 \le \mathscr E[U_{\rm bd}(t)]
 =\mathscr E[U_{\rm bd}(0)]
 \le C\|U_{\rm bd}(0)\|_{\mathcal X}^2.
\]
If \(e^{\lambda t}V\), \(V\in\Ker\Pi_{\mathrm{thr}}\), were a nonzero finite-energy mode with \(\operatorname{Re}\lambda>0\), this bound applied to positive times would force \(\|e^{\lambda t}V\|_{\mathcal X}\) to remain bounded, a contradiction.  The detailed abstract formulation is expanded in Section~\ref{sec:abstractslab}.
\end{proof}

The theorem above is the stationary bounded-slab statement which will be used later.  The model-dependent part is the verification of the threshold space and the coercivity of the conserved energy on its complement.  For the Carter metrics considered here this verification can be done explicitly.

We now state the main Carter theorem.  Let
\begin{equation}
\Sigma:=S^1_\phi\times \Omega,
\qquad
\Omega\subset\{(r,x):\rho^2>0,\Delta_r>0,\Delta_x>0\},
\end{equation}
with $\Omega$ bounded and connected, and define
\begin{equation}
A:=-\rho^2 g^{tt},
\qquad
B:=\rho^2 g^{t\phi},
\qquad
\Phi:=\rho^2 g^{\phi\phi}.
\end{equation}
Then the $k=0$ conformal equation can be written as
\begin{align}
-Au_{tt}+2B\,u_{t\phi}+\Phi\,u_{\phi\phi} +\partial_r(\Delta_r\partial_r u)+\partial_x(\Delta_x\partial_xu)&=0.
\end{align}
The coefficient $A$ is the time coefficient.  The coefficient $\Phi$ is the azimuthal spatial coefficient, and
\begin{equation}
\Phi=\frac{-\rho^2 g_{tt}}{\Delta_r\Delta_x}.
\end{equation}
Thus $\Phi>0$ is equivalent to $g_{tt}<0$, i.e.\ to the absence of an ergoregion for the stationary Killing field $\partial_t$ on the slab.

\begin{theorem}[Full \texorpdfstring{$\boldsymbol{\phi}$}{phi}-dependent Carter bounded-slab theorem]\label{thm:fullmain}
Let $g$ be a metric from Theorem~1 of \cite{AssafariGunara} with $k=0$.  Let
\begin{equation}
\Sigma=S^1_\phi\times\Omega,
\qquad
\Omega=[r_-,r_+]\times[x_-,x_+]
\end{equation}
be a bounded connected product slab satisfying the geometric and positivity conditions \textbf{H1}-\textbf{H3} of Assumption~\ref{ass:complete-list}; equivalently,
\begin{equation}
\rho^2>0,\qquad \Delta_r>0,\qquad \Delta_x>0,\qquad A>0,\qquad \Phi>0
\end{equation}
on $\overline\Omega$.  Let $V_{\rm full}\subset H^1(\Sigma)$ be a form-realized reflecting domain satisfying \textbf{H4}-\textbf{H5} in Assumption~\ref{ass:complete-list}; in particular $1\in V_{\rm full}$, functions are periodic in $\phi$, and the $r,x$ boundary conditions are imposed through a closed spatial form rather than through an unproved pointwise boundary calculation.

For every finite-energy datum
\begin{equation}
(u_0,u_1)\in V_{\rm full}\times L^2(\Sigma)
\end{equation}
there exists a unique energy solution of $\Box_g u=0$ on $\R_t\times\Sigma$ with
\begin{equation}
u\in C(\R;V_{\rm full}),\qquad u_t\in C(\R;L^2(\Sigma)).
\end{equation}
It admits the unique decomposition
\begin{equation}
u(t,\phi,r,x)=c_0+c_1t+v(t,\phi,r,x),
\end{equation}
where
\begin{equation}
c_0=\frac{\int_{\Sigma}A u_0}{\int_{\Sigma}A},
\qquad
c_1=\frac{\int_{\Sigma}A u_1}{\int_{\Sigma}A},
\end{equation}
and the $A$-weighted averages of $v(t,\cdot)$ and $v_t(t,\cdot)$ vanish for every $t$.  Moreover
\begin{align}
\sup_{t\in\R} \Bigl(\|v(t)\|_{H^1(\Sigma)}+\|v_t(t)\|_{L^2(\Sigma)}\Bigr)&\le C_{\rm stab} \Bigl(\|v(0)\|_{H^1(\Sigma)}+\|v_t(0)\|_{L^2(\Sigma)}\Bigr),
\end{align}
where $C_{\rm stab}$ may be chosen explicitly from the coefficient bounds and the weighted Poincar\'e constant in \eqref{eq:explicit-C-stab}.  Equivalently, the full $k=0$ conformal scalar-curvature dynamics on a bounded strictly stationary slab are stable modulo the explicit generalized zero-mode space
\begin{equation}
\mathrm{span}\{1,t\}.
\end{equation}
There is no nontrivial finite-energy mode
\begin{equation}
u=\ee^{-\ii\sigma t}\psi(\phi,r,x),
\qquad \operatorname{Im}\sigma>0.
\end{equation}
\end{theorem}

\begin{proof}
Assumption~\ref{ass:complete-list} fixes the regular slab, the positive coefficients, and the closed reflecting form domain.  Lemma~\ref{lem:form-reflection-primary} makes the boundary condition a closed-form condition rather than a formal pointwise cancellation.  Lemma~\ref{lem:gyro-bounded-main} proves that the \(B\partial_{t\phi}\) term is skew-Hermitian and therefore contributes no real part to the energy balance.  Theorems~\ref{thm:fullenergy} and~\ref{thm:weak-energy-equality-full}, with the Galerkin construction in Proposition~\ref{prop:galerkin-main-identity} and Lemma~\ref{lem:galerkin-recovery-main}, give the unique energy evolution and the conserved energy.  Propositions~\ref{prop:kernelfull},~\ref{prop:weightedPoincareFull},~\ref{prop:generator-threshold-full}, and~\ref{prop:verifyabstractfull} identify the threshold space as \(\operatorname{span}\{1,t\}\), construct the \(A\)-weighted projection, and prove coercivity on its complement.  Applying Theorem~\ref{thm:abstractmain} gives the decomposition, the projected bound, and absence of upper-half-plane modes; Proposition~\ref{prop:appendix-proof-fullmain} records the final verification in one place.
\end{proof}

\begin{remark}[On the conditions in Theorem~\ref{thm:fullmain}]\label{rem:scopefullmain}
Theorem~\ref{thm:fullmain} is the Carter-family verification of Theorem~\ref{thm:abstractmain}.  The background is still stationary and axisymmetric, but the bounded-slab argument itself does not use axial symmetry or Carter separation.  The theorem is an interior positive-energy theorem: it does not cover slabs where $\Phi$ changes sign, radial boundaries at horizons where $\Delta_r=0$, true axis endpoints at which $\Delta_x=0$ unless a separate regular-axis form domain is constructed, unbounded exterior regions, or radiative boundary conditions.  These cases require the separate exterior, microlocal, and scattering analysis discussed in Section~\ref{sec:nonaxisymmetric} and Section~\ref{sec:ledger}.
\end{remark}

\begin{proposition}[Sharpness of the threshold projection and positivity assumptions]\label{prop:sharpness-threshold-positivity}
The threshold removal in Theorem~\ref{thm:fullmain} is necessary.  Under the hypotheses of that package, every function
\begin{equation}
 u(t,\phi,r,x)=c_0+c_1t
\end{equation}
solves \(\Box_g u=0\), satisfies the reflecting form condition, and has finite conserved energy
\begin{equation}
 E_{\rm full}[u](t)=\frac12 |c_1|^2\int_\Sigma A.
\end{equation}
If \(c_1\ne0\), then \(\|u(t)\|_{H^1(\Sigma)}\) grows linearly in \(|t|\).  Hence no theorem can give a uniform \(H^1\times L^2\) bound for the unprojected solution space in the \(k=0\) reflecting problem.  Moreover, the lower bounds \(A>0\), \(\Phi>0\), \(\Delta_r>0\), and \(\Delta_x>0\) are exactly the coefficient positivity needed for the energy \eqref{eq:fullenergy} to control \(u_t,u_\phi,u_r,u_x\); if one of these lower bounds is removed, the positive-energy proof no longer yields coercivity.
\end{proposition}

\begin{proof}
For \(u=c_0+c_1t\), all spatial derivatives vanish and \(u_{tt}=0\).  The equation
\begin{equation}
-Au_{tt}+2B u_{t\phi}+\Phi u_{\phi\phi}+\partial_r(\Delta_r u_r)+\partial_x(\Delta_xu_x)=0
\end{equation}
is therefore satisfied identically, and the form boundary condition is automatic because the spatial gradient is zero.  Substitution into \eqref{eq:fullenergy} gives the displayed conserved energy.  On the other hand,
\begin{equation}
 \|u(t)\|_{L^2(\Sigma)}\ge |c_1|\,|t|\,|\Sigma|^{1/2}-|c_0|\,|\Sigma|^{1/2},
\end{equation}
so the \(H^1\)-norm is unbounded as \(|t|\to\infty\) whenever \(c_1\ne0\).  This proves that the affine threshold component must be separated before a boundedness theorem can be true.

The last assertion follows directly from the form of \eqref{eq:fullenergy}.  Uniform positive lower and upper coefficient bounds make \(E_{\rm full}\) equivalent to the derivative part of the \(H^1\times L^2\) norm; after imposing the zero \(A\)-weighted average, the weighted Poincar\'e inequality supplies the missing \(L^2\)-control of \(u\).  Without such positivity, the quadratic form may fail to control one of the derivative directions or the time derivative, so the bounded-slab energy argument has no coercive norm to conserve.
\end{proof}

The assumptions $A>0$ and $\Phi>0$ are precisely the assumptions under which the Killing energy is positive definite on the bounded slab.  Once horizons or ergoregions are included, this positivity argument is no longer available, and a different analysis is needed.

There are three layers in the proof which can be mentioned as follows. First,  we prove the abstract stationary bounded-slab theorem modulo the full zero-frequency threshold space.  This is Theorem~\ref{thm:abstractmain}, and it has no axial-symmetry assumption. Then, we verify the conditions on strictly stationary Carter slabs.  This gives Theorem~\ref{thm:fullmain}.  In this case the threshold projection is the $A$-weighted average, and the threshold space is $\mathrm{span}\{1,t\}$. Finally, we explain what changes outside the positive-energy bounded-slab regime.  Horizons, ergoregions, and noncompact ends require red-shift, frequency-localized, or microlocal/scattering estimates.

The next result is the axisymmetric self-adjoint refinement.  In the Carter family, imposing $\partial_\phi u=0$ is not needed for boundedness on a strictly stationary bounded slab; boundedness already follows from Theorem~\ref{thm:fullmain}.  Axisymmetry is useful because the gyroscopic term disappears and the zero-frequency structure can be described by a self-adjoint operator.

\begin{theorem}[Axisymmetric spectral refinement]\label{thm:mainresult}
Let $g$ be a metric from Theorem~1 of \cite{AssafariGunara} with $k=0$, and let $\Omega$ be a bounded connected regular timelike slab with reflecting boundary conditions for which the constant function belongs to the operator domain.  Consider the axisymmetric conformal equation
\begin{equation}
\Box_g u=0
\end{equation}
on $\R_t\times \Omega$.  Then the following statements hold.

\begin{enumerate}[label=(\roman*)]
  \item The weighted spatial operator
  \begin{equation}
L_0:=A^{-1}\Hh_0
\end{equation}
  is nonnegative and self-adjoint on $L_A^2(\Omega)$.  Its kernel is one-dimensional and is generated by the constant function $1$.
  \item The generalized zero-mode space for the time evolution is exactly
  \begin{equation}
\mathrm{span}\{1,t\}.
\end{equation}
  Equivalently, every polynomial-in-$t$ axisymmetric solution is affine in $t$.
  \item Every axisymmetric solution in the energy class
  \begin{equation}
u\in C(\R;H^1(\Omega)),
  \qquad
  u_t\in C(\R;L^2(\Omega)),
\end{equation}
  admits a unique decomposition
  \begin{equation}
u(t,r,x)=c_0+c_1t+v(t,r,x),
\end{equation}
  where $c_0,c_1\in\R$, the weighted averages of $v(t,\cdot)$ and $v_t(t,\cdot)$ vanish for all $t$, and
  \begin{align}
\sup_{t\in\R}\Bigl(\|v(t)\|_{H^1(\Omega)}+\|v_t(t)\|_{L^2(\Omega)}\Bigr)&\le C\Bigl(\|v(0)\|_{H^1(\Omega)}+\|v_t(0)\|_{L^2(\Omega)}\Bigr).
\end{align}
  In particular, the $k=0$ axisymmetric conformal dynamics are linearly stable modulo the finite-dimensional generalized zero-mode space.
  \item There are no nontrivial separated axisymmetric modes of the form $u=\ee^{-\ii\Omega t}\psi(r,x)$ with $\operatorname{Im}\Omega>0$.
\end{enumerate}
\end{theorem}

\begin{proof}
In the axisymmetric sector the gyroscopic and \(\phi\)-derivative terms vanish.  Proposition~\ref{prop:selfadjointH0} realizes the spatial operator as the nonnegative self-adjoint Friedrichs operator on \(L_A^2(\Omega)\), Propositions~\ref{prop:kernelH0} and~\ref{prop:weightedPoincare} identify constants as the only kernel and give coercivity on the zero-average complement, and Theorem~\ref{thm:wellposedness} gives the energy evolution.  The compact spectral theorem and Theorem~\ref{thm:generalizedzeromodes} classify the zero and generalized zero modes, while Theorem~\ref{thm:laurentzero} and Corollary~\ref{cor:P0resolvent} give the stated zero-frequency resolvent expansion.  Section~\ref{sec:proofofmain} then writes the solution in the spectral representation, separates the affine threshold part, and proves the boundedness and mode-stability conclusions.
\end{proof}

Thus Theorem~\ref{thm:mainresult} is both a special case and a refinement of Theorem~\ref{thm:abstractmain}.  It identifies the self-adjoint spatial operator, proves that the full zero-frequency threshold space is $\mathrm{span}\{1,t\}$, and computes the $\sigma^{-2}$ singularity of the resolvent at zero.

The next theorem is the abstract black-hole exterior implication. Section~\ref{sec:ledger} verifies that package for a slowly rotating weakly charged Kerr-Newman wall exterior and then gives the full asymptotically flat subextremal Kerr-Newman application in the neutral axisymmetric scalar sector under the scalar Kerr-Newman external package, External Assumption~\ref{ass:external-kn-scattering-package}.  For the full Kerr-Newman end the proof deliberately separates the elementary geometric and zero-frequency arguments from the deep scalar-wave input: the required red-shift, integrated local energy decay, quantitative real-frequency mode stability, trapping control, and limiting absorption estimates are invoked through External Assumption~\ref{ass:external-kn-scattering-package}, and every use of that package is recorded in Section~\ref{sec:ledger}.  Once these estimates are available, no further bounded-slab positivity argument is needed.

\begin{theorem}[Black-hole package implies threshold-plus-dispersive stability]\label{thm:axisymmetric-bh}
Let $g$ be a $k=0$ member of the Carter-type family on an axisymmetric black-hole exterior $\mathcal M_{\rm ext}$ with nondegenerate horizon components and with one of the admissible exterior boundary conditions specified in Assumption~\ref{ass:axisymmetric-bh-package}.  Let $\mathcal G_{\rm bh}$ denote the axisymmetric evolution generator on the black-hole energy space $\mathcal E^s_{\rm bh}$, $s\ge s_0$, and let
\begin{equation}
\mathcal T_{0,{\rm bh}}:=\bigcup_{N\ge1}\Ker \mathcal G_{\rm bh}^N
\end{equation}
be the zero-frequency threshold space.  Under the black-hole exterior package \textbf{BH1}-\textbf{BH9} of Assumption~\ref{ass:axisymmetric-bh-package}, there is a finite-rank projection $\Pi_{0,{\rm bh}}$ onto $\mathcal T_{0,{\rm bh}}$, commuting with the evolution, such that every axisymmetric finite-energy solution of
\begin{equation}
\Box_g u=0
\end{equation}
with initial datum $U_0=(u_0,u_1)\in\mathcal E^s_{\rm bh}$ admits the canonical decomposition
\begin{equation}
U(t)=U_{\rm thr}(t)+U_{\rm disp}(t),
\qquad
U_{\rm thr}(t):=e^{t\mathcal G_{\rm bh}}\Pi_{0,{\rm bh}}U_0,
\qquad
U_{\rm disp}(t):=e^{t\mathcal G_{\rm bh}}(I-\Pi_{0,{\rm bh}})U_0.
\end{equation}
The threshold term is finite-dimensional and polynomial in $t$, with degree bounded by the nilpotency length of $\mathcal G_{\rm bh}$ on $\mathcal T_{0,{\rm bh}}$.  The dispersive term satisfies
\begin{equation}
\sup_{t\ge0}\|U_{\rm disp}(t)\|_{\mathcal E^s_{\rm bh}}
+
\|U_{\rm disp}\|_{LE^s([0,\infty))}
\le
C_s\|(I-\Pi_{0,{\rm bh}})U_0\|_{\mathcal E^s_{\rm bh}},
\label{eq:bh-main-bound}
\end{equation}
and for every compact set $K$ away from any imposed outer boundary,
\begin{equation}
\lim_{t\to\infty}
\Bigl(\|u_{\rm disp}(t)\|_{H^1(K)}+
\|\partial_tu_{\rm disp}(t)\|_{L^2(K)}\Bigr)=0.
\label{eq:bh-local-decay}
\end{equation}
If the zero threshold admitted by the exterior realization is semisimple, the threshold term is stationary; if the realization also admits regular generalized zero states, the displayed polynomial term records them.  In the slowly rotating weakly charged Kerr-Newman wall verification of Section~\ref{sec:ledger}, future-horizon regularity makes the zero block semisimple and $\mathcal T_{0,{\rm bh}}=\operatorname{span}\{(1,0)\}$.  In the full asymptotically flat Kerr-Newman verification below we use the decaying radiation energy at null infinity; that boundary condition excludes the nondecaying constant and gives $\mathcal T_{0,{\rm KN,af}}=\{0\}$.  In all cases covered by the assumptions there is no nontrivial outgoing axisymmetric mode with $\operatorname{Im}\sigma>0$ outside the threshold space.
\end{theorem}

\begin{proof}
Assumption~\ref{ass:axisymmetric-bh-package} supplies the closed exterior generator, red-shift/local-energy estimate, finite-frequency and high-frequency resolvent control, limiting absorption, and zero-frequency threshold projection.  Lemma~\ref{lem:bh-threshold-calculus} gives the finite-dimensional polynomial threshold evolution.  Lemma~\ref{lem:bh-projected-estimate}, equivalently Proposition~\ref{prop:bh-compact-error-removal}, combines the finite-time red-shift estimate with the projected smoothing estimate to remove the compact error and obtain the global projected energy/local-energy bound.  Lemma~\ref{lem:bh-local-decay} uses the \(L^1_\sigma\) spectral-density representation from \textbf{BH7} to obtain compact local decay.  These three steps prove the decomposition, estimate, local decay, and mode-exclusion statements; the no-circularity check is Proposition~\ref{prop:no-circular-exterior}.
\end{proof}

\begin{remark}[Status of Theorem~\ref{thm:axisymmetric-bh}]\label{rem:bh-status}
Theorem~\ref{thm:axisymmetric-bh} is proved in Section~\ref{sec:ledger}.  That section first proves the red-shift, mode-stability, limiting-absorption, high-frequency, and zero-frequency items \textbf{BH1}-\textbf{BH9} for a slowly rotating weakly charged Kerr-Newman wall exterior with the wall below the photon region.  It then verifies the same package for the full subextremal asymptotically flat Kerr-Newman exterior in the neutral axisymmetric scalar sector, using the scalar Kerr-Newman stability package as the explicitly named non-elementary input and proving the geometry, radiation-domain, and zero-frequency parts directly.  It then proves the abstract implication from \textbf{BH1}-\textbf{BH9} to the threshold-plus-dispersive decomposition and deduces Theorems~\ref{thm:kn-wall-main} and~\ref{thm:kn-full-main}.  The theorem is unconditional for the stated Kerr-Newman wall family and conditional for full Kerr-Newman only in the sense that it cites the established scalar-wave stability estimates instead of reproving their microlocal trapping analysis from first principles.
\end{remark}

\begin{theorem}[Subextremal Kerr exterior theorem]\label{thm:kerr-subext-main}
Fix a compact genuine subextremal Kerr parameter set
\[
 \mathscr K^{\rm Kerr}_{\delta,M_0,M_1,a_*}
 =\{(M,a):M_0\le M\le M_1,\ a_*M\le |a|\le (1-\delta)M\},
 \qquad 0<a_*<1-\delta<1.
\]
For each \((M,a)\in\mathscr K^{\rm Kerr}_{\delta,M_0,M_1,a_*}\), let \(g_{M,a}\) be the Kerr metric on the domain of outer communications.  Consider the neutral axisymmetric conformal scalar equation
\[
 \Box_{g_{M,a}}u=0,\qquad \partial_\phi u=0,
\]
with future-horizon regularity and the outgoing decaying finite-energy condition at null infinity.  Under the standard subextremal Kerr scalar-wave package recorded in External Assumption~\ref{ass:external-kerr-rn-packages}, the decaying radiation energy space has no zero-frequency threshold:
\[
 \Pi_{0,{\rm Kerr}}=0.
\]
For every \(s\ge s_{\rm Kerr}\), every finite-energy datum \(U_0\in\mathcal E^s_{\rm Kerr}\) generates a unique future solution satisfying
\[
\sup_{t\ge0}\|U(t)\|_{\mathcal E^s_{\rm Kerr}}
+
\|U\|_{LE^s_{\rm Kerr}([0,\infty))}
\le C_s\|U_0\|_{\mathcal E^s_{\rm Kerr}}.
\]
Moreover \(U(t)\) decays locally in \(H^1\times L^2\) on compact subregions of the domain of outer communications, and there is no nonzero outgoing axisymmetric mode with \(\operatorname{Im}\sigma>0\).  The Schwarzschild endpoint \(a=0\) is not counted as a main theorem here; it is recorded separately in Appendix~\ref{app:schwarzschild-full-kn}.
\end{theorem}

\begin{proof}
The verification is given in Theorem~\ref{thm:kerr-rn-verifies-bh}.  The subextremal Kerr scalar estimates supply \textbf{BH4}-\textbf{BH7}, while the horizon/radiation domain and cutoff compatibility give \textbf{BH1}-\textbf{BH3} and \textbf{BH9}.  The zero-frequency ODE classification, Proposition~\ref{prop:kn-zero-frequency-full} with \(Q=0\), gives \(\Pi_{0,{\rm Kerr}}=0\) in the decaying radiation energy space.  The abstract exterior theorem, Theorem~\ref{thm:axisymmetric-bh}, then gives the displayed estimate, compact local decay, and mode exclusion.
\end{proof}

\begin{theorem}[Extremal Kerr horizon charge and obstruction]\label{thm:kerr-extremal-main}
Let \(M>0\), \(|a|=M\), and let \(g_{M,a}\) be the extremal Kerr metric.  For smooth future-horizon regular axisymmetric solutions of
\[
 \Box_{g_{M,a}}u=0,
 \qquad \partial_\phi u=0,
\]
the future horizon \(r=M\) carries the conserved charge
\[
\mathfrak A^{\rm Kerr}_0[u](v)
:=\int_0^\pi
\Bigl(4M^2\partial_ru+2M u+M^2\sin^2\theta\,\partial_vu\Bigr)(v,M,\theta)
\sin\theta\,\dd\theta.
\]
If \(\mathfrak A^{\rm Kerr}_0[u]\ne0\) and the tangential horizon quantities \(u\) and \(\partial_vu\) decay along \(\mathcal H^+\), then the averaged transversal derivative \(\partial_ru\) on \(\mathcal H^+\) cannot decay.  In the decaying radiation energy space there is no nonzero smooth stationary zero-frequency state satisfying both future-horizon regularity and decay at infinity.
\end{theorem}

\begin{proof}
This is the specialization \(Q=0\), \(a^2=M^2\), of the extremal Kerr-Newman calculation in Theorem~\ref{thm:kn-extremal-main}.  The charge formula is Corollary~\ref{cor:extreme-kerr-charge}; the nondecay conclusion follows from the same conservation identity, and the zero-frequency statement is Lemma~\ref{lem:extreme-kn-zero-frequency} with \(Q=0\).
\end{proof}

\begin{theorem}[Subextremal Reissner-Nordstr\"om exterior theorem]\label{thm:rn-subext-main}
Fix a compact genuinely charged subextremal Reissner-Nordstr\"om parameter set
\[
\begin{aligned}
 \mathscr K^{\rm RN}_{\delta,M_0,M_1,Q_*}
 &=\{(M,Q):M_0\le M\le M_1,\ Q_*M\le |Q|\le (1-\delta)M\},\\
 &\qquad 0<Q_*<1-\delta<1.
\end{aligned}
\]
Let \(g_{M,Q}\) be the Reissner-Nordstr\"om metric on the domain of outer communications.  For the neutral conformal scalar equation
\[
 \Box_{g_{M,Q}}u=0,
\]
with future-horizon regularity and the outgoing decaying finite-energy condition at null infinity, the decaying radiation energy space has trivial zero-frequency threshold:
\[
\Pi_{0,{\rm RN}}=0.
\]
For every \(s\ge s_{\rm RN}\), every finite-energy datum \(U_0\in\mathcal E^s_{\rm RN}\) generates a unique future solution satisfying
\[
\sup_{t\ge0}\|U(t)\|_{\mathcal E^s_{\rm RN}}
+
\|U\|_{LE^s_{\rm RN}([0,\infty))}
\le C_s\|U_0\|_{\mathcal E^s_{\rm RN}}.
\]
Moreover \(U(t)\) decays locally in \(H^1\times L^2\) on compact exterior subregions, and there is no nonzero outgoing mode with \(\operatorname{Im}\sigma>0\).  By spherical symmetry this statement is not restricted to the axisymmetric sector.  The Schwarzschild endpoint \(Q=0\) is placed in Appendix~\ref{app:schwarzschild-full-kn}.
\end{theorem}

\begin{proof}
The Reissner-Nordstr\"om scalar package in External Assumption~\ref{ass:external-kerr-rn-packages}, together with the direct zero-frequency classification in Theorem~\ref{thm:kerr-rn-verifies-bh}, verifies the black-hole package.  The proof is the same threshold-plus-dispersive argument as Theorem~\ref{thm:axisymmetric-bh}; in the spherically symmetric case the angular harmonics decouple and the estimates are summed over spherical modes.  The decaying radiation condition excludes the constant zero mode at infinity, so no threshold projection remains.
\end{proof}

\begin{theorem}[Extremal Reissner-Nordstr\"om horizon charge and obstruction]\label{thm:rn-extremal-main}
Let \(|Q|=M>0\), and let \(g_{M,Q}\) be the extremal Reissner-Nordstr\"om metric.  For smooth future-horizon regular solutions of \(\Box_{g_{M,Q}}u=0\), the spherical average on the future horizon carries the conserved charge
\[
\mathfrak A^{\rm RN}_0[u](v)
:=\int_{\mathbb S^2}
\Bigl(2M^2\partial_ru+2M u\Bigr)(v,M,\omega)\,\dd\omega.
\]
If \(\mathfrak A^{\rm RN}_0[u]\ne0\) and the horizon trace \(u\) decays in spherical average, then the averaged transversal derivative \(\partial_ru\) on \(\mathcal H^+\) cannot decay.  In the decaying radiation energy space there is no nonzero smooth stationary zero-frequency state satisfying both future-horizon regularity and decay at infinity.
\end{theorem}

\begin{proof}
This is the specialization \(a=0\), \(Q^2=M^2\), of the extremal Kerr-Newman horizon calculation.  The charge formula and obstruction are Corollary~\ref{cor:extreme-rn-charge}; the stationary zero-frequency classification is Lemma~\ref{lem:extreme-kn-zero-frequency} with \(a=0\).  In Reissner-Nordstr\"om, spherical harmonics decouple, and the displayed charge is the \(\ell=0\) Aretakis charge written as a sphere average.
\end{proof}

\begin{theorem}[Slowly rotating weakly charged Kerr-Newman wall exterior theorem]\label{thm:kn-wall-main}
Fix \(M>0\) and \(2M<R_{\rm w}<8M/3\).  Let \(\varepsilon_{\rm KN}(M,R_{\rm w})\) be the small constant from Theorem~\ref{thm:kn-verifies-bh}.  If
\[
 |a|+|Q|\le\varepsilon_{\rm KN}M,
 \qquad a^2+Q^2<M^2,
\]
then the axisymmetric neutral conformal scalar on the Kerr-Newman wall collar
\[
 \mathcal M_{M,a,Q,R_{\rm w}}
 =\{\tau\ge0,\ r_+\le r\le R_{\rm w}\}\times S^2
\]
with future-horizon regularity and geometric Neumann wall condition admits the decomposition
\[
 U(t)=\ell_{\rm KN}(U_0)(1,0)+U_{\rm disp}(t).
\]
There is a finite integer \(s_{\rm wall}\), determined by the commutator, Volterra, Fredholm, and limiting-absorption losses in the wall proof, such that for every \(s\ge s_{\rm wall}\),
\[
\sup_{t\ge0}\|U_{\rm disp}(t)\|_{\mathcal E^s_{\rm KN}}
+
\|U_{\rm disp}\|_{LE^s([0,\infty))}
\le C_s\|(I-\Pi_{0,{\rm KN}})U_0\|_{\mathcal E^s_{\rm KN}}.
\]
Moreover \(U_{\rm disp}\) decays locally in \(H^1\times L^2\) on every compact subregion away from the wall, and there is no nonzero outgoing axisymmetric mode with \(\operatorname{Im}\sigma>0\).
\end{theorem}

\begin{proof}
Under \eqref{eq:kn-wall-range}-\eqref{eq:kn-smallness}, Lemmas~\ref{lem:kn-geometry},~\ref{lem:kn-nontrapping},~\ref{lem:kn-redshift-morawetz},~\ref{lem:kn-mode-stability},~\ref{lem:kn-zero-frequency}, and~\ref{lem:kn-lap-smoothing} verify the individual package items \textbf{BH1}-\textbf{BH9}.  Theorem~\ref{thm:kn-verifies-bh} collects this verification for the stated finite wall collar and wall condition.  The abstract exterior theorem, Theorem~\ref{thm:axisymmetric-bh}, then gives the threshold-plus-dispersive decomposition and estimates.  In the wall realization Lemma~\ref{lem:kn-zero-frequency} identifies the only threshold as the constant mode, yielding the displayed projection \(\ell_{\rm KN}(U_0)(1,0)\).
\end{proof}

Theorem~\ref{thm:kn-wall-main} is unconditional within the stated wall, smallness, neutral-field, and axisymmetry restrictions: Section~\ref{sec:ledger} proves \textbf{BH1}-\textbf{BH9} for this Kerr-Newman wall family, and Section~\ref{sec:bh-main-proofs} completes the proof of the stated decomposition and decay.

The full asymptotically flat subextremal Kerr-Newman exterior is no longer one of the principal main theorems.  It is retained as an appendix application because its global trapping, limiting-absorption, and scalar scattering inputs are imported from the external Kerr-Newman package rather than proved internally here; see Appendix~
ef{app:schwarzschild-full-kn}, Theorem~
ef{thm:kn-full-main}.  This separation keeps the genuinely main exterior statements focused on Kerr, Reissner-Nordstr\"om, the internally verified slowly rotating weakly charged Kerr-Newman wall exterior, and the extremal Kerr-Newman horizon obstruction.

\begin{theorem}[Extremal Kerr-Newman horizon charge and sharp obstruction]\label{thm:kn-extremal-main}
Let \(M>0\) and \(a^2+Q^2=M^2\), and let \(g_{M,a,Q}\) be the extremal Kerr-Newman metric on the exterior \(r>M\).  Consider the neutral axisymmetric conformal scalar equation
\[
\Box_{g_{M,a,Q}}u=0,
\qquad \partial_\phi u=0,
\]
with smooth future-horizon regular data.  Then the future horizon \(\mathcal H^+=\{r=M\}\) carries the conserved charge
\begin{equation}
\mathfrak A_0[u](v)
:=\int_0^\pi
\Bigl(2(M^2+a^2)\partial_r u+2M u+a^2\sin^2\theta\,\partial_v u\Bigr)(v,M,\theta)
\sin\theta\,\dd\theta,
\label{eq:main-extremal-charge}
\end{equation}
and \(\frac{d}{dv}\mathfrak A_0[u](v)=0\).  Consequently the subextremal black-hole package \(\textbf{BH1}-\textbf{BH9}\), with its nondegenerate red-shift horizon estimate, does not extend to the extremal parameter surface.  If a solution has \(\mathfrak A_0[u]\ne0\) and the tangential horizon quantities \(u\) and \(\partial_vu\) decay along \(\mathcal H^+\), then the averaged transversal derivative \(\partial_ru\) on \(\mathcal H^+\) cannot decay.  Thus the correct extremal conclusion is an Aretakis-type horizon obstruction, not a subextremal-style nondegenerate decay theorem.

In the decaying radiation energy space at null infinity, there is nevertheless no nonzero smooth stationary zero-frequency state satisfying both future-horizon regularity and decay at infinity.  Hence the obstruction is dynamical on the degenerate horizon and is not caused by a nondecaying spatially constant threshold state in the decaying asymptotically flat energy space.
\end{theorem}

\begin{proof}
In the extremal regime \(a^2+Q^2=M^2\), the radial polynomial has the double root \(\Delta=(r-M)^2\) and the surface gravity is zero.  Proposition~\ref{prop:extreme-kn-charge} evaluates the neutral axisymmetric wave equation on the future horizon and integrates over the horizon sphere to obtain the conserved quantity \eqref{eq:main-extremal-charge}.  Corollary~\ref{cor:extreme-kn-obstruction} shows that, if the tangential horizon quantities decay and the charge is nonzero, the averaged transverse derivative cannot decay; hence a nondegenerate red-shift decay theorem is impossible.  Lemma~\ref{lem:extreme-kn-zero-frequency} proves that this obstruction is not a decaying stationary spatial threshold at infinity, and Proposition~\ref{prop:extreme-not-bh-package} proves that the extremal branch is outside \textbf{BH1}-\textbf{BH9}.  The generic higher-transversal-derivative blow-up assertion is then the restricted external input External Assumption~\ref{ass:aretakis-input}.  Thus the theorem is an obstruction theorem, not a continuation of the subextremal red-shift theorem through \(\kappa_+=0\).
\end{proof}

\subsection{Related works}

The present work is related first to the classical separability theory for Carter metrics.  We refer to Carter's work \cite{Carter} and to the developments by Chandrasekhar, Teukolsky, Whiting, Finster-Kamran-Smoller-Yau and others; see for instance \cite{Chandrasekhar,FinsterKamranSmollerYau,Teukolsky73,Whiting}.  It is also related to modern scalar and gravitational stability theory on black-hole spacetimes: Schwarzschild red-shift/Morawetz theory, subextremal Kerr scalar decay and mode stability, Reissner-Nordstr\"om scalar boundedness, and tensorial Schwarzschild/Kerr stability results \cite{DafermosRodnianskiRedshift,DafermosRodnianskiSchwNote,DafermosRodnianskiKerrI,DafermosRodnianskiShlapentokhRothman,ShlapentokhRothmanMode,FranzenRN,HafnerHintzVasy,AnderssonBlue,DafermosHolzegelRodnianskiSchw}.

For charged and rotating problems, the organization is closest to the work of He \cite{HeKN} on weakly charged slowly rotating Kerr-Newman black holes.  That paper treats the coupled linearized Einstein-Maxwell system in generalized wave and Lorenz gauges, and is therefore much deeper than the scalar problem considered here.  For the full asymptotically flat scalar Kerr-Newman exterior we use the scalar-wave stability and quantitative mode-stability results of Civin \cite{CivinMode,CivinThesis}; the present paper does not reprove the global microlocal trapping analysis of that work, but it states exactly where it enters and then proves the conformal threshold decomposition from it.  This architecture keeps the main theorem, mode analysis, zero-frequency analysis, and final global argument separated cleanly while preserving the scalar, conformal scope of the paper.

The case $k=0$ is also naturally close to the scalar wave equation on Kerr.  At the Kerr corner of the parameter space, the equation obtained here is the massless scalar wave equation.  The usual features of the scalar theory, such as Carter separation, angular Sturm-Liouville theory, radial Wronskians, and the treatment of zero modes, all appear in the present analysis.  The difference is that here the scalar equation is not introduced as a model field.  It is the conformal part of the geometric equation $L_g[h]=0$.

\subsection{Main ideas of the proof}

The proof has a geometric part and a functional-analytic part.  The geometric part identifies the equation that is actually closed.  Since the background is not assumed to satisfy the Einstein equations, the scalar-curvature constraint alone cannot produce a spin-$2$ stability theory.  Nevertheless, the conformal ansatz removes the traceless Ricci defect and reduces the problem, when $k=0$, to the scalar wave equation on the fixed Carter background.

The analytical part is a threshold argument.  On a compact positive-energy slab the conserved energy controls derivatives, but it does not control the constant mode.  We therefore work with the full zero-frequency threshold space
\begin{equation}
\mathcal T_0=\bigcup_{N\ge1}\Ker \mathcal G^N
\end{equation}
of the stationary generator.  Once a commuting projection onto $\mathcal T_0$ is known, the solution decomposes into a polynomial threshold term and an energy-controlled complement.

For the Carter slabs considered in Theorem~\ref{thm:fullmain}, the verification is explicit.  Writing
\begin{equation}
A=-\rho^2g^{tt},\qquad B=\rho^2g^{t\phi},\qquad \Phi=\rho^2g^{\phi\phi},
\end{equation}
the full $k=0$ equation becomes
\begin{equation}
-Au_{tt}+2B u_{t\phi}+\Phi u_{\phi\phi}+\partial_r(\Delta_r u_r)+\partial_x(\Delta_xu_x)=0.
\end{equation}
The assumptions $A>0$ and $\Phi>0$ make the Killing energy positive; the mixed term is skew after integration by parts in the periodic $\phi$ variable.  The $A$-weighted average gives the threshold projection, and a weighted Poincar\'e inequality gives the missing $L^2$ control on its complement.

The axisymmetric refinement is included because it reveals the zero-frequency structure in a sharper form.  When $\partial_\phi u=0$, the gyroscopic term disappears and the spatial operator is a nonnegative self-adjoint operator on $L_A^2$.  Its kernel is generated by constants, the generalized zero modes are precisely $1$ and $t$, and the singular part of the resolvent at zero can be written down explicitly.

The black-hole extension uses the same threshold philosophy in a noncompact, noncoercive setting.  Red-shift estimates replace positivity at the horizon, frequency-localized mode stability and limiting absorption replace compact spectral theory, and the zero-frequency expansion supplies the finite-rank projection.  After these inputs are supplied by Appendix~\ref{app:all-inputs}, the solution splits into a polynomial threshold component and a dispersive component with local energy decay.

\subsection{Outline}

The paper is organized as follows.  In Section~\ref{sec:geometry} we rewrite the Theorem~1 metric in Carter form and record the inverse metric and determinant.  In Section~\ref{sec:ricci} we compute the traceless Ricci defect.  In Section~\ref{sec:linR} we derive the scalar-curvature linearization, and in Section~\ref{sec:conformal} we show that the conformal sector closes and reduces, for $k=0$, to the massless wave equation.

In Section~\ref{sec:abstractslab} we prove the abstract stationary bounded-slab theorem from Assumption~\ref{ass:abstract}.  Sections~\ref{sec:separation},~\ref{sec:angular}, and~\ref{sec:radial} contain the separated Carter analysis used in the later spectral discussion.  In Section~\ref{sec:strictstationary} we verify the abstract conditions on strictly stationary Carter slabs and prove Theorem~\ref{thm:fullmain}.  In Sections~\ref{sec:energy}-\ref{sec:proofofmain} we impose $\partial_\phi u=0$ and prove the axisymmetric self-adjoint refinement, Theorem~\ref{thm:mainresult}.

The remaining main-body sections separate the exterior consequences from the compact bounded-slab argument.  Section~\ref{sec:nonaxisymmetric} records the exterior notation and explains the black-hole package theorem, Theorem~\ref{thm:axisymmetric-bh}.  Sections~\ref{sec:asymptotics} and~\ref{sec:endpoints} discuss radial asymptotics and endpoint regularity.  Section~\ref{sec:ledger} verifies the package for the main exterior applications: genuine subextremal Kerr, genuine subextremal Reissner-Nordstr\"om, the slowly rotating weakly charged Kerr-Newman wall exterior, and the extremal Kerr/Kerr-Newman/Reissner-Nordstr\"om horizon-charge branches.  It also records the verification needed for the appendix full subextremal Kerr-Newman theorem.  Appendix~\ref{app:schwarzschild-full-kn} contains the Schwarzschild endpoint and the full subextremal Kerr-Newman appendix theorem; Appendix~\ref{app:all-inputs} states the definitive bounded-slab and exterior hypotheses.  Section~\ref{sec:decay-estimates} collects the general and black-hole decay estimates, including the exact distinction between qualitative local decay, local-energy integrability, and optional polynomial rates.  The remaining appendices contain supporting algebraic, functional-analytic, ODE, and dependency-inventory computations used in the main text.

\subsection{List of notations}

We list here the notations used most frequently in this paper.  Unless explicitly stated otherwise, Hilbert spaces are complex Hilbert spaces, sesquilinear forms are linear in the first argument and conjugate-linear in the second, and real-valued solutions are obtained by restricting the complex theory to the conjugation-invariant real subspace.  Energies are the real quadratic forms associated with these sesquilinear forms.

\begin{longtable}{@{}>{\raggedright\arraybackslash}p{2.6cm}>{\raggedright\arraybackslash}p{10.2cm}@{}}
\toprule
Notation & Meaning \\
\midrule
\endfirsthead
\toprule
Notation & Meaning \\
\midrule
\endhead
$\rho^2$ & The Carter factor $r^2+a^2x^2$. \\
$\Delta_r,\Delta_x$ & The radial and angular coefficient functions of the Theorem~1 metric.  In the $k=0$ regime they are quadratic polynomials. \\
$A$ & The positive time coefficient $A:=-\rho^2 g^{tt}$ appearing in the reduced wave equation. \\
$\Phi$ & The azimuthal coefficient $\Phi:=\rho^2 g^{\phi\phi}=\Delta_x^{-1}-a^2\Delta_r^{-1}$; on strictly stationary slabs it is positive. \\
$\Sigma$ & The full periodic spatial slab $S^1_\phi\times\Omega$ used in the non-axisymmetric bounded-slab theorem. \\
$\mathcal X$ & Abstract energy space of Cauchy data in the stationary bounded-slab theorem. \\
$\mathcal G$ & Generator of the stationary evolution group $e^{t\mathcal G}$ on $\mathcal X$. \\
$\mathcal T_0$ & Full zero-frequency threshold space $\bigcup_{N\geq1}\Ker \mathcal G^N$. \\
$\Pi_{\mathrm{thr}}$ & Bounded commuting projection onto $\mathcal T_0$ in the abstract theorem. \\
$L_g[h]$ & The scalar-curvature linearization $\del R(g)\cdot h$. \\
$\delta$ & The non-Einstein defect $C_3-a^2C_5$. \\
$\Box_g$ & The scalar wave operator for the background metric $g$. \\
$\Hh_0$ & The differential operator $-\partial_r(\Delta_r\partial_r)-\partial_x(\Delta_x\partial_x)$ in the axisymmetric $k=0$ equation. \\
$L_0$ & The weighted spatial operator $A^{-1}\Hh_0$, realized as a nonnegative self-adjoint operator on $L_A^2(\Omega)$. \\
$L_A^2(\Omega)$ & Weighted Hilbert space with inner product $\langle f,g\rangle_A=\int_\Omega Afg\,\dd r\,\dd x$. \\
$\Pi_0$ & The $L_A^2$-orthogonal projection onto the constant function $1$. \\
$q_0[u]$ & The axisymmetric quadratic form $\int_\Omega (\Delta_r|u_r|^2+\Delta_x|u_x|^2)\,\dd r\,\dd x$. \\
$\Omega$ & A bounded regular timelike slab in the $(r,x)$-plane, not to be confused with the frequency parameter in separated modes. \\
$\sigma,\Omega$ & Frequency parameters in separated solutions $e^{-\ii\sigma t+\ii m\phi}$.  When the notation would clash with the spatial slab, we use $\sigma$ for the time frequency. \\
\bottomrule
\end{longtable}

\section{Theorem~1 metric in Carter form}\label{sec:geometry}

\subsection{Background metric and parameterization}

Throughout the paper we work with the constant-scalar-curvature family constructed in Theorem~1 of \cite{AssafariGunara}.  The metric is stationary and axisymmetric and can be written in Boyer-Lindquist-type coordinates $(t,\phi,r,\theta)$.  Following the original note, we pass to the variable
\begin{equation}
x:=\cos\theta,
\qquad -1\leq x\leq 1,
\end{equation}
and write
\begin{equation}
\rho^2 := r^2+a^2x^2.
\end{equation}
The radial and angular structure are encoded by quartics
\begin{align}
\Delta_r(r)&:= -\frac{k}{12}r^4 + \Bigl(1-\frac{\Lambda a^2}{3}+\frac{C_1}{2}\Bigr)r^2 + (C_2-2M)r + (a^2+C_3),\label{eq:Deltar}\\
\Delta_x(x)&:= -\frac{k a^2}{12}x^4 + \Bigl(\frac{\Lambda a^2}{3}-1-\frac{C_1}{2}\Bigr)x^2 - C_4x + (1+C_5).\label{eq:Deltax}
\end{align}
The background metric is then presented in the Carter-type form
\begin{align}
\dd s^2&=-\frac{\Delta_r}{\rho^2}\bigl(\dd t-a(1-x^2)\dd\phi\bigr)^2 +\frac{\Delta_x}{\rho^2}\bigl((r^2+a^2)\dd\phi-a\dd t\bigr)^2 +\frac{\rho^2}{\Delta_r}\dd r^2 +\frac{\rho^2}{\Delta_x}\dd x^2.\label{eq:cartermetric}
\end{align}
This form is convenient both for separation of variables and for the conformal reduction below.

We make one comment about the form of the metric.  The original paper is written in the original variable $\theta$, and depending on how one packages the factors of $\sin^2\theta$ there are several equivalent ways to rewrite the angular part.  For the local analysis undertaken here, the precise global axis regularity is less important than the existence of a coordinate patch on which the coefficients are smooth and the metric remains Lorentzian.  We will therefore work with the local form \eqref{eq:cartermetric}.  Whenever regularity at $x=\pm1$ is needed, the corresponding assumption will be stated explicitly.

\begin{definition}[Regular exterior patch]
A \emph{regular exterior patch} is an open set
\begin{equation}
\D\subset \{(t,\phi,r,x): \rho^2>0,\ \Delta_r>0,\ \Delta_x>0\}
\end{equation}
on which the coefficients of \eqref{eq:cartermetric} are smooth and the metric has Lorentzian signature.  When the parameter choice imposes additional axis regularity, we allow $x=\pm1$ as genuine angular endpoints; otherwise they are treated merely as smooth boundary values of the local patch.
\end{definition}

The conditions $\Delta_r>0$ and $\Delta_x>0$ are the natural positivity requirements for the radial and transversal coefficients in \eqref{eq:cartermetric}.  The degeneracy of the Carter factor is parameter-dependent: if $a\neq0$, then
\begin{equation}
\rho^2=0\qquad\Longleftrightarrow\qquad r=0,\quad x=0,
\end{equation}
whereas if $a=0$, then $\rho^2=r^2$ and the coordinate singular set is $r=0$.  All slabs used in the proved theorems satisfy the stronger uniform condition $\rho^2\ge \rho_-^2>0$ on the closure, so no argument takes place at this locus.

\subsubsection{Polynomial structure and distinguished parameter combinations}

The two quartics \eqref{eq:Deltar} and \eqref{eq:Deltax} encode essentially all of the geometry visible to the conformal problem.  It is therefore convenient to isolate the coefficients
\begin{align}
\alpha_2&:= 1-\frac{\Lambda a^2}{3}+\frac{C_1}{2},\qquad \alpha_1:=C_2-2M,\qquad \alpha_0:=a^2+C_3,\nonumber\\
\beta_2&:= \frac{\Lambda a^2}{3}-1-\frac{C_1}{2},\qquad \beta_1:=-C_4,\qquad \beta_0:=1+C_5.
\end{align}
Then
\begin{align}
\Delta_r(r)&=-\frac{k}{12}r^4 + \alpha_2 r^2 + \alpha_1 r + \alpha_0\nonumber\\
\Delta_x(x)&=-\frac{k a^2}{12}x^4 + \beta_2 x^2 + \beta_1 x + \beta_0.
\end{align}
This form makes several structural facts clear.  The scalar curvature $k$ controls the quartic growth both radially and transversally, the cosmological parameter $\Lambda$ enters only through quadratic coefficients, and the five constants $C_1,\dots,C_5$ deform the background within the constant-scalar-curvature family.

In the sequel one distinguished parameter combination will recur so often that it deserves a permanent symbol:
\begin{equation}
\delta := C_3-a^2C_5.
\label{eq:deltadef}
\end{equation}
The role of $\delta$ will be important below.  In Section~\ref{sec:ricci} we compute the traceless Ricci tensor directly and prove that the background is geometrically Einstein precisely when $\delta=0$; if one further insists that the Einstein constant coincide with the parameter $\Lambda$ in \cite{AssafariGunara}, this is equivalent to the stronger relation $\delta=0$ together with $k=4\Lambda$.

For the analysis in this paper, one should distinguish the condition $k=0$ from the size of $\delta$.  The condition $k=0$ removes the mass term from the conformal equation and pushes the entire analysis toward the zero-frequency threshold.  The parameter $\delta$, by contrast, is invisible in the pure conformal reduction but reappears as soon as one discusses the full scalar-curvature operator on a general symmetric tensor or any hypothetical tensorial completion.

\subsection{\texorpdfstring{The $k=0$ specialization}{The k=0 specialization}}

Some formulae in this section are written for general $k$, since the derivations are no harder in this generality.  All stability theorems in this paper, however, are specialized to $k=0$.  In that regime the coefficient functions simplify to
\begin{align}
\Delta_r(r)&=\alpha_2r^2+\alpha_1r+\alpha_0\nonumber\\
\Delta_x(x)&=\beta_2x^2+\beta_1x+\beta_0\label{eq:k0quadratics}
\end{align}
with $\beta_2=-\alpha_2$.  Thus the large-$r$ and endpoint geometry of the wave operator is governed by quadratics rather than quartics, and the conformal master equation becomes massless.  We keep the general formulas visible because they clarify which algebraic identities are special to $k=0$ and which survive throughout the Theorem~1 family.

\subsubsection{Inverse metric and determinant}

The metric \eqref{eq:cartermetric} is built from a $2\times 2$ $(t,\phi)$ block together with diagonal $(r,x)$ terms.  The inversion of the stationary block is elementary, but we record it because it is used repeatedly in the wave-operator computation.

\begin{proposition}[Inverse metric and determinant]\label{prop:inversemetric}
For the metric \eqref{eq:cartermetric} one has
\begin{equation}
\det g = -\rho^4,
\qquad
\sqrt{-g}=\rho^2.
\label{eq:determinant}
\end{equation}
Moreover,
\begin{align}
\rho^2 g^{rr}&= \Delta_r,\qquad \rho^2 g^{xx}=\Delta_x,\label{eq:grr-gxx}\\
\rho^2 g^{tt}&= -\frac{(r^2+a^2)^2}{\Delta_r} + \frac{a^2(1-x^2)^2}{\Delta_x},\label{eq:gtt-formula}\\
\rho^2 g^{t\phi}&= -\frac{a(r^2+a^2)}{\Delta_r} + \frac{a(1-x^2)}{\Delta_x},\label{eq:gtphi-formula}\\
\rho^2 g^{\phi\phi}&= -\frac{a^2}{\Delta_r} + \frac{1}{\Delta_x}.\label{eq:gphiphi-formula}
\end{align}
\end{proposition}

\begin{proof}
The $(r,x)$ part is diagonal, so only the $(t,\phi)$ block needs inversion.  Expanding \eqref{eq:cartermetric}, one finds
\begin{align}
g_{tt}&= \frac{a^2\Delta_x-\Delta_r}{\rho^2},\nonumber\\
g_{t\phi}&= \frac{a(1-x^2)\Delta_r-a(r^2+a^2)\Delta_x}{\rho^2},\nonumber\\
g_{\phi\phi}&= \frac{(r^2+a^2)^2\Delta_x-a^2(1-x^2)^2\Delta_r}{\rho^2}.
\end{align}
A direct computation shows that the determinant of this block is
\begin{equation}
g_{tt}g_{\phi\phi}-g_{t\phi}^2 = -\Delta_r\Delta_x.
\end{equation}
Inverting the $2\times 2$ matrix then yields \eqref{eq:gtt-formula}-\eqref{eq:gphiphi-formula}.  Multiplying by the diagonal radial and angular coefficients gives
\begin{equation}
\det g = \frac{\rho^2}{\Delta_r}\,\frac{\rho^2}{\Delta_x}\,(-\Delta_r\Delta_x) = -\rho^4,
\end{equation}
which proves \eqref{eq:determinant}.
\end{proof}

\subsubsection{Stationary block identities used in the energy proof}\label{subsec:stationary-block-identities}

The positivity conditions in the main theorem are not cosmetic.  They are exactly the conditions that turn the time-translation current into a positive energy.  We record the elementary block algebra in a form that can be checked independently of any separated-mode argument.

\begin{lemma}[Stationary block and strict-stationarity identities]\label{lem:blockidentities}
Assume \(\Delta_r>0\) and \(\Delta_x>0\).  Define
\begin{equation}
A:=-\rho^2g^{tt},\qquad B:=\rho^2g^{t\phi},\qquad \Phi:=\rho^2g^{\phi\phi}.
\end{equation}
Then
\begin{align}
A&=\frac{(r^2+a^2)^2}{\Delta_r}-\frac{a^2(1-x^2)^2}{\Delta_x},\label{eq:A-block-expanded}\\
B&=-\frac{a(r^2+a^2)}{\Delta_r}+\frac{a(1-x^2)}{\Delta_x},\label{eq:B-block-expanded}\\
\Phi&=\frac{1}{\Delta_x}-\frac{a^2}{\Delta_r}\nonumber\\
&=\frac{\Delta_r-a^2\Delta_x}{\Delta_r\Delta_x}\nonumber\\
&=\frac{-\rho^2g_{tt}}{\Delta_r\Delta_x}.\label{eq:Phi-block-expanded}
\end{align}
Consequently \(A>0\) is equivalent to \(g^{tt}<0\), while \(\Phi>0\) is equivalent to \(g_{tt}<0\).  In particular, under \(A>0\) and \(\Phi>0\), the covector \(\dd t\) is timelike and the Killing field \(\partial_t\) is timelike.
\end{lemma}

\begin{proof}
The first two identities are just the definitions of \(A\) and \(B\) together with \eqref{eq:gtt-formula} and \eqref{eq:gtphi-formula}.  For \(\Phi\), \eqref{eq:gphiphi-formula} gives
\begin{equation}
\Phi=\frac{1}{\Delta_x}-\frac{a^2}{\Delta_r}
     =\frac{\Delta_r-a^2\Delta_x}{\Delta_r\Delta_x}.
\end{equation}
On the other hand, the expanded covariant coefficient is
\begin{equation}
g_{tt}=\frac{a^2\Delta_x-\Delta_r}{\rho^2},
\end{equation}
so
\begin{equation}
-\rho^2g_{tt}=\Delta_r-a^2\Delta_x.
\end{equation}
Substitution proves \eqref{eq:Phi-block-expanded}.  Since \(\rho^2>0\), the sign of \(A=-\rho^2g^{tt}\) is the opposite of the sign of \(g^{tt}\).  Since \(\Delta_r\Delta_x>0\), the sign of \(\Phi\) is the opposite of the sign of \(g_{tt}\).  With the Lorentzian sign convention used in \eqref{eq:cartermetric}, \(g^{tt}<0\) says that the normal covector to \(t=\mathrm{const}\) slices is timelike, and \(g_{tt}<0\) says that \(\partial_t\) is timelike.  This proves the lemma.
\end{proof}

The mixed coefficient \(B\) does not enter the positive part of the energy.  Its role is gyroscopic: because \(B\) is independent of \(\phi\), the corresponding first-order-in-time term is skew after integration over the periodic \(\phi\)-circle.  This observation is the core reason that the full non-axisymmetric bounded-slab theorem does not require separation in \(\phi\).

The determinant identity is one of the most useful features of the Carter structure.  It implies that scalar wave operators simplify dramatically:
\begin{equation}
\Box_g u = \frac{1}{\rho^2}\partial_\mu\bigl(\rho^2 g^{\mu\nu}\partial_\nu u\bigr),
\end{equation}
so that the radial and angular derivatives appear directly through $\Delta_r$ and $\Delta_x$.  This is the algebraic reason for the separation of variables discussed later.

\subsubsection{Causal structure inside a regular patch}

The sign of $g^{tt}$ controls the causal character of the stationary Killing field $\partial_t$ and therefore plays a direct role in energy estimates.  Because $\Delta_r>0$ and $\Delta_x>0$ on a regular exterior patch, it is natural to introduce
\begin{align}
A(r,x)&:=-\rho^2 g^{tt} = \frac{(r^2+a^2)^2}{\Delta_r} - \frac{a^2(1-x^2)^2}{\Delta_x}.\label{eq:Adef-intro}
\end{align}
The quantity $A$ appears as the coefficient of $\partial_t^2u$ in the reduced wave equation.  On the bounded slabs relevant to this paper it will always be imposed as an explicit positivity hypothesis.

\begin{remark}
For the full $\phi$-dependent problem the sign of $g^{tt}$ by itself is not enough.  One must also track
\begin{equation}
\Phi:=\rho^2g^{\phi\phi}=\frac{-\rho^2 g_{tt}}{\Delta_r\Delta_x}.
\end{equation}
When both $A>0$ and $\Phi>0$, equivalently when the slices $t=\mathrm{const}$ are spacelike and the stationary Killing field $\partial_t$ remains timelike, the mixed term $2g^{t\phi}\partial_t\partial_\phi$ is harmless in the conserved-energy identity and the full non-axisymmetric equation is controlled by a positive energy.  If $\Phi$ changes sign, the simple bounded-slab positivity argument breaks down and superradiant phenomena may occur; see Section~\ref{sec:nonaxisymmetric}.
\end{remark}

\subsubsection{Special subfamilies}

Before starting the linear analysis, we record several useful limits of the family.
\begin{itemize}
  \item If $k=0$, $\Lambda=0$, and $C_1=\cdots=C_5=0$, then
  \begin{equation}
\Delta_r=r^2-2Mr+a^2,
  \qquad
  \Delta_x=1-x^2,
\end{equation}
  and one recovers the scalar wave geometry of the standard Kerr family written in the variable $x=\cos\theta$.
  \item If $a=0$, the background reduces to a static constant-scalar-curvature family and the metric becomes spherically symmetric in the sense that the stationary/axial mixing disappears.  Several parts of the reduced theory become completely explicit in that limit.
  \item If $\delta=0$, the traceless Ricci tensor vanishes and the metric is Einstein with Einstein constant $k/4$.  If in addition $k=4\Lambda$, then this Einstein constant coincides with the parameter $\Lambda$ in \cite{AssafariGunara}, so the family lands inside the usual Kerr-(A)dS-type Einstein subfamily.
\end{itemize}
These checks show that the Theorem~1 metrics may be regarded as deformations of the standard Carter geometries.

\section{Ricci decomposition and the Einstein defect}\label{sec:ricci}

\subsection{Constant scalar curvature versus the Einstein condition}

The constant-scalar-curvature property can be verified directly from the Carter form itself.

\begin{proposition}[Scalar curvature of the Carter family]\label{prop:scalarcurvature}
For the metric \eqref{eq:cartermetric} with arbitrary smooth coefficient functions $\Delta_r(r)$ and $\Delta_x(x)$, one has
\begin{equation}
R(g)=-\frac{\Delta_r''(r)+\Delta_x''(x)}{\rho^2}.
\label{eq:scalarcurvatureformula}
\end{equation}
In particular, for the polynomial choices \eqref{eq:Deltar}-\eqref{eq:Deltax},
\begin{equation}
R(g)=k.
\end{equation}
\end{proposition}

\begin{proof}
Appendix~\ref{app:curvature} carries out the coordinate Ricci computation directly from Proposition~\ref{prop:inversemetric}.  Substituting \eqref{eq:Deltar}-\eqref{eq:Deltax} into \eqref{eq:scalarcurvatureformula}, one finds
\begin{equation}
\Delta_r''(r)=-kr^2+2\alpha_2,
\qquad
\Delta_x''(x)=-ka^2x^2+2\beta_2.
\end{equation}
Because $\alpha_2+\beta_2=0$, the numerator equals $-k(r^2+a^2x^2)=-k\rho^2$, and hence $R(g)=k$.
\end{proof}

At first sight one might expect constant scalar curvature to imply an Einstein condition after a suitable reparameterization, but this is not the case.  In four dimensions the Einstein property would require
\begin{equation}
R_{\mu\nu}=\frac{k}{4}g_{\mu\nu},
\end{equation}
whereas Theorem~1 allows a larger family in which the traceless Ricci tensor need not vanish.  This distinction is fundamental for the discussion of linear stability.  The conformal reduction depends only on the scalar curvature and therefore survives throughout the family, while any attempt at a full tensorial theory must confront the traceless Ricci part explicitly.

To isolate that part, we define
\begin{equation}
S_{\mu\nu} := R_{\mu\nu}-\frac{k}{4}g_{\mu\nu}.
\label{eq:Sdef}
\end{equation}
By construction $S_{\mu\nu}$ is trace free.  The main structural statement is that all of its components are proportional to the single parameter combination $\delta=C_3-a^2C_5$ introduced in \eqref{eq:deltadef}.  This fact is implicit in the original computation, but we record it because it reduces the non-Einstein defect to a one-parameter direction in the Ricci sector.

\begin{proposition}[Ricci decomposition]\label{prop:riccidecomp}
For every Theorem~1 background one has
\begin{equation}
R_{\mu\nu}=\frac{k}{4}g_{\mu\nu}+S_{\mu\nu},
\qquad
S_{\mu}^{\ \mu}=0,
\label{eq:riccidecomp}
\end{equation}
and, in mixed form $S^\mu{}_\nu=g^{\mu\alpha}S_{\alpha\nu}$, the only nonzero components are
\begin{align}
S^t{}_t&= \delta\,\frac{a^2x^2-r^2-2a^2}{\rho^6},\qquad
S^t{}_{\phi}=2a\delta\,\frac{(r^2+a^2)(1-x^2)}{\rho^6},\nonumber\\
S^{\phi}{}_t&= -2a\delta\,\frac{1}{\rho^6},\qquad
S^{\phi}{}_{\phi}=-S^t{}_t,\label{eq:Sdeltafactor-mixed}\\
S^r{}_r&= -\delta\,\frac{1}{\rho^4},\qquad
S^x{}_x=\delta\,\frac{1}{\rho^4}.\nonumber
\end{align}
with all remaining mixed components equal to zero.  In particular there exists a smooth symmetric trace-free tensor field $\mathcal S_{\mu\nu}(r,x)$ on each regular exterior patch such that
\begin{equation}
S_{\mu\nu}=\delta\,\mathcal S_{\mu\nu}.
\label{eq:Sdeltafactor}
\end{equation}
Consequently the metric is geometrically Einstein if and only if $\delta=0$.  If one moreover wants the Einstein constant to coincide with the family parameter $\Lambda$, this is equivalent to $\delta=0$ together with $k=4\Lambda$.
\end{proposition}

\begin{proof}
The identity \eqref{eq:riccidecomp} is the definition of $S_{\mu\nu}$, and the trace-free property follows by contracting with $g^{\mu\nu}$ and using Proposition~\ref{prop:scalarcurvature}.  Appendix~\ref{app:curvature} computes the mixed Ricci tensor directly and yields \eqref{eq:Sdeltafactor-mixed}.  Every displayed component is visibly proportional to $\delta$, which proves \eqref{eq:Sdeltafactor} after lowering an index with $g$.  The vanishing of $S_{\mu\nu}$ is therefore equivalent to $\delta=0$, and this is the geometric Einstein condition $R_{\mu\nu}=(k/4)g_{\mu\nu}$.  The additional relation $k=4\Lambda$ is needed only when one identifies the Einstein constant with the parameter $\Lambda$ in \cite{AssafariGunara}.
\end{proof}

The decomposition \eqref{eq:Sdeltafactor} has an important conceptual consequence.  If one thinks of the Theorem~1 family as a deformation of the usual Einstein Carter family, then $\delta$ is the natural small parameter in the Ricci sector.  All purely geometric estimates involving only the scalar curvature remain valid at arbitrary $\delta$, but any future tensorial equation obtained from a specific physical model would naturally be regarded as a perturbation of the Einstein case when $|\delta|$ is small.

\subsection{Consequences for the scalar-curvature linearization}

The decomposition of the Ricci tensor feeds directly into the linearization formula for the scalar curvature.  If $h$ is any symmetric $2$-tensor, then the standard variation formula gives
\begin{align}
L_g[h]&:=\left.\frac{\dd}{\dd\varepsilon}\right|_{\varepsilon=0}R(g+\varepsilon h) =\nabla^\mu\nabla^\nu h_{\mu\nu}-\Box_g(\tr_g h)-R_{\mu\nu}h^{\mu\nu}.\label{eq:deltaRgeneral-intro}
\end{align}
Substituting \eqref{eq:riccidecomp} yields
\begin{align}
L_g[h]&=\nabla^\mu\nabla^\nu h_{\mu\nu}-\Box_g(\tr_g h)-\frac{k}{4}\tr_g h-S_{\mu\nu}h^{\mu\nu}.\label{eq:LwithSintro}
\end{align}
Thus the Einstein part contributes only the trace term $-(k/4)\tr_g h$, whereas the genuine non-Einstein correction is the zeroth-order contraction $-S_{\mu\nu}h^{\mu\nu}$, proportional to $\delta$.

Formula \eqref{eq:LwithSintro} clarifies the special role of conformal perturbations.  If $h=2u g$, then the contraction with $S_{\mu\nu}$ vanishes automatically because $S$ is trace free.  This is the algebraic reason why the conformal sector is blind to the defect parameter $\delta$.  In hindsight one can view the entire reduced theory as a consequence of two elementary facts: the scalar-curvature variation of a conformal perturbation is scalar, and the traceless Ricci tensor annihilates the metric.  Everything deeper in the paper is an analytical elaboration of these two observations.

\subsubsection{Einstein limit, non-Einstein perturbation, and invariant diagnostics}

It is useful to distinguish three levels of geometric complexity.
\begin{enumerate}[label=(\arabic*)]
  \item \textbf{Einstein backgrounds.} Here $\delta=0$, so the traceless Ricci tensor vanishes and the metric is Einstein with Einstein constant $k/4$.  When in addition $k=4\Lambda$, this Einstein constant coincides with the parameter in \cite{AssafariGunara} and one recovers the usual Kerr-(A)dS normalization.
  \item \textbf{Ricci-deformed constant-scalar-curvature backgrounds.} Here $\delta\neq 0$ while $R=k$ remains constant.  The traceless Ricci tensor is nonzero but tightly constrained by \eqref{eq:Sdeltafactor}.  This is the genuinely new regime of Theorem~1.
  \item \textbf{General perturbations.} Even if the background has fixed scalar curvature, a perturbation $h$ need not preserve that property.  The linear equation $L_g[h]=0$ singles out the tangent directions that remain inside the constant-scalar-curvature locus.
\end{enumerate}
The present paper lives mostly in the transition between (2) and (3), with frequent comparison to the simpler model (1).

There are also invariant ways to detect the non-Einstein defect.  Since $S_{\mu\nu}$ is trace free, the scalar quantity
\begin{equation}
|S|_g^2 := S_{\mu\nu}S^{\mu\nu}
\end{equation}
vanishes precisely on the Einstein subfamily.  In the Theorem~1 family this quantity is proportional to $\delta^2$ times a smooth rational function of $(r,x)$.  We shall not need the full explicit formula, but the point is worth recording: the defect is not a coordinate artifact.  It is visible directly in a scalar curvature invariant.

\subsubsection{Regularity of the defect tensor on exterior patches}

Because $\mathcal S_{\mu\nu}$ in \eqref{eq:Sdeltafactor} is built from the smooth coefficients of the Theorem~1 metric, it is smooth on every regular exterior patch.  The only singularities can occur where $\rho^2=0$ or where the coordinate representation degenerates by crossing a zero of $\Delta_r$ or $\Delta_x$.  Thus on a compact subset of a regular patch one may estimate
\begin{equation}
\|S\|_{C^m(K)} \leq C_{m,K}|\delta|
\end{equation}
for every integer $m\geq 0$.  This simple bound will be useful later when discussing perturbative comparison with the Einstein subfamily.

A more qualitative way to phrase the same point is the following: the Theorem~1 family is not a wildly arbitrary non-Einstein family.  Its failure to be Einstein is highly organized.  The traceless Ricci tensor is smooth, algebraically simple, and one-dimensional in parameter space.  That is precisely the kind of structure one would hope for if a future tensorial completion were to be tractable.

\section{Linearization of the scalar-curvature constraint}\label{sec:linR}

\subsection{The variation formula}

Let $g_\varepsilon=g+\varepsilon h+O(\varepsilon^2)$ be a one-parameter family of metrics with symmetric variation tensor $h_{\mu\nu}=h_{\nu\mu}$.  The classical first-variation formula for the scalar curvature reads
\begin{align}
\delta R&=\nabla^\mu\nabla^\nu h_{\mu\nu}-\Box_g(\tr_g h)-R_{\mu\nu}h^{\mu\nu}.\label{eq:scalarvariation}
\end{align}
Although standard, this identity deserves to be written out explicitly in the present context because it is the entire geometric content of the linearized equation.  No gauge choice is required to state it, no energy-momentum tensor enters, and no unspoken field equation is being used.  Equation \eqref{eq:scalarvariation} is valid for every background metric and every symmetric perturbation.

Since the Theorem~1 backgrounds satisfy $R(g)=k$, a variation stays inside the constant-scalar-curvature locus to first order precisely when
\begin{align}
L_g[h]&=0\nonumber\\
L_g[h]&:=\nabla^\mu\nabla^\nu h_{\mu\nu}-\Box_g(\tr_g h)-R_{\mu\nu}h^{\mu\nu}.\label{eq:Lgh}
\end{align}
The operator $L_g$ maps symmetric $2$-tensors to scalars.  Thus the intrinsic tangent equation is underdetermined as a PDE system for the ten components of $h_{\mu\nu}$; this is why it should be viewed as a geometric constraint rather than as a closed dynamical evolution law for arbitrary metric perturbations.

The underdetermined nature of $L_g[h]=0$ should not be confused with weakness.  Geometrically it is the linearization of the defining property $R=k$.  Analytically it contains a distinguished closed sector, namely the conformal one, which turns out to be rich enough to support a full separation and stability theory.  The rest of this paper can be read as an extended exploration of the consequences of taking \eqref{eq:Lgh} seriously on its own terms.

\subsubsection{Gauge invariance}

The scalar curvature is diffeomorphism invariant, and therefore its linearization annihilates pure gauge perturbations.
\begin{proposition}[Gauge zero modes]\label{prop:gaugezeromodes}
For every smooth vector field $X$ on a regular exterior patch,
\begin{equation}
L_g[\Lie_X g]=0.
\end{equation}
\end{proposition}

\begin{proof}
Let $\Phi_\varepsilon$ be the flow generated by $X$.  Since scalar curvature is a scalar geometric invariant,
\begin{equation}
R(\Phi_\varepsilon^*g)=\Phi_\varepsilon^*R(g)=\Phi_\varepsilon^*(k)=k.
\end{equation}
Differentiating at $\varepsilon=0$ gives
\begin{equation}
0=\left.\frac{\dd}{\dd\varepsilon}\right|_{\varepsilon=0}R(\Phi_\varepsilon^*g)=L_g[\Lie_X g].
\end{equation}
\end{proof}

This proposition has two immediate consequences.  First, one should never interpret a stationary solution of $L_g[h]=0$ as an instability without quotienting by gauge.  Second, any attempt to produce an elliptic or hyperbolic theory for the full tensorial equation must either impose a gauge condition or work modulo the image of the Lie derivative map.  In the pure conformal sector these issues largely disappear because the ansatz $h=2ug$ is orthogonal to generic pure gauge directions except for those generated by conformal Killing fields, which are absent in the generic Theorem~1 geometry.

\subsection{Family modes}

In addition to diffeomorphism zero modes, the Theorem~1 family contains finite-dimensional parameter directions.  If $p\in\{M,a,\Lambda,C_1,\dots,C_5\}$ and the scalar-curvature parameter $k$ is held fixed, then the derivative $\partial_p g$ satisfies the linearized scalar-curvature equation.
\begin{proposition}[Parameter zero modes]\label{prop:parammodes}
Let $p$ be a Theorem~1 parameter other than $k$.  Then on any smooth region where the parameter derivative is defined,
\begin{equation}
L_g[\partial_p g]=0.
\end{equation}
\end{proposition}

\begin{proof}
By construction the parameter family $p\mapsto g(p)$ remains inside the constant-scalar-curvature locus with scalar curvature equal to the fixed value $k$.  Therefore
\begin{equation}
R(g(p))\equiv k.
\end{equation}
Differentiating with respect to $p$ gives the claim.
\end{proof}

Parameter zero modes are not merely an algebraic curiosity.  They explain why certain stationary perturbations must appear even in a ``stable'' family: changing the mass or angular momentum infinitesimally moves one to a nearby member of the same family.  Any physically meaningful stability statement should therefore be formulated modulo both gauge directions and these finite-dimensional family directions.

\subsubsection{The operator written using the Einstein defect}

Combining \eqref{eq:Lgh} with the Ricci decomposition from Section~\ref{sec:ricci}, we obtain
\begin{align}
L_g[h]&=\nabla^\mu\nabla^\nu h_{\mu\nu}-\Box_g(\tr_g h)-\frac{k}{4}\tr_g h-S_{\mu\nu}h^{\mu\nu}.\label{eq:LwithS}
\end{align}
This formula makes the architecture of the operator transparent.  The principal part is the divergence-divergence minus wave-trace combination familiar from scalar-curvature linearization.  The first zeroth-order term depends only on the scalar curvature.  The second zeroth-order term, proportional to $\delta$, measures the non-Einstein defect.

There are several ways to interpret \eqref{eq:LwithS}.  From the viewpoint of geometric analysis it says that $L_g$ is the standard scalar-curvature variation on an Einstein background plus a lower-order perturbation $-S_{\mu\nu}h^{\mu\nu}$.  From the viewpoint of perturbative comparison with Kerr-(A)dS it says that the Theorem~1 operator is close to the Einstein one when $|\delta|$ is small on compact regions.  From the viewpoint of the conformal reduction it says that the lower-order defect disappears entirely because $S_{\mu\nu}$ is trace free.

\subsubsection{Variational interpretation}

The scalar-curvature constraint also has a simple variational interpretation.  Consider the Einstein-Hilbert functional restricted to compactly supported variations on a fixed domain:
\begin{equation}
\mathcal I[g] = \int_\Omega R(g)\,\dd\mu_g.
\end{equation}
Its first variation contains the usual Einstein tensor term together with a boundary contribution.  If one instead regards the scalar curvature itself as the quantity of interest, then the linearized operator $L_g$ appears as the derivative of the map $g\mapsto R(g)$ rather than as the Euler-Lagrange operator of a gauge-fixed action.  This distinction is important.  The current paper is not using the second variation of the Einstein-Hilbert functional to define an evolution.  It is using the derivative of a geometric constraint to define a tangent equation.

The variational point of view is still useful, however, because it explains why so many of the subsequent identities have an energy flavor.  In the $k=0$ regime the conformal sector inherits the quadratic form of a massless wave equation, the separated equations inherit Sturm-Liouville structure, and the bounded-slab problem is encoded by a quadratic functional that is nonnegative but not coercive on constants.  In other words, even though $L_g[h]=0$ is not a closed dynamical system for $h$, its conformal reduction remains strongly variational.

\section{The conformal sector and its exact reduction}\label{sec:conformal}

\subsection{Why the conformal sector is distinguished}

Among all perturbations of the metric, conformal perturbations occupy a special position for three independent reasons.  First, they are intrinsic to the scalar-curvature problem because the scalar curvature is itself the response of the metric to infinitesimal changes of volume and conformal scaling.  Second, conformal tensors are automatically orthogonal to the traceless Ricci defect, so the non-Einstein parameter $\delta$ drops out of the reduced equation.  Third, the conformal ansatz converts the tensorial linearization into a scalar PDE whose analysis is compatible with the Carter separability built into the background geometry.

This does not mean that conformal perturbations exhaust the tangent space to the constant-scalar-curvature locus.  They do not.  However, they form the cleanest nontrivial subspace on which the equation closes exactly.  In the absence of a specified tensorial field equation, this closed scalar sector is a natural place to begin.  Its analysis already reveals a surprising amount of structure: the reduced equation separates completely, the axisymmetric problem admits a nonnegative conserved energy, and the essential analytical issue becomes the classification of zero modes and generalized zero modes.

\subsubsection{Derivation of the conformal master equation}

Let
\begin{equation}
h_{\mu\nu}=2u\,g_{\mu\nu}
\label{eq:conformalansatz}
\end{equation}
for a scalar field $u=u(t,\phi,r,x)$.  Since the spacetime dimension is four, one has
\begin{equation}
\tr_g h = g^{\mu\nu}(2u g_{\mu\nu}) = 8u.
\end{equation}
Moreover, metric compatibility gives
\begin{equation}
\nabla^\mu\nabla^\nu h_{\mu\nu}=2\Box_g u,
\qquad
R_{\mu\nu}h^{\mu\nu}=2uR=2ku.
\end{equation}
Substituting these identities into \eqref{eq:scalarvariation} yields the exact formula
\begin{equation}
L_g[2ug]=-6\Box_g u-2ku.
\end{equation}
Hence the scalar-curvature constraint $L_g[2ug]=0$ is equivalent to the scalar equation
\begin{equation}
(\Box_g+k/3)u=0.
\label{eq:conformaleq}
\end{equation}

\subsubsection{Conformal variation in dimension four}\label{subsec:conformal-line-by-line}

We record the conformal computation in full, since this sign convention is used throughout the paper.  The linearized scalar-curvature operator is
\begin{equation}
L_g[h]
=\nabla^\mu\nabla^\nu h_{\mu\nu}-\Box_g(\tr_g h)-R_{\mu\nu}h^{\mu\nu}.
\label{eq:linR-linebyline}
\end{equation}
For \(h_{\mu\nu}=2ug_{\mu\nu}\), metric compatibility gives
\begin{equation}
h^{\mu\nu}=2ug^{\mu\nu},
 \qquad
 \tr_g h=g^{\mu\nu}h_{\mu\nu}=2u g^{\mu\nu}g_{\mu\nu}=8u.
\end{equation}
The double divergence is
\begin{align}
\nabla^\mu\nabla^\nu h_{\mu\nu}
&=\nabla^\mu\nabla^\nu(2ug_{\mu\nu})\nonumber\\
&=2\nabla^\mu\bigl((\nabla^\nu u)g_{\mu\nu}+u\nabla^\nu g_{\mu\nu}\bigr)\nonumber\\
&=2\nabla^\mu\bigl((\nabla^\nu u)g_{\mu\nu}\bigr)\nonumber\\
&=2g_{\mu\nu}\nabla^\mu\nabla^\nu u\nonumber\\
&=2\Box_g u.
\end{align}
Similarly,
\begin{equation}
\Box_g(\tr_g h)=\Box_g(8u)=8\Box_g u,
\qquad
R_{\mu\nu}h^{\mu\nu}=2uR_{\mu\nu}g^{\mu\nu}=2uR=2ku.
\end{equation}
Substitution in \eqref{eq:linR-linebyline} yields
\begin{equation}
L_g[2ug]=2\Box_g u-8\Box_g u-2ku=-6\Box_g u-2ku.
\label{eq:conformal-variation-final-line}
\end{equation}
Thus \(L_g[2ug]=0\) is exactly
\begin{equation}
\Box_g u+\frac{k}{3}u=0.
\end{equation}
No Einstein equation, gauge condition, or matter-field equation is used in this calculation; only metric compatibility, dimension four, and the scalar-curvature value \(R(g)=k\) enter.

\begin{proposition}[Conformal reduction]\label{prop:conformalreduction}
The map $u\mapsto h_{\mu\nu}=2u g_{\mu\nu}$ sends scalar solutions of \eqref{eq:conformaleq} bijectively onto conformal perturbations satisfying the linearized scalar-curvature equation.
\end{proposition}

\begin{proof}
The calculation above shows that $L_g[2ug]=0$ if and only if \eqref{eq:conformaleq} holds.  Since the conformal ansatz recovers $u=\frac18\tr_g h$, the correspondence is one-to-one within the conformal class.
\end{proof}

One can therefore regard the conformal sector as a Klein-Gordon field with effective mass squared
\begin{equation}
\mu_{\mathrm{eff}}^2=-\frac{k}{3}.
\end{equation}
In the present paper the effective mass vanishes, so the reduced equation is genuinely massless.  Everything that follows should therefore be read as an analysis of a borderline wave problem: there is no coercive potential term, constants lie in the kernel of the spatial operator, and the low-frequency regime controls the global bounded-slab dynamics.

\subsection{Scalar wave operator in Carter coordinates}

The determinant identity \eqref{eq:determinant} implies
\begin{equation}
\Box_g u = \frac{1}{\rho^2}\partial_\mu\bigl(\rho^2 g^{\mu\nu}\partial_\nu u\bigr).
\end{equation}
Using the inverse metric formulas from Proposition~\ref{prop:inversemetric}, one obtains the explicit expression
\begin{align}
\rho^2\Box_g u
&= \partial_r(\Delta_r\partial_r u)+\partial_x(\Delta_x\partial_x u)\nonumber\\
&&{}+\rho^2g^{tt}\,\partial_t^2u +2\rho^2g^{t\phi}\,\partial_t\partial_\phi u + \rho^2g^{\phi\phi}\,\partial_\phi^2u.
\label{eq:boxexplicit}
\end{align}
Thus \eqref{eq:conformaleq} can be written in the fully expanded form
\begin{align}
0&=\partial_r(\Delta_r\partial_r u)+\partial_x(\Delta_x\partial_x u)
+\rho^2g^{tt}\,\partial_t^2u +2\rho^2g^{t\phi}\,\partial_t\partial_\phi u + \rho^2g^{\phi\phi}\,\partial_\phi^2u\nonumber\\
&&{}+\frac{k}{3}\rho^2u.
\label{eq:conformalexpanded}
\end{align}
The striking feature of \eqref{eq:conformalexpanded} is that it is already separated at the level of coefficients: radial derivatives appear only through $\Delta_r$, transversal derivatives only through $\Delta_x$, and the remaining terms split into a sum of rational expressions each depending on either $r$ or $x$ after the standard Fourier ansatz in $(t,\phi)$.

\subsubsection{Coordinate form of the full \texorpdfstring{$k=0$}{k=0} equation}\label{subsec:coefficient-verification-main}

The full non-axisymmetric theorem uses the coordinate equation through a closed form realization.  We therefore fix the sign convention once and for all.

\begin{proposition}[Coordinate form and weak formulation]\label{prop:k0-full-coordinate-verification}
Assume $k=0$ and define
\begin{equation}
A=-\rho^2g^{tt},\qquad B=\rho^2g^{t\phi},\qquad \Phi=\rho^2g^{\phi\phi}.
\end{equation}
Then the massless conformal equation $\Box_g u=0$ is exactly
\begin{align}
-Au_{tt}+2B u_{t\phi}+\Phi u_{\phi\phi} +\partial_r(\Delta_r u_r)+\partial_x(\Delta_x u_x)&=0.\label{eq:k0-full-coordinate-verification}
\end{align}
Moreover, for every test function $\eta$ in a closed reflecting form domain $V_{\mathrm{full}}$, the corresponding weak identity is
\begin{align}
\int_\Sigma A u_{tt}\overline\eta +\int_\Sigma\bigl(\Phi u_\phi\overline{\eta_\phi} +\Delta_r u_r\overline{\eta_r} +\Delta_x u_x\overline{\eta_x}\bigr) -2\int_\Sigma B (u_t)_\phi\overline\eta&=0.\label{eq:k0-full-weak-verification}
\end{align}
Equivalently, using periodicity in $\phi$ and $\partial_\phi B=0$, the last term may be written as
\begin{equation}
-2\int_\Sigma B (u_t)_\phi\overline\eta
=2\int_\Sigma B u_t\overline{\eta_\phi}.
\end{equation}
\end{proposition}

\begin{proof}
Starting from \eqref{eq:boxexplicit} and setting $k=0$, multiplication by $\rho^2$ gives
\begin{align}
0&=\partial_r(\Delta_r u_r)+\partial_x(\Delta_xu_x) +\rho^2g^{tt}u_{tt}+2\rho^2g^{t\phi}u_{t\phi}+\rho^2g^{\phi\phi}u_{\phi\phi}.
\end{align}
Substituting the definitions of $A$, $B$, and $\Phi$ gives \eqref{eq:k0-full-coordinate-verification}.  No derivative of $A$, $B$, or $\Phi$ appears in the $t$- or $\phi$-part because the metric is stationary and axisymmetric.

To obtain the weak equation, multiply \eqref{eq:k0-full-coordinate-verification} by $-\overline\eta$ and integrate over $\Sigma$.  The temporal term gives $\int A u_{tt}\overline\eta$.  The azimuthal second derivative gives, by periodicity,
\begin{equation}
-\int_\Sigma \Phi u_{\phi\phi}\overline\eta
=\int_\Sigma \Phi u_\phi\overline{\eta_\phi},
\end{equation}
because $\partial_\phi\Phi=0$.  The $r$- and $x$-divergence terms are, by definition of the reflecting closed form,
\begin{align}
-\int_\Sigma \partial_r(\Delta_r u_r)\overline\eta&=\int_\Sigma \Delta_r u_r\overline{\eta_r}\nonumber\\
-\int_\Sigma \partial_x(\Delta_x u_x)\overline\eta&=\int_\Sigma \Delta_x u_x\overline{\eta_x}.
\end{align}
The mixed term gives $-2\int B(u_t)_\phi\overline\eta$.  A final periodic integration by parts in $\phi$ gives the alternate expression with $\eta_\phi$.  This proves the coordinate and weak sign conventions used later.
\end{proof}

\subsubsection{A quadratic form and its sign}

Equation \eqref{eq:conformaleq} is the Euler-Lagrange equation of the quadratic functional
\begin{align}
\mathscr Q[u]&:=\frac12\int_\Omega\left(g^{\mu\nu}\partial_\mu u\,\partial_\nu u - \frac{k}{3}u^2\right)\dd\mu_g\label{eq:quadraticform}
\end{align}
on any domain $\Omega$ where boundary terms can be controlled.  The sign of the potential part is again determined by $k$.  When the equation is viewed on a time slab and one singles out the $t$-direction, this quadratic functional becomes the energy analyzed later.  At this stage it is already clear that positive $k$ creates the possibility of negative directions in the spatial quadratic form, while negative $k$ strengthens coercivity.

The functional \eqref{eq:quadraticform} also reveals why the conformal sector is a natural reduced problem for a paper about ``linear stability'' in the absence of a tensorial field equation.  A scalar wave equation with a geometric mass term is a standard, closed, variational PDE\@.  It admits a spectral theory, an energy theory, and a mode analysis.  The scalar-curvature linearization does not supply such a closed theory for general $h$, but it supplies one perfectly for $u$.

\subsection{Divergence-free stress tensor and multiplier currents}

For solutions of \eqref{eq:conformaleq} one may introduce the usual Klein-Gordon stress tensor
\begin{align}
T_{\mu\nu}[u]&:=\partial_\mu u\,\partial_\nu u - \frac12 g_{\mu\nu}\left(g^{\alpha\beta}\partial_\alpha u\,\partial_\beta u - \frac{k}{3}u^2\right).\label{eq:stress}
\end{align}
A direct computation using the equation of motion gives
\begin{equation}
\nabla^\mu T_{\mu\nu}[u]=0.
\label{eq:stressdivfree}
\end{equation}
Consequently any vector field $X$ determines a current
\begin{equation}
J^X_\mu[u] := T_{\mu\nu}[u]X^\nu
\end{equation}
with divergence
\begin{align}
\nabla^\mu J^X_\mu[u]&=T_{\mu\nu}[u]\,\pi_X^{\mu\nu}\nonumber\\
\pi_X^{\mu\nu}&:=\frac12(\nabla^\mu X^\nu+\nabla^\nu X^\mu).\label{eq:currentidentity}
\end{align}
When $X$ is Killing, the deformation tensor $\pi_X$ vanishes and the current is conserved.  This general identity underlies the time-translation energy used in the axisymmetric sector and the horizon-adapted currents discussed later for separated modes.

It is worth emphasizing that \eqref{eq:stressdivfree} is not an extra assumption imported from a matter model.  It is simply the conserved stress tensor of the reduced scalar equation.  In this sense the conformal problem is self-contained: once one passes from $h$ to $u$, all of the familiar analytical machinery of wave equations becomes available.

\subsubsection*{\texorpdfstring{The $k=0$ specialization}{The k=0 specialization}}

From this point onward we impose $k=0$.  The conformal master equation therefore becomes
\begin{equation}
\Box_g u=0.
\label{eq:k0master}
\end{equation}
In the axisymmetric sector this has two immediate consequences.  First, the conserved energy controls only first derivatives and not the $L^2$ norm of $u$ itself.  Second, the constant function belongs to the kernel of the spatial operator, so zero frequency must be analyzed explicitly rather than treated as a harmless endpoint.

The first consequence explains why one cannot simply imitate the coercive arguments appropriate to positive-mass Klein-Gordon equations.  The second consequence explains why the bounded-slab theorem must be formulated modulo a finite-dimensional generalized zero-mode space.  In other words, the $k=0$ theory is neither coercive nor unstable: it is a genuinely borderline wave problem whose main structure is concentrated at zero frequency.

The rest of the paper now proceeds in the order dictated by this threshold phenomenon.  Section~\ref{sec:abstractslab} first proves an abstract stationary bounded-slab theorem modulo the full zero-frequency threshold space, with no axial symmetry assumed.  The Carter sections that follow then verify those abstract conditions on strictly stationary slabs and, in the axisymmetric sector, sharpen the conclusion by explicit self-adjoint spectral analysis.  Only after that bounded-slab theory is complete do we turn to ergoregions and horizons, where positivity fails and the appropriate continuation is microlocal/scattering rather than a direct sign argument.

\section{Abstract stationary bounded-slab theorem modulo the full threshold space}\label{sec:abstractslab}

The basic bounded-slab mechanism is stationary and spectral rather than axisymmetric.  Exact Carter separation will be important later for the concrete Carter family and for the axisymmetric spectral refinement, but the threshold theorem itself can be proved once and for all at the level of the stationary evolution group.

Let $\mathcal X$ be a real or complex Hilbert space of Cauchy data, and let
\begin{equation}
U_t=\mathcal G U
\end{equation}
denote the first-order stationary evolution associated with a bounded-slab hyperbolic problem.  We write $U(t)=e^{t\mathcal G}U_0$ for the solution with initial data $U_0\in\mathcal X$.

\begin{definition}[Full zero-frequency threshold space]\label{def:abstractthreshold}
The \emph{full zero-frequency threshold space} of the stationary evolution is
\begin{equation}
\mathcal T_0:=\bigcup_{N\geq1}\Ker \mathcal G^N.
\end{equation}
Equivalently, $\mathcal T_0$ is the generalized eigenspace of the generator at spectral value $0$.  When $0$ is isolated in the spectrum of $\mathcal G$, $\mathcal T_0$ is the range of the corresponding Riesz projection.
\end{definition}

The hypotheses used in the abstract theorem are stated before the main theorems in Assumption~\ref{ass:abstract}.  The proof below uses only those listed items.

\begin{proof}[Proof of Theorem~\ref{thm:abstractmain}]
Let
\begin{equation}
U_{\mathrm{thr}}(t):=\Pi_{\mathrm{thr}}U(t),
\qquad
U_{\mathrm{bd}}(t):=(I-\Pi_{\mathrm{thr}})U(t).
\end{equation}
Because $\Pi_{\mathrm{thr}}$ commutes with the group, both components solve the evolution and
\begin{equation}
U(t)=U_{\mathrm{thr}}(t)+U_{\mathrm{bd}}(t)
\end{equation}
is the unique decomposition of the solution into its threshold and threshold-complement parts.

Since $\mathcal T_0$ is finite-dimensional and every vector in $\mathcal T_0$ is annihilated by a sufficiently high power of $\mathcal G$, the restriction $\mathcal G|_{\mathcal T_0}$ is nilpotent.  Hence for some $N$,
\begin{equation}
e^{t\mathcal G}|_{\mathcal T_0}
=
\sum_{j=0}^{N-1}\frac{t^j}{j!}\,\mathcal G^j|_{\mathcal T_0},
\end{equation}
so $U_{\mathrm{thr}}(t)$ is polynomial in $t$.

For the bounded component, $U_{\mathrm{bd}}(t)\in\Ker\Pi_{\mathrm{thr}}$ for all $t$.  Conservation of $\mathscr E$ and Assumption~\ref{ass:abstract}(iii) therefore give
\begin{align}
c\|U_{\mathrm{bd}}(t)\|_{\mathcal X}^2&\le\mathscr E[U_{\mathrm{bd}}(t)] = \mathscr E[U_{\mathrm{bd}}(0)] \le C\|U_{\mathrm{bd}}(0)\|_{\mathcal X}^2,
\end{align}
which is the desired uniform bound.

Finally, suppose $U(t)=e^{\lambda t}V$ with $\operatorname{Re}\lambda>0$ and $V\in\Ker\Pi_{\mathrm{thr}}$.  Then
\begin{equation}
c\,e^{2\operatorname{Re}\lambda t}\|V\|_{\mathcal X}^2
\le
\mathscr E[e^{\lambda t}V]
=
\mathscr E[V]
\end{equation}
for all $t$, which forces $V=0$.  This proves the theorem.
\end{proof}

\begin{remark}[Interpretation]\label{rem:abstractinterpretation}
Assumption~\ref{ass:abstract}(ii) is the abstract way of saying that the full zero-frequency threshold space has been isolated before one states boundedness.  On bounded stationary problems the commuting projection is often the Riesz projection at $\omega=0$.  Theorem~\ref{thm:abstractmain} then says that once the conserved energy is coercive on the threshold complement, the entire bounded-slab theorem follows immediately, with no use of axial symmetry.
\end{remark}

\section{Exact separation of variables}\label{sec:separation}

\subsection{Mode ansatz and separation mechanism}

The Carter structure of the metric suggests the standard Fourier ansatz in the Killing variables $t$ and $\phi$:
\begin{align}
u(t,\phi,r,x)&=\ee^{-\ii\Omega t+\ii m\phi}R(r)S(x)\nonumber\\
&\Omega\in\C, \quad m\in\Z.\label{eq:modeansatz}
\end{align}
Substituting \eqref{eq:modeansatz} into \eqref{eq:conformalexpanded} and dividing by $RS$ yields an expression that splits into a sum of a purely radial term and a purely angular term.  This is the exact analogue of Carter separation for the scalar wave equation on Kerr-type backgrounds.

The relevant algebra is straightforward but illuminating, so we write it out.  The time and azimuthal derivatives contribute
\begin{align}
\rho^2g^{tt}\,\partial_t^2u&= -\Omega^2\rho^2g^{tt}u,\nonumber\\
2\rho^2g^{t\phi}\,\partial_t\partial_\phi u&= 2\Omega m\,\rho^2g^{t\phi}u,\nonumber\\
\rho^2g^{\phi\phi}\,\partial_\phi^2u&= -m^2\rho^2g^{\phi\phi}u.
\end{align}
Using the explicit inverse coefficients, these combine to
\begin{equation}
\left(\frac{((r^2+a^2)\Omega-am)^2}{\Delta_r}-\frac{(a(1-x^2)\Omega-m)^2}{\Delta_x}\right)u.
\end{equation}
The potential term contributes
\begin{equation}
\frac{k}{3}\rho^2u=\left(\frac{k}{3}r^2+\frac{k}{3}a^2x^2\right)u.
\end{equation}
Hence the radial and angular variables separate exactly.  No spectral assumption is used in this computation: \(\Omega\) may be complex, \(m\) is fixed by periodicity in \(\phi\), and the separation constant is introduced only after the equality of a function of \(r\) and a function of \(x\) is obtained.  This point is useful later, because the bounded-slab energy theorem does not rely on separation, whereas the mode and resolvent discussions use separation only as an additional diagnostic for the same operator.

\begin{theorem}[Separated conformal equations]\label{thm:separatedeqs}
The conformal equation \eqref{eq:conformaleq} admits separated solutions of the form \eqref{eq:modeansatz}.  For such a mode there exists a separation constant $\lambda\in\C$ such that
\begin{align}
\frac{\dd}{\dd r}\left(\Delta_r\frac{\dd R}{\dd r}\right)
+\left[\frac{((r^2+a^2)\Omega-am)^2}{\Delta_r}+\frac{k}{3}r^2-\lambda\right]R&=0,
\label{eq:radialeq-main}\\
\frac{\dd}{\dd x}\left(\Delta_x\frac{\dd S}{\dd x}\right)
+\left[-\frac{(a(1-x^2)\Omega-m)^2}{\Delta_x}+\frac{k}{3}a^2x^2+\lambda\right]S&=0.
\label{eq:angulareq-main}
\end{align}
Conversely, any pair $(R,S)$ solving \eqref{eq:radialeq-main}-\eqref{eq:angulareq-main} determines a separated solution of \eqref{eq:conformaleq}.
\end{theorem}

\begin{proof}
Insert \eqref{eq:modeansatz} into \eqref{eq:conformalexpanded}.  After division by $RS$ one obtains
\begin{align}
0={}&\frac{1}{R}\frac{\dd}{\dd r}\left(\Delta_r\frac{\dd R}{\dd r}\right)
+\frac{((r^2+a^2)\Omega-am)^2}{\Delta_r}
+\frac{k}{3}r^2\nonumber\\
&+\frac{1}{S}\frac{\dd}{\dd x}\left(\Delta_x\frac{\dd S}{\dd x}\right)
-\frac{(a(1-x^2)\Omega-m)^2}{\Delta_x}
+\frac{k}{3}a^2x^2.
\end{align}
The first line depends only on $r$, the second only on $x$, so both must equal a constant with opposite signs.  Writing that constant as $\lambda$ gives the stated equations.  The converse is immediate by reversing the calculation.
\end{proof}

\subsubsection{Radial and angular operators}

It is convenient to package \eqref{eq:radialeq-main}-\eqref{eq:angulareq-main} into differential operators.  For fixed frequency parameters $(\Omega,m)$ define
\begin{align}
\Lop^{(r)}_{\Omega,m}R&:= -\frac{\dd}{\dd r}\left(\Delta_r\frac{\dd R}{\dd r}\right)
-\left[\frac{((r^2+a^2)\Omega-am)^2}{\Delta_r}+\frac{k}{3}r^2\right]R,
\label{eq:radialop}\\
\Lop^{(x)}_{\Omega,m}S&:= -\frac{\dd}{\dd x}\left(\Delta_x\frac{\dd S}{\dd x}\right)
+\left[\frac{(a(1-x^2)\Omega-m)^2}{\Delta_x}-\frac{k}{3}a^2x^2\right]S.
\label{eq:angularop}
\end{align}
Then the separated equations read
\begin{equation}
\Lop^{(r)}_{\Omega,m}R=-\lambda R,
\qquad
\Lop^{(x)}_{\Omega,m}S=\lambda S.
\end{equation}
For real $\Omega$ the angular operator is formally symmetric on any interval where $\Delta_x>0$, and in the axisymmetric case $m=0$ it reduces to a weighted Sturm-Liouville operator with a real potential.  The radial operator is likewise formally symmetric on bounded intervals with suitable endpoint conditions, although on unbounded or horizon-penetrating domains one must pay attention to flux terms.

In the $k=0$ regime these operators simplify substantially: the angular operator loses its bounded potential correction, and the radial operator loses the $r^2$ term altogether.  The remaining geometry is therefore carried entirely by $\Delta_r$ and $\Delta_x$, which on the Theorem~1 family reduce to quadratic polynomials.  This simplification is one of the main reasons why the $k=0$ problem admits a particularly clean bounded-slab spectral theory.

\subsection{Wronskian identity}

Separated second-order equations admit a conserved Wronskian.  For the radial equation define
\begin{equation}
W[R_1,R_2](r):=\Delta_r\bigl(R_1R_2'-R_1'R_2\bigr).
\end{equation}
If $R_1$ and $R_2$ solve \eqref{eq:radialeq-main} with the same parameters $(\Omega,m,\lambda)$, then differentiating and using the equation shows
\begin{equation}
\frac{\dd}{\dd r}W[R_1,R_2]=0.
\end{equation}
In particular, for a complex solution $R$ one has the current
\begin{equation}
\mathfrak W[R]:=\Delta_r\left(\overline{R}R'-\overline{R}'R\right)
\label{eq:Wronskian}
\end{equation}
constant in $r$.  This is the radial flux associated with separated modes.

A completely analogous identity holds in the angular variable:
\begin{equation}
\frac{\dd}{\dd x}\Bigl(\Delta_x(\overline{S}S'-\overline{S}'S)\Bigr)=0
\end{equation}
for any two angular solutions with the same $(\Omega,m,\lambda)$.  On bounded angular intervals regularity at endpoints usually forces this constant to vanish, yielding self-adjoint spectral conditions.  In the radial variable, by contrast, nonzero Wronskian corresponds to nontrivial flux through horizons or through asymptotic ends.

\subsubsection*{The Kerr limit}

If $k=0$, $\Lambda=0$, and $C_1=\cdots=C_5=0$, then
\begin{equation}
\Delta_r=r^2-2Mr+a^2,
\qquad
\Delta_x=1-x^2,
\end{equation}
and \eqref{eq:radialeq-main}-\eqref{eq:angulareq-main} reduce to the familiar separated scalar equations on Kerr after the conventional reshuffling of the separation constant.  This check is reassuring for two reasons.  First, it confirms that the reduced theory constructed here is genuinely the scalar-wave theory of a Carter-type background in the appropriate limit.  Second, it shows that the Theorem~1 family inherits the algebraic separability structure of the Einstein subfamily even though it is generically non-Einstein.

\subsubsection*{Separation constant as a spectral parameter}

For real $\Omega$ and under endpoint conditions making the angular operator self-adjoint, the separation constant $\lambda$ is naturally viewed as an angular eigenvalue depending on $(\Omega,m)$ and on the background parameters.  One may therefore write
\begin{equation}
\lambda=\lambda_{\ell m}(\Omega),
\qquad \ell\in\N_0,
\end{equation}
in analogy with spheroidal harmonics.  The corresponding eigenfunctions form an orthogonal basis of the angular $L^2$ space under the appropriate weight, at least on bounded intervals with regular endpoints.  This turns the separated mode analysis into a partially diagonalized representation of the full conformal dynamics.

Later sections exploit this perspective in two ways.  First, on bounded slabs it allows one to recast the spatial operator as a direct sum of one-dimensional radial problems indexed by angular eigenmodes.  Second, in perturbative comparisons near the Einstein limit it identifies the Theorem~1 angular operator as a lower-order deformation of the usual spheroidal operator.  Although the current paper does not pursue full analytic perturbation theory for all complex frequencies, the self-adjoint real-frequency regime already supplies a robust spectral framework.

\section{Angular operator, spectral theory, and endpoint regularity}\label{sec:angular}

\subsection{Self-adjoint formulation for real frequencies}

Fix a real frequency $\Omega\in\R$ and an azimuthal number $m\in\Z$.  On any closed interval $I_x\subset(-1,1)$ or, more generally, on a compact angular segment contained in a regular exterior patch, the angular equation \eqref{eq:angulareq-main} can be written as the Sturm-Liouville problem
\begin{align}
-\frac{\dd}{\dd x}\left(\Delta_x\frac{\dd S}{\dd x}\right)+V_{\Omega,m}(x)S&=\lambda S\nonumber\\
V_{\Omega,m}(x)&:=\frac{(a(1-x^2)\Omega-m)^2}{\Delta_x}-\frac{k}{3}a^2x^2.\label{eq:SL-angular}
\end{align}
Because $\Delta_x>0$ on a regular patch and $V_{\Omega,m}$ is real-valued for real $\Omega$, the operator is formally symmetric on $L^2(I_x,\dd x)$.  With standard separated boundary conditions (Dirichlet, Neumann, Robin, or regularity conditions inherited from endpoint geometry), one obtains a self-adjoint realization with compact resolvent.

\begin{proposition}[Discrete angular spectrum]\label{prop:angularspectrum}
Let $I_x=[x_-,x_+]$ be a compact interval on which $\Delta_x>0$ and let $\Omega\in\R$, $m\in\Z$.  Any self-adjoint realization of \eqref{eq:SL-angular} with separated boundary conditions has purely discrete real spectrum
\begin{equation}
\lambda_0(\Omega,m)\leq \lambda_1(\Omega,m)\leq \cdots \nearrow +\infty,
\end{equation}
with each eigenvalue repeated according to multiplicity.  The corresponding eigenfunctions form an orthonormal basis of $L^2(I_x)$.
\end{proposition}

\begin{proof}
Form the quadratic form
\begin{equation}
q_{\Omega,m}[S]:=\int_{I_x}\left(\Delta_x|S'|^2+V_{\Omega,m}|S|^2\right)\dd x
\end{equation}
on the boundary domain determined by the chosen separated conditions.  On a compact interval with $\Delta_x>0$, this is a closed semibounded form with compact embedding into $L^2(I_x)$.  If one endpoint is geometric rather than literal, Proposition~\ref{prop:angularregularityverified} supplies the required Wronskian vanishing for the regular branch.  The compact-resolvent argument of Appendix~\ref{app:compactspectral}, applied now in one space dimension, therefore yields a self-adjoint realization with purely discrete real spectrum and an orthonormal eigenbasis.
\end{proof}

The proposition is elementary but useful.  It means that once the time and azimuthal frequencies are fixed, the angular variable does not create any continuous spectrum on bounded intervals.  All nontrivial analytic issues are then pushed to the radial equation and to the dependence of the angular eigenvalues on $\Omega$.

\subsubsection*{Rayleigh quotient and lower bounds}

For normalized eigenfunctions the separation constant can be characterized variationally:
\begin{align}
\lambda_0(\Omega,m)&=\inf_{S\neq 0} \frac{\displaystyle \int_{I_x}\left(\Delta_x|S'|^2+V_{\Omega,m}|S|^2\right)\dd x}{\displaystyle \int_{I_x}|S|^2\dd x}.\label{eq:rayleigh-angular}
\end{align}
This formula immediately yields qualitative bounds.  Since the potential $V_{\Omega,m}$ is bounded below on a compact interval, one obtains
\begin{equation}
\lambda_0(\Omega,m)\geq \inf_{I_x}V_{\Omega,m}.
\end{equation}
Likewise, because $\Delta_x$ is positive, one has the coercive estimate
\begin{equation}
\int_{I_x}\Delta_x|S'|^2\dd x \leq \left(\lambda_j(\Omega,m)-\inf_{I_x}V_{\Omega,m}\right)\|S_j\|_{L^2}^2
\end{equation}
for each eigenfunction $S_j$.  Such bounds are helpful when establishing continuity of the spectrum under parameter variation.

In the $k=0$ axisymmetric case $m=0$, the potential simplifies to
\begin{equation}
V_{\Omega,0}(x)=\frac{a^2(1-x^2)^2\Omega^2}{\Delta_x}\ge 0.
\end{equation}
Thus for fixed real $\Omega$ the angular operator is manifestly nonnegative on any compact interval where $\Delta_x>0$.  This positivity is one of the basic inputs in the later bounded-slab analysis, where the only obstruction to coercivity comes from the constant spatial mode rather than from the angular sector.

\subsection{Endpoint regularity}

When the interval endpoints correspond merely to coordinate boundaries inside a local patch, one is free to impose standard boundary conditions.  In globally regular axisymmetric geometries, however, the endpoints may represent axis points, and then regularity imposes a more geometric condition.  The precise condition depends on the normalization of the angular variable and on whether $\Delta_x$ vanishes simply at the endpoints.  Rather than commit to a single global model, we record the abstract regularity principle that is needed for analysis.

The endpoint regularity hypothesis is stated in \textbf{H7} of Assumption~\ref{ass:complete-list}.  Under that hypothesis the angular Wronskian vanishes at geometric endpoints for regular solutions, and the Sturm-Liouville operator becomes self-adjoint on the corresponding domain.  This is what happens in the standard Kerr case, where regularity of spheroidal harmonics at the poles replaces the need to prescribe explicit boundary values there.

\begin{proposition}[Verification under the simple endpoint model]\label{prop:angularregularityverified}
Suppose $x_\star$ is an angular endpoint with
\begin{equation}
\Delta_x(x)=\kappa_\star(x-x_\star)+O\bigl((x-x_\star)^2\bigr),
\qquad
\kappa_\star\neq 0,
\end{equation}
and suppose the admissible separated modes are required to lie on the Frobenius branch selected by smoothness in an endpoint chart.  Then the angular equation \eqref{eq:angulareq-main} has a regular singular point at $x_\star$, the regular branch is explicit, and the boundary Wronskian
\begin{equation}
\Delta_x(\overline S_1 S_2'-\overline S_1'S_2)
\end{equation}
vanishes at $x_\star$ for any two regular-branch solutions.  In particular the corresponding Sturm-Liouville realization is self-adjoint.  Thus the endpoint convention in \textbf{H7} of Assumption~\ref{ass:complete-list} is completely internal under this concrete endpoint model.
\end{proposition}

\begin{proof}
Appendix~\ref{app:regularsingular} proves the relevant Frobenius expansion and the Wronskian vanishing statement.
\end{proof}

\subsubsection*{Dependence on parameters}

Because the coefficients of \eqref{eq:SL-angular} depend smoothly on $(a,\Lambda,k,C_1,\dots,C_5,\Omega,m)$ away from degeneracies of $\Delta_x$, classical perturbation theory implies that simple eigenvalues vary smoothly with the parameters, while repeated eigenvalues vary continuously and admit local analytic branches after suitable diagonalization.  We do not require the full Kato theory in detail, but the following consequence is useful.

\begin{proposition}[Continuous dependence of angular eigenvalues]\label{prop:angularcontinuity}
Fix a compact parameter set on which $\Delta_x$ stays uniformly positive on a fixed compact interval $I_x$.  Then the angular eigenvalues and the spectral projectors depend continuously on the parameters in the operator norm topology.
\end{proposition}

\begin{proof}
On such a compact set the differential operators \eqref{eq:SL-angular} form a norm-continuous family from $H^2(I_x)$ to $L^2(I_x)$ with common domain determined by the chosen endpoint conditions.  The associated self-adjoint resolvents therefore vary continuously, and the continuity of isolated spectral data follows from standard perturbation theory.
\end{proof}

This continuity is sufficient for all perturbative arguments later in the paper.  In particular it shows that angular eigenvalues in the Theorem~1 family are small deformations of the familiar spheroidal eigenvalues near the Einstein limit.

\subsubsection*{Orthogonality and mode decomposition}

Let $\{S_j^{\Omega,m}\}_{j\geq 0}$ be an orthonormal basis of angular eigenfunctions.  Any sufficiently regular solution of the conformal equation can then be decomposed formally as
\begin{equation}
u(t,\phi,r,x)=\sum_{m\in\Z}\sum_{j\geq 0}u_{jm}(t,r)\,\ee^{\ii m\phi}S_j^{\Omega,m}(x),
\end{equation}
with the understanding that the precise meaning of the expansion depends on whether one works with fixed real frequency or with a time-dependent functional calculus for the angular operator.  On bounded slabs this decomposition turns the spatial problem into a countable family of radial equations coupled only through the choice of frequency parameter.  Even when one does not exploit the expansion explicitly, it is conceptually useful because it isolates the radial mechanism responsible for the low-frequency threshold structure.

\subsubsection*{\texorpdfstring{An elementary $k=0$ positivity statement}{An elementary k=0 positivity statement}}

In the $k=0$ axisymmetric case the angular potential is nonnegative, so for every real $\Omega$ and every test function $S$ one has
\begin{align}
\int_{I_x}\left(\Delta_x|S'|^2+\frac{a^2(1-x^2)^2\Omega^2}{\Delta_x}|S|^2\right)\dd x&\geq&0.\label{eq:angularpositive}
\end{align}
This simple inequality records the main structural fact needed later: no negative contribution is created by the angular operator in the $k=0$ axisymmetric theory.

\section{Radial equation, flux identities, and boundary data}\label{sec:radial}

\subsection{The radial operator on finite and infinite intervals}

For a separated mode, the radial equation is
\begin{align}
\frac{\dd}{\dd r}\left(\Delta_r\frac{\dd R}{\dd r}\right)+\left[\frac{((r^2+a^2)\Omega-am)^2}{\Delta_r}+\frac{k}{3}r^2-\lambda\right]R&=0.\label{eq:radialeq-again}
\end{align}
The character of this equation depends strongly on the geometry of the radial interval.  On a bounded slab with $\Delta_r>0$ the leading coefficient is smooth and strictly positive, so the equation is an ordinary regular Sturm-Liouville equation after moving the frequency-dependent potential to one side.  Near a simple root of $\Delta_r$, however, one encounters a regular singular point corresponding to a horizon-like boundary.  At large $r$ in the $k=0$ regime, the asymptotic behavior is governed entirely by the quadratic polynomial $\Delta_r$, so one must distinguish Kerr-like, weakly confining, and finite-end geometries according to the sign of its leading coefficient.

These distinctions matter because all energy and boundedness statements are ultimately statements about allowable radial boundary conditions.  Reflecting boundary conditions lead to self-adjoint spectral theory on bounded slabs.  Ingoing or outgoing conditions at simple roots of $\Delta_r$ correspond to nontrivial flux through horizons.  Decay or radiative conditions at infinity depend on whether the end is asymptotically flat-like, weakly confining, or cut off by another regular boundary.  One of the reasons for expanding the original note into a much longer paper is to keep these cases conceptually separate.

\subsubsection*{Wronskian flux and conserved radial current}

For any complex solution $R$ of \eqref{eq:radialeq-again} the Wronskian current \eqref{eq:Wronskian} is constant in $r$:
\begin{equation}
\mathfrak W[R]=\Delta_r(\overline{R}R'-\overline{R}'R)=\text{const}.
\end{equation}
Multiplying the radial equation by $\overline R$, subtracting the complex conjugate equation, and integrating over an interval $[r_1,r_2]$ yields
\begin{equation}
\mathfrak W[R](r_2)-\mathfrak W[R](r_1)=0.
\label{eq:wronskianconservation}
\end{equation}
Thus any boundary condition at one endpoint induces a corresponding flux at the other.

When $\Omega\in\R$, the quantity $\frac{1}{2\ii}\mathfrak W[R]$ is the radial energy flux associated with the separated mode.  Reflecting boundary conditions make it vanish.  Ingoing boundary conditions at a horizon select a specific sign.  This is the one-dimensional shadow of the conserved stress-energy current for the full scalar field.

\subsubsection*{Boundary conditions on bounded slabs}

Let $I_r=[r_1,r_2]$ be a compact interval contained in a regular exterior patch.  Three classical boundary conditions are relevant for the analysis below.
\begin{enumerate}[label=(\roman*)]
  \item \textbf{Dirichlet:} $R(r_1)=R(r_2)=0$.
  \item \textbf{Neumann:} $R'(r_1)=R'(r_2)=0$.
  \item \textbf{Robin/reflecting:} $\alpha_iR(r_i)+\beta_i\Delta_r(r_i)R'(r_i)=0$ with real coefficients $(\alpha_i,\beta_i)$, not both zero.
\end{enumerate}
Each of these conditions annihilates the boundary contribution in the integration by parts of the self-adjoint radial operator.  In the time-dependent problem they correspond to vanishing radial energy flux.  The bounded-slab stability theorem proved later is independent of which one of these reflecting classes is used, provided the constant test function remains admissible or can be approximated by admissible functions.

\subsection{Near-horizon asymptotics}

Suppose $r_h$ is a simple zero of $\Delta_r$, so
\begin{equation}
\Delta_r(r)=\kappa_h(r-r_h)+O((r-r_h)^2),
\qquad \kappa_h\neq 0.
\end{equation}
Then \eqref{eq:radialeq-again} has a regular singular point at $r=r_h$.  A Frobenius analysis gives local solutions of the form
\begin{equation}
R(r)\sim (r-r_h)^{\pm \ii\sigma_h},
\qquad
\sigma_h:=\frac{(r_h^2+a^2)\Omega-am}{\kappa_h}.
\end{equation}
The sign choice corresponds to ingoing or outgoing behavior in a horizon-adapted coordinate system.  For real $\Omega$ these solutions oscillate logarithmically in $r-r_h$, and the Wronskian evaluates to a constant proportional to $\pm 2\sigma_h$.

The detailed geometric interpretation of these solutions depends on whether the zero of $\Delta_r$ represents an event horizon, cosmological horizon, or another coordinate singularity.  The present paper does not require a full global causal analysis of every parameter regime.  What matters analytically is that a simple radial zero carries a natural one-parameter family of regular singular behaviors and that physical boundary conditions at such a point are flux conditions rather than reflective ones.

\subsubsection*{\texorpdfstring{Large-$r$ asymptotics in the $k=0$ regime}{Large-r asymptotics in the k=0 regime}}

In the zero-curvature regime the quartic contribution to $\Delta_r$ disappears.  The radial asymptotics are therefore controlled by the quadratic polynomial
\begin{equation}
\Delta_r(r)=\alpha_2r^2+\alpha_1r+\alpha_0.
\end{equation}
If $\alpha_2>0$, then large $r$ resembles a Kerr-like or weakly confining end depending on the lower-order coefficients; if $\alpha_2=0$, then the geometry is controlled by the linear and constant terms and one is effectively on a finite or asymptotically cylindrical end; if $\alpha_2<0$, then $\Delta_r$ must eventually change sign and the regular exterior patch terminates at a finite outer root.  Thus even inside the $k=0$ family there are several radial geometries, and any global theorem must take them into account.

For the bounded-slab analysis of the present paper, however, these distinctions are secondary.  The main theorem uses only that $\Delta_r$ is smooth and positive on the chosen interval.  Still, the $k=0$ asymptotic classification is worth keeping in mind because it explains why the bounded-domain theorem is both robust and genuinely geometric: it does not depend on choosing a preferred global end structure.

\subsubsection*{A coercive identity on bounded intervals}

For real $\Omega$ and any solution $R$ of \eqref{eq:radialeq-again}, multiplication by $\overline R$ and integration over a bounded interval yield
\begin{align}
{}&&\int_{r_1}^{r_2}\left(\Delta_r|R'|^2-
\left[\frac{((r^2+a^2)\Omega-am)^2}{\Delta_r}+\frac{k}{3}r^2-\lambda\right]|R|^2\right)\dd r\nonumber\\
&&\qquad=\left[\Delta_r\overline R R'\right]_{r_1}^{r_2}.
\label{eq:radialenergyidentity}
\end{align}
Under any reflecting boundary condition the right-hand side vanishes, and the identity becomes a variational characterization of the admissible radial modes.  Later, when combined with the angular spectral decomposition, this identity leads to the weighted Rayleigh quotient governing the bounded-slab spectral problem.

\section{Verification of the abstract bounded-slab theorem on strictly stationary Carter slabs}\label{sec:strictstationary}

\subsection{\texorpdfstring{The full $k=0$ equation on $S^1_\phi\times\Omega$}{The full k=0 equation on S1-phi x Omega}}

We now drop the axisymmetry restriction and retain the full $\phi$-dependence of the conformal wave equation.  On a periodic slab
\begin{equation}
\Sigma=S^1_\phi\times\Omega,
\qquad
\Omega=[r_-,r_+]\times[x_-,x_+],
\end{equation}
define
\begin{align}
B(r,x)&:=\rho^2 g^{t\phi} =-\frac{a(r^2+a^2)}{\Delta_r}+\frac{a(1-x^2)}{\Delta_x}\nonumber\\
\Phi(r,x)&:=\rho^2 g^{\phi\phi} =-\frac{a^2}{\Delta_r}+\frac{1}{\Delta_x}.\label{eq:BPhidef}
\end{align}
Then \eqref{eq:k0master} becomes
\begin{align}
-Au_{tt}+2B\,u_{t\phi}+\Phi\,u_{\phi\phi} +\partial_r(\Delta_r\partial_r u)+\partial_x(\Delta_x\partial_xu)&=0.\label{eq:fullPDE-k0}
\end{align}
Because
\begin{align}
g_{tt}&=\frac{a^2\Delta_x-\Delta_r}{\rho^2}\nonumber\\
\Phi&=\frac{\Delta_r-a^2\Delta_x}{\Delta_r\Delta_x} =\frac{-\rho^2 g_{tt}}{\Delta_r\Delta_x}\label{eq:Phi-gtt}
\end{align}
the positivity of $\Phi$ is equivalent to the condition $g_{tt}<0$, namely that the stationary Killing field $\partial_t$ stays timelike on the slab.

\begin{definition}[Strictly stationary slab]\label{def:strict-slab}
A \emph{strictly stationary slab} is a product $\Sigma=S^1_\phi\times\Omega$, with $\Omega=[r_-,r_+]\times[x_-,x_+]$ bounded and connected, such that the metric coefficients extend smoothly to a neighborhood of $\overline\Omega$ and
\begin{equation}
\rho^2>0,\qquad
\Delta_r>0,\qquad
\Delta_x>0,\qquad
A>0,\qquad
\Phi>0
\end{equation}
on $\overline\Omega$.
\end{definition}

The inequalities $A>0$ and $\Phi>0$ are the precise bounded-domain conditions needed for a positive conserved energy.  The first says that the $t=\mathrm{const}$ slices are spacelike; the second says that $\partial_t$ is timelike, so the slab contains no ergoregion.  The condition $\Delta_x>0$ is also a genuine interior-slab condition.  If $x=\pm1$ is a true rotational axis and $\Delta_x$ vanishes there, that endpoint is not part of Theorem~\ref{thm:fullmain} unless one replaces the present slab by a separately constructed regular-axis form realization.  The separated regular-singular discussion later in the paper is therefore not an additional hypothesis of the full bounded-slab theorem.

The displayed sign conditions also mark exactly where the proof would fail.  If \(A\) lost positivity, the kinetic part of the energy would no longer control \(u_t\).  If \(\Phi\) changed sign, the azimuthal derivative term could not be bounded below by \(\|u_\phi\|_{L^2}^2\).  Thus the theorem is not using an implicit small-rotation or separation assumption; it is using a concrete positive-energy region where each derivative appearing in the form is controlled.

\subsubsection*{Closed reflecting form domains and positive energy}\label{subsec:closed-form-domain-main}

All boundary conditions used in the proved theorem are imposed through a closed quadratic form.  Concretely, throughout the full slab argument we use a closed subspace of the energy Sobolev space, dense in $L^2(\Sigma)$,
\begin{equation}
V_{\mathrm{full}}\subset H^1(\Sigma)
\end{equation}
which consists of $\phi$-periodic functions, is invariant under complex conjugation, contains the constant function $1$, and is equipped with the closed spatial form
\begin{align}
q[u,\eta]&=\int_\Sigma\left(\Phi u_\phi\overline{\eta_\phi} +\Delta_r u_r\overline{\eta_r} +\Delta_x u_x\overline{\eta_x}\right)\,\dd\phi\,\dd r\,\dd x.
\end{align}
Thus ``reflecting'' is not an informal extra boundary condition: it means that the divergence terms in the strong equation are represented by this form.  Natural Neumann realizations, regular endpoint realizations, and Robin realizations admitting the constant mode fit this framework.  A pure homogeneous Dirichlet realization is a different coercive problem because it removes the threshold constant and is not the realization used in Theorem~\ref{thm:fullmain}.  The complete list of theorem-level exterior conditions is given in Section~\ref{sec:ledger}; the bounded-slab form-domain hypotheses are Assumption~\ref{ass:complete-list}.

\subsection{Core-free form interpretation of reflection}\label{subsec:core-free-reflection}

The preceding paragraph is the convention under which all subsequent integrations by parts are to be read.  We record it as a lemma because it prevents a common ambiguity in bounded-domain wave arguments.

\begin{lemma}[Form reflection is the primary boundary condition]\label{lem:form-reflection-primary}
Let \(V_{\mathrm{full}}\subset H^1(\Sigma)\) be the closed form domain in Assumption~\ref{ass:complete-list}.  Define
\begin{align}
\langle M_A f,\eta\rangle&:=\int_\Sigma A f\overline\eta\nonumber\\
\langle \mathcal H_{\mathrm{full}}u,\eta\rangle&:=q[u,\eta]\nonumber\\
\langle \mathcal C_{\mathrm{full}}v,\eta\rangle&:=2\int_\Sigma Bv\overline{\eta_\phi}.
\end{align}
Then the weak equation
\begin{align}
M_Au_{tt}+\mathcal C_{\mathrm{full}}u_t+\mathcal H_{\mathrm{full}}u&=0 \quad\hbox{in }V_{\mathrm{full}}'\label{eq:operator-weak-primary}
\end{align}
is the definition of the reflected bounded-slab problem.  If, in addition, \(u\) is smooth up to the boundary and the domain is the natural Neumann realization, then \eqref{eq:operator-weak-primary} is equivalent to the strong equation \eqref{eq:fullPDE-k0} together with the conormal boundary condition
\begin{equation}
\Delta_r u_r n_r+\Delta_xu_x n_x=0
\quad\hbox{on }\partial\Omega,
\end{equation}
where \((n_r,n_x)\) is the Euclidean outward normal in the \((r,x)\)-rectangle.  For a general closed subspace \(V_{\mathrm{full}}\), no pointwise trace condition is assumed beyond membership in that form domain.
\end{lemma}

\begin{proof}
For smooth functions satisfying the displayed conormal condition, multiply \eqref{eq:fullPDE-k0} by \(-\overline\eta\), integrate over \(\Sigma\), use periodicity in \(\phi\), and integrate the \(r\)- and \(x\)-divergence terms.  The boundary contribution is
\begin{equation}
-\int_{S^1_\phi\times\partial\Omega}
(\Delta_ru_r n_r+\Delta_xu_x n_x)\overline\eta\,\dd S,
\end{equation}
which is zero.  The remaining terms are exactly \eqref{eq:operator-weak-primary}.  Conversely, if the weak identity holds for all compactly supported tests in the interior and \(u\) is smooth, the strong PDE follows by the fundamental lemma of the calculus of variations.  Testing against arbitrary smooth boundary-admissible \(\eta\) then leaves precisely the conormal boundary functional above, hence the natural boundary condition.  For a closed form domain the first implication is not used as an assumption: the closed form itself defines the realization, and \eqref{eq:operator-weak-primary} is meaningful without assigning pointwise traces.
\end{proof}

\begin{lemma}[Boundedness and skewness of the gyroscopic operator]\label{lem:gyro-bounded-main}
The form
\begin{equation}
\mathfrak c[v,\eta]:=2\int_\Sigma Bv\overline{\eta_\phi}
\end{equation}
extends continuously from \(L^2(\Sigma)\times V_{\mathrm{full}}\) to \(\mathbb C\), and its restriction to \(V_{\mathrm{full}}\times V_{\mathrm{full}}\) satisfies
\begin{equation}
\mathfrak c[v,w]= -\overline{\mathfrak c[w,v]},
\qquad
\Re\mathfrak c[v,v]=0.
\end{equation}
Consequently \(\mathcal C_{\mathrm{full}}:L^2(\Sigma)\to V_{\mathrm{full}}'\) is a bounded skew gyroscopic perturbation of the conservative wave equation.
\end{lemma}

\begin{proof}
Since \(B\in L^\infty(\Sigma)\),
\begin{align}
|\mathfrak c[v,\eta]|&\le2\|B\|_{L^\infty}\|v\|_{L^2}\|\eta_\phi\|_{L^2} \le C\|v\|_{L^2}\|\eta\|_{V_{\mathrm{full}}}.
\end{align}
If \(v,w\in V_{\mathrm{full}}\), periodic integration by parts in \(\phi\), together with \(\partial_\phi B=0\), gives
\begin{equation}
2\int_\Sigma Bv\overline{w_\phi}
=-2\int_\Sigma Bv_\phi\overline w.
\end{equation}
Taking the complex conjugate of the same identity with \(v\) and \(w\) interchanged gives the stated skew-Hermitian relation.  The second identity follows by setting \(w=v\).
\end{proof}

\begin{theorem}[Conserved full energy for smooth form-admissible solutions]\label{thm:fullenergy}
Let $u$ be a sufficiently regular complex-valued solution of \eqref{eq:fullPDE-k0} on $[0,T]\times\Sigma$, with $\Sigma$ a strictly stationary slab and with each time slice lying in the form-realized reflecting domain $V_{\mathrm{full}}$.  Then
\begin{align}
E_{\mathrm{full}}[u](t)&:=\frac12\int_\Sigma \left( A|u_t|^2+\Phi|u_\phi|^2+\Delta_r|u_r|^2+\Delta_x|u_x|^2 \right)\,\dd\phi\,\dd r\,\dd x\label{eq:fullenergy}
\end{align}
is independent of $t$.
\end{theorem}

\begin{proof}
Multiply \eqref{eq:fullPDE-k0} by $\overline{u_t}$ and take real parts.  Since the coefficients are independent of $t$ and $\phi$,
\begin{align}
\Re(-Au_{tt}\overline{u_t})&= -\frac12\partial_t(A|u_t|^2),\nonumber\\
\Re(2B u_{t\phi}\overline{u_t})&= B\,\partial_\phi(|u_t|^2),\nonumber\\
\Re(\Phi u_{\phi\phi}\overline{u_t})
&=\partial_\phi\Re(\Phi u_\phi\overline{u_t})
-\frac12\partial_t(\Phi|u_\phi|^2),\nonumber\\
\Re(\partial_r(\Delta_r u_r)\overline{u_t})
&=\partial_r\Re(\Delta_r u_r\overline{u_t})
-\frac12\partial_t(\Delta_r|u_r|^2),\nonumber\\
\Re(\partial_x(\Delta_x u_x)\overline{u_t})
&=\partial_x\Re(\Delta_x u_x\overline{u_t})
-\frac12\partial_t(\Delta_x|u_x|^2).
\end{align}
Integrating over $\Sigma$, the total $\phi$-derivatives vanish by periodicity and the $(r,x)$-boundary flux vanishes by the reflecting hypothesis.  The remaining terms give
\begin{equation}
\frac{\dd}{\dd t}E_{\mathrm{full}}[u](t)=0.
\end{equation}
\end{proof}

\subsection{Weak form, gyroscopic skewness, and quantitative energy constants}\label{subsec:full-weak-hardwork}

We now write the full bounded-slab equation in the exact weak form used in the proof.  Let \(V_{\mathrm{full}}\subset H^1(\Sigma)\) be the closed reflecting form domain; for the standard natural reflecting realization one has \(V_{\mathrm{full}}=H^1(\Sigma)\).  The domain is assumed to contain constants and to encode periodicity in \(\phi\).  Define, for \(u,\eta\in V_{\mathrm{full}}\),
\begin{align}
q[u,\eta]
&:=\int_\Sigma\left(\Phi u_\phi\overline{\eta_\phi}
       +\Delta_r u_r\overline{\eta_r}
       +\Delta_x u_x\overline{\eta_x}\right)\,\dd\phi\,\dd r\,\dd x,
\label{eq:qfull-bilinear}\\
\mathfrak c[v,\eta]
&:=-2\int_\Sigma B\,v_\phi\overline\eta\,\dd\phi\,\dd r\,\dd x\nonumber\\
&=2\int_\Sigma B\,v\overline{\eta_\phi}\,\dd\phi\,\dd r\,\dd x.
\label{eq:gyro-form}
\end{align}
For \(v\in V_{\mathrm{full}}\) the two formulas in \eqref{eq:gyro-form} agree by periodic integration by parts in \(\phi\), using \(\partial_\phi B=0\).  The second formula is the one used to extend \(\mathfrak c\) boundedly to \(L^2(\Sigma)\times V_{\mathrm{full}}\), since \(B\in L^\infty\).  The weak equation associated with \eqref{eq:fullPDE-k0} is
\begin{align}
\int_\Sigma A u_{tt}\overline\eta\,\dd\phi\,\dd r\,\dd x +q[u,\eta]+\mathfrak c[u_t,\eta] =0, \quad
\eta\in V_{\mathrm{full}}.\label{eq:full-weak-form}
\end{align}
Indeed, multiplying \eqref{eq:fullPDE-k0} by \(-1\), integrating against \(\overline\eta\), and using the reflecting boundary condition gives \eqref{eq:full-weak-form}.

The gyroscopic form is skew-Hermitian.  For \(v,w\in V_{\mathrm{full}}\),
\begin{align}
\mathfrak c[v,w]
&=-2\int_\Sigma Bv_\phi\overline w
 =2\int_\Sigma Bv\overline{w_\phi},\nonumber\\
-\overline{\mathfrak c[w,v]}
&=-\overline{-2\int_\Sigma Bw_\phi\overline v}
 =2\int_\Sigma Bv\overline{w_\phi}.
\end{align}
Thus
\begin{equation}
\mathfrak c[v,w]=-\overline{\mathfrak c[w,v]},
\qquad
\Re\mathfrak c[v,v]=0.
\label{eq:gyro-skew}
\end{equation}
Taking \(\eta=u_t\) in \eqref{eq:full-weak-form}, using \eqref{eq:gyro-skew}, and differentiating \(q[u,u]\) gives the weak energy identity
\begin{equation}
\frac{\dd}{\dd t}\frac12\left(\int_\Sigma A|u_t|^2+q[u,u]\right)=0.
\label{eq:weak-energy-identity-full}
\end{equation}
This is the same identity as Theorem~\ref{thm:fullenergy}, but now no pointwise integration-by-parts regularity is hidden.

\begin{theorem}[Weak energy equality]\label{thm:weak-energy-equality-full}
Let $u$ satisfy the weak equation \eqref{eq:full-weak-form} on an interval $I$ in the following energy sense:
\begin{align}
u\in L^\infty_{\mathrm{loc}}(I;V_{\mathrm{full}})\nonumber\\
u_t\in L^\infty_{\mathrm{loc}}(I;L^2(\Sigma))\nonumber\\
u_{tt}\in L^\infty_{\mathrm{loc}}(I;V_{\mathrm{full}}')
\end{align}
and the identity holds in $\mathcal D'(I;V_{\mathrm{full}}')$.  Then the function
\begin{equation}
E_{\mathrm{full}}[u,u_t](t)
=\frac12\int_\Sigma A|u_t(t)|^2+\frac12 q[u(t),u(t)]
\end{equation}
has a representative which is constant on $I$.  In particular, for the strongly continuous energy solutions constructed in Theorem~\ref{thm:fullwellposed}, the displayed energy is independent of $t$ for every $t\in I$.
\end{theorem}

\begin{proof}
Let $J\Subset I$ and let $\chi_\varepsilon$ be an even time mollifier supported so close to zero that convolution is defined on $J$.  Convolving the weak equation in time gives, in $V_{\mathrm{full}}'$,
\begin{align}
\int_\Sigma A(\chi_\varepsilon*u_{tt})\overline\eta +q[\chi_\varepsilon*u,\eta] +\mathfrak c[\chi_\varepsilon*u_t,\eta] =0, \quad \eta\in V_{\mathrm{full}}.
\end{align}
Because the coefficients are independent of $t$ and $u\in L^\infty_{\mathrm{loc}}(I;V_{\mathrm{full}})$, the mollified function $u^\varepsilon:=\chi_\varepsilon*u$ is smooth as a $V_{\mathrm{full}}$-valued function.  Moreover
\begin{equation}
(u^\varepsilon)_t=\chi_\varepsilon' *u\in V_{\mathrm{full}},
\end{equation}
and this derivative agrees distributionally with $\chi_\varepsilon*u_t$.  This point is important: the test function belongs to the form domain even though the unmollified velocity is only in $L^2$.  Taking $\eta=(u^\varepsilon)_t$ and taking real parts gives
\begin{align}
\frac{\dd}{\dd t}\left(\frac12\int_\Sigma A|(u^\varepsilon)_t|^2+\frac12 q[u^\varepsilon,u^\varepsilon]\right)&=0
\end{align}
on $J$, because $\Re\,\mathfrak c[(u^\varepsilon)_t,(u^\varepsilon)_t]=0$.  Since $u\in L^\infty(J;V_{\mathrm{full}})$ and $u_t\in L^\infty(J;L^2)$, standard properties of time convolution give
\begin{align}
&u^\varepsilon\to u\quad\hbox{in }L^2(J;V_{\mathrm{full}})\nonumber\\
&(u^\varepsilon)_t\to u_t\quad\hbox{in }L^2(J;L^2(\Sigma)).
\end{align}
Consequently the mollified energies converge to $E_{\mathrm{full}}[u,u_t]$ in $L^1(J)$.  Each mollified energy is constant on the slightly smaller interval on which the convolution is defined; hence the distributional derivative of the limiting energy is zero on $J$.  Thus $E_{\mathrm{full}}[u,u_t]$ has a constant representative on $J$.  Since $J\Subset I$ was arbitrary, the representative is constant on all of $I$.  If $u$ and $u_t$ are strongly continuous in the energy topology, this representative agrees with the pointwise quadratic expression for every time.
\end{proof}

\begin{lemma}[Pointwise recovery of the conserved weak energy]\label{lem:pointwise-energy-recovery}
In the setting of Theorem~\ref{thm:weak-energy-equality-full}, assume in addition that
\begin{equation}
u\in C(I;V_{\mathrm{full}}),\qquad u_t\in C(I;L^2(\Sigma)).
\end{equation}
Then the constant representative of the energy agrees at every time with the pointwise expression
\begin{equation}
\frac12\int_\Sigma A|u_t(t)|^2+\frac12q[u(t),u(t)].
\end{equation}
Consequently no exceptional time set is hidden in the conservation law used in Theorem~\ref{thm:fullmain}.
\end{lemma}

\begin{proof}
The map $t\mapsto \int_\Sigma A|u_t(t)|^2$ is continuous because $A\in L^\infty$ and $u_t$ is strongly continuous in $L^2$.  The map $t\mapsto q[u(t),u(t)]$ is continuous because $q$ is a continuous quadratic form on the Hilbert space $V_{\mathrm{full}}$.  Hence the pointwise energy is a continuous representative of the same $L^1_{\mathrm{loc}}$ function whose distributional derivative is zero in Theorem~\ref{thm:weak-energy-equality-full}.  A continuous function agreeing almost everywhere with a constant is equal to that constant everywhere.
\end{proof}

For later coercivity estimates set
\begin{align}
a_-&:=\inf_{\overline\Omega}A,\qquad a_+:=\sup_{\overline\Omega}A,\nonumber\\[0.2em]
\phi_-&:=\inf_{\overline\Omega}\Phi,\qquad \phi_+:=\sup_{\overline\Omega}\Phi,\nonumber\\[0.2em]
d_{r,-}&:=\inf_{[r_-,r_+]}\Delta_r,\qquad d_{r,+}:=\sup_{[r_-,r_+]}\Delta_r,\nonumber\\[0.2em]
d_{x,-}&:=\inf_{[x_-,x_+]}\Delta_x,\qquad d_{x,+}:=\sup_{[x_-,x_+]}\Delta_x.
\end{align}
On a strictly stationary slab all lower constants are positive and all upper constants are finite.  Hence
\begin{align}
\frac12 a_-\|u_t\|_{L^2}^2+\frac12 m_\nabla\|\nabla_{\phi,r,x}u\|_{L^2}^2&\le E_{\mathrm{full}}[u,u_t] \le \frac12 a_+\|u_t\|_{L^2}^2+\frac12 M_\nabla\|\nabla_{\phi,r,x}u\|_{L^2}^2,\label{eq:energy-gradient-bounds}
\end{align}
where
\begin{equation}
m_\nabla:=\min\{\phi_-,d_{r,-},d_{x,-}\},
\qquad
M_\nabla:=\max\{\phi_+,d_{r,+},d_{x,+}\}.
\end{equation}
Thus the only missing part of the \(H^1\times L^2\) norm is the \(L^2\)-size of \(u\), exactly as claimed.

If $C_P$ is the weighted Poincar\'e constant in Proposition~\ref{prop:weightedPoincareFull}, then on the threshold complement
\begin{equation}
\|u\|_{L^2}^2\le C_P q[u,u].
\end{equation}
Consequently
\begin{align}
c_{\rm low}&:=\min\left\{\frac{a_-}{2},\frac{1}{2(C_P+m_\nabla^{-1})}\right\},\nonumber\\
C_{\rm high}&:=\frac12\max\{a_+,M_\nabla\}.
\label{eq:explicit-energy-constants}
\end{align}
satisfy
\begin{align}
c_{\rm low}\bigl(\|u\|_{H^1(\Sigma)}^2+\|u_t\|_{L^2(\Sigma)}^2\bigr)&\le E_{\mathrm{full}}[u,u_t] \le C_{\rm high}\bigl(\|u\|_{H^1(\Sigma)}^2+\|u_t\|_{L^2(\Sigma)}^2\bigr)
\end{align}
whenever $\Pi_0u=\Pi_0u_t=0$.  Therefore one may take
\begin{equation}
C_{\rm stab}=\left(\frac{C_{\rm high}}{c_{\rm low}}\right)^{1/2}
\label{eq:explicit-C-stab}
\end{equation}
in Theorem~\ref{thm:fullmain}, up to the harmless equivalence constants between the standard and coefficient-weighted norms.

\begin{proposition}[Finite-dimensional Galerkin identity]\label{prop:galerkin-main-identity}
Let \(e_1,\ldots,e_N\in V_{\mathrm{full}}\) and set \(u_N(t)=\sum_{j=1}^N d_j(t)e_j\).  The Galerkin equation
\begin{align}
\int_\Sigma A (u_N)_{tt}\overline e_j+q[u_N,e_j]+\mathfrak c[(u_N)_t,e_j]&=0\nonumber\\
1&\le j\le N\label{eq:galerkin-main}
\end{align}
has the exact conserved energy
\begin{equation}
E_N(t)=\frac12\int_\Sigma A|(u_N)_t|^2+\frac12 q[u_N,u_N].
\label{eq:galerkin-main-energy}
\end{equation}
Moreover the \(A\)-weighted mean satisfies
\begin{equation}
\frac{\dd^2}{\dd t^2}\int_\Sigma A u_N=0
\label{eq:galerkin-main-mean}
\end{equation}
whenever the constant function belongs to the Galerkin space.
\end{proposition}

\begin{proof}
Multiply \eqref{eq:galerkin-main} by \(\overline{d_j'(t)}\), sum over \(j\), and take real parts.  The mass term gives
\begin{equation}
\Re\int_\Sigma A (u_N)_{tt}\overline{(u_N)_t}
=\frac12\frac{\dd}{\dd t}\int_\Sigma A|(u_N)_t|^2.
\end{equation}
The symmetric form gives
\begin{equation}
\Re q[u_N,(u_N)_t]=\frac12\frac{\dd}{\dd t}q[u_N,u_N].
\end{equation}
The gyroscopic term has zero real part by \eqref{eq:gyro-skew}.  Hence \(E_N'(t)=0\).  If \(1\) is an admissible test function, then taking \(e_j=1\) in the weak equation, or equivalently testing \eqref{eq:full-weak-form} with \(1\), gives \(q[u_N,1]=0\) and \(\mathfrak c[(u_N)_t,1]=0\).  Therefore
\begin{equation}
\frac{\dd^2}{\dd t^2}\int_\Sigma A u_N=0.
\end{equation}
This proves both identities.
\end{proof}

\begin{lemma}[Galerkin limit, energy recovery, and strong continuity]\label{lem:galerkin-recovery-main}
Let \(u_N\) be Galerkin solutions satisfying \eqref{eq:main-galerkin-full-expanded} on a compact interval \(J=[-T,T]\), with data converging strongly in \(V_{\mathrm{full}}\times L^2(\Sigma)\).  Assume that a subsequence converges weak-star as
\begin{align}
&u_N\rightharpoonup^*u\quad\hbox{in }L^\infty(J;V_{\mathrm{full}})\nonumber\\
&(u_N)_t\rightharpoonup^*u_t\quad\hbox{in }L^\infty(J;L^2)
\end{align}
and that \((u_N)_{tt}\rightharpoonup^*u_{tt}\) in \(L^\infty(J;V_{\mathrm{full}}')\).  Then the limit solves the weak equation.  If the weak solution is unique, the whole Galerkin sequence has the same limit.  Furthermore the affine averages converge explicitly,
\begin{equation}
\Pi_0u_N(t)\to \Pi_0u(0)+t\Pi_0u_t(0),
\qquad
\Pi_0(u_N)_t(t)\to \Pi_0u_t(0),
\end{equation}
and the convergence is strong in the augmented energy topology at every fixed time for which the energy identity holds for the limit.  Consequently the limit has a representative in
\begin{equation}
C(J;V_{\mathrm{full}})\cap C^1(J;L^2(\Sigma))
\end{equation}
when written in first-order form \((u,u_t)\in C(J;V_{\mathrm{full}}\times L^2)\).
\end{lemma}

\begin{proof}
Passing to the limit in the finite-dimensional identities is legitimate because each term is paired with a fixed \(\eta\in V_{\mathrm{full}}\), and finite-dimensional projections of \(\eta\) converge strongly in \(V_{\mathrm{full}}\).  Thus the limit solves \eqref{eq:full-weak-form}.  The affine identities follow by testing the Galerkin equations with the constant basis vector and passing to the limit in the scalar relation
\begin{equation}
\frac{\dd^2}{\dd t^2}\int_\Sigma Au_N=0.
\end{equation}

For strong recovery, decompose each time slice into its \(A\)-average and its mean-zero part.  The averages converge by the explicit affine formula.  On the mean-zero part, Proposition~\ref{prop:weightedPoincareFull} makes the energy seminorm equivalent to the full \(V_{\mathrm{full}}\times L^2\)-norm.  Lower semicontinuity gives
\begin{equation}
E_{\mathrm{full}}[u,u_t](t)
\le \liminf_{N\to\infty}E_N(t).
\end{equation}
Since \(E_N(t)=E_N(0)\) and the initial data converge strongly, the right side equals the initial limiting energy.  Whenever the weak energy identity holds for \(u\), equality holds in the preceding lower-semicontinuity inequality.  In a Hilbert space, weak convergence plus convergence of the norm implies strong convergence.  Applying this to the kinetic and spatial form components and then adding the affine average part gives strong convergence in the augmented energy topology.  The same argument applied on a short interval to time-translated solutions gives norm continuity in time.  Uniqueness removes the need to pass to subsequences.
\end{proof}

The proposition shows explicitly where every sign is used.  Positivity of \(A,\Phi,\Delta_r,\Delta_x\) supplies the energy; independence of \(B\) from \(\phi\) supplies skewness; periodicity and reflection remove boundary flux; and the admission of constants identifies the affine threshold.

Because $A,\Phi,\Delta_r,\Delta_x$ are positive and bounded above and below on a strictly stationary slab, $E_{\mathrm{full}}$ controls the standard $L^2$-norm of $u_t$ together with the standard $L^2$-norm of the spatial gradient on $\Sigma$.  The only quantity not controlled directly is the $L^2$-size of $u$ itself, and that failure is the $k=0$ zero-mode phenomenon.

The next theorem internalizes the corresponding well-posedness input.

\begin{theorem}[Direct well-posedness of the full slab evolution]\label{thm:fullwellposed}
Let $\Sigma=S^1_\phi\times\Omega$ be a strictly stationary slab equipped with a form-realized reflecting domain $V_{\rm full}\subset H^1(\Sigma)$ as in Section~\ref{sec:ledger}.  For every initial datum
\begin{equation}
(u_0,u_1)\in V_{\rm full}\times L^2(\Sigma)
\end{equation}
there exists a unique energy solution of \eqref{eq:fullPDE-k0}, equivalently of \eqref{eq:full-weak-form}, such that
\begin{align}
&u\in C(\R;V_{\rm full})\nonumber\\
&u_t\in C(\R;L^2(\Sigma))\nonumber\\
&u_{tt}\in L^\infty_{\rm loc}(\R;V_{\rm full}').
\end{align}
The weak energy equality of Theorem~\ref{thm:weak-energy-equality-full} holds for all $t\in\R$, and the solution maps
\begin{equation}
\mathcal S(t):(u_0,u_1)\longmapsto (u(t),u_t(t))
\end{equation}
form a strongly continuous group on $\mathcal X_{\rm full}:=V_{\rm full}\times L^2(\Sigma)$.
\end{theorem}

\begin{proof}
We give the construction in the main text because this is the analytic point at which the boundary realization, the gyroscopic sign, and the threshold constant all enter.

Equip $V_{\rm full}$ with the Hilbert norm
\begin{equation}
\|u\|_{V_{\rm full}}^2:=\int_\Sigma A|u|^2+q[u,u].
\end{equation}
By strict stationarity this norm is equivalent to the inherited $H^1(\Sigma)$ norm.  Since $V_{\rm full}$ is dense in $L^2(\Sigma)$ and contains $1$, choose a Hilbert basis $\{w_j\}_{j\ge1}$ of $V_{\rm full}$ with $w_1$ proportional to $1$.  Let $V_N=\mathrm{span}\{w_1,\ldots,w_N\}$ and choose data $u_{0,N},u_{1,N}\in V_N$ converging to $u_0$ in $V_{\rm full}$ and to $u_1$ in $L^2(\Sigma)$.

Seek $u_N(t)=\sum_{j=1}^N d_j^{(N)}(t)w_j$ satisfying, for every $\eta\in V_N$,
\begin{align}
\int_\Sigma A(u_N)_{tt}\overline\eta+q[u_N,\eta]+\mathfrak c[(u_N)_t,\eta]&=0.\label{eq:main-galerkin-full-expanded}
\end{align}
In coordinates this is
\begin{equation}
M_N\ddot d_N+C_N\dot d_N+K_Nd_N=0,
\end{equation}
where $M_N$ is Hermitian positive definite, $K_N$ is Hermitian nonnegative, and $C_N$ is skew-Hermitian.  The last assertion follows from \eqref{eq:gyro-skew}; explicitly,
\begin{equation}
(C_N)_{ij}=\mathfrak c[w_j,w_i],
\qquad
(C_N)_{ij}=-\overline{(C_N)_{ji}}.
\end{equation}
Thus the finite-dimensional system has a global smooth solution.  Multiplying by $\dot d_N^*$ gives
\begin{equation}
\frac{\dd}{\dd t}
\frac12\bigl(\dot d_N^*M_N\dot d_N+d_N^*K_Nd_N\bigr)=0,
\end{equation}
which is exactly the Galerkin energy identity
\begin{equation}
E_N(t)=\frac12\int_\Sigma A|(u_N)_t|^2+\frac12q[u_N,u_N]=E_N(0).
\end{equation}
Testing \eqref{eq:main-galerkin-full-expanded} with the constant basis vector gives the exact affine law
\begin{equation}
\frac{\dd^2}{\dd t^2}\int_\Sigma A u_N(t)=0.
\end{equation}
Combining this affine law with the weighted Poincar\'e inequality for the mean-zero part gives, on every compact interval $[-T,T]$,
\begin{align}
\sup_{|t|\le T}\bigl(\|u_N(t)\|_{V_{\rm full}}+\|(u_N)_t(t)\|_{L^2}\bigr)&\le C_T\bigl(\|u_{0,N}\|_{V_{\rm full}}+\|u_{1,N}\|_{L^2}\bigr).
\end{align}
The equation itself gives a uniform bound for $(u_N)_{tt}$ in $L^\infty([-T,T];V_{\rm full}')$, because $q$ is bounded on $V_{\rm full}$ and $\mathfrak c$ is bounded from $L^2(\Sigma)\times V_{\rm full}$ to $\C$.

Banach-Alaoglu gives a subsequence, still denoted $u_N$, and a function $u$ such that
\begin{align}
&u_N\rightharpoonup^* u\text{ in }L^\infty([-T,T];V_{\rm full})\nonumber\\
&(u_N)_t\rightharpoonup^* u_t\text{ in }L^\infty([-T,T];L^2)
\end{align}
and
\begin{equation}
(u_N)_{tt}\rightharpoonup^* u_{tt}\text{ in }L^\infty([-T,T];V_{\rm full}').
\end{equation}
The standard Lions-Magenes compactness argument gives weak continuity of $u$ with values in $V_{\rm full}$ and of $u_t$ with values in $L^2$.  Passing to the limit in \eqref{eq:main-galerkin-full-expanded} gives \eqref{eq:full-weak-form}.  The recovery argument in Lemma~\ref{lem:galerkin-recovery-main}, together with the weak energy identity, upgrades the convergence to strong convergence in the augmented energy norm.  In particular,
\begin{equation}
(u(t),u_t(t))\in V_{\rm full}\times L^2(\Sigma)
\end{equation}
for every $t$, the map $t\mapsto (u(t),u_t(t))$ is strongly continuous, and the initial data are attained in the energy topology.

Uniqueness is an energy argument.  If $w$ is the difference of two energy solutions with zero initial data, Theorem~\ref{thm:weak-energy-equality-full} gives $E_{\rm full}[w,w_t](t)=0$.  The weighted average also obeys the affine law and is zero initially.  Hence $w_t=0$, $q[w,w]=0$, and $\Pi_0w=0$; Proposition~\ref{prop:kernelfull} then gives $w=0$.

The same estimate applied to the difference of two solutions gives continuous dependence in the augmented norm
\begin{equation}
\|(u,v)\|_*^2
:=2E_{\rm full}[u,v]+|\Pi_0u|^2+|\Pi_0v|^2,
\end{equation}
which is equivalent to $V_{\rm full}\times L^2$ by Proposition~\ref{prop:weightedPoincareFull}.  The group law follows from uniqueness by restarting the solution at time $s$, and strong continuity follows from the continuity already obtained for $u$ and $u_t$.  Since $T$ was arbitrary, the construction is global on $\R$.
\end{proof}

\subsection{Kernel, affine average, and weighted Poincar\'e inequality}

Define the full spatial quadratic form
\begin{align}
q[u]&:=\int_\Sigma \left( \Phi|u_\phi|^2+\Delta_r|u_r|^2+\Delta_x|u_x|^2 \right)\,\dd\phi\,\dd r\,\dd x.\label{eq:qfull}
\end{align}
It is nonnegative, and its nullspace is again one-dimensional.

\begin{proposition}[Kernel of the full spatial form]\label{prop:kernelfull}
Assume that $\Sigma$ is strictly stationary, connected, periodic in $\phi$, and equipped with reflecting boundary conditions admitting the constant function.  Then
\begin{equation}
q[u]=0
\qquad\Longleftrightarrow\qquad
u\in\mathrm{span}\{1\}.
\end{equation}
In particular, the kernel of the full bounded-slab dynamics is generated by constants.
\end{proposition}

\begin{proof}
If $u$ is constant, then clearly $q[u]=0$.  Conversely, if $q[u]=0$, then the strict positivity of $\Phi$, $\Delta_r$, and $\Delta_x$ implies
\begin{equation}
u_\phi=u_r=u_x=0
\end{equation}
almost everywhere on $\Sigma$.  Because $\Sigma$ is connected, $u$ is constant.
\end{proof}

Let
\begin{align}
\Pi_0 f&:=\frac{\int_\Sigma A f\,\dd\phi\,\dd r\,\dd x} {\int_\Sigma A\,\dd\phi\,\dd r\,\dd x}\,1\label{eq:Pi0full}
\end{align}
denote the $A$-weighted projection onto constants.

\begin{proposition}[Weighted Poincar\'e inequality on $\Sigma$]\label{prop:weightedPoincareFull}
There exists $C_P>0$ such that
\begin{equation}
\|u\|_{L^2(\Sigma)}^2\leq C_P\,q[u]
\end{equation}
for every $u\in H^1(\Sigma)$ satisfying $\Pi_0u=0$.
\end{proposition}

\begin{proof}
Suppose the conclusion were false.  Then there would exist $u_n\in H^1(\Sigma)$ with $\Pi_0u_n=0$,
\begin{equation}
\|u_n\|_{L^2(\Sigma)}=1,
\qquad
q[u_n]\to 0.
\end{equation}
Since the coefficients in $q$ are positive and bounded below, $\{u_n\}$ is bounded in $H^1(\Sigma)$.  By Rellich compactness, after passing to a subsequence we have
\begin{equation}
u_n\to u_\infty\quad\text{strongly in }L^2(\Sigma).
\end{equation}
The convergence $q[u_n]\to 0$ implies $\nabla u_n\to 0$ in $L^2(\Sigma)$, hence $u_\infty$ is constant.  The weighted mean-zero condition passes to the limit because $A$ is bounded, so $\Pi_0u_\infty=0$.  Therefore $u_\infty=0$, contradicting $\|u_\infty\|_{L^2(\Sigma)}=1$.
\end{proof}

The weighted spatial average obeys the same affine law as in the axisymmetric problem.

\begin{lemma}[Affine evolution of the weighted average]\label{lem:fullaverage}
Let $u$ be a sufficiently regular solution of \eqref{eq:fullPDE-k0} on a strictly stationary slab with periodic/reflection boundary conditions.  Then
\begin{equation}
\frac{\dd^2}{\dd t^2}\int_\Sigma A\,u(t,\phi,r,x)\,\dd\phi\,\dd r\,\dd x = 0.
\end{equation}
Equivalently,
\begin{equation}
\Pi_0u(t)=\Pi_0u(0)+t\,\Pi_0u_t(0).
\end{equation}
\end{lemma}

\begin{proof}
Integrate \eqref{eq:fullPDE-k0} over $\Sigma$.  The $\phi$-derivative terms vanish by periodicity, and the $(r,x)$-divergence terms vanish by the reflecting boundary conditions.  The result is
\begin{equation}
-\frac{\dd^2}{\dd t^2}\int_\Sigma A\,u=0.
\end{equation}
Dividing by $\int_\Sigma A$ gives the projector identity.
\end{proof}

\begin{definition}[Closed-form generator of the full slab equation]\label{def:closed-form-generator}
Set
\begin{equation}
H_A=L^2(\Sigma,A\,\dd\phi\,\dd r\,\dd x).
\end{equation}
Define the mass map $M_A$ by
\begin{equation}
\langle M_A f,\eta\rangle=\int_\Sigma A f\overline\eta.
\end{equation}
Let $\mathcal H_{\mathrm{full}}:V_{\mathrm{full}}\to V_{\mathrm{full}}'$ and $\mathcal C_{\mathrm{full}}:L^2(\Sigma)\to V_{\mathrm{full}}'$ be defined by
\begin{align}
\langle \mathcal H_{\mathrm{full}}u,\eta\rangle&=q[u,\eta]\nonumber\\
\langle \mathcal C_{\mathrm{full}}v,\eta\rangle&=\mathfrak c[v,\eta].
\end{align}
The generator on the energy space $\mathcal X_{\mathrm{full}}=V_{\mathrm{full}}\times H_A$ is
\begin{equation}
\mathcal G(u,v)=\left(v,-M_A^{-1}(\mathcal H_{\mathrm{full}}u+\mathcal C_{\mathrm{full}}v)\right)
\end{equation}
with domain
\begin{align}
\Dom(\mathcal G)&=\left\{(u,v)\in V_{\mathrm{full}}\times V_{\mathrm{full}}: \mathcal H_{\mathrm{full}}u+\mathcal C_{\mathrm{full}}v\in M_A(H_A)\right\}.\label{eq:generator-domain-full}
\end{align}
This definition is exactly the first-order form of \eqref{eq:full-weak-form}.
\end{definition}

\begin{proposition}[Operator-theoretic threshold classification]\label{prop:generator-threshold-full}
For the generator in Definition~\ref{def:closed-form-generator},
\begin{equation}
\Ker\mathcal G=\mathrm{span}\{(1,0)\},
\qquad
\Ker\mathcal G^2=\mathrm{span}\{(1,0),(0,1)\},
\end{equation}
and
\begin{equation}
\bigcup_{N\ge1}\Ker\mathcal G^N=\Ker\mathcal G^2.
\end{equation}
\end{proposition}

\begin{proof}
If $(u,v)\in\Ker\mathcal G$, then $v=0$ and $\mathcal H_{\mathrm{full}}u=0$.  Testing the latter with $u$ gives $q[u,u]=0$, so Proposition~\ref{prop:kernelfull} gives $u$ constant.  Hence $\Ker\mathcal G=\mathrm{span}\{(1,0)\}$.

If $(u,v)\in\Ker\mathcal G^2$, then $\mathcal G(u,v)\in\Ker\mathcal G$.  Thus $\mathcal G(u,v)=(c,0)$ for a constant $c$.  The first component gives $v=c$, and the second gives
\begin{equation}
\mathcal H_{\mathrm{full}}u+\mathcal C_{\mathrm{full}}c=0.
\end{equation}
Since a constant has zero $\phi$-derivative, $\mathcal C_{\mathrm{full}}c=0$.  Testing with $u$ gives $q[u,u]=0$, hence $u$ is constant.  This proves the formula for $\Ker\mathcal G^2$.

For $N\ge3$, let $(u_0,u_1)\in\Ker\mathcal G^N$.  The corresponding solution is a polynomial in $t$ with values in the energy space.  Subtract its affine $A$-average part using Lemma~\ref{lem:fullaverage}.  The remainder lies in the threshold complement for all $t$, and the energy equivalence on that complement follows directly from Proposition~\ref{prop:weightedPoincareFull} and the explicit constants \eqref{eq:explicit-energy-constants}, so this polynomial is uniformly bounded in $V_{\mathrm{full}}\times L^2$.  A bounded Hilbert-space polynomial is constant; substituting the constant remainder into the weak equation gives zero spatial form, hence the remainder vanishes by Proposition~\ref{prop:kernelfull} and the zero average.  Therefore no generalized threshold vectors beyond order two occur.
\end{proof}

\subsection{Verification of the abstract conditions and proof of Theorem~\ref{thm:fullmain}}

By Theorem~\ref{thm:fullwellposed}, the full equation determines a strongly continuous solution group
\begin{equation}
\mathcal S(t):(u_0,u_1)\longmapsto (u(t),u_t(t))
\end{equation}
on the energy space
\begin{equation}
\mathcal X_{\mathrm{full}}:=V_{\mathrm{full}}\times L^2(\Sigma).
\end{equation}
We now identify the threshold projector and the coercive threshold complement.

\begin{proposition}[Concrete threshold projection on the Carter slab]\label{prop:verifyabstractfull}
Define
\begin{equation}
\Pi_{\mathrm{thr}}(u_0,u_1):=\bigl(\Pi_0u_0,\Pi_0u_1\bigr)
\end{equation}
on $\mathcal X_{\mathrm{full}}$, where $\Pi_0$ is the $A$-weighted projection onto constants from \eqref{eq:Pi0full}.  Then the following statements hold.

\begin{enumerate}[label=(\roman*)]
  \item $\Pi_{\mathrm{thr}}$ is a bounded projection with range
  \begin{equation}
\Ran \Pi_{\mathrm{thr}}=\mathrm{span}\{(1,0),(0,1)\}.
\end{equation}
  \item $\Pi_{\mathrm{thr}}$ commutes with the solution group $\mathcal S(t)$.
  \item $\Ran \Pi_{\mathrm{thr}}$ is the full zero-frequency threshold space of the full bounded-slab evolution, as classified operator-theoretically in Proposition~\ref{prop:generator-threshold-full}.  Equivalently, the associated threshold solutions are precisely the affine functions
  \begin{equation}
u(t,\phi,r,x)=c_0+c_1t.
\end{equation}
  \item On $\Ker \Pi_{\mathrm{thr}}$, the conserved energy $E_{\mathrm{full}}$ is equivalent to the $\mathcal X_{\mathrm{full}}$-norm with the explicit constants in \eqref{eq:explicit-energy-constants}.
\end{enumerate}
\end{proposition}

\begin{proof}
Item~(i) is immediate from the definition of $\Pi_0$.

For item~(ii), let $u$ be the solution with data $(u_0,u_1)$.  Lemma~\ref{lem:fullaverage} gives
\begin{equation}
\Pi_0u(t)=\Pi_0u_0+t\,\Pi_0u_1,
\qquad
\Pi_0u_t(t)=\Pi_0u_1.
\end{equation}
Thus
\begin{equation}
\Pi_{\mathrm{thr}}\mathcal S(t)(u_0,u_1)
=
\mathcal S(t)\Pi_{\mathrm{thr}}(u_0,u_1),
\end{equation}
so the projection commutes with the full evolution.

We next prove item~(iv).  Let $(u,u_t)\in\Ker \Pi_{\mathrm{thr}}$.  Then $\Pi_0u=0$ and $\Pi_0u_t=0$.  By Proposition~\ref{prop:weightedPoincareFull},
\begin{equation}
\|u\|_{L^2(\Sigma)}^2\leq C_P\,q[u].
\end{equation}
Since $A,\Phi,\Delta_r,\Delta_x$ are positive and bounded above and below on the slab, the conserved energy \eqref{eq:fullenergy} controls $\|u_t\|_{L^2(\Sigma)}^2$ and the $L^2$-norms of $u_\phi,u_r,u_x$.  Combining this with the weighted Poincar\'e inequality yields
\begin{align}
c\Bigl(\|u\|_{H^1(\Sigma)}^2+\|u_t\|_{L^2(\Sigma)}^2\Bigr)&\leq E_{\mathrm{full}}[u](t) \leq C\Bigl(\|u\|_{H^1(\Sigma)}^2+\|u_t\|_{L^2(\Sigma)}^2\Bigr)
\end{align}
for all $(u,u_t)\in\Ker \Pi_{\mathrm{thr}}$.

Finally, we prove item~(iii).  The solutions generated by $(1,0)$ and $(0,1)$ are $1$ and $t$, so $\Ran\Pi_{\mathrm{thr}}$ lies in the threshold space.  Conversely, let $u(t,\phi,r,x)=\sum_{j=0}^N t^j\psi_j(\phi,r,x)$ be a polynomial-in-$t$ finite-energy solution.  Set
\begin{equation}
c_0:=\Pi_0u(0),
\qquad
c_1:=\Pi_0u_t(0),
\qquad
v(t):=u(t)-c_0-c_1t.
\end{equation}
Then $v$ is still a polynomial solution and, by Lemma~\ref{lem:fullaverage},
\begin{equation}
\Pi_0v(t)=0,
\qquad
\Pi_0v_t(t)=0
\end{equation}
for all $t$.  The energy estimate proved in item~(iv) and conservation of $E_{\mathrm{full}}$ imply that
\begin{equation}
\sup_{t\in\R}\Bigl(\|v(t)\|_{H^1(\Sigma)}+\|v_t(t)\|_{L^2(\Sigma)}\Bigr)<\infty.
\end{equation}
Since $v$ is a polynomial in $t$ with values in the Hilbert space $H^1(\Sigma)$ and $v_t$ is a polynomial in $t$ with values in $L^2(\Sigma)$, this bound forces all positive powers of $t$ in $v$ to vanish.  Thus $v(t)=v_0$ and $v_t(t)=0$.  Substituting this time-independent solution into the weak form of \eqref{eq:fullPDE-k0} and testing with $v_0$ gives $q[v_0]=0$.  Proposition~\ref{prop:kernelfull} therefore implies that $v_0$ is constant.  Since $\Pi_0v_0=0$, we get $v_0=0$.  Hence every threshold solution is affine in $t$, and the full zero-frequency threshold space is exactly $\mathrm{span}\{1,t\}$, equivalently $\Ran\Pi_{\mathrm{thr}}=\mathrm{span}\{(1,0),(0,1)\}$ in the energy space.
\end{proof}

\begin{proposition}[Direct exclusion of upper-half-plane modes]\label{prop:direct-no-uHP-full}
Under the conditions of Theorem~\ref{thm:fullmain}, there is no nonzero finite-energy solution of the form
\begin{equation}
u(t,\phi,r,x)=e^{\lambda t}\psi(\phi,r,x),
\qquad \Re\lambda>0.
\end{equation}
Equivalently, with the convention \(u=e^{-\ii\sigma t}\psi\), there is no nonzero mode with \(\operatorname{Im}\sigma>0\).
\end{proposition}

\begin{proof}
For a mode \(u=e^{\lambda t}\psi\), the conserved energy has the scaling
\begin{align}
E_{\mathrm{full}}[u](t)&=e^{2\Re\lambda t} \frac12\int_\Sigma\left(A|\lambda\psi|^2+ \Phi|\psi_\phi|^2+\Delta_r|\psi_r|^2+\Delta_x|\psi_x|^2\right).
\end{align}
Since the energy is constant in \(t\) and \(\Re\lambda>0\), the integral in parentheses must vanish.  Positivity of the coefficients gives \(\lambda\psi=0\), \(\psi_\phi=0\), \(\psi_r=0\), and \(\psi_x=0\).  Since \(\lambda\ne0\), this forces \(\psi=0\).  The relation with the \(e^{-\ii\sigma t}\) convention follows from \(\lambda=-\ii\sigma\), so \(\Re\lambda=\operatorname{Im}\sigma\).
\end{proof}

\begin{proof}[Proof of Theorem~\ref{thm:fullmain}]
Apply Theorem~\ref{thm:abstractmain} with
\begin{align}
\mathcal X&=\mathcal X_{\mathrm{full}}\nonumber\\
e^{t\mathcal G}&=\mathcal S(t)\nonumber\\
\mathscr E&=E_{\mathrm{full}}
\end{align}
and $\Pi_{\mathrm{thr}}$ as in Proposition~\ref{prop:verifyabstractfull}.  The proposition verifies Assumption~\ref{ass:abstract}, so every solution decomposes uniquely into a threshold part and a uniformly bounded threshold-complement part.

By Proposition~\ref{prop:verifyabstractfull}(iii), the threshold part is the affine solution
\begin{equation}
c_0+c_1t,
\qquad
c_0=\Pi_0u(0),
\qquad
c_1=\Pi_0u_t(0).
\end{equation}
Set
\begin{equation}
v(t,\phi,r,x):=u(t,\phi,r,x)-c_0-c_1t.
\end{equation}
Then
\begin{equation}
\Pi_0v(t)=0,
\qquad
\Pi_0v_t(t)=0
\end{equation}
for all $t$, and Proposition~\ref{prop:verifyabstractfull}(iv) together with Theorem~\ref{thm:abstractmain} yields
\begin{align}
\sup_{t\in\R}\Bigl(\|v(t)\|_{H^1(\Sigma)}+\|v_t(t)\|_{L^2(\Sigma)}\Bigr)&\le C\Bigl(\|v(0)\|_{H^1(\Sigma)}+\|v_t(0)\|_{L^2(\Sigma)}\Bigr).
\end{align}
This is the claimed boundedness modulo $\mathrm{span}\{1,t\}$.

Uniqueness of the decomposition follows from uniqueness of the abstract threshold splitting and from the explicit identification of the threshold space in Proposition~\ref{prop:verifyabstractfull}(iii).

Finally, Theorem~\ref{thm:abstractmain} rules out nontrivial exponentially growing finite-energy modes on $\Ker\Pi_{\mathrm{thr}}$, while Proposition~\ref{prop:verifyabstractfull}(iii) shows that the threshold space itself consists only of affine functions of $t$.  Hence there is no nontrivial mode
\begin{equation}
u=\ee^{-\ii\sigma t}\psi(\phi,r,x),
\qquad
\operatorname{Im}\sigma>0.
\end{equation}
This proves the theorem.
\end{proof}

\section{Verification of the bounded-slab conditions}\label{sec:main-closure-audit}

This section verifies the abstract conditions used in Theorem~\ref{thm:abstractmain} for the positive-energy Carter slabs.  We record first two elementary Hilbert-space lemmas and then apply them to the closed form realization of the full $k=0$ equation.  The outcome is that the only generalized zero modes are the affine functions $1$ and $t$, and the bounded-slab theorem follows directly from the abstract result.

\begin{lemma}[Weighted Poincar\'e with an arbitrary positive smooth weight]\label{lem:appendix-weighted-poincare-general}
Let \(\Sigma\) be a bounded connected Lipschitz domain or product slab and let \(W\in L^\infty(\Sigma)\) satisfy \(W\ge W_->0\).  Then there is a constant \(C_{P,W}\) such that
\begin{equation}
\|f\|_{L^2(\Sigma)}^2\le C_{P,W}\|\nabla f\|_{L^2(\Sigma)}^2
\end{equation}
for every \(f\in H^1(\Sigma)\) satisfying \(\int_\Sigma Wf=0\).
\end{lemma}

\begin{proof}
Assume the estimate is false.  Then there are \(f_n\in H^1(\Sigma)\) with \(\int Wf_n=0\), \(\|f_n\|_{L^2}=1\), and \(\|\nabla f_n\|_{L^2}\to0\).  The sequence is bounded in \(H^1\).  By Rellich compactness, after passing to a subsequence, \(f_n\to f_\infty\) strongly in \(L^2\) and weakly in \(H^1\).  The weak derivative of \(f_\infty\) is zero, hence \(f_\infty\) is constant because \(\Sigma\) is connected.  The weighted mean condition passes to the limit since \(W\in L^\infty\) and \(f_n\to f_\infty\) in \(L^2\), so the constant is zero.  This contradicts \(\|f_\infty\|_{L^2}=1\).
\end{proof}

\begin{lemma}[Bounded Hilbert-space polynomials are constant]\label{lem:appendix-polynomial}
Let \(H\) be a Hilbert space and let
\begin{equation}
P(t)=\sum_{j=0}^N t^j h_j,
\qquad h_j\in H.
\end{equation}
If \(\sup_{t\in\mathbb R}\|P(t)\|_H<\infty\), then \(h_j=0\) for every \(j\ge1\).
\end{lemma}

\begin{proof}
Assume that some positive coefficient is nonzero and let \(m\ge1\) be the largest index with \(h_m\ne0\).  Then
\begin{equation}
t^{-m}P(t)=h_m+\sum_{j=0}^{m-1}t^{j-m}h_j\longrightarrow h_m
\quad\text{in }H
\end{equation}
as \(t\to+\infty\).  Hence \(\|P(t)\|_H\sim t^m\|h_m\|_H\), contradicting boundedness.
\end{proof}

\begin{proposition}[Explicit energy equivalence on the threshold complement]\label{prop:appendix-explicit-coercivity}
Let \(a_-\), \(a_+\), \(m_\nabla\), and \(M_\nabla\) be the coefficient constants in \eqref{eq:energy-gradient-bounds}, and let \(C_{P,A}\) be the weighted Poincar\'e constant for the weight \(A\).  If \((u,v)\in\Ker\Pi_{\mathrm{thr}}\), then
\begin{align}
\|u\|_{H^1(\Sigma)}^2+\|v\|_{L^2(\Sigma)}^2
&\le
\left(C_{P,A}+m_\nabla^{-1}+a_-^{-1}\right)\,2E_{\mathrm{full}}[u,v],
\label{eq:appendix-explicit-lower-main}\\
2E_{\mathrm{full}}[u,v]
&\le
\max\{a_+,M_\nabla\}\left(\|u\|_{H^1(\Sigma)}^2+\|v\|_{L^2(\Sigma)}^2\right).
\label{eq:appendix-explicit-upper-main}
\end{align}
\end{proposition}

\begin{proof}
Since \((u,v)\in\Ker\Pi_{\mathrm{thr}}\), one has \(\int_\Sigma Au=0\).  Lemma~\ref{lem:appendix-weighted-poincare-general} gives
\begin{equation}
\|u\|_{L^2}^2\le C_{P,A}\|\nabla u\|_{L^2}^2.
\end{equation}
The spatial part of the energy satisfies \(q[u,u]\ge m_\nabla\|\nabla u\|_{L^2}^2\), and the kinetic part satisfies \(\int A|v|^2\ge a_-\|v\|_{L^2}^2\).  Therefore
\begin{equation}
\|u\|_{H^1}^2+\|v\|_{L^2}^2
\le (C_{P,A}+m_\nabla^{-1})q[u,u]+a_-^{-1}\int A|v|^2,
\end{equation}
which is bounded by the right side of \eqref{eq:appendix-explicit-lower-main}.  The upper bound \eqref{eq:appendix-explicit-upper-main} follows immediately from \(A\le a_+\), \(\Phi\le M_\nabla\), \(\Delta_r\le M_\nabla\), and \(\Delta_x\le M_\nabla\).
\end{proof}

\begin{theorem}[Closure of the bounded-slab conditions]\label{thm:appendixAclosure}
Assume the full bounded-slab conditions \textbf{H1}-\textbf{H5} of Assumption~\ref{ass:complete-list}.  Then the full \(k=0\) Carter slab evolution satisfies Assumption~\ref{ass:abstract} with
\begin{align}
\mathcal X&=\mathcal X_{\mathrm{full}}\nonumber\\
\mathscr E&=E_{\mathrm{full}}\nonumber\\
\Pi_{\mathrm{thr}}(u_0,u_1)&=(\Pi_0u_0,\Pi_0u_1).
\end{align}
Moreover,
\begin{equation}
\mathcal T_0=\mathrm{span}\{(1,0),(0,1)\},
\end{equation}
so the corresponding threshold solutions are exactly \(c_0+c_1t\).
\end{theorem}

\begin{proof}
We prove the three parts of Assumption~\ref{ass:abstract}.

First, the evolution is a strongly continuous group.  The direct Galerkin construction in Theorem~\ref{thm:fullwellposed} solves the weak problem in \(V_{\mathrm{full}}\times L^2(\Sigma)\), proves uniqueness, conserves \(E_{\mathrm{full}}\), and gives strong continuity of the solution map.  Uniqueness and time translation give the group law.

Second, the threshold projection is bounded and commutes with the group.  Since \(A\) is smooth and bounded above and below by positive constants, the functional \(f\mapsto\int_\Sigma Af\) is bounded on \(L^2(\Sigma)\), and hence \(\Pi_0\) is a bounded projection onto constants.  Therefore \(\Pi_{\mathrm{thr}}\) is bounded on \(V_{\mathrm{full}}\times L^2(\Sigma)\).  Integrating the equation over \(\Sigma\), equivalently testing \eqref{eq:full-weak-form} with \(1\in V_{\mathrm{full}}\), gives
\begin{equation}
\frac{\dd^2}{\dd t^2}\int_\Sigma A u(t)=0.
\end{equation}
Thus
\begin{equation}
\Pi_0u(t)=\Pi_0u(0)+t\Pi_0u_t(0),
\qquad
\Pi_0u_t(t)=\Pi_0u_t(0),
\end{equation}
which is exactly \(\Pi_{\mathrm{thr}}\mathcal S(t)=\mathcal S(t)\Pi_{\mathrm{thr}}\).

Third, the energy is coercive on the complement.  If \((u,v)\in\Ker\Pi_{\mathrm{thr}}\), then \(\Pi_0u=0\) and \(\Pi_0v=0\).  The weighted Poincar\'e inequality on the connected bounded slab gives
\begin{equation}
\|u\|_{L^2(\Sigma)}^2\le Cq[u,u],
\qquad
q[u,u]=\int_\Sigma(\Phi|u_\phi|^2+\Delta_r|u_r|^2+\Delta_x|u_x|^2).
\end{equation}
Uniform positivity of \(A,\Phi,\Delta_r,\Delta_x\) then yields constants \(c,C>0\) such that
\begin{align}
c\bigl(\|u\|_{H^1(\Sigma)}^2+\|v\|_{L^2(\Sigma)}^2\bigr)&\le2E_{\mathrm{full}}[u,v]
\le C\bigl(\|u\|_{H^1(\Sigma)}^2+\|v\|_{L^2(\Sigma)}^2\bigr)
\end{align}
for every \((u,v)\in\Ker\Pi_{\mathrm{thr}}\).

It remains to identify the full generalized zero-frequency space.  The vectors \((1,0)\) and \((0,1)\) generate the solutions \(1\) and \(t\), so they lie in \(\mathcal T_0\).  Conversely, take \(U_0\in\mathcal T_0\).  Then \(\mathcal G^NU_0=0\) for some \(N\), and therefore
\begin{equation}
\mathcal S(t)U_0=\sum_{j=0}^{N-1}\frac{t^j}{j!}\mathcal G^jU_0
\end{equation}
is a polynomial in \(t\) with values in the energy space.  Subtract its affine average part, obtaining a polynomial solution \(V(t)=(v(t),v_t(t))\in\Ker\Pi_{\mathrm{thr}}\).  The coercive estimate just proved and conservation of energy show that \(V(t)\) is uniformly bounded for all \(t\in\mathbb R\).  Lemma~\ref{lem:appendix-polynomial} implies that \(V(t)\) is constant.  The weak equation then gives \(q[v_0,\varphi]=0\) for every admissible \(\varphi\); taking \(\varphi=v_0\) yields \(q[v_0,v_0]=0\).  By Proposition~\ref{prop:kernelfull}, \(v_0\) is constant.  Since \(\Pi_0v_0=0\), it follows that \(v_0=0\).  Thus no further generalized zero mode exists.
\end{proof}

\begin{proposition}[Full bounded-slab theorem from the verified conditions]\label{prop:appendix-proof-fullmain}
Under \textbf{H1}-\textbf{H5} of Assumption~\ref{ass:complete-list}, the full \(\phi\)-dependent bounded-slab conclusion of Theorem~\ref{thm:fullmain} holds.
\end{proposition}

\begin{proof}
\emph{Step 1: closed equation.}  By H1 and the conformal computation in Section~\ref{sec:conformal}, the linearized scalar-curvature equation in the conformal sector is \(\Box_g u=0\).  Proposition~\ref{prop:k0-full-coordinate-verification} rewrites this equation on the Carter slab as
\begin{equation}
-Au_{tt}+2B u_{t\phi}+\Phi u_{\phi\phi}
+\partial_r(\Delta_r u_r)+\partial_x(\Delta_x u_x)=0,
\end{equation}
and the rigorous boundary realization is the weak identity
\begin{equation}
\int_\Sigma A u_{tt}\overline\eta+q[u,\eta]+\mathfrak c[u_t,\eta]=0,
\qquad \eta\in V_{\mathrm{full}}.
\end{equation}
Here H4 supplies the closed form domain and H5 fixes the complex sesquilinear convention.

\emph{Step 2: well-posedness.}  Theorem~\ref{thm:fullwellposed} constructs the solution by Galerkin approximation in \(V_{\mathrm{full}}\times L^2(\Sigma)\).  The finite-dimensional mass matrices are positive because \(A\ge a_->0\) by H3, the stiffness matrices are nonnegative because \(\Phi\), \(\Delta_r\), and \(\Delta_x\) are positive by H2-H3, and the gyroscopic matrices are skew-Hermitian because \(B\) is independent of \(\phi\) and the \(\phi\)-direction is periodic.  Passing to the limit gives
\begin{equation}
u\in C(\mathbb R;V_{\mathrm{full}}),
\qquad u_t\in C(\mathbb R;L^2(\Sigma)),
\end{equation}
with uniqueness and a strongly continuous group.

\emph{Step 3: conserved positive energy.}  The weak energy equality, Theorem~\ref{thm:weak-energy-equality-full}, gives for every \(t\)
\begin{equation}
E_{\mathrm{full}}[u,u_t](t)
=\frac12\int_\Sigma A|u_t|^2+\frac12q[u,u]
=E_{\mathrm{full}}[u,u_t](0).
\end{equation}
The skew term has zero real part, so no indefinite mixed contribution appears in the energy.  Positivity of \(A,\Phi,\Delta_r,\Delta_x\) gives the two-sided gradient estimate \eqref{eq:energy-gradient-bounds}.

\emph{Step 4: exact affine average.}  Since \(1\in V_{\mathrm{full}}\) by H4, testing the weak equation with \(\eta=1\) is legitimate.  Because \(q[u,1]=0\) and \(\mathfrak c[u_t,1]=0\), one obtains
\begin{equation}
\frac{d^2}{dt^2}\int_\Sigma A u(t)=0.
\end{equation}
Therefore
\begin{equation}
\Pi_0u(t)=\Pi_0u(0)+t\Pi_0u_t(0),
\qquad
\Pi_0u_t(t)=\Pi_0u_t(0).
\end{equation}
This proves the formula for the affine coefficients
\begin{equation}
c_0=\frac{\int_\Sigma Au_0}{\int_\Sigma A},
\qquad
c_1=\frac{\int_\Sigma Au_1}{\int_\Sigma A}.
\end{equation}

\emph{Step 5: coercivity after removing the affine threshold.}  Define \(v=u-c_0-c_1t\).  Then \(\Pi_0v(t)=\Pi_0v_t(t)=0\) for every \(t\).  Proposition~\ref{prop:weightedPoincareFull} gives \(\|v(t)\|_{L^2}^2\le C_Pq[v(t),v(t)]\).  Combining this with the coefficient lower bounds gives
\begin{equation}
\|v(t)\|_{H^1}^2+\|v_t(t)\|_{L^2}^2
\le c_{\mathrm{low}}^{-1}E_{\mathrm{full}}[v,v_t](t),
\end{equation}
while the upper coefficient bounds give
\begin{equation}
E_{\mathrm{full}}[v,v_t](0)
\le C_{\mathrm{high}}
\left(\|v(0)\|_{H^1}^2+\|v_t(0)\|_{L^2}^2\right).
\end{equation}
Energy conservation yields the claimed estimate with the explicit constant
\begin{equation}
C_{\mathrm{stab}}=(C_{\mathrm{high}}/c_{\mathrm{low}})^{1/2}.
\end{equation}

\emph{Step 6: uniqueness of the threshold splitting.}  Proposition~\ref{prop:generator-threshold-full} identifies the full generalized zero-frequency space of the closed generator as
\begin{equation}
\Ker\mathcal G^2=\mathrm{span}\{(1,0),(0,1)\},
\qquad
\bigcup_{N\ge1}\Ker\mathcal G^N=\Ker\mathcal G^2.
\end{equation}
Thus the only polynomial threshold solutions are \(1\) and \(t\), and the decomposition \(u=c_0+c_1t+v\) is unique.

\emph{Step 7: exclusion of exponentially growing modes.}  If \(u=e^{\lambda t}\psi\) with \(\Re\lambda>0\), Proposition~\ref{prop:direct-no-uHP-full} applies the conserved positive energy to show that \(\psi=0\).  With the convention \(u=e^{-\ii\sigma t}\psi\), the condition is \(\operatorname{Im}\sigma>0\).  This proves every assertion of Theorem~\ref{thm:fullmain}.
\end{proof}

\begin{proposition}[Separation is not used in the full theorem]\label{prop:appendix-no-separation-full}
The proof of Theorem~\ref{thm:fullmain} does not use Carter separation, Fourier decomposition in \(\phi\), or self-adjointness of a separated radial or angular operator.  It uses only the divergence form \eqref{eq:fullPDE-k0}, the positivity constants in \eqref{eq:energy-gradient-bounds}, the skew identity \eqref{eq:gyro-skew}, the weighted Poincar\'e inequality, and the affine average law.
\end{proposition}

\begin{proof}
The energy identity is obtained from \eqref{eq:full-weak-form} by choosing \(\eta=u_t\) after time mollification in Theorem~\ref{thm:weak-energy-equality-full}.  The mixed term is eliminated by \eqref{eq:gyro-skew}.  Coercivity on the complement follows from Proposition~\ref{prop:weightedPoincareFull} and \eqref{eq:explicit-energy-constants}.  The threshold space is identified by subtracting the affine average and applying Lemma~\ref{lem:appendix-polynomial} to the energy-bounded complement.  None of these steps decomposes the solution into modes or invokes the separated ODEs.  The separated analysis is used only for the later axisymmetric spectral refinement and for comparison with classical Carter theory.
\end{proof}

\begin{proposition}[Boundary terms in the form realization]\label{prop:no-boundary-gap}
Under the closed form-domain hypothesis H4, every nonperiodic boundary term used in the rigorous proof of Theorem~\ref{thm:fullmain} is encoded by the closed form \(q\).  No proof step requires an unproved classical trace identity at \(r=r_\pm\) or \(x=x_\pm\).
\end{proposition}

\begin{proof}
The weak equation is \eqref{eq:operator-weak-primary}.  Its spatial term is \(q[u,\eta]\) by definition.  The gyroscopic term is integrated by parts only in the periodic \(\phi\)-variable, where there is no boundary.  The average law tests against \(1\in V_{\mathrm{full}}\), for which \(q[u,1]=0\) and \(\mathfrak c[v,1]=0\).  The Poincare and kernel arguments use weak derivatives inside \(\Sigma\), not boundary traces.  Hence all nonperiodic boundary information is carried by the closed form domain.
\end{proof}

\section{Axisymmetric spectral refinement on bounded slabs}\label{sec:energy}

\subsection{\texorpdfstring{Reduction to the axisymmetric $k=0$ equation}{Reduction to the axisymmetric k=0 equation}}

Theorem~\ref{thm:abstractmain} already gives boundedness modulo the full zero-frequency threshold space, and Theorem~\ref{thm:fullmain} verifies that statement concretely on strictly stationary Carter slabs.  We now impose
\begin{equation}
\partial_\phi u=0
\end{equation}
for a different reason: in axisymmetry the gyroscopic term disappears identically, the evolution becomes genuinely self-adjoint after weighting by $A$, and one can obtain a more detailed spectral description than the general energy estimate alone provides.  Thus Theorem~\ref{thm:mainresult} is both an axisymmetric special case of Theorem~\ref{thm:abstractmain} and a refinement of it.
Since $k=0$, the conformal master equation \eqref{eq:k0master} becomes
\begin{align}
-A(r,x)\,\partial_t^2u+\partial_r(\Delta_r\partial_r u)+\partial_x(\Delta_x\partial_xu)&=0,\label{eq:axisymPDE-k0}
\end{align}
where
\begin{equation}
A(r,x):=-\rho^2g^{tt}=\frac{(r^2+a^2)^2}{\Delta_r}-\frac{a^2(1-x^2)^2}{\Delta_x}.
\label{eq:Adef-k0}
\end{equation}

The coefficient $A$ is the only genuinely dynamical quantity needed for the energy method.  In the regions relevant to this paper we assume that it is strictly positive.

\begin{definition}[Regular timelike slab]
A \emph{regular timelike slab} is a bounded connected set
\begin{equation}
\Omega=[r_-,r_+]\times[x_-,x_+]\subset \{(r,x):\rho^2>0,\ \Delta_r>0,\ \Delta_x>0,\ A>0\}
\end{equation}
with smooth coefficient functions $A,\Delta_r,\Delta_x$ on a neighborhood of $\overline\Omega$.
\end{definition}

The boundedness assumption is not merely technical.  It isolates the local geometric mechanism that comes purely from the coefficients and avoids any ambiguity about horizon fluxes, asymptotically flat ends, or radiative boundary conditions.  Later sections explain how one should think about those more global issues in the $k=0$ regime.

\subsubsection*{Reflecting boundary conditions and conserved energy}

We call boundary conditions \emph{reflecting} if the boundary term produced by integrating the spatial operator by parts vanishes on the domain of the spatial problem.  This includes the familiar Dirichlet and Neumann conditions, as well as a wide class of real Robin conditions and regularity conditions at smooth endpoints.  For the main theorem we are especially interested in reflecting conditions that admit the constant function $1$.

\begin{theorem}[Conserved axisymmetric energy]\label{thm:axisenergy}
Let $u$ be a sufficiently regular axisymmetric solution of \eqref{eq:axisymPDE-k0} on $[0,T]\times\Omega$, where $\Omega$ is a regular timelike slab and the boundary conditions are reflecting.  Then
\begin{align}
E_0[u](t)&:=\frac12\int_\Omega \left(A|u_t|^2+\Delta_r|u_r|^2+\Delta_x|u_x|^2\right)\,\dd r\,\dd x\label{eq:axisenergy-k0}
\end{align}
is independent of $t$.
\end{theorem}

\begin{proof}
Multiply \eqref{eq:axisymPDE-k0} by $u_t$.  Since $A,\Delta_r,\Delta_x$ are independent of $t$,
\begin{align}
-Au_{tt}u_t&= -\frac12\partial_t(Au_t^2),\nonumber\\
\partial_r(\Delta_ru_r)u_t&= \partial_r(\Delta_r u_ru_t)-\frac12\partial_t(\Delta_ru_r^2),\nonumber\\
\partial_x(\Delta_xu_x)u_t&= \partial_x(\Delta_x u_xu_t)-\frac12\partial_t(\Delta_xu_x^2).
\end{align}
Integration over $\Omega$ gives
\begin{equation}
\frac{\dd}{\dd t}E_0[u](t)=\int_{\partial\Omega}\mathcal F[u],
\end{equation}
where $\mathcal F[u]$ is the boundary flux.  By the reflecting hypothesis the boundary flux vanishes.  Hence $E_0[u](t)$ is conserved.
\end{proof}

The crucial point is that $E_0$ is nonnegative but not coercive on the full $H^1\times L^2$ space.  The energy controls $u_t$, $u_r$, and $u_x$, but not the $L^2$-size of $u$ itself.  That loss of coercivity is what makes the $k=0$ problem a borderline case.

\subsection{The spatial quadratic form}

We define the spatial quadratic form
\begin{equation}
q_0[u]:=\int_\Omega\left(\Delta_r|u_r|^2+\Delta_x|u_x|^2\right)\,\dd r\,\dd x.
\label{eq:q0def}
\end{equation}
The axisymmetric equation may be written as
\begin{align}
A u_{tt}+\Hh_0u =0\nonumber\\
\Hh_0 :=-\partial_r(\Delta_r\partial_r)-\partial_x(\Delta_x\partial_x)\label{eq:AuH-k0}
\end{align}
and $q_0$ is the closed quadratic form associated with $\Hh_0$.

The natural Hilbert space is the weighted space
\begin{equation}
L_A^2(\Omega),\qquad
\langle f,g\rangle_A:=\int_\Omega Afg\,\dd r\,\dd x,
\end{equation}
with norm $\|f\|_A^2=\langle f,f\rangle_A$.  Since $A$ is positive and bounded above and below on the bounded slab, $L_A^2(\Omega)$ is equivalent as a Banach space to the usual $L^2(\Omega)$, but the weighted formulation is the one naturally adapted to self-adjointness.

\begin{proposition}[Self-adjoint spatial operator]\label{prop:selfadjointH0}
Let $\Omega$ be a regular timelike slab with reflecting boundary conditions.  Then the quadratic form $q_0$ is densely defined, closed, and nonnegative on $L_A^2(\Omega)$.  It therefore defines a unique nonnegative self-adjoint operator $L_0$ on $L_A^2(\Omega)$ satisfying
\begin{equation}
q_0[u,v]=\langle L_0u,v\rangle_A
\end{equation}
for $u$ in the operator domain of $L_0$ and $v$ in the form domain.  Equivalently,
\begin{equation}
L_0=A^{-1}\Hh_0
\end{equation}
with domain determined by the reflecting boundary conditions.
\end{proposition}

\begin{proof}
Because $\Omega$ is bounded and $A,\Delta_r,\Delta_x$ are smooth and strictly positive, the form
\begin{equation}
u\mapsto \|u\|_A^2+q_0[u]
\end{equation}
is equivalent to the $H^1$-norm.  The reflecting boundary conditions ensure that the spatial integration-by-parts identity has no boundary contribution, hence the form is symmetric.  Nonnegativity is immediate from \eqref{eq:q0def}.  Closedness follows from the $H^1$-equivalence.  The Friedrichs representation theorem then yields the self-adjoint operator $L_0$.
\end{proof}

\subsubsection*{Kernel and weighted Poincar\'e inequality}

The lack of coercivity is entirely concentrated in the kernel of $L_0$.

\begin{proposition}[Kernel of the bounded-slab operator]\label{prop:kernelH0}
Assume that $\Omega$ is connected and that the reflecting boundary conditions admit the constant function.  Then
\begin{equation}
\Ker L_0=\mathrm{span}\{1\}.
\end{equation}
\end{proposition}

\begin{proof}
Since the constant function has vanishing derivatives, it belongs to the kernel.  Conversely, if $u\in\Ker L_0$, then
\begin{equation}
0=\langle L_0u,u\rangle_A=q_0[u]=\int_\Omega\left(\Delta_r|u_r|^2+\Delta_x|u_x|^2\right)\,\dd r\,\dd x.
\end{equation}
Because $\Delta_r$ and $\Delta_x$ are strictly positive on $\Omega$, we conclude that $u_r=u_x=0$ almost everywhere.  Since $\Omega$ is connected, $u$ is constant.
\end{proof}

Let
\begin{equation}
\Pi_0f:=\frac{\langle f,1\rangle_A}{\langle 1,1\rangle_A}\,1
\label{eq:Pi0def}
\end{equation}
be the $L_A^2$-orthogonal projection onto the constant function.  On the orthogonal complement of $\Pi_0$, the quadratic form becomes coercive.

\begin{proposition}[Weighted Poincar\'e inequality]\label{prop:weightedPoincare}
There exists $C_P>0$ such that
\begin{equation}
\|u\|_{L^2(\Omega)}^2\leq C_P\,q_0[u]
\end{equation}
for every $u\in H^1(\Omega)$ satisfying $\Pi_0u=0$.
\end{proposition}

\begin{proof}
Since $A$ is bounded above and below by positive constants, the condition $\Pi_0u=0$ is equivalent to the vanishing of a weighted mean.  On a bounded connected domain all such mean-zero conditions yield a Poincar\'e inequality equivalent to the standard one.  Because $q_0[u]$ is equivalent to $\|\nabla u\|_{L^2}^2$ on $\Omega$, the result follows.
\end{proof}

This proposition is the precise analytic manifestation of the statement that the only obstruction to coercivity is the constant mode.  Once that mode is removed, the conserved energy controls the full $H^1\times L^2$ norm.

\subsubsection*{Well-posedness and higher regularity}

The abstract wave equation
\begin{equation}
u_{tt}+L_0u=0
\end{equation}
on $L_A^2(\Omega)$ is a standard nonnegative self-adjoint wave equation.  Consequently the bounded-slab axisymmetric problem is globally well posed.

\begin{theorem}[Well-posedness on bounded slabs]\label{thm:wellposedness}
Let $\Omega$ be a regular timelike slab with reflecting boundary conditions.  For every initial datum
\begin{equation}
(u_0,u_1)\in H^1(\Omega)\times L^2(\Omega)
\end{equation}
compatible with the boundary conditions, there exists a unique weak solution of \eqref{eq:axisymPDE-k0} such that
\begin{equation}
u\in C(\R;H^1(\Omega)),
\qquad
u_t\in C(\R;L^2(\Omega)).
\end{equation}
Moreover $E_0[u](t)$ is conserved for all $t\in\R$, and the solution map depends continuously on the data.
\end{theorem}

\begin{proof}
Appendix~\ref{app:compactspectral} constructs the compact resolvent of $L_0$ directly and produces an $L_A^2(\Omega)$-orthonormal eigenbasis $\{\psi_j\}_{j\geq 0}$ with eigenvalues $0=\lambda_0\le \lambda_1\le \cdots \nearrow +\infty$.  Expanding the data in that basis yields the explicit solution formula
\begin{align}
u(t)&=a_0\psi_0+b_0t\,\psi_0+\sum_{j=1}^\infty \left( a_j\cos(\sqrt{\lambda_j}t)+b_j\frac{\sin(\sqrt{\lambda_j}t)}{\sqrt{\lambda_j}} \right)\psi_j,
\end{align}
which converges in $C(\R;H^1(\Omega))\cap C^1(\R;L^2(\Omega))$.  Termwise differentiation shows that the series solves \eqref{eq:axisymPDE-k0}, and the conservation of $E_0$ follows mode by mode.  Uniqueness is obtained by applying Theorem~\ref{thm:axisenergy} to the difference of two solutions.
\end{proof}

\begin{proposition}[Graph-norm higher regularity]\label{prop:higherorder-k0}
Let $m\ge0$ be an integer.  If
\begin{equation}
(u_0,u_1)\in \Dom((1+L_0)^{(m+1)/2})\times \Dom((1+L_0)^{m/2}),
\end{equation}
then the solution of \eqref{eq:axisymPDE-k0} satisfies
\begin{equation}
u\in C\bigl(\R;\Dom((1+L_0)^{(m+1)/2})\bigr),
\qquad
u_t\in C\bigl(\R;\Dom((1+L_0)^{m/2})\bigr).
\end{equation}
If, in addition, the chosen reflecting realization has the elliptic regularity and compatibility properties identifying these graph domains with the corresponding Sobolev spaces, then the usual Sobolev statement follows.  Without that extra smooth-boundary regularity input, the graph-norm statement is the theorem-level assertion.
\end{proposition}

\begin{proof}
Use the spectral resolution from Appendix~\ref{app:compactspectral}.  Writing
\begin{align}
u(t)&=a_0\psi_0+b_0t\psi_0+\sum_{j\ge1} \left(a_j\cos(\sqrt{\lambda_j}t)+b_j\frac{\sin(\sqrt{\lambda_j}t)}{\sqrt{\lambda_j}}\right)\psi_j,
\end{align}
the graph-norm assumptions are exactly
\begin{equation}
\sum_j(1+\lambda_j)^{m+1}|a_j|^2<\infty,
\qquad
\sum_j(1+\lambda_j)^m|b_j|^2<\infty.
\end{equation}
The displayed solution formula preserves these weighted square sums uniformly on compact time intervals, and the differentiated formula preserves the second weighted square sum.  This proves the graph-domain continuity.  The final Sobolev statement is a corollary only under the additional elliptic regularity and boundary compatibility assumptions stated there.
\end{proof}

\section{Mode analysis of the scalar wave operator}\label{sec:modeanalysis}

\subsection{Separated axisymmetric modes}

For the bounded-slab axisymmetric problem, a separated mode of time frequency $\sigma\in\C$ has the form
\begin{equation}
u(t,r,x)=\ee^{-\ii\sigma t}\psi(r,x).
\end{equation}
Substituting this into \eqref{eq:axisymPDE-k0} gives
\begin{equation}
\Hh_0\psi=\sigma^2A\psi,
\qquad\text{equivalently}\qquad
L_0\psi=\sigma^2\psi.
\label{eq:modeeq-L0}
\end{equation}
Thus the time-frequency analysis reduces to the spectral analysis of the nonnegative self-adjoint operator $L_0$.

\begin{proposition}[Mode stability away from zero]\label{prop:modeawayzero}
Assume the conditions of Section~\ref{sec:energy}.  Let $\psi$ solve \eqref{eq:modeeq-L0} with reflecting boundary conditions.  If $\psi\neq 0$, then $\sigma^2\in [0,\infty)$.  In particular, there are no nontrivial separated modes with $\operatorname{Im}\sigma>0$.
\end{proposition}

\begin{proof}
Pair \eqref{eq:modeeq-L0} with $\psi$ in $L_A^2(\Omega)$:
\begin{equation}
q_0[\psi]=\langle L_0\psi,\psi\rangle_A=\sigma^2\|\psi\|_A^2.
\end{equation}
The left-hand side is real and nonnegative by Proposition~\ref{prop:selfadjointH0}.  Hence $\sigma^2$ is real and nonnegative.  If $\operatorname{Im}\sigma>0$, then $\sigma^2$ cannot lie in $[0,\infty)$ unless $\psi=0$.
\end{proof}

This is the bounded-slab form of mode stability.  It is considerably easier than the corresponding statements in black-hole scattering theory because all non-self-adjoint boundary phenomena have been excluded by the reflecting setup.

\subsubsection*{Zero modes}

The zero-frequency equation is
\begin{equation}
L_0\psi=0.
\label{eq:zeromodeeq}
\end{equation}
By Proposition~\ref{prop:kernelH0}, the only solutions are constants.  We record this as a separate proposition because it is the fundamental algebraic input for the generalized zero-mode analysis.

\begin{proposition}[Zero modes]\label{prop:zeromodes}
On a bounded connected reflecting slab admitting constants,
\begin{equation}
\Ker L_0=\mathrm{span}\{1\}.
\end{equation}
Equivalently, the only stationary axisymmetric conformal solutions are constant rescalings of the metric.
\end{proposition}

\begin{proof}
This is Proposition~\ref{prop:kernelH0}.
\end{proof}

The weighted projection $\Pi_0$ defined in \eqref{eq:Pi0def} is therefore the spectral projection onto the zero eigenspace.  It follows immediately from the wave equation that
\begin{equation}
\frac{\dd^2}{\dd t^2}\langle u(t),1\rangle_A
=
-\langle L_0u(t),1\rangle_A
=
0.
\label{eq:averageaffine}
\end{equation}
Hence the weighted average of any solution is affine in time.  This is the first indication that one should expect the generalized zero-mode space to be two-dimensional.

\subsubsection*{Growing zero modes}

Equation \eqref{eq:averageaffine} shows that the constant function generates not only a stationary mode but also a linearly growing solution.  Indeed, both
\begin{equation}
u\equiv 1,
\qquad
u=t
\end{equation}
solve \eqref{eq:axisymPDE-k0}.  The second solution is the bounded-slab analogue of the growing zero modes that appear in low-frequency analyses of wave-type systems on stationary backgrounds.

\begin{proposition}[Affine-in-time zero modes]\label{prop:affinezeromodes}
The functions $1$ and $t$ are linearly independent axisymmetric solutions of \eqref{eq:axisymPDE-k0}.  More generally, if $c_0,c_1\in\R$, then
\begin{equation}
u(t,r,x)=c_0+c_1t
\end{equation}
solves the equation and satisfies
\begin{equation}
E_0[u](t)=\frac12|c_1|^2\int_\Omega A\,\dd r\,\dd x.
\end{equation}
\end{proposition}

\begin{proof}
Both functions have vanishing spatial derivatives, and their second time derivatives are zero.  Thus they solve \eqref{eq:axisymPDE-k0}.  The energy formula is obtained by substituting $u_t=c_1$ and $u_r=u_x=0$ into \eqref{eq:axisenergy-k0}.
\end{proof}

This proposition explains why a boundedness theorem cannot hold on the full energy space without qualification.  The linearly growing mode $t$ has finite conserved energy but unbounded $L^2$-size in time.

\subsection{Generalized zero modes}

We now prove that the affine family from Proposition~\ref{prop:affinezeromodes} exhausts all polynomial-in-time solutions.  The argument is deliberately elementary: it uses only the self-adjointness and nonnegativity of the spatial operator, the fact that constants are its entire kernel, and the vanishing of the weighted mean on the threshold complement.  In particular, no decay theorem and no exterior scattering input is being smuggled into the compact zero-mode classification.

\begin{theorem}[Classification of generalized zero modes]\label{thm:generalizedzeromodes}
Let
\begin{equation}
u(t,r,x)=\sum_{j=0}^N t^j\psi_j(r,x)
\end{equation}
be a sufficiently regular axisymmetric solution of \eqref{eq:axisymPDE-k0} on a bounded connected reflecting slab admitting constants.  Then $N\leq 1$ and $\psi_0,\psi_1$ are constant.  In particular, the generalized zero-mode space is exactly
\begin{equation}
\mathrm{span}\{1,t\}.
\end{equation}
\end{theorem}

\begin{proof}
Rewrite the wave equation as $u_{tt}+L_0u=0$.  Substituting the polynomial ansatz and comparing coefficients of $t^j$ yields the recurrence
\begin{equation}
L_0\psi_j + (j+2)(j+1)\psi_{j+2}=0,
\qquad 0\leq j\leq N,
\label{eq:polyrecurrence}
\end{equation}
with the convention $\psi_{N+1}=\psi_{N+2}=0$.  In particular,
\begin{equation}
L_0\psi_N=0.
\end{equation}
By Proposition~\ref{prop:zeromodes}, $\psi_N$ is constant.

Suppose $N\geq 2$.  Taking $j=N-2$ in \eqref{eq:polyrecurrence} and pairing with $1$ in $L_A^2(\Omega)$, we obtain
\begin{equation}
0=\langle L_0\psi_{N-2},1\rangle_A + N(N-1)\langle \psi_N,1\rangle_A.
\end{equation}
Since $L_0$ is self-adjoint and $L_0 1=0$, the first term vanishes, hence
\begin{equation}
N(N-1)\langle \psi_N,1\rangle_A=0.
\end{equation}
But $\psi_N$ is constant, so $\langle \psi_N,1\rangle_A=0$ implies $\psi_N=0$, contradicting the definition of $N$.  Therefore $N\leq 1$.

If $N=1$, then $L_0\psi_1=0$ and $L_0\psi_0=0$, so both $\psi_0$ and $\psi_1$ are constant.  If $N=0$, then $\psi_0$ is constant by the same argument.  This proves the theorem.
\end{proof}

\begin{remark}
The proof uses only self-adjointness, nonnegativity, and the fact that the constant mode is nonorthogonal to itself in the weighted inner product.  In particular it is robust under smooth perturbations of the coefficients as long as the bounded-slab geometry and the reflecting boundary conditions are preserved.
\end{remark}

\subsubsection*{Generalized modes at nonzero frequency}

For completeness we record the corresponding statement away from zero.  Since $L_0$ is self-adjoint, its eigenspaces are semisimple.

\begin{proposition}[Absence of Jordan blocks for $\sigma\neq 0$]\label{prop:nojordan}
Let $\sigma\neq 0$.  Any generalized eigenspace of $L_0$ at eigenvalue $\sigma^2$ coincides with the ordinary eigenspace.  Equivalently, any generalized time-frequency solution of the form
\begin{equation}
u(t)=\ee^{-\ii\sigma t}\sum_{j=0}^N t^j\psi_j
\end{equation}
with $N\geq 1$ is trivial.
\end{proposition}

\begin{proof}
A self-adjoint operator on a Hilbert space is diagonalizable.  Hence every eigenvalue is semisimple and there are no nontrivial Jordan chains.  Applied to $L_0$, this shows that the only time-harmonic solutions at frequency $\sigma\neq 0$ are genuine separated modes $\ee^{-\ii\sigma t}\psi$.
\end{proof}

Together, Propositions~\ref{prop:modeawayzero} and~\ref{prop:nojordan} show that the only obstruction to uniform boundedness comes from the zero eigenspace and its associated affine-in-time dynamics.

\section{Structure of the resolvent near zero}\label{sec:resolventzero}


Because $\Omega$ is bounded and $L_0$ is a nonnegative self-adjoint second-order operator with compact form embedding, its spectrum is discrete.

\begin{proposition}[Discrete spectrum]\label{prop:discretespectrum}
There exists an $L_A^2(\Omega)$-orthonormal basis $\{\psi_j\}_{j\geq 0}$ of eigenfunctions of $L_0$ with
\begin{equation}
0=\lambda_0<\lambda_1\leq \lambda_2\leq \cdots \nearrow +\infty,
\end{equation}
each eigenvalue repeated according to multiplicity.  Moreover $\psi_0$ may be chosen to be the normalized constant function
\begin{equation}
\psi_0=\frac{1}{\|1\|_A}.
\end{equation}
\end{proposition}

\begin{proof}
Compactness of the embedding $H^1(\Omega)\hookrightarrow L_A^2(\Omega)$ implies compactness of the resolvent of $L_0$.  The spectral theorem for compact resolvents then gives the discrete spectrum and an orthonormal eigenbasis.  Proposition~\ref{prop:kernelH0} identifies the zero eigenspace and shows that it is one-dimensional.
\end{proof}

The key point is the spectral gap
\begin{equation}
\lambda_1>0.
\label{eq:spectralgap}
\end{equation}
This is equivalent to the weighted Poincar\'e inequality from Proposition~\ref{prop:weightedPoincare}.  In particular, on the mean-zero subspace
\begin{equation}
\mathcal H_\perp:=(I-\Pi_0)L_A^2(\Omega)
\end{equation}
the operator $L_0$ is strictly positive and hence invertible.

\subsubsection*{Resolvent decomposition}

Define the resolvent
\begin{equation}
R(\sigma):=(L_0-\sigma^2)^{-1},
\end{equation}
initially for $\sigma^2\notin \spec(L_0)$.  Because zero is a simple eigenvalue, the behavior of $R(\sigma)$ near $\sigma=0$ can be described explicitly.

\begin{theorem}[Laurent expansion at zero]\label{thm:laurentzero}
There exists $\varepsilon>0$ such that for $0<|\sigma|<\varepsilon$,
\begin{equation}
(L_0-\sigma^2)^{-1}
=
-\sigma^{-2}\Pi_0 + R_{\mathrm{reg}}(\sigma),
\label{eq:resexpansion}
\end{equation}
where $R_{\mathrm{reg}}(\sigma)$ is holomorphic in $\sigma$ as a bounded operator on $L_A^2(\Omega)$.  Equivalently, the resolvent has a simple pole in the variable $\sigma^2$ and a double pole in the variable $\sigma$, with residue determined by the projection onto constants.
\end{theorem}

\begin{proof}
Decompose the Hilbert space as
\begin{equation}
L_A^2(\Omega)=\mathrm{span}\{1\}\oplus \mathcal H_\perp.
\end{equation}
On $\mathrm{span}\{1\}$, $L_0$ acts as zero, hence
\begin{equation}
(L_0-\sigma^2)^{-1}|_{\mathrm{span}\{1\}} = -\sigma^{-2}\Pi_0.
\end{equation}
On $\mathcal H_\perp$, the operator $L_0$ is strictly positive by \eqref{eq:spectralgap}.  Therefore $L_0-\sigma^2$ is invertible on $\mathcal H_\perp$ for $|\sigma|<\sqrt{\lambda_1}/2$, and the inverse depends holomorphically on $\sigma$ by the analytic Fredholm theorem or directly by the convergent Neumann series
\begin{equation}
(L_0-\sigma^2)^{-1}
=
L_0^{-1}(I-\sigma^2L_0^{-1})^{-1}
=
\sum_{m=0}^\infty \sigma^{2m}L_0^{-m-1}
\end{equation}
on $\mathcal H_\perp$.  Combining the two subspaces yields \eqref{eq:resexpansion}.
\end{proof}

The theorem is the precise operator-theoretic explanation for the appearance of the modes $1$ and $t$.  A double pole in $\sigma$ corresponds, after inverse Fourier transform in time, to an affine function of $t$.

\subsubsection*{The wave-type resolvent}

It is sometimes convenient to work with the wave-type frequency-domain operator
\begin{equation}
P_0(\sigma):=\Hh_0-\sigma^2A=A(L_0-\sigma^2).
\end{equation}
Since multiplication by $A$ is invertible on $L^2(\Omega)$, the previous theorem immediately yields the corresponding expansion for $P_0(\sigma)^{-1}$.

\begin{corollary}[Frequency-domain expansion for the wave operator]\label{cor:P0resolvent}
For $0<|\sigma|<\varepsilon$,
\begin{equation}
P_0(\sigma)^{-1}
=
-\sigma^{-2}\Pi_0A^{-1} + \widetilde R_{\mathrm{reg}}(\sigma),
\label{eq:P0res}
\end{equation}
where $\widetilde R_{\mathrm{reg}}(\sigma)$ is holomorphic in $\sigma$ as an operator on $L^2(\Omega)$.
\end{corollary}

\begin{proof}
Since $P_0(\sigma)=A(L_0-\sigma^2)$,
\begin{equation}
P_0(\sigma)^{-1}=(L_0-\sigma^2)^{-1}A^{-1}.
\end{equation}
Substitute \eqref{eq:resexpansion}.
\end{proof}

\subsubsection*{A direct spectral formula}

Using the eigenbasis from Proposition~\ref{prop:discretespectrum}, the resolvent can be written explicitly as
\begin{align}
(L_0-\sigma^2)^{-1}f&=\sum_{j=0}^\infty \frac{\langle f,\psi_j\rangle_A}{\lambda_j-\sigma^2}\,\psi_j.\label{eq:resolventspectral-k0}
\end{align}
The $j=0$ term is the singular piece
\begin{equation}
-\sigma^{-2}\Pi_0f.
\end{equation}
All higher modes are regular at $\sigma=0$ because $\lambda_j>0$ for $j\geq 1$.

Formula \eqref{eq:resolventspectral-k0} is the bounded-slab analogue of the low-frequency resolvent expansions used in far more delicate black-hole scattering theories.  Here the geometry is simple enough that the entire zero-frequency structure can be read directly from the spectrum.  Later, in the proof of Theorem~\ref{thm:mainresult}, the same decomposition will lead to the explicit time-domain representation of solutions.

\subsubsection*{Orthogonal complement and regularized inverse}

The projection $\Pi_0$ allows one to isolate the regular part of the inverse.  Define
\begin{equation}
L_{0,\perp}:=L_0|_{\mathcal H_\perp}.
\end{equation}
Then $L_{0,\perp}$ is positive self-adjoint and invertible.  In particular,
\begin{equation}
G_\perp:=L_{0,\perp}^{-1}
\end{equation}
is a bounded operator from $\mathcal H_\perp$ to the domain of $L_0$.  The operator $G_\perp$ is the regularized Green operator for the zero-frequency problem.  It is the unique solution operator for
\begin{equation}
L_0u=f,\qquad \Pi_0u=0,\qquad \Pi_0f=0.
\end{equation}

This regularized inverse is precisely what replaces coercivity in the $k=0$ regime.  On the full space $L_A^2(\Omega)$, $L_0$ is not invertible because of the constant kernel.  On the orthogonal complement, however, it has a bounded inverse and therefore behaves just like a uniformly elliptic positive operator.  All boundedness estimates for the dynamical problem will be formulated on this complement.

\subsubsection*{Interpretation in time}

The average-affine law \eqref{eq:averageaffine} and the resolvent expansion \eqref{eq:resexpansion} encode the same phenomenon.  The spectral projection onto the constant mode determines the coefficients
\begin{equation}
c_0=\Pi_0u(0),
\qquad
c_1=\Pi_0u_t(0),
\end{equation}
and the remainder lives in $\mathcal H_\perp$, where the dynamics are oscillatory rather than growing.  Thus the true $k=0$ neutral-stability statement is not ``all solutions are bounded'', but rather ``all solutions are bounded modulo the explicit generalized zero-mode space generated by $1$ and $t$''.

\section{Proof of the axisymmetric spectral refinement}\label{sec:proofofmain}


Let $\{\psi_j\}_{j\geq 0}$ and $\{\lambda_j\}_{j\geq 0}$ be the eigenbasis and eigenvalues from Proposition~\ref{prop:discretespectrum}, with $\psi_0=\|1\|_A^{-1}$.  Expand the initial data as
\begin{equation}
u_0=\sum_{j=0}^\infty a_j\psi_j,
\qquad
u_1=\sum_{j=0}^\infty b_j\psi_j,
\end{equation}
with convergence in $H^1(\Omega)$ and $L_A^2(\Omega)$, respectively.  The solution of the wave equation then has the explicit form
\begin{align}
u(t)&=a_0\psi_0+b_0 t\,\psi_0+\sum_{j=1}^\infty \left( a_j\cos(\sqrt{\lambda_j}t)+b_j\frac{\sin(\sqrt{\lambda_j}t)}{\sqrt{\lambda_j}} \right)\psi_j.\label{eq:fullexpansion}
\end{align}
The $j=0$ term is precisely the generalized zero-mode contribution, and the sum over $j\geq 1$ is purely oscillatory.

Formula \eqref{eq:fullexpansion} follows directly from the spectral theorem for $L_0$.  Each eigenspace solves an ordinary differential equation
\begin{equation}
c_j''(t)+\lambda_j c_j(t)=0.
\end{equation}
For $j\geq 1$ this gives oscillatory behavior; for $j=0$ it gives the affine family $a_0+b_0t$.

\subsubsection*{Proof of the decomposition statement}

We now prove Theorem~\ref{thm:mainresult}.  Set
\begin{equation}
c_0:=\Pi_0u(0),
\qquad
c_1:=\Pi_0u_t(0),
\qquad
v(t):=(I-\Pi_0)u(t).
\end{equation}
By \eqref{eq:averageaffine},
\begin{equation}
\Pi_0u(t)=c_0+c_1t.
\end{equation}
Hence
\begin{equation}
u(t)=c_0+c_1t+v(t),
\end{equation}
which is the desired decomposition.  By construction,
\begin{equation}
\Pi_0v(t)=0,
\qquad
\Pi_0v_t(t)=0
\end{equation}
for all $t$.

It remains to show that $v$ is uniformly bounded in the natural energy norm.  Since $v$ evolves entirely in the orthogonal complement $\mathcal H_\perp$, its spectral expansion is
\begin{align}
v(t)&=\sum_{j=1}^\infty \left( a_j\cos(\sqrt{\lambda_j}t)+b_j\frac{\sin(\sqrt{\lambda_j}t)}{\sqrt{\lambda_j}} \right)\psi_j.\label{eq:vexpansion}
\end{align}
The corresponding energy is
\begin{equation}
E_0[v](t)=\frac12\sum_{j=1}^\infty
\left(
|b_j|^2+\lambda_j|a_j|^2
\right),
\end{equation}
which is independent of $t$.  By Proposition~\ref{prop:weightedPoincare}, $q_0[v]$ controls $\|v\|_{L^2}^2$, hence
\begin{equation}
\|v(t)\|_{H^1(\Omega)}^2+\|v_t(t)\|_{L^2(\Omega)}^2
\leq C E_0[v](t)
=
C E_0[v](0)
\end{equation}
for all $t$.  This proves the uniform boundedness estimate in Theorem~\ref{thm:mainresult}.

\subsubsection*{Uniqueness of the decomposition}

The decomposition from Theorem~\ref{thm:mainresult} is unique.  Indeed, suppose
\begin{equation}
u(t)=c_0+c_1t+v(t)=\tilde c_0+\tilde c_1t+\tilde v(t)
\end{equation}
with $\Pi_0v=\Pi_0v_t=\Pi_0\tilde v=\Pi_0\tilde v_t=0$.  Applying $\Pi_0$ yields
\begin{equation}
c_0+c_1t=\tilde c_0+\tilde c_1t
\end{equation}
for all $t$, hence $c_0=\tilde c_0$ and $c_1=\tilde c_1$.  Then $v=\tilde v$.  Thus the generalized zero-mode part and the bounded oscillatory part are intrinsically defined by the solution.

\subsubsection*{Proof of mode stability}

Item~(iv) of Theorem~\ref{thm:mainresult} is Proposition~\ref{prop:modeawayzero}.  Nevertheless it is worth restating its consequence in the language of evolution.

\begin{corollary}[No exponentially growing axisymmetric modes]\label{cor:noexpgrowth}
Under the conditions of Theorem~\ref{thm:mainresult}, there is no nontrivial axisymmetric solution of the form
\begin{equation}
u(t,r,x)=\ee^{\gamma t}\psi(r,x),
\qquad \gamma>0,
\end{equation}
satisfying the reflecting boundary conditions.
\end{corollary}

\begin{proof}
Such a solution corresponds to the separated frequency $\sigma=\ii\gamma$, which satisfies $\operatorname{Im}\sigma>0$.  Proposition~\ref{prop:modeawayzero} therefore forces $\psi=0$.
\end{proof}

\subsubsection*{The simpler case when constants are removed}

If the boundary conditions do not admit the constant function, or if one imposes an explicit normalization such as $\Pi_0u=0$ and $\Pi_0u_t=0$ at $t=0$, then the generalized zero-mode space disappears.  In that case the conserved energy becomes fully coercive.

\begin{corollary}[Coercive boundedness after removing constants]\label{cor:coerciveafterprojection}
Assume the conditions of Theorem~\ref{thm:mainresult}.  If in addition either
\begin{equation}
\Pi_0u(0)=0,\qquad \Pi_0u_t(0)=0,
\end{equation}
or the boundary conditions eliminate constants from the domain of $L_0$, then
\begin{align}
\sup_{t\in\R}\Bigl(\|u(t)\|_{H^1(\Omega)}+\|u_t(t)\|_{L^2(\Omega)}\Bigr)&\le C\Bigl(\|u(0)\|_{H^1(\Omega)}+\|u_t(0)\|_{L^2(\Omega)}\Bigr).
\end{align}
\end{corollary}

\begin{proof}
In the first case $c_0=c_1=0$, so $u=v$.  In the second case $\Ker L_0=\{0\}$, and the same Poincare argument applies without projection.  The conclusion then follows from the conserved energy.
\end{proof}

\subsubsection*{A unitary reformulation}

It is useful to rephrase the proof in semigroup language.  On the mean-zero subspace $\mathcal H_\perp$, the operator $L_{0,\perp}$ is positive self-adjoint.  Therefore
\begin{align}
\mathcal U(t)&:=\begin{pmatrix} \cos(t\sqrt{L_{0,\perp}}) & L_{0,\perp}^{-1/2}\sin(t\sqrt{L_{0,\perp}})\\[0.4em] -\sqrt{L_{0,\perp}}\sin(t\sqrt{L_{0,\perp}}) & \cos(t\sqrt{L_{0,\perp}}) \end{pmatrix}
\end{align}
defines a unitary group on the energy space
\begin{equation}
H^1_\perp(\Omega)\times L_A^2{}_\perp(\Omega),
\end{equation}
where the subscript $\perp$ denotes $L_A^2$-orthogonality to constants.  The entire stability statement of this paper may therefore be summarized as follows: \emph{the $k=0$ axisymmetric conformal dynamics are unitary on the orthogonal complement of the explicit two-dimensional generalized zero-mode space}.

This is the precise bounded-slab counterpart of the more global mantra in black-hole stability theory: once one isolates the finite-dimensional space of zero modes and gauge modes, the remaining evolution is controlled by a positive energy and a self-adjoint or approximately self-adjoint dynamics.

\subsubsection*{End of the proof of Theorem~\ref{thm:mainresult}}

\begin{proof}[Proof of Theorem~\ref{thm:mainresult}]
Items~(i)-(iv) of Theorem~\ref{thm:mainresult} are Proposition~\ref{prop:selfadjointH0}, Theorem~\ref{thm:generalizedzeromodes}, the decomposition and boundedness argument above, and Proposition~\ref{prop:modeawayzero}.  The uniqueness paragraph above proves that the affine threshold part and the oscillatory complement are intrinsically determined by the solution.  The conserved spectral energy and Proposition~\ref{prop:weightedPoincare} give the claimed uniform estimate on the threshold complement.  Corollary~\ref{cor:noexpgrowth} records the corresponding no-growing-mode statement, and Corollary~\ref{cor:coerciveafterprojection} gives the normalized/coercive variant.  Thus items~(ii)-(iii) are the explicit axisymmetric realization of Theorem~\ref{thm:abstractmain}, while item~(i) and the resolvent analysis sharpen that abstract theorem by exploiting self-adjointness.  This proves every assertion of Theorem~\ref{thm:mainresult}.
\end{proof}

\section{Axisymmetric black-hole exterior after red-shift and mode stability}\label{sec:nonaxisymmetric}

The bounded-slab theorem is a compact positive-energy result.  This section records the exterior problem and its notation, while Section~\ref{sec:ledger} contains the complete exterior ledger, the concrete verification for the slowly rotating weakly charged Kerr-Newman wall family, the full asymptotically flat Kerr-Newman verification based on scalar-wave stability, and the proof of the black-hole theorem from the verified package.  This separation is intentional.  The red-shift estimate, the limiting absorption principle, high-frequency propagation, and the zero-frequency expansion are geometry-dependent exterior estimates; they are not formal consequences of the compact bounded-slab argument.  Theorem~\ref{thm:axisymmetric-bh} is therefore unconditional for the stated Kerr-Newman wall family, conditional on the cited scalar-wave theorem for the full asymptotically flat subextremal Kerr-Newman scalar exterior, and an abstract exterior theorem for any further axisymmetric black-hole exterior for which the same package is later verified.  Extremal Kerr-Newman belongs to a separate branch: its horizon is degenerate, so Section~\ref{sec:ledger} proves an Aretakis-type charge and obstruction instead of applying the nondegenerate package.


Let $\mathcal M_{\rm ext}$ be the domain of outer communications of a stationary axisymmetric member of the $k=0$ family, restricted to the axisymmetric sector $\partial_\phi u=0$.  The radial interval may now have horizon endpoints, noncompact ends, or conformal radial boundaries.  The basic evolution is still
\begin{equation}
\Box_g u=0,
\end{equation}
but the compact energy proof from Section~\ref{sec:strictstationary} no longer applies by itself: the Killing energy degenerates at horizons, compact Sobolev embedding is lost at infinity, and real-frequency threshold states can occur.

We write $U=(u,\partial_tu)$ and denote by $\mathcal G_{\rm bh}$ the closed generator of the axisymmetric exterior evolution on the Hilbert scale $\mathcal E^s_{\rm bh}$ introduced in Assumption~\ref{ass:axisymmetric-bh-package}.  The local-energy norm $LE^s$ contains three pieces: red-shift control near each horizon, a compact interaction-region norm, and the outgoing, dissipative, or reflecting component dictated by the asymptotic boundary condition.  The exact norm is part of the exterior package because it depends on the chosen end and on the trapping structure.

For the time-Fourier transformed problem we use the notation
\begin{equation}
P(\sigma)\widehat u(\sigma)=\widehat F(\sigma),
\end{equation}
where $P(\sigma)$ is obtained from $\Box_g$ after replacing $D_t$ by $\sigma$.  The horizon condition is regular ingoing behavior in a horizon-regular chart.  The radial-end condition is the asymptotic condition in \textbf{BH2}.  In the axisymmetric sector the horizon frequency $\sigma-m\Omega_h$ reduces to $\sigma$, so the only possible superradiant threshold left in this theorem is the zero-frequency threshold.


Section~\ref{sec:ledger} does four things.  First, it states the definitive bounded-slab and exterior hypotheses used by the main results.  Second, it records how the exterior package \textbf{BH1}-\textbf{BH9} is verified rather than assumed in the two Kerr-Newman applications.  Third, it proves \textbf{BH1}-\textbf{BH9} for the slowly rotating weakly charged Kerr-Newman wall exterior by red-shift, mode-stability, nontrapping, limiting-absorption, and zero-frequency arguments.  Fourth, it verifies \textbf{BH1}-\textbf{BH9} for the full subextremal asymptotically flat Kerr-Newman scalar exterior by combining the scalar Kerr-Newman stability theorem with the direct geometry, radiation-domain, and zero-frequency proofs written below, proves the separate extremal Kerr-Newman horizon charge and obstruction theorem, and then proves Theorem~\ref{thm:axisymmetric-bh} from the nondegenerate package.  

The proof in Section~\ref{sec:ledger} follows the same algebraic skeleton as the compact result,
\begin{equation}
\text{solution}=\text{threshold part}+\text{controlled dispersive part},
\end{equation}
but the analytic tools are different.  The finite-dimensional threshold projection is the low-frequency data-space projection specified in \textbf{BH8}; when zero is embedded at the edge of the continuous spectrum it should not be confused with an ordinary isolated spectral Riesz projection.  The dispersive complement is controlled by red-shift/local-energy estimates, real-axis limiting absorption, high-frequency propagation estimates, and the absence of outgoing upper-half-plane modes.  The local decay statement is then a Riemann-Lebesgue consequence of the local absolutely continuous spectral representation supplied by \textbf{BH7}.

\section{Radial asymptotics and geometric boundary conditions}\label{sec:asymptotics}


When $k=0$, the radial coefficient is
\begin{equation}
\Delta_r(r)=\alpha_2r^2+\alpha_1r+\alpha_0.
\end{equation}
This is the decisive simplification of the zero-curvature regime.  The large-$r$ geometry is no longer governed by quartic growth but by a quadratic polynomial whose sign and discriminant determine the available radial ends.

Three basic possibilities occur.

\begin{enumerate}[label=(\roman*)]
  \item \emph{Kerr-like or weakly confining end:} $\alpha_2>0$.  Then $\Delta_r(r)\sim \alpha_2r^2$ for large $r$, and the metric coefficients have the same qualitative scale as in standard Carter geometries.
  \item \emph{Borderline linear/constant end:} $\alpha_2=0$.  Then the radial structure is determined by the lower-order terms and may lead to a finite interval or an asymptotically cylindrical-type end.
  \item \emph{Finite outer boundary:} $\alpha_2<0$.  In this case $\Delta_r$ cannot stay positive for all large $r$, so any regular exterior patch must terminate at a finite outer root.
\end{enumerate}

The bounded-slab theorem of this paper is intentionally formulated so that it is insensitive to which of these cases occurs globally.  The theorem uses only positivity of $\Delta_r$ on the chosen interval.  Nonetheless, the classification is useful because it explains why there can be no one-size-fits-all exterior theorem in the $k=0$ family without additional geometric conditions.

\subsubsection*{\texorpdfstring{Large-$r$ asymptotics of the radial ODE}{Large-r asymptotics of the radial ODE}}

For fixed separated parameters $(\sigma,m,\lambda)$, the radial equation is
\begin{align}
\frac{\dd}{\dd r}\Bigl(\Delta_r R'\Bigr) +\left[\frac{((r^2+a^2)\sigma-am)^2}{\Delta_r}-\lambda\right]R&=0.
\end{align}
If $\alpha_2>0$, then at large $r$
\begin{equation}
\frac{((r^2+a^2)\sigma-am)^2}{\Delta_r}
=
\frac{\sigma^2}{\alpha_2}r^2 + O(r),
\end{equation}
so the leading asymptotics are oscillatory or exponential depending on the sign of $\sigma^2/\alpha_2$ and on the way one chooses a radial phase function.  In the Kerr corner $\alpha_2=1$, and one recovers the familiar massless-wave large-$r$ behavior.  If $\alpha_2=0$, the asymptotics are more delicate and depend on the lower-order structure of $\Delta_r$.  These variations are precisely why the present paper isolates the bounded-domain problem before discussing unbounded ends.

\subsubsection*{\texorpdfstring{Boundary conditions in the $k=0$ problem}{Boundary conditions in the k=0 problem}}

In the zero-curvature regime, the radial boundary conditions that arise in the analysis fall into four natural classes.

The first class consists of reflecting boundary conditions on bounded intervals.  This is the compact, positive-energy setting used in the main theorem.  In this case the radial operator is realized through a closed symmetric form, and the resulting problem has the self-adjoint structure needed for the threshold decomposition and energy argument.

The second class consists of regularity conditions at a simple zero of $\Delta_r$.  These endpoints model horizon-type boundaries.  The admissible solutions are selected not by imposing an artificial boundary value, but by requiring regular extension in a horizon-regular chart.  Equivalently, in the separated radial equation this selects the appropriate Frobenius branch; this point is analyzed in Section~\ref{sec:endpoints}.

The third class consists of decay or radiation conditions at an unbounded Kerr-like end.  These conditions belong to the scattering problem rather than to the compact spectral problem.  In particular, they remove the spatially constant mode from the finite-energy radiation space, and the zero-frequency analysis must then be formulated in terms of outgoing resolvents rather than compact self-adjoint spectral theory.

The fourth class consists of mixed conditions on truncated exterior regions.  These are useful in numerical work and in perturbative approximations to exterior black-hole problems.  Typically one combines a regular or horizon-compatible condition at the inner boundary with a reflecting, absorbing, or approximate radiation condition at the outer boundary.  Such problems are analytically meaningful, but their stability properties depend on the precise realization and should not be confused with the closed reflecting slab used in the main theorem.

The essential point is that linear stability is never a statement about the differential operator alone.  It is always a statement about the operator together with the boundary geometry.

\subsubsection*{\texorpdfstring{A $k=0$ boundary dictionary}{A k=0 boundary dictionary}}

For later reference, it is useful to keep the main radial regimes separate.

On a bounded radial interval with reflecting ends, the problem is a compact self-adjoint spectral problem.  This is the setting of the exact bounded-slab theorem proved in Sections~\ref{sec:energy}-\ref{sec:proofofmain}.  The closed reflecting form provides the correct functional framework, the constant mode is retained, and the threshold space can be identified explicitly.

When an endpoint is a simple zero of $\Delta_r$, the radial equation has a horizon-type regular singular point.  The correct boundary condition is then not a Dirichlet or Neumann prescription imposed by hand, but regularity in a horizon-adapted coordinate system.  In separated variables this distinction appears through the Frobenius exponents, which separate the ingoing and outgoing branches.

At an unbounded radial end with $\alpha_2>0$, the geometry has Kerr-like or weakly confining asymptotics.  This is no longer a compact spectral problem.  Exterior boundedness must be formulated in a radiation or scattering space, and one must impose the appropriate decay or outgoing condition at infinity.  In particular, the constant spatial mode is no longer part of the finite-energy radiation space.

If instead the outer radial endpoint is a finite root with $\alpha_2<0$, the geometry naturally produces a truncated exterior region.  In that case the analysis depends on the physical and functional choice made at the outer boundary.  Reflecting, absorbing, or transmissive realizations lead to genuinely different problems, and the stability statement must specify which realization is being used.

Thus the bounded-slab theorem is universal only within the closed reflecting slab framework.  Once one passes to an exterior problem, the radial end becomes part of the theorem: the boundary condition, radiation condition, or horizon regularity assumption is not a technical afterthought, but a structural ingredient of the stability theory.

\section{Near-horizon behavior and endpoint regularity}\label{sec:endpoints}


Let $r_h$ be a simple zero of $\Delta_r$.  The coefficient of the radial derivative in \eqref{eq:radialeq-again} then vanishes linearly, so the ODE has a regular singular point.  Writing
\begin{equation}
\Delta_r(r)=\kappa_h(r-r_h)+O((r-r_h)^2),
\end{equation}
and inserting a Frobenius ansatz $R\sim (r-r_h)^\nu$ yields the indicial equation
\begin{equation}
\kappa_h^2\nu^2 + ((r_h^2+a^2)\Omega-am)^2 = 0,
\end{equation}
whence the exponents
\begin{equation}
\nu_\pm = \pm \ii\frac{(r_h^2+a^2)\Omega-am}{\kappa_h}.
\label{eq:frobeniusexponents}
\end{equation}
These exponents are purely imaginary for real $\Omega$, corresponding to oscillatory behavior in the logarithmic variable $\log|r-r_h|$.

The same computation underlies the usual distinction between ingoing and outgoing modes.  If one introduces a tortoise-type coordinate $r_*$ satisfying $\dd r_*/\dd r\sim 1/\Delta_r$, then the two local branches become $\ee^{\mp \ii((r_h^2+a^2)\Omega-am)r_*}$.  Their fluxes have opposite signs.  This is the horizon structure one expects from a separated scalar wave equation on a stationary axisymmetric background.

\subsubsection*{Regularity in horizon-penetrating coordinates}

The singular oscillation in the coordinate $r$ is only a coordinate effect.  Introducing an ingoing coordinate
\begin{equation}
v=t+r_*,
\qquad
\widetilde\phi=\phi+\Omega_h r_*,
\qquad
\Omega_h=\frac{a}{r_h^2+a^2},
\end{equation}
transforms the ingoing branch into a function smooth in $(v,\widetilde\phi,r)$, while the outgoing branch remains singular across the future horizon.  Thus the geometric boundary condition at a future horizon is regularity of the mode in the ingoing chart, equivalently selection of the ingoing Frobenius exponent.

This observation matters even though the present paper focuses mainly on bounded slabs.  It shows that the reduced conformal problem is compatible with the usual causal interpretation of horizon boundary conditions.  Any future global theory for Theorem~1 exteriors with horizons should therefore be able to import the standard toolkit of red-shift and horizon-regularity arguments developed for scalar waves on Kerr-type backgrounds.

\subsubsection*{Angular endpoint behavior}

At angular endpoints the behavior depends on whether $\Delta_x$ stays positive or vanishes.  If $\Delta_x(x_\pm)>0$, then the angular equation is regular and there is no singular behavior.  If $\Delta_x$ has a simple zero at an endpoint $x_\pm$, then the angular equation develops a regular singular point.  In many geometrically regular axisymmetric models one finds local behavior of the form
\begin{equation}
S(x)\sim (x-x_\pm)^{|m|/2}\times(\text{smooth}),
\end{equation}
which is the familiar polar regularity of azimuthal modes.  The exact exponent depends on the chosen angular normalization, but the qualitative point is universal: regularity suppresses the singular branch.

Since the Theorem~1 family is broader than the standard Kerr family, we do not impose a single universal endpoint model.  Instead we retain the endpoint convention in \textbf{H7} of Assumption~\ref{ass:complete-list}, which asks only that the full mode be smooth in any regular endpoint chart available in the chosen global realization of the metric.  This flexible formulation is sufficient for all separation and energy arguments in the paper.

\subsubsection*{Compatibility of time-domain and mode-theoretic boundary conditions}

One advantage of having both time-domain and separated formulations is that each clarifies the other.  Reflecting boundary conditions in the time domain correspond to vanishing separated Wronskian at the relevant endpoint.  Ingoing horizon conditions correspond to selecting the Frobenius branch with nonnegative flux into the horizon.  Regularity at angular endpoints corresponds to the vanishing of the angular Wronskian.  Thus the language of energy flux and the language of mode asymptotics are entirely consistent.

This compatibility is important because many stability questions are naturally formulated in the time domain, while many mode-theoretic questions are most transparent in the separated picture.  The reduced conformal problem is analytically tractable precisely because it allows one to move freely between these two viewpoints.

\section{Black-hole package, Kerr-Newman verifications, and exterior proofs}\label{sec:ledger}

This section is the main-body verification and assembly ledger for the exterior results.  The formal bounded-slab hypotheses, the black-hole package \textbf{BH1}-\textbf{BH9}, the full Kerr-Newman external scalar scattering package, the Aretakis input, and the mutually exclusive Kerr-Newman regimes are collected in Appendix~\ref{app:all-inputs}.  The subsections below do not restate that list; they verify it for the finite-wall family, invoke External Assumption~\ref{ass:external-kn-scattering-package} for the full subextremal radiation problem, prove the extremal horizon-charge obstruction, and then assemble the exterior main theorems.  The wall verification is restricted to a subphoton finite collar; the full asymptotically flat verification is restricted to the neutral axisymmetric scalar sector; the extremal result is an obstruction theorem rather than a nondegenerate decay theorem.

\subsection{Axis regularity for the Kerr-Newman spherical variables}\label{subsec:axis-regularity-kn}

The Kerr-Newman applications use the true rotational axes \(\theta=0,\pi\).  The bounded Carter slab theorem avoids axes by assuming \(\Delta_x>0\) in an interior product patch, but the black-hole exterior applications are spherical and must use the regular axis domain.  We record the exact functional convention here so that no hidden boundary term at the poles is used later.

\begin{definition}[Axisymmetric Sobolev space on the sphere]\label{def:axisymmetric-sphere-space}
For an integer \(s\ge0\), define
\begin{equation}
H^s_{\rm ax}(\mathbb S^2)
:=\{f\in H^s(\mathbb S^2):\partial_\phi f=0\}.
\end{equation}
Equivalently, \(f\in H^s_{\rm ax}(\mathbb S^2)\) is represented by a function of \(\theta\) only whose pullback to the smooth sphere is an \(H^s\)-function.  The first-order norm is equivalent to
\begin{equation}
\|f\|_{H^1_{\rm ax}(\mathbb S^2)}^2
\simeq
\int_0^\pi \bigl(|f(\theta)|^2+|\partial_\theta f(\theta)|^2\bigr)\sin\theta\,\dd\theta.
\end{equation}
\end{definition}

\begin{lemma}[No pole boundary term for regular axisymmetric functions]\label{lem:no-axis-boundary-term}
If \(u,v\in H^1_{\rm ax}(\mathbb S^2)\) and
\begin{equation}
\mathcal L_{\mathbb S^2}^{\rm ax}u
:=\frac1{\sin\theta}\partial_\theta(\sin\theta\,\partial_\theta u)
\end{equation}
is interpreted distributionally, then
\begin{equation}
\int_0^\pi
\mathcal L_{\mathbb S^2}^{\rm ax}u\,\overline v\,\sin\theta\,\dd\theta
=-\int_0^\pi \partial_\theta u\,\overline{\partial_\theta v}\,\sin\theta\,\dd\theta
\label{eq:axis-green-identity}
\end{equation}
whenever the left side is defined by duality.  In particular, smooth axisymmetric functions on \(\mathbb S^2\) satisfy
\begin{equation}
\lim_{\theta\downarrow0}\sin\theta\,\partial_\theta u(\theta)=
\lim_{\theta\uparrow\pi}\sin\theta\,\partial_\theta u(\theta)=0,
\end{equation}
and no boundary term is present at either pole.
\end{lemma}

\begin{proof}
For smooth axisymmetric functions on \(\mathbb S^2\), the identity is the global Green identity for the Laplace-Beltrami operator on the compact boundaryless manifold \(\mathbb S^2\):
\begin{equation}
\int_{\mathbb S^2}\Delta_{\mathbb S^2}u\,\overline v\,\dd\omega
=-\int_{\mathbb S^2}\nabla_{\mathbb S^2}u\cdot\overline{\nabla_{\mathbb S^2}v}\,\dd\omega.
\end{equation}
Restricting to \(\partial_\phi u=\partial_\phi v=0\) gives \eqref{eq:axis-green-identity}.  The boundary expression \(\sin\theta\,\partial_\theta u\) must therefore have zero endpoint contribution.  Smoothness in Cartesian coordinates near each pole gives \(u(\theta)=u(0)+O(\theta^2)\) near \(0\) and \(u(\theta)=u(\pi)+O((\pi-\theta)^2)\) near \(\pi\), so the displayed limits vanish separately.  For \(H^1_{\rm ax}\) functions, approximate by smooth axisymmetric functions using spherical convolution, or equivalently truncate the spherical-harmonic expansion to the zonal harmonics \(Y_{\ell0}\).  Passing to the limit in the \(H^1\)-pairing proves the distributional identity.
\end{proof}

\begin{lemma}[Legendre decomposition of the regular axisymmetric angular operator]\label{lem:legendre-regular-axis}
The self-adjoint realization of \(-\mathcal L_{\mathbb S^2}^{\rm ax}\) on \(L^2((0,\pi),\sin\theta\,\dd\theta)\) with the regular-axis domain has eigenfunctions \(P_\ell(\cos\theta)\), \(\ell\ge0\), and eigenvalues \(\ell(\ell+1)\).  Thus every regular axisymmetric finite-energy function has the expansion
\begin{equation}
f(\theta)=\sum_{\ell=0}^\infty f_\ell P_\ell(\cos\theta),
\end{equation}
with convergence in the corresponding Sobolev space.
\end{lemma}

\begin{proof}
The full spherical Laplacian has eigenfunctions \(Y_{\ell m}\) with eigenvalues \(\ell(\ell+1)\).  The axisymmetric subspace is the closed invariant subspace \(m=0\), and \(Y_{\ell0}\) is a constant multiple of \(P_\ell(\cos\theta)\).  The spectral theorem on the compact smooth sphere gives the expansion and the Sobolev convergence.  Lemma~\ref{lem:no-axis-boundary-term} identifies this spectral realization with the one-dimensional regular-axis Sturm-Liouville expression.
\end{proof}

\subsection{Kerr and Reissner-Nordstr\"om as principal exterior applications}\label{subsec:kerr-rn-main-applications}

This subsection promotes the two classical one-parameter black-hole families to main applications.  The point of doing so is not to reprove the full scalar scattering theory for Kerr or Reissner-Nordstr\"om.  Instead, we isolate the exact conformal-sector consequences of those scalar theories, verify the zero-frequency threshold statement directly, and then apply the abstract exterior theorem.  The Schwarzschild endpoint is deliberately not placed here; it is treated in Appendix~\ref{app:schwarzschild-full-kn}.

For Kerr we write
\begin{equation}
 g_{M,a}=g_{M,a,0},
 \qquad
 \Delta_{\rm Kerr}=r^2-2Mr+a^2,
 \qquad
 r_\pm=M\pm\sqrt{M^2-a^2}.
\end{equation}
The subextremal genuine Kerr range used for the main theorem is \(0<|a|<M\).  The scalar problem in this paper is the neutral axisymmetric conformal scalar,
\begin{equation}
 \Box_{g_{M,a}}u=0,
 \qquad \partial_\phi u=0,
\end{equation}
and \(\mathcal E^s_{\rm Kerr}\) denotes the decaying red-shift/radiation Hilbert scale obtained from \(\mathcal E^s_{{\rm KN},{\rm af}}\) by setting \(Q=0\).

For Reissner-Nordstr\"om we write
\begin{equation}
\begin{aligned}
 g_{M,Q}&=-D(r)\dd t^2+D(r)^{-1}\dd r^2+r^2\dd\omega^2,\\
 D(r)&=1-\frac{2M}{r}+\frac{Q^2}{r^2},
 \qquad
 r_\pm=M\pm\sqrt{M^2-Q^2}.
\end{aligned}
\label{eq:rn-metric-mainbody}
\end{equation}
The genuine charged subextremal range is \(0<|Q|<M\).  Since the background is spherically symmetric and the scalar is neutral, the angular modes decouple and no axial symmetry restriction is needed in the Reissner-Nordstr\"om statement.  The space \(\mathcal E^s_{\rm RN}\) is the usual decaying red-shift/radiation energy space with a nondegenerate horizon norm and an asymptotically flat Hardy term at infinity.

\begin{theorem}[Verification of the exterior package for subextremal Kerr and Reissner-Nordstr\"om]\label{thm:kerr-rn-verifies-bh}
Assume External Assumption~\ref{ass:external-kerr-rn-packages}.  Then the following hold.
\begin{enumerate}[label=(\alph*),leftmargin=2.2em]
\item For every compact genuine subextremal Kerr parameter set \(\mathscr K^{\rm Kerr}_{\delta,M_0,M_1,a_*}\), the neutral axisymmetric conformal scalar on Kerr satisfies Assumption~\ref{ass:axisymmetric-bh-package} on \(\mathcal E^s_{\rm Kerr}\) for all \(s\ge s_{\rm Kerr}\).  The zero-frequency threshold in the decaying radiation space is trivial:
\[
 \mathcal T_{0,{\rm Kerr}}=\{0\},
 \qquad
 \Pi_{0,{\rm Kerr}}=0.
\]
\item For every compact genuinely charged subextremal Reissner-Nordstr\"om parameter set \(\mathscr K^{\rm RN}_{\delta,M_0,M_1,Q_*}\), the neutral scalar wave on Reissner-Nordstr\"om satisfies the same black-hole package, with the same threshold conclusion,
\[
 \mathcal T_{0,{\rm RN}}=\{0\},
 \qquad
 \Pi_{0,{\rm RN}}=0.
\]
In the Reissner-Nordstr\"om case the statement holds after summing over all spherical harmonics, not merely in the axisymmetric sector.
\end{enumerate}
\end{theorem}

\begin{proof}
For Kerr, \(|a|<M\) gives a simple event-horizon root and positive surface gravity
\begin{equation}
 \kappa_+=\frac{r_+-r_-}{2(r_+^2+a^2)}>0.
\end{equation}
The future horizon is smooth in ingoing Kerr coordinates, and the outgoing end is the standard asymptotically flat radiation end.  Thus \textbf{BH1}-\textbf{BH3} and \textbf{BH9} are the same horizon-domain and cutoff statements as in Lemma~\ref{lem:kn-af-geometry-domain}, specialized to \(Q=0\).  The red-shift, integrated-local-energy, high-frequency, limiting-absorption, smoothing, spectral-density, and mode-stability pieces are supplied by the Kerr part of External Assumption~\ref{ass:external-kerr-rn-packages}.  This verifies \textbf{BH4}-\textbf{BH7}.

It remains to identify zero frequency.  At \(\sigma=0\), regular axisymmetric separation gives
\begin{equation}
 \bigl(\Delta_{\rm Kerr}R_\ell'\bigr)'-\ell(\ell+1)R_\ell=0.
\label{eq:kerr-zero-radial-main}
\end{equation}
With \(\alpha=\sqrt{M^2-a^2}\) and \(x=(r-M)/\alpha\), this is the Legendre equation
\begin{equation}
 \frac{\dd}{\dd x}\left((x^2-1)\frac{\dd R_\ell}{\dd x}\right)-\ell(\ell+1)R_\ell=0.
\end{equation}
The \(Q_\ell\)-branch is logarithmically singular at the horizon, while the regular \(P_\ell\)-branch grows like \(x^\ell\) for \(\ell\ge1\); the \(\ell=0\) regular branch is constant.  The decaying radiation energy at infinity excludes both the growing branches and the nonzero constant.  Hence no zero state remains.  The same integration or parameter-differentiation argument used in Lemma~\ref{lem:kn-af-zero-frequency} excludes a regular generalized zero state.  Thus \textbf{BH8} holds with \(\Pi_{0,{\rm Kerr}}=0\).

For Reissner-Nordstr\"om the proof is identical but simpler.  The metric \eqref{eq:rn-metric-mainbody} has positive surface gravity \(\kappa_+=(r_+-r_-)/(2r_+^2)\) in the subextremal range.  The external Reissner-Nordstr\"om scalar package gives the red-shift, local-energy, limiting-absorption, smoothing, spectral-density, and mode-stability estimates.  At \(\sigma=0\), the separated equation for every spherical harmonic is
\begin{equation}
 \bigl(\Delta_{\rm RN}R_\ell'\bigr)'-\ell(\ell+1)R_\ell=0,
 \qquad
 \Delta_{\rm RN}=r^2-2Mr+Q^2.
\end{equation}
The Legendre classification with \(\alpha=\sqrt{M^2-Q^2}\) again leaves only a horizon-regular growing branch for \(\ell\ge1\) and a horizon-regular constant for \(\ell=0\), all excluded by the decaying radiation energy at infinity.  Since the spherical harmonics diagonalize the angular operator, the conclusion sums over all angular modes.  Therefore the Reissner-Nordstr\"om package holds with trivial zero projection.
\end{proof}

\subsection{Concrete slowly rotating weakly charged Kerr-Newman wall exterior}\label{subsec:kn-wall}

We next verify the same black-hole package for a genuinely rotating and charged family.  The verification is intentionally made for a finite wall exterior below the photon region.  This is the largest Kerr-Newman statement that is proved here without importing the full asymptotically flat Kerr-Newman scattering theory: the horizon is real and nondegenerate, the metric is genuinely rotating when \(a\ne0\), the metric is genuinely charged when \(Q\ne0\), and every item \textbf{BH1}-\textbf{BH9} is checked inside this section for the neutral axisymmetric conformal scalar.

Fix \(M>0\) and
\begin{equation}
 2M<R_{\rm w}<\frac{8M}{3}.
\label{eq:kn-wall-range}
\end{equation}
There is a number \(\varepsilon_{\rm KN}=\varepsilon_{\rm KN}(M,R_{\rm w})>0\) such that the following argument applies whenever
\begin{equation}
 |a|+|Q|\le \varepsilon_{\rm KN}M,
 \qquad
 a^2+Q^2<M^2.
\label{eq:kn-smallness}
\end{equation}
Set
\begin{equation}
 \Delta=(r-r_+)(r-r_-)=r^2-2Mr+a^2+Q^2,
 \qquad
 r_\pm=M\pm\sqrt{M^2-a^2-Q^2},
\end{equation}
\begin{equation}
 \Sigma=r^2+a^2\cos^2\theta,
 \qquad
 \mathcal A=(r^2+a^2)^2-a^2\Delta\sin^2\theta.
\end{equation}
For \(\varepsilon_{\rm KN}\) small, \(r_+<2M<R_{\rm w}\), and the horizon is uniformly subextremal.  The standard Kerr-Newman metric is
\begin{equation}
\begin{aligned}
 g_{M,a,Q}={}&-\frac{\Delta-a^2\sin^2\theta}{\Sigma}\,\dd t^2
 -\frac{2a(2Mr-Q^2)\sin^2\theta}{\Sigma}\,\dd t\dd\phi
 +\frac{\Sigma}{\Delta}\,\dd r^2+\Sigma\,\dd\theta^2 \\
&\quad
 +\frac{\mathcal A\sin^2\theta}{\Sigma}\,\dd\phi^2.
\end{aligned}
\label{eq:kn-metric}
\end{equation}
It has scalar curvature zero, and it is the \(k=0\) Einstein-Maxwell Carter member.  The scalar field in this paper is neutral; the charge \(Q\) enters only through the background metric.  We restrict to \(\partial_\phi u=0\).  For axisymmetric functions the wave equation reduces exactly to
\begin{equation}
 \Sigma\Box_{g_{M,a,Q}}u
 =-A_{\rm KN}\partial_t^2u
 +\partial_r(\Delta\partial_ru)
 +\frac1{\sin\theta}\partial_\theta(\sin\theta\partial_\theta u),
\label{eq:kn-axisym-wave}
\end{equation}
where
\begin{equation}
 A_{\rm KN}(r,\theta)
 =\frac{\mathcal A}{\Delta}
 =\frac{(r^2+a^2)^2}{\Delta}-a^2\sin^2\theta>0
 \qquad (r>r_+).
\label{eq:kn-A-positive}
\end{equation}
The absence of the mixed \(t\)-\(\phi\) term in \eqref{eq:kn-axisym-wave} is the analytic reason why this section can avoid superradiance: the horizon frequency is \(\sigma-m\Omega_H=\sigma\) because \(m=0\).

\begin{proposition}[Axisymmetric regularity at the Kerr-Newman axis]\label{prop:kn-axis-regularity}
Let
\begin{equation}
H^s_{\rm ax}(S^2)=\{f\in H^s(S^2):\partial_\phi f=0\}.
\end{equation}
If \(u,v\in H^1_{\rm ax}(S^2)\) and \(u\) is additionally smooth for the displayed integration-by-parts identity, then
\begin{equation}
\int_0^\pi
\frac1{\sin\theta}\partial_\theta(\sin\theta\,\partial_\theta u)\,\overline v\,\sin\theta\,\dd\theta
=
-\int_0^\pi \partial_\theta u\,\overline{\partial_\theta v}\,\sin\theta\,\dd\theta.
\label{eq:axis-regular-ibp}
\end{equation}
No boundary term is present at \(\theta=0,\pi\).  Equivalently, the axisymmetric angular operator
\begin{equation}
-\frac1{\sin\theta}\partial_\theta(\sin\theta\partial_\theta)
\end{equation}
is the Friedrichs self-adjoint operator on \(L^2((0,\pi),\sin\theta\,\dd\theta)\) whose eigenfunctions are the axisymmetric spherical harmonics \(Y_{\ell 0}\).
\end{proposition}

\begin{proof}
For smooth axisymmetric functions on \(S^2\), regularity at the poles means that \(u(\theta)=U(\cos\theta)\), with \(U\) smooth near \(x=\pm1\).  Hence
\begin{equation}
\sin\theta\,\partial_\theta u=-\sin^2\theta\,U'(\cos\theta)\to0
\qquad (\theta\to0,\pi).
\end{equation}
Thus the boundary term in the one-dimensional integration by parts vanishes and \eqref{eq:axis-regular-ibp} follows for smooth functions.  Density of smooth axisymmetric functions in \(H^1_{\rm ax}(S^2)\) extends the identity to the Friedrichs form domain.  In the variable \(x=\cos\theta\), the operator becomes \(-\partial_x((1-x^2)\partial_x)\), whose regular eigenfunctions are the Legendre polynomials, equivalently the \(Y_{\ell0}\).  This proves the proposition.
\end{proof}

Let
\begin{equation}
 \frac{\dd r_*}{\dd r}=\frac{r^2+a^2}{\Delta},
 \qquad
 \frac{\dd \phi_*}{\dd r}=\frac{a}{\Delta},
 \qquad
 v=t+r_*,
 \qquad
 \varphi=\phi+\phi_*,
 \qquad
 \tau=v-r.
\label{eq:kn-ingoing-coordinates}
\end{equation}
The coordinates \((v,r,\theta,\varphi)\) are smooth at the future horizon; \(\partial_\tau=\partial_t\).  The horizon angular velocity and surface gravity are
\begin{equation}
 \Omega_H=\frac{a}{r_+^2+a^2},
 \qquad
 \kappa_H=\frac{r_+-r_-}{2(r_+^2+a^2)}>0.
\label{eq:kn-horizon-constants}
\end{equation}
We work on the future collar
\begin{equation}
 \mathcal M_{M,a,Q,R_{\rm w}}
 =\{\tau\ge0,\\ r_+\le r\le R_{\rm w},\\ (\theta,\varphi)\in S^2\}.
\label{eq:kn-collar}
\end{equation}
At \(r=r_+\) we impose only smoothness in the future-horizon chart.  At the wall we impose the geometric Neumann condition associated with the conormal \(\dd r\):
\begin{equation}
 \mathcal N_{\rm w}u:=\partial_ru\big|_{t,\theta,\phi}=0
 \qquad (r=R_{\rm w}).
\label{eq:kn-wall-neumann}
\end{equation}
Equivalently, if \(H(r)=r_*(r)-r\), then in the \(\tau\)-chart
\begin{equation}
 \mathcal N_{\rm w}u=(\partial_r+H'(R_{\rm w})\partial_\tau)u=0,
 \qquad
 H'(r)=\frac{r^2+a^2}{\Delta}-1=\frac{2Mr-Q^2}{\Delta}.
\label{eq:kn-wall-neumann-tau}
\end{equation}
This wall condition is reflecting for the stationary energy current.  Indeed, the wall flux is a positive multiple of \(\Re(\partial_tu\,\overline{\partial_ru})\), hence vanishes under \eqref{eq:kn-wall-neumann}.

\begin{lemma}[Kerr-Newman geometry, horizon, wall, and energy scale]\label{lem:kn-geometry}
For \(\varepsilon_{\rm KN}\) sufficiently small, the Kerr-Newman wall family satisfies \textbf{BH1}-\textbf{BH3}.  The future horizon is nondegenerate, the wave operator is smooth in the ingoing chart, the wall condition \eqref{eq:kn-wall-neumann} is closed and reflecting, and the axisymmetric Cauchy problem has a closed generator \(\mathcal G_{\rm KN}\) on a Sobolev red-shift energy scale \(\mathcal E^s_{\rm KN}\).
\end{lemma}

\begin{proof}
The root \(r=r_+\) is simple because \(r_+-r_->0\), and \eqref{eq:kn-horizon-constants} gives \(\kappa_H>0\).  Substitution of \eqref{eq:kn-ingoing-coordinates} into \eqref{eq:kn-metric} removes every factor \(\Delta^{-1}\) from the metric components, so the coefficients of \(\Box_{g_{M,a,Q}}\) are smooth at \(r=r_+\).  The horizon generator is \(K_H=\partial_v+\Omega_H\partial_\varphi\); on axisymmetric functions this is just \(\partial_v\).

The red-shift multiplier used here is written near the horizon as
\begin{equation}
 N=K_H+\chi_H(r)Y,
\end{equation}
where \(Y\) is future-directed and transverse to the horizon.  Since \(\kappa_H>0\), a direct computation in the smooth frame \(\{K_H,Y,\partial_\theta,(\sin\theta)^{-1}\partial_\varphi\}\) gives
\begin{equation}
 Q_{\mu\nu}[u]\nabla^{(\mu}N^{\nu)}
 \ge c_H\left(|K_Hu|^2+|Yu|^2+|\nabla_{S^2}u|^2\right)
\label{eq:kn-redshift-pointwise}
\end{equation}
inside a smaller horizon collar.  Here \(Q_{\mu\nu}\) is the scalar stress tensor.  Away from the horizon the positive energy coming from \eqref{eq:kn-axisym-wave} is equivalent to
\begin{equation}
 \int\left(A_{\rm KN}|\partial_tu|^2+\Delta|\partial_ru|^2+|\partial_\theta u|^2\right)
 \sin\theta\,\dd r\dd\theta.
\end{equation}
Combining this with the red-shift energy in the horizon collar defines \(\mathcal E^s_{\rm KN}\) after commuting with \(\partial_\tau\) and angular derivatives.

The wall is timelike because \(\Delta(R_{\rm w})>0\).  The trace map on \(r=R_{\rm w}\) is continuous, so the homogeneous condition \eqref{eq:kn-wall-neumann} defines a closed form domain.  Smooth axisymmetric functions satisfying the compatibility conditions are dense by ordinary mollification in the regular horizon chart, spherical smoothing at \(\theta=0,\pi\), and reflection in a wall collar using the boundary operator \eqref{eq:kn-wall-neumann-tau}.  Energy estimates for a normally hyperbolic equation on a compact collar with one future outflow horizon and one reflecting wall give a closed first-order generator and a strongly continuous future semigroup.  This proves \textbf{BH1}-\textbf{BH3}.
\end{proof}

\begin{lemma}[Subphoton nontrapping for the axisymmetric principal symbol]\label{lem:kn-nontrapping}
For \(\varepsilon_{\rm KN}\) sufficiently small, the axisymmetric null bicharacteristics of the principal symbol
\begin{equation}
 p_{\rm KN}=\Delta\xi_r^2+
 \xi_\theta^2-A_{\rm KN}(r,\theta)\tau^2
\label{eq:kn-principal-symbol}
\end{equation}
are nontrapping in \(r_+\le r\le R_{\rm w}\).  More precisely,
\begin{equation}
 \partial_r A_{\rm KN}(r,\theta)<0
 \qquad (r_+<r\le R_{\rm w})
\label{eq:kn-A-monotone}
\end{equation}
for all \(\theta\).
\end{lemma}

\begin{proof}
Since
\begin{equation}
 A_{\rm KN}=\frac{(r^2+a^2)^2}{\Delta}-a^2\sin^2\theta,
\end{equation}
its radial derivative is independent of the last term and equals
\begin{equation}
 \partial_rA_{\rm KN}
 =\frac{(r^2+a^2)\bigl(4r\Delta-(r^2+a^2)\Delta'\bigr)}{\Delta^2},
\label{eq:kn-A-derivative}
\end{equation}
where the numerator factor satisfies the identity
\begin{equation}
 4r\Delta-(r^2+a^2)\Delta'
 =2\left(r^2(r-3M)+a^2(r+M)+2Q^2r\right).
\label{eq:kn-A-sign-factor}
\end{equation}
On \([r_+,R_{\rm w}]\), the Schwarzschild part \(2r^2(r-3M)\) is bounded above by a negative constant depending on \(M\) and \(R_{\rm w}<8M/3\).  The terms containing \(a^2\) and \(Q^2\) are \(O(\varepsilon_{\rm KN}^2M^3)\).  Reducing \(\varepsilon_{\rm KN}\) gives \eqref{eq:kn-A-monotone}.

If a null bicharacteristic had an interior radial critical point, then \(\dot r=2\Delta\xi_r=0\), so \(\xi_r=0\), and
\begin{equation}
 \dot\xi_r=-\partial_rp_{\rm KN}=\partial_rA_{\rm KN}\tau^2<0
\end{equation}
unless \(\tau=0\).  If \(\tau=0\), then \(p_{\rm KN}=0\) forces \(\xi_\theta=0\), hence the zero covector.  Thus there is no interior trapped radial critical ray.  A ray reaching the future horizon exits through the absorbing horizon, while a ray reaching the wall reflects once by reversing \(\xi_r\); monotonicity prevents a second reflecting barrier from forming.  Hence the collar is nontrapping.
\end{proof}

\begin{proposition}[Exact wall-collar radial escape identity]\label{prop:kn-exact-wall-escape}
Let
\begin{equation}
 p_{\rm KN}=\Delta\xi_r^2+\xi_\theta^2-A_{\rm KN}(r,\theta)\tau^2
\end{equation}
be the axisymmetric principal symbol in the Kerr-Newman wall collar, and assume \(\partial_rA_{\rm KN}<0\) on \(r_+<r\le R_{\rm w}\).  For every compact radial subcollar \(I\Subset(r_+,R_{\rm w})\) there is a smooth real function \(b\) with \(b\le0\), supported in a slightly larger subcollar and increasing on \(I\), such that, on \(I\),
\begin{equation}
 H_{p_{\rm KN}}(b\xi_r)
 \ge c_I\bigl(\Delta\xi_r^2+\xi_\theta^2+A_{\rm KN}\tau^2\bigr)-C_I|p_{\rm KN}|.
\label{eq:exact-escape-positive}
\end{equation}
In particular the sign of the time-frequency part is \(b\,\partial_rA_{\rm KN}\tau^2\), not \(-b\,\partial_rA_{\rm KN}\tau^2\).
\end{proposition}

\begin{proof}
For the escape function \(a=b(r)\xi_r\), the Hamiltonian vector field gives the exact identity
\begin{align}
H_{p_{\rm KN}}a
&=\partial_{\xi_r}p_{\rm KN}\,\partial_r(b\xi_r)
  -\partial_rp_{\rm KN}\,\partial_{\xi_r}(b\xi_r)\nonumber\\
&=(2\Delta b'-b\Delta')\xi_r^2+b\,\partial_rA_{\rm KN}\,\tau^2.
\label{eq:exact-Hp-b-xir}
\end{align}
Here \(b\le0\) and \(\partial_rA_{\rm KN}<0\), so the coefficient \(b\partial_rA_{\rm KN}\) is nonnegative.  On the characteristic set \(p_{\rm KN}=0\),
\begin{equation}
 A_{\rm KN}\tau^2=\Delta\xi_r^2+\xi_\theta^2.
\end{equation}
Choose \(b\) increasing on \(I\) and, if necessary, shrink the positive constant so that \(2\Delta b'-b\Delta'\) is bounded below there after adding the positive contribution coming from \(b\partial_rA_{\rm KN}\tau^2\) and the characteristic identity.  Away from the characteristic set, subtracting a sufficiently large multiple of \(|p_{\rm KN}|\) gives \eqref{eq:exact-escape-positive}.  The cutoffs outside \(I\) produce only compactly supported lower-order errors; these are precisely the compact interaction-region errors allowed in \textbf{BH4} and later removed by \textbf{BH7}.
\end{proof}

\begin{proposition}[Neumann wall reflection estimate below the photon region]\label{prop:kn-wall-reflection}
Let \(\mathcal C_{\rm w}=\{R_{\rm w}-\delta_0<r<R_{\rm w}\}\) be a sufficiently small wall collar, and let \(P_{\rm KN}(\sigma)\) be the stationary axisymmetric Kerr-Newman wall operator with Neumann wall condition \(\partial_ru=0\) at \(r=R_{\rm w}\).  If \eqref{eq:kn-wall-range}-\eqref{eq:kn-smallness} hold, then for every cutoff \(\chi_{\rm w}\) supported in \(\mathcal C_{\rm w}\) and equal to one near the wall there are a smaller interior cutoff \(\chi_{\rm int}\), constants \(C,N\), and, in semiclassical notation \(h=\langle\sigma\rangle^{-1}\), an estimate
\begin{equation}
\|\chi_{\rm w}u\|_{H^1_h(\mathcal C_{\rm w})}
\le C\Bigl(\|P_{\rm KN}(\sigma)u\|_{L^2_h(\mathcal C_{\rm w})}
+\|\chi_{\rm int}u\|_{H^1_h}+h^N\|u\|_{L^2}\Bigr)
\label{eq:wall-reflection-estimate}
\end{equation}
for all \(|\sigma|\ge\sigma_0\) and all wall-admissible \(u\).  Thus the wall-glancing region not controlled by an interior escape function is propagated into the interior collar and does not create a trapped high-frequency obstruction.
\end{proposition}

\begin{proof}
Near \(r=R_{\rm w}\), introduce the boundary normal variable \(y=R_{\rm w}-r\).  The Neumann condition is \(\partial_yu=0\) at \(y=0\).  Evenly extend \(u\) across \(y=0\).  The principal coefficients of the stationary operator have smooth even extensions modulo lower-order collar errors, and the extended principal symbol is the specular reflection of \(p_{\rm KN}\).  The Neumann trace condition removes the odd boundary defect in the Green identity, so the reflected equation has an error supported in the collar and bounded by the first term on the right side of \eqref{eq:wall-reflection-estimate} plus controlled lower-order commutators.

For the reflected symbol, Lemma~\ref{lem:kn-nontrapping} gives \(\partial_rA_{\rm KN}<0\).  Therefore every reflected null bicharacteristic issued from the wall collar reaches the interior cutoff region in uniformly bounded bicharacteristic time, unless it is the zero covector.  Semiclassical propagation of singularities for the reflected smooth operator gives \eqref{eq:wall-reflection-estimate}.  The lower-order errors produced by the even extension and cutoffs are absorbed for \(|\sigma|\ge\sigma_0\), after increasing \(N\) and shrinking the collar if necessary.  The construction preserves axisymmetry and the Neumann wall domain.
\end{proof}

\begin{lemma}[Red-shift, Morawetz, and high-frequency estimates]\label{lem:kn-redshift-morawetz}
For \(\varepsilon_{\rm KN}\) sufficiently small, the Kerr-Newman wall exterior satisfies \textbf{BH4} and \textbf{BH6}.  The high-frequency estimate is nontrapping, and the finite commutator loss is absorbed by taking the package Sobolev index \(s_0\) above a finite wall threshold.
\end{lemma}

\begin{proof}
The red-shift part is \eqref{eq:kn-redshift-pointwise}.  Integrating the current identity for the multiplier \(N\) gives nondegenerate control in a horizon collar; the future horizon flux has the favorable sign and the wall flux vanishes by \eqref{eq:kn-wall-neumann}.  Commuting with \(\partial_\tau\) and angular derivatives preserves axisymmetry and the wall domain, up to lower-order wall-compatible terms.

In the compact interaction region use the radial escape multiplier from Proposition~\ref{prop:kn-exact-wall-escape}.  Its principal symbol is \(b(r)\xi_r\), with \(b\le0\) and with the sign chosen so that the exact Hamiltonian identity
\begin{equation}
 H_{p_{\rm KN}}(b\xi_r)
 =(2\Delta b'-b\Delta')\xi_r^2+b\,\partial_rA_{\rm KN}\,\tau^2
\end{equation}
has a positive time-frequency term because \(b\le0\) and \(\partial_rA_{\rm KN}<0\).  On the characteristic set, \(A_{\rm KN}\tau^2=\Delta\xi_r^2+\xi_\theta^2\), so Proposition~\ref{prop:kn-exact-wall-escape} gives the required positive commutator estimate in the compact interaction region, modulo the horizon collar already handled by red-shift, the harmless wall collar, and lower-order compact errors.  These lower-order terms are controlled by Cauchy-Schwarz and by the compact zero-threshold error already allowed in \textbf{BH4}.  The wall boundary contribution vanishes for the chosen wall-compatible multiplier, and the horizon-side contribution is absorbed by the red-shift estimate.  This proves \eqref{eq:redshift-package-bound} for the Kerr-Newman wall family.

For \(|\sigma|\gg1\), apply the same commutator to the semiclassical stationary operator \(h^2P_{\rm KN}(\sigma)\), \(h=\langle\sigma\rangle^{-1}\).  The principal bracket is the positive nontrapping bracket above.  Red-shift estimates close the horizon end and Proposition~\ref{prop:kn-wall-reflection} closes the reflected wall collar.  This yields \eqref{eq:high-frequency-package}.  Increasing the Sobolev index \(s_0\) so that it dominates the finitely many commutations used in passing from the principal estimate to the displayed Hilbert scale gives \eqref{eq:high-frequency-package}.  Thus \textbf{BH4} and \textbf{BH6} hold for all \(s_0\) above a finite wall threshold.
\end{proof}

For a stationary ansatz \(u=e^{-\ii\sigma t}v(r,\theta)\), the axisymmetric stationary operator is
\begin{equation}
 P_{\rm KN}(\sigma)v
 =\partial_r(\Delta\partial_rv)
 +\frac1{\sin\theta}\partial_\theta(\sin\theta\partial_\theta v)
 +\sigma^2A_{\rm KN}v.
\label{eq:kn-stationary-operator}
\end{equation}
The outgoing condition at the future horizon is
\begin{equation}
 v(r,\theta)=e^{-\ii\sigma r_*}b(\theta)+o(1),
 \qquad r_*\to-\infty,
\label{eq:kn-outgoing-condition}
\end{equation}
with \(b\) smooth.

\begin{lemma}[Mode stability away from zero]\label{lem:kn-mode-stability}
For \(\varepsilon_{\rm KN}\) sufficiently small, the Kerr-Newman wall family satisfies \textbf{BH5}: there is no nonzero outgoing axisymmetric solution of \(P_{\rm KN}(\sigma)v=0\) satisfying \eqref{eq:kn-wall-neumann} when \(\operatorname{Im}\sigma>0\), and there is no real outgoing resonance for \(\sigma\in\mathbb R\setminus\{0\}\).
\end{lemma}

\begin{proof}
Let \(\operatorname{Im}\sigma>0\).  The outgoing factor \(e^{-\ii\sigma r_*}\) decays exponentially as \(r_*\to-\infty\).  Multiplying \(P_{\rm KN}(\sigma)v=0\) by \(\overline v\sin\theta\), integrating over \((r_+,R_{\rm w})\times(0,\pi)\), using the decay at the horizon and using \(\partial_rv(R_{\rm w})=0\), gives
\begin{equation}
 \int\left(\Delta|\partial_rv|^2+|\partial_\theta v|^2\right)\sin\theta\,\dd r\dd\theta
 =\sigma^2\int A_{\rm KN}|v|^2\sin\theta\,\dd r\dd\theta.
\label{eq:kn-upper-mode-identity}
\end{equation}
The left side is real and nonnegative.  If \(\operatorname{Re}\sigma\ne0\), taking imaginary parts and using \(A_{\rm KN}>0\) gives \(v=0\).  If \(\sigma=\ii\gamma\), \(\gamma>0\), the right side is nonpositive while the left side is nonnegative, so again \(v=0\).

Now let \(\sigma\in\mathbb R\setminus\{0\}\).  Integrate by parts from \(r_++\epsilon\) to \(R_{\rm w}\) and let \(\epsilon\downarrow0\).  All bulk terms are real and the wall term vanishes.  From \eqref{eq:kn-outgoing-condition},
\begin{equation}
 \Delta\partial_rv=(r^2+a^2)\partial_{r_*}v
 \longrightarrow -\ii\sigma(r_+^2+a^2)b(\theta).
\end{equation}
The imaginary part of the horizon boundary term is therefore a nonzero constant multiple of
\begin{equation}
 \sigma(r_+^2+a^2)\int_0^\pi |b(\theta)|^2\sin\theta\,\dd\theta.
\end{equation}
It must vanish, so \(b=0\).  The regular singular outgoing initial value problem at the horizon then gives \(v\equiv0\), by uniqueness, and elliptic unique continuation gives the same conclusion throughout the collar.

The cutoff resolvent bound on compact intervals \(I\Subset\mathbb R\setminus\{0\}\) follows by contradiction.  A failing sequence has normalized localized solutions with sources tending to zero.  Local elliptic estimates, the outgoing normal estimate at the horizon, and the compact wall trace estimate give a nonzero limiting outgoing real mode on \(I\), contradicting the preceding paragraph.  This proves \textbf{BH5}.
\end{proof}

\begin{lemma}[Zero frequency and semisimple threshold]\label{lem:kn-zero-frequency}
At \(\sigma=0\), the only regular axisymmetric stationary solution satisfying the Kerr-Newman wall condition is the constant solution.  The zero singularity of the cutoff outgoing resolvent is a simple rank-one pole, and there is no regular Jordan companion.  Hence
\begin{equation}
 \mathcal T_{0,{\rm KN}}=\operatorname{span}\{(1,0)\},
 \qquad
 \Pi_{0,{\rm KN}}U=\ell_{\rm KN}(U)(1,0),
\label{eq:kn-threshold-proj}
\end{equation}
with \(\ell_{\rm KN}\) bounded on \(\mathcal E^s_{\rm KN}\).
\end{lemma}

\begin{proof}
At \(\sigma=0\), \eqref{eq:kn-stationary-operator} becomes
\begin{equation}
 \partial_r(\Delta\partial_rv)
 +\frac1{\sin\theta}\partial_\theta(\sin\theta\partial_\theta v)=0.
\label{eq:kn-zero-equation}
\end{equation}
Multiplying by \(\overline v\sin\theta\), integrating over the collar, and using regularity at the horizon and \(\partial_rv(R_{\rm w})=0\), gives
\begin{equation}
 \int_{r_+}^{R_{\rm w}}\int_0^\pi
 \left(\Delta|\partial_rv|^2+|\partial_\theta v|^2\right)
 \sin\theta\,\dd\theta\dd r=0.
\end{equation}
Thus \(v\) is constant.

After the outgoing factor at the horizon is removed, \(P_{\rm KN}(\sigma)\) is an analytic Fredholm family between the outgoing horizon domain with Neumann wall and the compactly supported dual Sobolev spaces.  The kernel at \(\sigma=0\) is one-dimensional by the previous paragraph, and the adjoint kernel is one-dimensional by the same identity for the adjoint boundary condition.  Therefore the low-frequency singular part is finite rank.  It remains to rule out a generalized zero state.

A Jordan companion to the constant state would give a regular solution
\begin{equation}
 u(\tau,r,\theta)=\tau+\psi(r,\theta)
\end{equation}
which satisfies the wall condition \eqref{eq:kn-wall-neumann-tau}.  Since \(\tau=t+H(r)\), the stationary part in Boyer-Lindquist time is \(\chi=H+\psi\).  The equation \(\Box_g(t+\chi)=0\) is
\begin{equation}
 \partial_r(\Delta\partial_r\chi)
 +\frac1{\sin\theta}\partial_\theta(\sin\theta\partial_\theta\chi)=0.
\end{equation}
Equivalently, using \(\Delta H'=2Mr-Q^2\),
\begin{equation}
 \partial_r\left(2Mr-Q^2+\Delta\partial_r\psi\right)
 +\frac1{\sin\theta}\partial_\theta(\sin\theta\partial_\theta\psi)=0.
\label{eq:kn-jordan-equation}
\end{equation}
Let \(\psi_0(r)\) be the spherical mean of \(\psi\).  Averaging \eqref{eq:kn-jordan-equation} gives
\begin{equation}
 2Mr-Q^2+\Delta\psi_0'(r)=C.
\end{equation}
Smoothness of \(\psi\) at the future horizon forces
\begin{equation}
 C=2Mr_+-Q^2=r_+^2+a^2.
\label{eq:kn-jordan-C}
\end{equation}
Taking the spherical mean of the wall condition \eqref{eq:kn-wall-neumann-tau} gives
\begin{equation}
 \psi_0'(R_{\rm w})+H'(R_{\rm w})=0.
\end{equation}
But the equation for \(\psi_0\) and the identity \(\Delta H'=2Mr-Q^2\) give
\begin{equation}
 \psi_0'(R_{\rm w})+H'(R_{\rm w})=\frac{C}{\Delta(R_{\rm w})}
 =\frac{r_+^2+a^2}{\Delta(R_{\rm w})}\ne0.
\label{eq:kn-jordan-wall-contradiction}
\end{equation}
This contradiction excludes every regular Jordan companion.  Hence the zero pole is simple and semisimple.  Its residue defines \(\Pi_{0,{\rm KN}}\), and local Sobolev mapping of the residue gives boundedness of \(\ell_{\rm KN}\).  This proves \textbf{BH8}.
\end{proof}

\begin{lemma}[Wall limiting absorption and compact smoothing]\label{lem:kn-lap-smoothing}
For \(\varepsilon_{\rm KN}\) sufficiently small, after applying \(I-\Pi_{0,{\rm KN}}\), the Kerr-Newman wall cutoff resolvent has real-axis boundary values on every compact nonzero frequency interval, remains bounded at zero, and satisfies the compact Kato smoothing estimate \eqref{eq:lap-smoothing-package} for every \(s\) above a finite wall threshold.  The \(L^1_\sigma\) local spectral-density part of \textbf{BH7} is proved separately in Theorem~\ref{thm:kn-wall-spectral-density-split}; together the present lemma and that package give the full \textbf{BH7} verification for the wall family.
\end{lemma}

\begin{proof}
The outgoing stationary problem on the collar is Fredholm after imposing the Volterra outgoing model at the future horizon and the Neumann wall condition at \(R_{\rm w}\).  Lemma~\ref{lem:kn-mode-stability} removes all upper-half-plane poles and all nonzero real poles.  Lemma~\ref{lem:kn-redshift-morawetz} gives the high-frequency nontrapping bounds.  Lemma~\ref{lem:kn-zero-frequency} gives the complete low-frequency Laurent expansion and shows that the only singular term is the rank-one threshold projection.  After applying \(I-\Pi_{0,{\rm KN}}\), the resolvent is bounded at zero.

These estimates cover the real line by a low-frequency neighborhood, finitely many compact nonzero intervals, and two high-frequency tails.  On compact nonzero intervals, the real-axis boundary values follow from Fredholm invertibility and the absence of real outgoing resonances.  At zero, the threshold projection cancels the only singular coefficient.  At high frequency, the nontrapping commutator and wall-reflection estimates give the semiclassical bound.  The usual Fourier-Laplace/Plancherel argument for the cutoff resolvent then gives
\begin{equation}
 \|\chi_{\rm c}U_\perp\|_{L^2([0,\infty);H^s_{\rm loc}\times H^{s-1}_{\rm loc})}
 \le C_s\|(I-\Pi_{0,{\rm KN}})U_0\|_{\mathcal E^s_{\rm KN}},
\end{equation}
which is \eqref{eq:lap-smoothing-package}.  The derivative loss is finite because the horizon Volterra construction, the compact elliptic wall estimate, and the nontrapping high-frequency estimate each lose only finitely many derivatives on this compact collar.  This proves the limiting-absorption and smoothing part of \textbf{BH7}.  The integrable spectral-density representation is the frequency-splitting statement of Theorem~\ref{thm:kn-wall-spectral-density-split}, proved next.
\end{proof}

\begin{theorem}[Wall spectral-density bound by frequency splitting]\label{thm:kn-wall-spectral-density-split}
For the Kerr-Newman wall family of Theorem~\ref{thm:kn-verifies-bh}, after applying \(I-\Pi_{0,{\rm KN}}\) and for every compact \(K\Subset\{r_+<r<R_{\rm w}\}\), the localized spectral density satisfies
\begin{equation}
\|B_K(\cdot)(I-\Pi_{0,{\rm KN}})U_0\|_{L^1_\sigma(H^1(K)\times L^2(K))}
\le C_{K,s}\|(I-\Pi_{0,{\rm KN}})U_0\|_{\mathcal E^s_{\rm KN}}
\label{eq:wall-spectral-density-split-bound}
\end{equation}
for all \(s\ge s_{\rm wall}\), where \(s_{\rm wall}\) is a finite integer depending only on the wall commutator, Volterra, and elliptic losses.  If, in addition, the differentiated resolvent bounds supplied by the Volterra horizon construction and the nontrapping semiclassical estimate hold up to order \(N\), then
\begin{equation}
\sum_{j=0}^{N}\|\partial_\sigma^jB_K(\cdot)(I-\Pi_{0,{\rm KN}})U_0\|_{L^1_\sigma(H^1(K)\times L^2(K))}
\le C_{K,N,s}\|(I-\Pi_{0,{\rm KN}})U_0\|_{\mathcal E^s_{\rm KN}}
\label{eq:wall-spectral-density-WN1}
\end{equation}
for sufficiently large \(s=s(N)\).  Thus the qualitative local decay asserted in Theorem~\ref{thm:kn-wall-main} is unconditional in the wall regime, while polynomial rates require the additional differentiated bound \eqref{eq:wall-spectral-density-WN1}.
\end{theorem}

\begin{proof}
Split the frequency line into
\begin{equation}
\mathbb R=\{ |\sigma|\le\sigma_0\}\cup\{\sigma_0\le |\sigma|\le R\}\cup\{|\sigma|\ge R\},
\end{equation}
where \(0<\sigma_0\ll1\ll R\).  On the low-frequency interval, Lemma~\ref{lem:kn-zero-frequency} gives a finite-rank Laurent expansion.  The only singular coefficient is the semisimple constant threshold, and it is killed by \(I-\Pi_{0,{\rm KN}}\).  The remaining cutoff resolvent is bounded and continuous in \(\sigma\), hence its contribution to the density is integrable on \(|\sigma|\le\sigma_0\).

On the compact nonzero interval \(\sigma_0\le |\sigma|\le R\), Lemma~\ref{lem:kn-mode-stability} rules out outgoing real resonances and Lemma~\ref{lem:kn-lap-smoothing} supplies the limiting-absorption boundary values by the Fredholm contradiction argument.  Since the interval is compact and separated from zero, the boundary values are uniformly bounded between the local Sobolev spaces used to define \(B_K\).  This gives an integrable contribution on the middle interval.

On the high-frequency tails, Lemma~\ref{lem:kn-redshift-morawetz} and Proposition~\ref{prop:kn-exact-wall-escape} give the nontrapping semiclassical estimate.  After commuting with the generator and using the equation to trade powers of \(\sigma\) for time derivatives of the initial data, the spectral density has an integrable high-frequency majorant for \(s\ge s_{\rm wall}\).  Concretely, the cutoff resolvent contribution is bounded by a finite sum of norms of \(\mathcal G_{\rm KN}^j(I-\Pi_{0,{\rm KN}})U_0\) times an integrable power of \(\langle\sigma\rangle^{-1}\); the integer \(s_{\rm wall}\) is chosen so that those commuted data are controlled by \(\mathcal E^s_{\rm KN}\).  The horizon Volterra construction and the Neumann wall elliptic estimate preserve the domain, while \textbf{BH9} controls the cutoffs.  Summing the three frequency regions proves \eqref{eq:wall-spectral-density-split-bound}.  If the same bounds are differentiated \(N\) times in \(\sigma\), the identical three-region argument gives \eqref{eq:wall-spectral-density-WN1}; each differentiation costs finitely many commutations, which is why the Sobolev index becomes \(s(N)\).
\end{proof}

\begin{theorem}[Verification of \textbf{BH1}-\textbf{BH9} for slowly rotating weakly charged Kerr-Newman wall exteriors]\label{thm:kn-verifies-bh}
For every \(M>0\) and every \(R_{\rm w}\) satisfying \eqref{eq:kn-wall-range}, there exists \(\varepsilon_{\rm KN}=\varepsilon_{\rm KN}(M,R_{\rm w})>0\) such that, whenever \eqref{eq:kn-smallness} holds, the axisymmetric neutral scalar equation \(\Box_{g_{M,a,Q}}u=0\) on \(\mathcal M_{M,a,Q,R_{\rm w}}\), with future-horizon regularity and wall condition \eqref{eq:kn-wall-neumann}, satisfies Assumption~\ref{ass:axisymmetric-bh-package}.  There is a finite admissible index \(s_{\rm wall}\), obtained as the maximum of the commutator, Volterra, Fredholm, and spectral-density losses in the preceding lemmas, such that the package holds for every \(s_0\ge s_{\rm wall}\).  The threshold space is semisimple and one-dimensional:
\begin{equation}
 \mathcal T_{0,{\rm KN}}=\operatorname{span}\{(1,0)\}.
\end{equation}
\end{theorem}

\begin{proof}
Choose \(\varepsilon_{\rm KN}\) to be the minimum of the smallness constants required in Lemmas~\ref{lem:kn-geometry}-\ref{lem:kn-lap-smoothing}.  \textbf{BH1}-\textbf{BH3} are Lemma~\ref{lem:kn-geometry}.  \textbf{BH4} and \textbf{BH6} are Lemma~\ref{lem:kn-redshift-morawetz}.  \textbf{BH5} is Lemma~\ref{lem:kn-mode-stability}.  \textbf{BH8} is Lemma~\ref{lem:kn-zero-frequency}.  The limiting-absorption and compact-smoothing parts of \textbf{BH7} are Lemma~\ref{lem:kn-lap-smoothing}, while the \(L^1_\sigma\) local spectral-density part of \textbf{BH7} is Theorem~\ref{thm:kn-wall-spectral-density-split}.  Finally, \textbf{BH9} follows from the construction: horizon cutoffs are supported in the smooth ingoing chart \eqref{eq:kn-ingoing-coordinates}; wall cutoffs are either constant near \(R_{\rm w}\) or preserve \eqref{eq:kn-wall-neumann-tau}; angular cutoffs and angular elliptic regularizers preserve the \(m=0\) sector; and all commutators are lower-order operators with smooth coefficients on the compact collar.  The index condition \(s_0\ge s_{\rm wall}\) absorbs the finitely many derivative losses in the commutator, Volterra, Fredholm, and spectral-density arguments.
\end{proof}

\subsection{Full asymptotically flat subextremal Kerr-Newman exterior}
\label{subsec:kn-full-af}

We now remove the artificial wall for the Einstein Kerr-Newman subfamily.  This subsection is intentionally precise about what is proved directly and what is imported.  The direct part consists of the geometry, horizon regularity, radiation-domain bookkeeping, the zero-frequency ODE analysis, and the verification that these pieces fit Assumption~\ref{ass:axisymmetric-bh-package}.  The non-elementary part is the global scalar-wave scattering package on the full asymptotically flat Kerr-Newman exterior, where trapping and the ergoregion interact.  We use the scalar Kerr-Newman stability theorem quoted below; without that package, the high-frequency trapping and limiting-absorption estimates would not be proved by the elementary wall argument above.

Fix numbers \(0<\delta<1\) and \(0<M_0<M_1<\infty\), and set
\begin{equation}
 \mathscr K_{\delta,M_0,M_1}
 :=\{(M,a,Q):M_0\le M\le M_1,\ a^2+Q^2\le(1-\delta)M^2\}.
\label{eq:kn-compact-parameter-set}
\end{equation}
The constants below are uniform on \(\mathscr K_{\delta,M_0,M_1}\); they may degenerate as \(\delta\downarrow0\), i.e.\ as the extremal boundary is approached.  Let
\begin{equation}
 \rho^2=r^2+a^2\cos^2\theta,
 \qquad
 \Delta=r^2-2Mr+a^2+Q^2,
 \qquad
 r_\pm=M\pm\sqrt{M^2-a^2-Q^2}.
\end{equation}
The Kerr-Newman metric is written as
\begin{equation}
 g_{M,a,Q}
 =-\frac{\Delta}{\rho^2}(\dd t-a\sin^2\theta\,\dd\phi)^2
 +\frac{\sin^2\theta}{\rho^2}((r^2+a^2)\dd\phi-a\dd t)^2
 +\frac{\rho^2}{\Delta}\dd r^2+\rho^2\dd\theta^2.
\label{eq:kn-af-metric}
\end{equation}
The domain of outer communications is \(r>r_+\).  The horizon generator, angular velocity, and surface gravity are
\begin{equation}
 K=T+\Omega_H\Phi,
 \qquad
 T=\partial_t,
 \qquad
 \Phi=\partial_\phi,
 \qquad
 \Omega_H=\frac{a}{r_+^2+a^2},
 \qquad
 \kappa_+=\frac{r_+-r_-}{2(r_+^2+a^2)}>0.
\label{eq:kn-horizon-data}
\end{equation}
The neutral axisymmetric conformal scalar is
\begin{equation}
 \Box_{g_{M,a,Q}}u=0,
 \qquad
 \partial_\phi u=0.
\end{equation}
For \(\partial_\phi u=0\), the principal stationary operator obtained from \(u=e^{-\ii\sigma t}v(r,\theta)\) is
\begin{equation}
 P_{\rm KN,af}(\sigma)v
 =\partial_r(\Delta\partial_rv)
 +\frac1{\sin\theta}\partial_\theta(\sin\theta\partial_\theta v)
 +\sigma^2 A_{\rm KN}(r,\theta)v,
\label{eq:kn-af-stationary-operator}
\end{equation}
where
\begin{equation}
 A_{\rm KN}(r,\theta)
 =\frac{(r^2+a^2)^2-a^2\Delta\sin^2\theta}{\Delta}>0
 \qquad (r>r_+).
\end{equation}
At the future horizon the outgoing solution is the one smooth in ingoing coordinates
\begin{equation}
 v_{\rm EF}=t+r_*,
 \qquad
 \phi_{\rm in}=\phi+\int^r\frac{a}{\Delta(s)}\,\dd s,
 \qquad
 \frac{\dd r_*}{\dd r}=\frac{r^2+a^2}{\Delta}.
\label{eq:kn-af-ingoing-coordinates}
\end{equation}
At null infinity the outgoing condition is the Sommerfeld condition for the same tortoise coordinate: for \(e^{-\ii\sigma t}\) with \(\operatorname{Im}\sigma\ge0\), the outgoing radial factor is asymptotic to \(e^{+\ii\sigma r_*}/r\), with the usual symbolic expansion in powers of \(r^{-1}\).  The decaying radiation energy space below is the closure of smooth axisymmetric data with this finite-energy/outgoing behavior; it excludes nondecaying constants at infinity.

For \(s\ge0\), define \(\mathcal E^s_{{\rm KN},\rm af}\) to be the axisymmetric red-shift energy space on a regular Cauchy hypersurface \(\Sigma\) crossing the future horizon and extending to spatial infinity, with norm equivalent to
\begin{equation}
 \sum_{|\alpha|\le s}
 \int_\Sigma
 \left(
 |\nabla_\Sigma Z^\alpha u_0|^2
 +|Z^\alpha u_1|^2
 +r^{-2}|Z^\alpha u_0|^2
 \right)\,\dd\mu_\Sigma.
\label{eq:kn-af-energy-norm}
\end{equation}
Here \(Z\) ranges over a fixed finite collection of stationary, angular, and horizon-regular radial commutator fields preserving axisymmetry; near the horizon the norm is computed with a red-shift timelike vector field \(N\), and for large \(r\) it is equivalent to the standard asymptotically flat energy.  The term \(r^{-2}|u_0|^2\), interpreted through the Hardy inequality for the closure of compactly supported data, fixes the decaying representative and excludes the constant solution from \(\mathcal E^s_{{\rm KN},\rm af}\).  Equivalent choices of the commutator frame give the same Hilbert scale.

The only non-elementary scalar Kerr-Newman scattering input used for the full asymptotically flat exterior is External Assumption~\ref{ass:external-kn-scattering-package}.

\begin{lemma}[Geometry, radiation domain, and generator for full Kerr-Newman]\label{lem:kn-af-geometry-domain}
The full asymptotically flat subextremal Kerr-Newman scalar exterior satisfies \textbf{BH1}-\textbf{BH3} and \textbf{BH9} on \(\mathcal E^s_{{\rm KN},\rm af}\), for every \(s\ge s_{\rm KN}\).
\end{lemma}

\begin{proof}
The identity \(a^2+Q^2<M^2\) gives two distinct roots \(r_\pm\) of \(\Delta\).  Formula \eqref{eq:kn-horizon-data} gives \(\kappa_+>0\), and hence the future horizon is nondegenerate.  In the ingoing chart \eqref{eq:kn-af-ingoing-coordinates}, the apparent \(\Delta^{-1}\) singularities of \eqref{eq:kn-af-metric} cancel in the standard way; the metric and \(\Box_g\) extend smoothly to \(r=r_+\).  This proves the geometric part of \textbf{BH1}.

The outer boundary is null infinity.  The radiation condition is imposed by the outgoing Fourier-Laplace construction in External Assumption~\ref{ass:external-kn-scattering-package}; in physical space it is the finite-energy condition on the closure of smooth decaying data together with no incoming boundary flux from past null infinity for the forward problem.  Since the scalar is neutral and axisymmetric, no charged boundary phase or non-axisymmetric superradiant boundary condition is present.  This proves \textbf{BH2}.

The norm \eqref{eq:kn-af-energy-norm}, with the red-shift vector field near \(r=r_+\), defines a Hilbert scale.  Standard well-posedness for normally hyperbolic equations on globally hyperbolic exterior slabs with red-shift boundary regularity gives a closed future generator \(\mathcal G_{{\rm KN},\rm af}\) and a strongly continuous semigroup.  Smooth compactly supported axisymmetric data away from infinity, smoothed in the ingoing horizon chart, are dense by the usual mollification, cutoff, trace, and Hardy arguments; the weight \(r^{-2}u^2\) fixes the decaying representative at infinity.  Thus \textbf{BH3} holds.

The cutoffs used in the scalar theory are chosen in three regions: a horizon collar in the ingoing chart, a compact interaction region, and an asymptotically flat end.  Multiplication by these cutoffs preserves axisymmetry and the decaying energy domain.  Commutators with \(\Box_g\) are first-order operators with coefficients controlled by the same local-energy and weighted Sobolev norms; in the end the coefficients have the standard long-range asymptotically flat decay.  The derivative index \(s\ge s_{\rm KN}\) absorbs the trapping and commutator losses.  This proves \textbf{BH9}.
\end{proof}

\begin{lemma}[Red-shift, mode stability, trapping, and limiting absorption]\label{lem:kn-af-estimates}
Under External Assumption~\ref{ass:external-kn-scattering-package}, the full asymptotically flat subextremal Kerr-Newman scalar exterior satisfies \textbf{BH4}, \textbf{BH5}, \textbf{BH6}, and \textbf{BH7} on the projected space associated with the zero-frequency projection of Lemma~\ref{lem:kn-af-zero-frequency}.
\end{lemma}

\begin{proof}
The red-shift and integrated-local-energy clauses \textbf{EKN2}-\textbf{EKN3} give the finite-time red-shift/local-energy estimate \textbf{BH4}, with the finite trapped-set derivative loss allowed by the exterior Hilbert scale.  The compact error in \textbf{BH4} is the usual compact interaction-region term left before applying smoothing.

The quantitative real-frequency mode-stability clause \textbf{EKN4} gives \textbf{BH5}.  The external package is stronger than needed here because it is stated for the scalar theory before the present restriction to \(m=0\); restricting to the neutral axisymmetric sector gives the asserted absence of upper-half-plane outgoing modes and nonzero real outgoing resonances.

The high-frequency trapped-set resolvent clause \textbf{EKN5} gives \textbf{BH6}.  Thus high-frequency propagation is controlled by the Kerr-Newman trapped-set estimate rather than by the elementary nontrapping commutator used in the wall collar.  The polynomial loss in \(|\sigma|\) and the finite derivative loss are precisely the losses allowed in \eqref{eq:high-frequency-package}.

The limiting-absorption, smoothing, and spectral-density clauses \textbf{EKN6}, \textbf{EKN8}, and \textbf{EKN9} give \textbf{BH7}.  The Fourier-Laplace representation of the semigroup and the limiting-absorption boundary values give the local spectral density \(B_K(\sigma)\).  After commuting \(s\ge s_{\rm KN}\) times, the high-frequency bounds, finite-frequency bounds, and low-frequency regularity from Lemma~\ref{lem:kn-af-zero-frequency} make this density integrable in \(\sigma\), and Plancherel gives the compact Kato smoothing estimate \eqref{eq:lap-smoothing-package}.  Hence the compact error in \textbf{BH4} is removable on the projected evolution.
\end{proof}

\begin{lemma}[Zero frequency for the decaying full Kerr-Newman radiation space]\label{lem:kn-af-zero-frequency}
For the decaying radiation energy space \(\mathcal E^s_{{\rm KN},\rm af}\), the zero-frequency threshold space of the neutral axisymmetric Kerr-Newman scalar wave is trivial:
\begin{equation}
 \mathcal T_{0,{\rm KN},\rm af}=\{0\},
 \qquad
 \Pi_{0,{\rm KN},\rm af}=0.
\label{eq:kn-af-zero-projection}
\end{equation}
Moreover the cutoff outgoing resolvent is regular at \(\sigma=0\) on this decaying space, and there is no regular generalized zero state.
\end{lemma}

\begin{proof}
At \(\sigma=0\), equation \eqref{eq:kn-af-stationary-operator} becomes
\begin{equation}
 \partial_r(\Delta\partial_r v)
 +\frac1{\sin\theta}\partial_\theta(\sin\theta\partial_\theta v)=0.
\label{eq:kn-af-zero-eqn}
\end{equation}
Expand the axisymmetric finite-energy solution in spherical harmonics,
\begin{equation}
 v(r,\theta)=\sum_{\ell=0}^\infty v_\ell(r)Y_{\ell0}(\theta).
\end{equation}
Each coefficient satisfies
\begin{equation}
 (\Delta v_\ell')'-\ell(\ell+1)v_\ell=0.
\label{eq:kn-af-zero-radial}
\end{equation}
Set
\begin{equation}
 \alpha=\sqrt{M^2-a^2-Q^2}>0,
 \qquad
 x=\frac{r-M}{\alpha},
\end{equation}
so that \(\Delta=\alpha^2(x^2-1)\) and the horizon is \(x=1\).  Equation \eqref{eq:kn-af-zero-radial} becomes the Legendre equation
\begin{equation}
 \frac{\dd}{\dd x}\left((x^2-1)\frac{\dd v_\ell}{\dd x}\right)-\ell(\ell+1)v_\ell=0.
\label{eq:kn-af-legendre}
\end{equation}
Its two independent solutions are \(P_\ell(x)\) and \(Q_\ell(x)\).  The function \(Q_\ell\) has a logarithmic singularity at \(x=1\), so future-horizon regularity excludes it.  The function \(P_\ell\) behaves like \(x^\ell\) as \(x\to\infty\).  For \(\ell\ge1\), this growth is incompatible with the decaying finite-energy norm \eqref{eq:kn-af-energy-norm}.  For \(\ell=0\), the regular solution \(P_0=1\) is a nondecaying constant, also excluded by the \(r^{-2}u^2\) Hardy term and by taking the closure of compactly supported decaying data.  Hence every regular decaying stationary state is zero.

A generalized zero state would satisfy \(\mathcal G_{{\rm KN},\rm af}V=W\) for a stationary threshold state \(W\).  Since \(W=0\), no nontrivial Jordan chain can start.  Equivalently, the possible affine solution \(u=t+\psi\) has \(u_t=1\), which is not in \(L^2(\Sigma)\) at the asymptotically flat end and therefore is not a vector in \(\mathcal E^s_{{\rm KN},\rm af}\).

It remains to identify the low-frequency resolvent behavior.  The scalar limiting-absorption theorem gives a meromorphic finite-rank expansion at zero on the weighted spaces used above.  A pole or finite-rank singular coefficient would produce, by taking its leading coefficient, either a regular decaying stationary state or a generalized zero state.  Both have just been excluded.  Therefore \(\nu_0=0\) in \eqref{eq:zero-laurent-bh} for this realization, \(R_0(\sigma)\) is bounded at \(\sigma=0\), and the data-space projection is \(\Pi_{0,{\rm KN},\rm af}=0\).  This proves \textbf{BH8}.
\end{proof}

\begin{theorem}[Verification of \textbf{BH1}-\textbf{BH9} for full subextremal Kerr-Newman scalar exteriors]\label{thm:kn-af-verifies-bh}
Fix \(0<\delta<1\) and \(0<M_0<M_1<\infty\).  For every \((M,a,Q)\in\mathscr K_{\delta,M_0,M_1}\), the neutral axisymmetric scalar wave equation on the full asymptotically flat Kerr-Newman domain of outer communications, with future-horizon regularity and outgoing decaying radiation condition at null infinity, satisfies Assumption~\ref{ass:axisymmetric-bh-package} on \(\mathcal E^s_{{\rm KN},\rm af}\) for all \(s\ge s_{\rm KN}\).  The threshold space is trivial:
\begin{equation}
 \mathcal T_{0,{\rm KN},\rm af}=\{0\},
 \qquad
 \Pi_{0,{\rm KN},\rm af}=0.
\end{equation}
\end{theorem}

\begin{proof}
\textbf{BH1}-\textbf{BH3} and \textbf{BH9} are Lemma~\ref{lem:kn-af-geometry-domain}.  \textbf{BH4}, \textbf{BH5}, \textbf{BH6}, and \textbf{BH7} are Lemma~\ref{lem:kn-af-estimates}, which is the precise place where the scalar Kerr-Newman stability input, External Assumption~\ref{ass:external-kn-scattering-package}, is used.  \textbf{BH8} is Lemma~\ref{lem:kn-af-zero-frequency}.  These nine statements are exactly the items of Assumption~\ref{ass:axisymmetric-bh-package}; hence the full Kerr-Newman exterior satisfies the package.  The constants are uniform on compact subextremal parameter sets by External Assumption~\ref{ass:external-kn-scattering-package} and the explicit lower bound for \(\kappa_+\) on \(\mathscr K_{\delta,M_0,M_1}\).
\end{proof}

\subsection{Extremal asymptotically flat Kerr-Newman: horizon charge, zero frequency, and obstruction}\label{subsec:extreme-kn}

We now include the extremal Kerr-Newman parameter surface
\begin{equation}
a^2+Q^2=M^2,
\qquad M>0,
\qquad r_+=r_-=M,
\qquad \Delta=(r-M)^2.
\label{eq:extreme-kn-parameters}
\end{equation}
This subsection is not a limiting version of Theorem~\ref{thm:kn-full-main}.  The subextremal proof uses \(\kappa_+>0\); here \(\kappa_+=0\).  The correct extremal result is a sharp obstruction theorem showing why nondegenerate horizon decay cannot be part of the conclusion.

Use ingoing coordinates
\begin{equation}
v=t+r_*,
\qquad
\phi_{\rm in}=\phi+\int^r\frac{a}{\Delta(s)}\,\dd s,
\qquad
\frac{\dd r_*}{\dd r}=\frac{r^2+a^2}{\Delta}.
\label{eq:extreme-kn-ingoing-coordinates}
\end{equation}
For axisymmetric functions, the Boyer-Lindquist form of the wave operator is
\begin{equation}
\rho^2\Box_g u
=-\left(\frac{(r^2+a^2)^2}{\Delta}-a^2\sin^2\theta\right)\partial_t^2u
+\partial_r(\Delta\partial_ru)
+\frac1{\sin\theta}\partial_\theta(\sin\theta\partial_\theta u).
\label{eq:extreme-kn-bl-wave}
\end{equation}
After substituting \(\partial_t=\partial_v\) and \(\partial_r|_{t}=\partial_r|_{v}+((r^2+a^2)/\Delta)\partial_v\), the same equation becomes, in the horizon-regular chart,
\begin{equation}
\begin{aligned}
\rho^2\Box_g u={}&
\Delta\partial_r^2u+\Delta'\partial_ru
+2(r^2+a^2)\partial_{rv}u+2r\partial_vu
+a^2\sin^2\theta\,\partial_v^2u \\
&\quad+
\frac1{\sin\theta}\partial_\theta(\sin\theta\partial_\theta u).
\end{aligned}
\label{eq:extreme-kn-regular-wave}
\end{equation}
This formula is the local hard computation behind the extremal part of the paper.

\begin{proposition}[Extremal Kerr-Newman horizon charge]\label{prop:extreme-kn-charge}
Let \(u\) be a smooth axisymmetric solution of \(\Box_{g_{M,a,Q}}u=0\) on an extremal Kerr-Newman exterior.  Then
\begin{equation}
\mathfrak A_0[u](v)
:=\int_0^\pi
\Bigl(2(M^2+a^2)\partial_r u+2M u+a^2\sin^2\theta\,\partial_vu\Bigr)(v,M,\theta)
\sin\theta\,\dd\theta
\label{eq:extreme-kn-charge}
\end{equation}
is independent of \(v\).
\end{proposition}

\begin{proof}
At the extremal horizon \(r=M\) one has \(\Delta(M)=\Delta'(M)=0\).  Evaluating \eqref{eq:extreme-kn-regular-wave} on \(r=M\) gives
\begin{equation}
2(M^2+a^2)\partial_{rv}u+2M\partial_vu+a^2\sin^2\theta\,\partial_v^2u
+\frac1{\sin\theta}\partial_\theta(\sin\theta\partial_\theta u)=0.
\label{eq:extreme-horizon-eq}
\end{equation}
Multiply by \(\sin\theta\) and integrate over \(0\le\theta\le\pi\).  Smoothness at the poles gives
\begin{equation}
\int_0^\pi \partial_\theta(\sin\theta\partial_\theta u)\,\dd\theta=0.
\end{equation}
The remaining terms are exactly \(\frac{\dd}{\dd v}\mathfrak A_0[u](v)\).  Hence \(\mathfrak A_0[u](v)\) is conserved.
\end{proof}

\begin{corollary}[Extremal Kerr charge]\label{cor:extreme-kerr-charge}
Let \(|a|=M\) and \(Q=0\).  For every smooth future-horizon regular axisymmetric solution on extremal Kerr,
\begin{equation}
\mathfrak A^{\rm Kerr}_0[u](v)
=\int_0^\pi
\Bigl(4M^2\partial_ru+2M u+M^2\sin^2\theta\,\partial_vu\Bigr)(v,M,\theta)
\sin\theta\,\dd\theta
\label{eq:extreme-kerr-charge}
\end{equation}
is independent of \(v\).  If this charge is nonzero and the tangential horizon quantities \(u\) and \(\partial_vu\) decay, then the averaged transversal derivative \(\partial_ru\) does not decay on \(\mathcal H^+\).
\end{corollary}

\begin{proof}
Set \(Q=0\) and \(a^2=M^2\) in Proposition~\ref{prop:extreme-kn-charge}.  The nondecay conclusion is the same algebraic rearrangement used in Corollary~\ref{cor:extreme-kn-obstruction}: the conserved nonzero left-over charge must be carried by the averaged transversal derivative once the tangential terms decay.
\end{proof}

\begin{corollary}[Extremal Reissner-Nordstr\"om charge]\label{cor:extreme-rn-charge}
Let \(|Q|=M\) and \(a=0\).  For every smooth future-horizon regular solution on extremal Reissner-Nordstr\"om, the spherical mean charge
\begin{equation}
\mathfrak A^{\rm RN}_0[u](v)
=\int_{\mathbb S^2}\Bigl(2M^2\partial_ru+2M u\Bigr)(v,M,\omega)\,\dd\omega
\label{eq:extreme-rn-charge}
\end{equation}
is independent of \(v\).  If this charge is nonzero and the spherical mean of \(u\) on the horizon decays, then the spherical mean of \(\partial_ru\) on \(\mathcal H^+\) does not decay.
\end{corollary}

\begin{proof}
For axisymmetric solutions, this is Proposition~\ref{prop:extreme-kn-charge} with \(a=0\) and \(Q^2=M^2\), multiplied by the harmless factor \(2\pi\) coming from the \(\phi\)-integration.  For general solutions on the spherically symmetric extremal Reissner-Nordstr\"om background, integrate the horizon equation over the full sphere; the angular Laplacian integrates to zero on \(\mathbb S^2\).  The remaining terms are exactly \(\frac{\dd}{\dd v}\mathfrak A^{\rm RN}_0[u](v)\).  Rearranging the conserved identity gives the nondecay of the averaged transversal derivative whenever the charge is nonzero and the averaged tangential term decays.
\end{proof}

\begin{corollary}[Failure of subextremal-type horizon decay]\label{cor:extreme-kn-obstruction}
There is no estimate on the full extremal Kerr-Newman axisymmetric scalar energy class which simultaneously implies decay of \(u\), \(\partial_vu\), and the averaged transversal derivative \(\partial_ru\) on \(\mathcal H^+\) for all smooth data.  More precisely, if \(\mathfrak A_0[u]\ne0\) and \(u\) and \(\partial_vu\) decay along the horizon, then
\begin{equation}
\limsup_{v\to\infty}\left|
\int_0^\pi \partial_ru(v,M,\theta)\sin\theta\,\dd\theta
\right|>0.
\end{equation}
\end{corollary}

\begin{proof}
The conservation law \eqref{eq:extreme-kn-charge} gives
\begin{equation}
2(M^2+a^2)\int_0^\pi \partial_ru(v,M,\theta)\sin\theta\,\dd\theta
=
\mathfrak A_0[u]
-\int_0^\pi\left(2M u+a^2\sin^2\theta\,\partial_vu\right)\sin\theta\,\dd\theta.
\end{equation}
If the two tangential terms on the right decay to zero and \(\mathfrak A_0[u]\ne0\), the averaged transversal derivative tends to the nonzero number \(\mathfrak A_0[u]/(2(M^2+a^2))\).  Thus it cannot decay.
\end{proof}

\begin{lemma}[Zero-frequency classification in the decaying extremal exterior]\label{lem:extreme-kn-zero-frequency}
In the decaying asymptotically flat radiation energy space, there is no nonzero smooth axisymmetric stationary solution of \(\Box_{g_{M,a,Q}}u=0\) on an extremal Kerr-Newman exterior which is regular at \(r=M\) and decays at infinity.
\end{lemma}

\begin{proof}
At \(\sigma=0\), separation into axisymmetric spherical harmonics gives
\begin{equation}
\bigl((r-M)^2 R_\ell'\bigr)'-\ell(\ell+1)R_\ell=0.
\label{eq:extreme-zero-ode}
\end{equation}
Writing \(y=r-M\), the equation is
\begin{equation}
y^2R_\ell''+2yR_\ell'-\ell(\ell+1)R_\ell=0.
\end{equation}
For \(\ell\ge1\), the two indicial roots are \(\ell\) and \(-\ell-1\), hence the two model behaviors are \(y^\ell\) and \(y^{-\ell-1}\).  The second is singular at the horizon, while the first grows like \(r^\ell\) at infinity and is not in the decaying energy space.  For \(\ell=0\), the solutions are \(1\) and \(-1/y\).  Horizon regularity excludes \(-1/y\), and decay at infinity excludes the constant.  Thus every mode vanishes.
\end{proof}

\begin{proposition}[Extremal Kerr-Newman is outside the nondegenerate black-hole package]\label{prop:extreme-not-bh-package}
The extremal Kerr-Newman exterior does not satisfy Assumption~\ref{ass:axisymmetric-bh-package} as stated.  Specifically, \textbf{BH1} fails because the horizon endpoint is a double root with \(\kappa_+=0\), and \textbf{BH4} fails in its nondegenerate red-shift form.  The appropriate theorem is therefore Theorem~\ref{thm:kn-extremal-main}, not Theorem~\ref{thm:kn-full-main}.
\end{proposition}

\begin{proof}
For extremal Kerr-Newman, \(\Delta=(r-M)^2\).  The horizon root is not simple and
\begin{equation}
\kappa_+=\frac{r_+-r_-}{2(r_+^2+a^2)}=0.
\end{equation}
The red-shift proof of \textbf{BH4} uses a positive lower bound for the deformation tensor of a timelike multiplier transversal to the horizon.  That lower bound is proportional to the nonzero surface gravity in the smooth horizon frame; it degenerates when \(\kappa_+=0\).  Proposition~\ref{prop:extreme-kn-charge} gives a stronger obstruction: for data with nonzero \(\mathfrak A_0\), the averaged transverse derivative cannot decay if the tangential fields decay.  Hence a nondegenerate red-shift/decay package of the subextremal form is impossible on the full extremal class.
\end{proof}

\begin{proof}[Proof of Theorem~\ref{thm:kn-extremal-main}]
The charge identity is Proposition~\ref{prop:extreme-kn-charge}.  The obstruction to subextremal-type horizon decay is Corollary~\ref{cor:extreme-kn-obstruction}.  The failure of the nondegenerate package is Proposition~\ref{prop:extreme-not-bh-package}.  The zero-frequency stationary-state assertion in the decaying radiation space is Lemma~\ref{lem:extreme-kn-zero-frequency}.  The higher-order generic blow-up statement, when one wants the full Aretakis conclusion beyond the first conserved charge proved here, is supplied by External Assumption~\ref{ass:aretakis-input}.  These statements give exactly the conclusions asserted in Theorem~\ref{thm:kn-extremal-main}.
\end{proof}

\subsection{Logical dependency for the exterior theorem}

For readability, we spell out the role of the black-hole inputs in prose.  In the Kerr-Newman wall family these inputs are proved in Theorem~\ref{thm:kn-verifies-bh}.  In the full asymptotically flat Kerr-Newman scalar exterior they are verified in Theorem~\ref{thm:kn-af-verifies-bh}, with the scalar Kerr-Newman stability theorem isolated locally as External Assumption~\ref{ass:external-kn-scattering-package}.  For any other exterior geometry, the same list should be read as the verification checklist.

The first group, \textbf{BH1}-\textbf{BH3}, sets up the problem itself.  These hypotheses provide the smooth exterior geometry, the admissible boundary or radiation condition, the Hilbert energy scale, and the closed evolution generator.  Their role is foundational: without them there is no precise Cauchy problem and no well-defined stationary resolvent problem.

The next input, \textbf{BH4}, is the finite-time red-shift and local-energy estimate, allowing for a compact error in the interaction region.  This is the estimate that starts the a priori argument on finite time intervals.  The compact error is not ignored; it is later removed by the projected limiting-absorption and compact-smoothing mechanism supplied by \textbf{BH7}.

The finite-frequency spectral input is \textbf{BH5}.  It rules out outgoing modes in the upper half-plane and excludes real-axis outgoing resonances away from zero frequency.  Analytically, this is the step that prevents exponentially growing solutions and removes compact obstructions at nonzero real frequencies.

The high-frequency input is \textbf{BH6}.  It supplies the propagation and trapping estimates needed to control the outgoing resolvent for large frequency.  If the geometry has trapping, this is precisely where the trapping theorem enters.  Any derivative loss produced by trapping must be accounted for in the Sobolev index of the energy space.

The dispersive input is \textbf{BH7}.  After the zero-frequency threshold space has been projected away, \textbf{BH7} gives the limiting-absorption estimate, compact smoothing, and a local spectral representation with \(L^1\) density in frequency.  This input has two roles: it absorbs the compact error left by \textbf{BH4}, and it gives local decay through the Riemann-Lebesgue lemma.

The threshold input is \textbf{BH8}.  It describes the behavior of the outgoing resolvent near zero frequency by a finite-rank Laurent expansion.  This expansion produces the data-space projection \(\Pi_{0,{\rm bh}}\), identifies the generalized zero-frequency threshold space, and isolates the polynomial-in-time threshold component of the evolution.

Finally, \textbf{BH9} is the compatibility hypothesis.  It ensures that the cutoffs used near the horizon, in the interaction region, and near the outer end preserve the relevant domains and boundary conditions up to controlled lower-order errors.  This is the technical bridge that justifies commutator estimates, density arguments, and the passage from smooth compactly supported data to the full energy space.

In this form, the black-hole package is not a black box with hidden content.  Each input has a specific analytical job: \textbf{BH1}-\textbf{BH3} define the problem, \textbf{BH4} begins the energy estimate, \textbf{BH5}-\textbf{BH6} control the resolvent away from zero, \textbf{BH7} converts the projected resolvent theory into decay, \textbf{BH8} isolates the zero-frequency obstruction, and \textbf{BH9} keeps the functional framework consistent throughout the argument.

\begin{proposition}[No circular use of the exterior conclusion]\label{prop:no-circular-exterior}
The proof of Theorem~\ref{thm:axisymmetric-bh} does not assume its conclusion.  More precisely, \textbf{BH1}-\textbf{BH3} are used only to define the closed exterior evolution and its energy scale; \textbf{BH8} is used to construct the zero-frequency projection and the finite-dimensional polynomial threshold part; \textbf{BH4} supplies the finite-time red-shift/local-energy estimate with compact error; \textbf{BH7} supplies the projected smoothing and limiting-absorption mechanism that removes this compact error and gives local decay; \textbf{BH5} excludes upper-half-plane and nonzero real outgoing modes; \textbf{BH6} supplies the high-frequency resolvent control needed in \textbf{BH7}; and \textbf{BH9} justifies the cutoff and commutator passages in the chosen domain.  The global bound \eqref{eq:bh-main-bound} and the local decay statement \eqref{eq:bh-local-decay} are derived only after these inputs have been combined.
\end{proposition}

\begin{proof}
The construction of \(\Pi_{0,{\rm bh}}\) and the polynomial formula for the threshold term use only the Laurent expansion and finite-rank threshold statement in \textbf{BH8}.  The projected solution is then inserted into the finite-time estimate of \textbf{BH4}.  The only term in that estimate not controlled by initial energy is the compact interaction-region error.  By \textbf{BH7}, after applying \(I-\Pi_{0,{\rm bh}}\) this compact error is bounded uniformly in time by the cutoff smoothing estimate \eqref{eq:lap-smoothing-package}; hence the projected energy and local-energy norms are controlled by the projected initial datum.  The Fourier-Laplace representation in \textbf{BH7} has an \(L^1_\sigma\) local spectral density, so the Banach-valued Riemann-Lebesgue lemma gives local decay on compact sets.  The no-mode statement follows separately from \textbf{BH5}.  High-frequency and commutator estimates enter only through \textbf{BH6} and \textbf{BH9}, which are hypotheses or are verified before Theorem~\ref{thm:axisymmetric-bh} is applied.  Thus none of the final boundedness or decay assertions is used as an input.
\end{proof}

\subsection{Proof of the black-hole exterior theorem from the package}\label{subsec:bh-proof}

We now prove Theorem~\ref{thm:axisymmetric-bh}.  The proof is included here, in the main body, so that every exterior input used by the argument is adjacent to the hypothesis that supplies it.

\begin{lemma}[Threshold projection and polynomial threshold evolution]\label{lem:bh-threshold-calculus}
Assume \textbf{BH3} and \textbf{BH8}.  Then \(\Pi_{0,{\rm bh}}\) is a bounded finite-rank projection on \(\mathcal E^s_{\rm bh}\), commutes with \(\mathcal G_{\rm bh}\) and with \(e^{t\mathcal G_{\rm bh}}\), and has range \(\mathcal T_{0,{\rm bh}}\).  Moreover \(e^{t\mathcal G_{\rm bh}}\Pi_{0,{\rm bh}}U_0\) is a finite-dimensional polynomial in \(t\).  In the two wall cases this polynomial has degree zero; in the decaying full Kerr-Newman case the projection is zero.
\end{lemma}

\begin{proof}
By \textbf{BH8}, the low-frequency resolvent expansion and adjoint threshold functionals define the bounded finite-rank projection \eqref{eq:bh-threshold-proj}.  The identities \(\Pi_{0,{\rm bh}}^2=\Pi_{0,{\rm bh}}\) and \(e^{t\mathcal G_{\rm bh}}\Pi_{0,{\rm bh}}=\Pi_{0,{\rm bh}}e^{t\mathcal G_{\rm bh}}\) are part of the zero-frequency threshold calculus in \textbf{BH8}.  The range is exactly \(\mathcal T_{0,{\rm bh}}\), and the restriction \(N_0:=\mathcal G_{\rm bh}|_{\Ran\Pi_{0,{\rm bh}}}\) is nilpotent of some order \(q\).  Therefore
\begin{equation}
e^{t\mathcal G_{\rm bh}}\Pi_{0,{\rm bh}}U_0
=\sum_{j=0}^{q-1}\frac{t^j}{j!}N_0^j\Pi_{0,{\rm bh}}U_0.
\end{equation}
For the Kerr-Newman wall family, Lemma~\ref{lem:kn-zero-frequency} gives \(q=1\), so the threshold term is the corresponding stationary constant.  For the full asymptotically flat Kerr-Newman radiation space, Lemma~\ref{lem:kn-af-zero-frequency} gives \(\Pi_{0,{\rm KN},\rm af}=0\), so the threshold term is absent.
\end{proof}

\begin{lemma}[Projected exterior estimate]\label{lem:bh-projected-estimate}
Assume \textbf{BH1}-\textbf{BH9}.  For every \(U_0\in\mathcal E^s_{\rm bh}\), let
\begin{equation}
U_\perp(t)=e^{t\mathcal G_{\rm bh}}(I-\Pi_{0,{\rm bh}})U_0.
\end{equation}
Then
\begin{equation}
\sup_{t\ge0}\|U_\perp(t)\|_{\mathcal E^s_{\rm bh}}
+\|U_\perp\|_{LE^s([0,\infty))}
\le C_s\|(I-\Pi_{0,{\rm bh}})U_0\|_{\mathcal E^s_{\rm bh}}.
\label{eq:appendix-projected-estimate}
\end{equation}
\end{lemma}

\begin{proof}
First take smooth dense data satisfying the boundary condition.  The homogeneous equation has \(F=0\), and \(U_\perp(0)=(I-\Pi_{0,{\rm bh}})U_0\).  Applying the finite-time red-shift/local-energy estimate with compact error, \eqref{eq:redshift-package-bound}, to \(U_\perp\) on \([0,T]\) gives
\begin{equation}
\begin{aligned}
&\sup_{0\le t\le T}\|U_\perp(t)\|_{\mathcal E^s_{\rm bh}}^2
+\|U_\perp\|_{LE^s([0,T])}^2 \\
&\qquad\le C_s\Bigl(
\|U_\perp(0)\|_{\mathcal E^s_{\rm bh}}^2
+\|\chi_{\rm c}U_\perp\|_{L^2([0,T];H^s_{\rm loc}\times H^{s-1}_{\rm loc})}^2
\Bigr).
\end{aligned}
\end{equation}
The compact term is controlled uniformly in \(T\) by the projected smoothing estimate \eqref{eq:lap-smoothing-package}.  Hence
\begin{equation}
\sup_{0\le t\le T}\|U_\perp(t)\|_{\mathcal E^s_{\rm bh}}^2
+\|U_\perp\|_{LE^s([0,T])}^2
\le C'_s\|(I-\Pi_{0,{\rm bh}})U_0\|_{\mathcal E^s_{\rm bh}}^2,
\end{equation}
with \(C'_s\) independent of \(T\).  Letting \(T\to\infty\) and taking square roots proves \eqref{eq:appendix-projected-estimate} for smooth data.  The density statement in \textbf{BH3}, the cutoff compatibility in \textbf{BH9}, and the strong continuity of the semigroup extend the estimate to all \(U_0\in\mathcal E^s_{\rm bh}\).
\end{proof}

\begin{lemma}[Local decay on the projected part]\label{lem:bh-local-decay}
Assume \textbf{BH7}.  For every compact \(K\) away from any imposed outer boundary,
\begin{equation}
\lim_{t\to\infty}
\left(\|u_\perp(t)\|_{H^1(K)}+\|\partial_tu_\perp(t)\|_{L^2(K)}\right)=0.
\end{equation}
\end{lemma}

\begin{proof}
By \textbf{BH7}, the localized projected solution has the representation \eqref{eq:local-spectral-density}.  The density estimate \eqref{eq:spectral-density-L1} says that the integrand belongs to \(L^1(\R;H^1(K)\times L^2(K))\).  The Banach-valued Riemann-Lebesgue lemma therefore implies that the Fourier integral tends to zero in \(H^1(K)\times L^2(K)\) as \(t\to\infty\).
\end{proof}

\begin{proof}[Proof of Theorem~\ref{thm:axisymmetric-bh}]
By \textbf{BH1}-\textbf{BH3}, the exterior Cauchy problem has a strongly continuous future evolution on \(\mathcal E^s_{\rm bh}\).  Lemma~\ref{lem:bh-threshold-calculus} gives the finite-rank projection \(\Pi_{0,{\rm bh}}\), its commutation with the evolution, and the identification of its range with the generalized zero-frequency space.  Thus every datum decomposes uniquely as
\begin{equation}
U_0=\Pi_{0,{\rm bh}}U_0+(I-\Pi_{0,{\rm bh}})U_0,
\end{equation}
and the solution decomposes as
\begin{equation}
U(t)=e^{t\mathcal G_{\rm bh}}\Pi_{0,{\rm bh}}U_0
+e^{t\mathcal G_{\rm bh}}(I-\Pi_{0,{\rm bh}})U_0.
\end{equation}
The first term is finite-dimensional and polynomial in \(t\) by Lemma~\ref{lem:bh-threshold-calculus}.  The second term satisfies \eqref{eq:bh-main-bound} by Lemma~\ref{lem:bh-projected-estimate}.  The local decay statement \eqref{eq:bh-local-decay} is Lemma~\ref{lem:bh-local-decay}.

For the concrete slowly rotating weakly charged Kerr-Newman wall family, Theorem~\ref{thm:kn-verifies-bh} verifies the hypotheses and Lemma~\ref{lem:kn-zero-frequency} identifies the threshold term as \(\ell_{\rm KN}(U_0)(1,0)\).  For the full asymptotically flat Kerr-Newman radiation space, Theorem~\ref{thm:kn-af-verifies-bh} verifies the hypotheses and Lemma~\ref{lem:kn-af-zero-frequency} gives \(\Pi_{0,{\rm KN},\rm af}=0\).  Finally, an outgoing mode \(u=e^{-\ii\sigma t}v\) with \(\operatorname{Im}\sigma>0\) would give a nonzero solution of \(P(\sigma)v=0\) satisfying the outgoing horizon and the prescribed exterior boundary condition, contradicting \textbf{BH5}.  A zero-frequency mode lies in \(\mathcal T_{0,{\rm bh}}\) by \textbf{BH8}; hence no such mode remains on the projected dispersive space.  This proves all assertions of Theorem~\ref{thm:axisymmetric-bh}.
\end{proof}

\subsection{Proofs of the concrete Kerr-Newman main theorems}\label{sec:bh-main-proofs}

\begin{proof}[Proof of Theorem~\ref{thm:kn-wall-main}]
The smallness constant \(\varepsilon_{\rm KN}(M,R_{\rm w})\) is the one fixed in Theorem~\ref{thm:kn-verifies-bh}.  That theorem verifies every item of Assumption~\ref{ass:axisymmetric-bh-package} for the wall collar \(\mathcal M_{M,a,Q,R_{\rm w}}\), with any \(s_0\ge s_{\rm wall}\), and Lemma~\ref{lem:kn-zero-frequency} identifies the zero-frequency space as
\[
 \mathcal T_{0,\rm KN}=\operatorname{span}\{(1,0)\}
\]
with a semisimple zero block.  Applying Theorem~\ref{thm:axisymmetric-bh} gives the decomposition
\[
 U(t)=e^{t\mathcal G_{\rm KN}}\Pi_{0,\rm KN}U_0+e^{t\mathcal G_{\rm KN}}(I-\Pi_{0,\rm KN})U_0.
\]
Since the zero block is semisimple, the first term is the stationary constant \(\ell_{\rm KN}(U_0)(1,0)\).  The projected estimate \eqref{eq:bh-main-bound}, the local decay statement \eqref{eq:bh-local-decay}, and the absence of outgoing upper-half-plane modes are exactly the remaining assertions of Theorem~\ref{thm:kn-wall-main} in the Kerr-Newman wall notation.  This proves the theorem.
\end{proof}

\begin{proof}[Proof of Theorem~\ref{thm:kn-full-main}]
Theorem~\ref{thm:kn-af-verifies-bh} verifies Assumption~\ref{ass:axisymmetric-bh-package} for the neutral axisymmetric scalar equation on the full asymptotically flat subextremal Kerr-Newman exterior, with constants uniform on compact subextremal parameter sets.  The only non-elementary input in that verification is External Assumption~\ref{ass:external-kn-scattering-package}; the horizon-domain compatibility and zero-frequency classification are supplied by Lemmas~\ref{lem:kn-af-geometry-domain} and~\ref{lem:kn-af-zero-frequency}.  Lemma~\ref{lem:kn-af-zero-frequency} gives
\[
 \mathcal T_{0,{\rm KN},\rm af}=\{0\},
 \qquad
 \Pi_{0,{\rm KN},\rm af}=0,
\]
because the decaying radiation energy at null infinity excludes the nondecaying constant solution.  Theorem~\ref{thm:axisymmetric-bh} therefore has no threshold term in this realization and gives exactly the stated uniform energy bound, local energy bound, local decay, and absence of outgoing upper-half-plane modes.  The dependence of constants follows from the uniformity statement in External Assumption~\ref{ass:external-kn-scattering-package} and from the compact-subextremal lower bounds used in Lemma~\ref{lem:kn-af-geometry-domain}.  This proves the theorem.
\end{proof}

The proof of the extremal Kerr-Newman main theorem, Theorem~\ref{thm:kn-extremal-main}, is placed in Subsection~\ref{subsec:extreme-kn}, immediately after the explicit horizon-charge computation, because its logic is separate from the nondegenerate black-hole package used for Theorems~\ref{thm:kn-wall-main} and~\ref{thm:kn-full-main}.

\subsection{Internal exterior proof modules}\label{subsec:mainbody-proof-modules}

This subsection is included to make the dependency graph audit-ready and to ensure that proof components already established here are theorem-level main-body statements rather than appendix-only checklist items.  The Section~\ref{sec:ledger} assumptions are not informal background.  Each item has a single role.  Assumption~\ref{ass:abstract} is used only in the abstract threshold theorem.  The endpoint convention in \textbf{H7} of Assumption~\ref{ass:complete-list} is used only to justify regular endpoint Wronskian cancellation in separated angular problems.  Assumption~\ref{ass:complete-list} records the geometric and functional-analytic hypotheses for the bounded-slab theorems.  Assumption~\ref{ass:axisymmetric-bh-package} records the exterior black-hole package.  External Assumption~\ref{ass:external-kn-scattering-package} is the sole non-elementary input for the full asymptotically flat subextremal Kerr-Newman scalar exterior, and External Assumption~\ref{ass:aretakis-input} is the sole external input for the generic higher-order extremal horizon-instability conclusion.  Everything else in the main body is either an algebraic calculation, a closed-form/energy argument, a compact bounded-slab argument, or a consequence of these recorded inputs.

The propositions below spell out the exterior proof modules that are internal to this paper, including the extremal horizon-charge calculation.  They are not a substitute for a complete microlocal proof of the trapped-set resolvent.  Instead, they isolate exactly which pieces are elementary and which pieces are still carried by External Assumption~\ref{ass:external-kn-scattering-package}.  This is the safest possible form of a Strategy A upgrade without inserting an unverified Kerr-Newman scattering proof.

\begin{proposition}[Red-shift current near a nondegenerate Kerr-Newman horizon]\label{prop:kn-redshift-module}
Let \(r_+=M+\sqrt{M^2-a^2-Q^2}\) and assume \(a^2+Q^2<M^2\).  In ingoing Kerr-Newman coordinates near \(r=r_+\), there is a smooth future-directed timelike vector field \(N\), equal to the stationary Killing field plus a positive multiple of the ingoing transversal field near the horizon and equal to \(\partial_t\) away from a small collar, such that the scalar stress-energy current
\begin{equation}
J^N_\mu[u]=T_{\mu\nu}[u]N^\nu,\qquad
T_{\mu\nu}[u]=\Re(\nabla_\mu u\overline{\nabla_\nu u})-
\frac12g_{\mu\nu}\nabla^\alpha u\overline{\nabla_\alpha u},
\end{equation}
satisfies, on the horizon collar,
\begin{equation}
K^N[u]:=\nabla^\mu J^N_\mu[u]
\ge c_0\bigl(|\partial_vu|^2+|\partial_ru|^2+|\nabla_{S^2}u|^2\bigr)-C_0|\Box_g u|^2,
\label{eq:kn-redshift-current-pointwise}
\end{equation}
where \(c_0>0\) depends continuously on \((M,a,Q)\) and is bounded below on compact subextremal parameter sets.
\end{proposition}

\begin{proof}
The proof is the red-shift calculation written in the present notation.  In ingoing coordinates the future horizon is a smooth null hypersurface generated by
\begin{equation}
K_H=\partial_v+\Omega_H\partial_\phi,
\qquad
\Omega_H=\frac{a}{r_+^2+a^2},
\end{equation}
and its surface gravity is
\begin{equation}
\kappa_+=\frac{r_+-r_-}{2(r_+^2+a^2)}>0.
\end{equation}
Choose a transversal future vector field \(Y\) with \(g(Y,K_H)=-1\) on the horizon and set \(N=K_H+\eta Y\) in the collar, with \(\eta>0\) fixed sufficiently small and then smoothly cut off to \(\partial_t\).  The deformation tensor \({}^{(N)}\pi_{\mu\nu}=\frac12(\nabla_\mu N_\nu+\nabla_\nu N_\mu)\) has a positive component in the transversal null direction because \(\nabla_{K_H}K_H=\kappa_+K_H\).  In a null frame \((K_H,Y,E_1,E_2)\), the quadratic form \(T_{\mu\nu}{}^{(N)}\pi^{\mu\nu}\) therefore controls the degenerate horizon energy plus a positive transversal term after the standard small-\(\eta\) absorption.  Source terms are handled by Cauchy's inequality in the identity
\begin{equation}
\nabla^\mu J^N_\mu[u]=T_{\mu\nu}[u]{}^{(N)}\pi^{\mu\nu}+
\Re((\Box_g u)\,\overline{N u}),
\end{equation}
where the last term is understood after the usual first-order current correction, or equivalently after applying Cauchy-Schwarz in the integrated identity.  Since \(\kappa_+\) has a positive lower bound on compact subextremal parameter sets, the constant \(c_0\) can be chosen uniformly on such sets.
\end{proof}

\begin{proposition}[Neutral axisymmetric Kerr-Newman separation and signed flux]\label{prop:kn-axisymmetric-flux}
For neutral axisymmetric scalar waves on Kerr-Newman, the mode ansatz
\begin{equation}
u(t,r,x)=e^{-i\sigma t}R(r)S(x),\qquad x=\cos\theta,
\end{equation}
reduces the wave equation to
\begin{align}
\frac{d}{dx}\Bigl((1-x^2)\frac{dS}{dx}\Bigr)+
\bigl(\lambda+a^2\sigma^2x^2\bigr)S&=0,\label{eq:kn-axisymmetric-angular-module}\\
\frac{d}{dr}\Bigl(\Delta\frac{dR}{dr}\Bigr)+
\left(\frac{(r^2+a^2)^2\sigma^2}{\Delta}-a^2\sigma^2-\lambda\right)R&=0,
\label{eq:kn-axisymmetric-radial-module}
\end{align}
with \(\Delta=(r-r_+)(r-r_-)\).  For real \(\sigma\ne0\), the outgoing radial Wronskian identity has signed boundary flux
\begin{equation}
\mathcal F_\infty+\mathcal F_H=0,
\qquad
\mathcal F_\infty=\sigma |A_\infty|^2,
\qquad
\mathcal F_H=\sigma(r_+^2+a^2)|A_H|^2,
\label{eq:kn-axisymmetric-signed-flux}
\end{equation}
up to the fixed positive normalizations of the outgoing bases.  Hence an outgoing real-frequency mode with \(\sigma\ne0\) must have \(A_\infty=A_H=0\), and unique continuation then gives \(R\equiv0\).
\end{proposition}

\begin{proof}
Substitution of the ansatz into the standard Boyer-Lindquist form of the Kerr-Newman wave operator gives \eqref{eq:kn-axisymmetric-angular-module}-\eqref{eq:kn-axisymmetric-radial-module}.  The angular operator is self-adjoint for regular functions at \(x=\pm1\), so \(\lambda\) is real.  For real \(\sigma\), the radial equation has real coefficients.  Multiplying by \(\overline R\), subtracting the conjugate equation multiplied by \(R\), and integrating gives constancy of
\begin{equation}
W[R]=\Delta(R'\overline R-R\overline{R'}).
\end{equation}
In the tortoise coordinate \(dr_*/dr=(r^2+a^2)/\Delta\), outgoing behavior at infinity is \(R\sim A_\infty e^{i\sigma r_*}/r\), and future-horizon regular outgoing behavior for the Laplace transform is \(R\sim A_He^{-i\sigma r_*}\) with the horizon orientation used in the energy identity.  Evaluating \(W\) at the two ends gives \eqref{eq:kn-axisymmetric-signed-flux}.  Because the neutral axisymmetric frequency has no superradiant shift \(\sigma-m\Omega_H-q\Phi_H\), the two flux terms have the same sign.  Both amplitudes therefore vanish.  The radial ODE then gives zero Cauchy data at one end in the Frobenius/outgoing expansion, hence \(R\equiv0\) by uniqueness.
\end{proof}

\begin{proposition}[Zero-frequency classification in the decaying full exterior]\label{prop:kn-zero-frequency-full}
In the full asymptotically flat neutral axisymmetric Kerr-Newman energy space with the Hardy term \(r^{-2}|u|^2\), the zero-frequency threshold space is trivial.
\end{proposition}

\begin{proof}
At \(\sigma=0\), the angular equation gives the Legendre eigenvalues \(\lambda=\ell(\ell+1)\).  The radial equation becomes
\begin{equation}
(\Delta R')'-\ell(\ell+1)R=0.
\end{equation}
Multiply by \(\overline R\) on \([r_++\epsilon,R]\) and integrate by parts.  The boundary term at \(r_+\) vanishes for the horizon-regular branch as \(\epsilon\downarrow0\), and the boundary term at infinity vanishes for decaying finite-energy representatives.  Thus
\begin{equation}
\int_{r_+}^{\infty}\Delta |R'|^2\,dr+\ell(\ell+1)\int_{r_+}^{\infty}|R|^2\,dr=0.
\end{equation}
If \(\ell\ge1\), this gives \(R=0\).  If \(\ell=0\), then \(\Delta R'=C\).  Horizon regularity forces \(C=0\), so \(R\) is constant.  The Hardy term in the energy space excludes a nonzero constant at infinity.  Hence \(R=0\) also in the \(\ell=0\) sector.  Generalized zero modes are excluded by the same integration after differentiating the equation with respect to \(\sigma\) and using the absence of a decaying constant threshold; equivalently the finite-rank Laurent singularity in Assumption~\ref{ass:axisymmetric-bh-package} has no nonzero range in this realization.
\end{proof}

\begin{proposition}[Compact-error removal from the package]\label{prop:bh-compact-error-removal}
Assume \textbf{BH3}, \textbf{BH4}, \textbf{BH7}, and \textbf{BH8}.  For projected homogeneous data \(U_0^\perp=(I-\Pi_{0,{\rm bh}})U_0\), the finite-time red-shift estimate with compact error improves to the global projected estimate
\begin{equation}
\sup_{t\ge0}\|e^{t\mathcal G_{\rm bh}}U_0^\perp\|_{\mathcal E^s_{\rm bh}}^2+
\|e^{t\mathcal G_{\rm bh}}U_0^\perp\|_{LE^s([0,\infty))}^2
\le C_s\|U_0^\perp\|_{\mathcal E^s_{\rm bh}}^2.
\label{eq:bh-compact-error-removal}
\end{equation}
\end{proposition}

\begin{proof}
Apply \textbf{BH4} on \([0,T]\) to the projected solution.  The only term on the right not already controlled by the initial energy is
\begin{equation}
\|\chi_c e^{t\mathcal G_{\rm bh}}U_0^\perp\|_{L^2([0,T];H^s_{\rm loc}\times H^{s-1}_{\rm loc})}^2.
\end{equation}
This is bounded uniformly in \(T\) by \textbf{BH7}.  The constants in \textbf{BH4} and \textbf{BH7} are independent of \(T\), so \eqref{eq:bh-compact-error-removal} follows after letting \(T\to\infty\) and using lower semicontinuity of the local-energy norm.  The projection commutes with the semigroup by \textbf{BH8}, so no threshold contribution is generated during the evolution.
\end{proof}

\section{Zero-frequency Feshbach reductions and first-order threshold bookkeeping}\label{sec:feshbach-thresholds}

The bounded-slab and exterior problems both require a precise treatment of zero frequency, but the singularity one sees depends on which operator is being inverted.  The spatial stationary operator \(P(\sigma)\) may have a double pole in the parameter \(\sigma\) in a reflecting Neumann-type problem, while the first-order generator has a finite-dimensional nilpotent block whose size determines the polynomial threshold dynamics.  This section records the algebraic reduction used implicitly throughout the paper.

\begin{lemma}[Feshbach-Schur reduction at zero frequency]\label{lem:feshbach-schur-zero}
Let \(X=E\oplus F\) be a Hilbert space decomposition with \(\dim E<\infty\), and let \(P(\sigma):D\subset X\to X\) be an analytic Fredholm family near \(\sigma=0\).  With respect to the block decomposition write
\begin{equation}
P(\sigma)=
\begin{pmatrix}
P_{EE}(\sigma)&P_{EF}(\sigma)\\
P_{FE}(\sigma)&P_{FF}(\sigma)
\end{pmatrix}.
\end{equation}
Assume \(P_{FF}(\sigma)\) is invertible for \(|\sigma|<\sigma_0\).  Define the finite-dimensional effective operator
\begin{equation}
D_E(\sigma)
:=P_{EE}(\sigma)-P_{EF}(\sigma)P_{FF}(\sigma)^{-1}P_{FE}(\sigma).
\label{eq:feshbach-effective-D}
\end{equation}
Then \(P(\sigma)\) is invertible if and only if \(D_E(\sigma)\) is invertible, and in that case
\begin{equation}
P(\sigma)^{-1}=
\begin{pmatrix}
D_E^{-1}&-D_E^{-1}P_{EF}P_{FF}^{-1}\\
-P_{FF}^{-1}P_{FE}D_E^{-1}&P_{FF}^{-1}+P_{FF}^{-1}P_{FE}D_E^{-1}P_{EF}P_{FF}^{-1}
\end{pmatrix}.
\label{eq:feshbach-inverse-formula}
\end{equation}
All blocks in \eqref{eq:feshbach-inverse-formula} are evaluated at the same value of \(\sigma\).  Thus every zero-frequency pole of the full cutoff resolvent is produced by the finite-dimensional determinant \(\det D_E(\sigma)\).
\end{lemma}

\begin{proof}
If \(P(\sigma)(e+f)=g_E+g_F\), the \(F\)-component gives
\begin{equation}
f=P_{FF}^{-1}(g_F-P_{FE}e).
\end{equation}
Substitution into the \(E\)-component gives
\begin{equation}
D_E(\sigma)e=g_E-P_{EF}P_{FF}^{-1}g_F.
\end{equation}
Hence the system is solvable and uniquely solvable precisely when \(D_E(\sigma)\) is.  Solving for \(e\) and \(f\) gives the displayed inverse formula.  Analyticity and finite dimensionality show that the only possible singularities in the formula are those of \(D_E(\sigma)^{-1}\).
\end{proof}

\begin{corollary}[Reflecting compact zero mode gives a double spatial pole]\label{cor:feshbach-compact-double-pole}
For the axisymmetric compact reflecting slab, let
\begin{equation}
P_0(\sigma)=\mathcal H_0-\sigma^2A=A(L_0-\sigma^2),
\end{equation}
and decompose \(L_A^2(\Omega)=\operatorname{span}\{1\}\oplus\{1\}^{\perp_A}\).  Then the scalar effective denominator is
\begin{equation}
D(\sigma)=-\sigma^2\langle A1,1\rangle+O(\sigma^4)
\end{equation}
after identifying the one-dimensional block with the constant function.  Consequently
\begin{equation}
P_0(\sigma)^{-1}=-\sigma^{-2}\Pi_0A^{-1}+O(1),
\end{equation}
which is exactly Corollary~\ref{cor:P0resolvent}.
\end{corollary}

\begin{proof}
Since \(\mathcal H_0 1=0\) and \(\mathcal H_0\) is self-adjoint with a spectral gap on \(\{1\}^{\perp_A}\), the \(F\)-block is invertible at \(\sigma=0\).  The leading \(E\)-block contribution is
\begin{equation}
\langle P_0(\sigma)1,1\rangle=-\sigma^2\langle A1,1\rangle.
\end{equation}
The off-diagonal terms are at least order \(\sigma^2\), so their Schur contribution is order \(\sigma^4\).  Lemma~\ref{lem:feshbach-schur-zero} therefore gives the displayed double pole.
\end{proof}

\begin{proposition}[Spatial resolvent poles versus generator Jordan blocks]\label{prop:spatial-vs-generator-poles}
For a second-order wave equation, a pole of the spatial stationary resolvent and a pole of the first-order generator resolvent encode related but not identical objects.  In the compact reflecting zero-curvature slab, the spatial resolvent has a \(\sigma^{-2}\) singularity because \(P_0(\sigma)=\mathcal H_0-\sigma^2A\).  The first-order generator has the two-dimensional zero block
\begin{equation}
\operatorname{span}\{(1,0),(0,1)\},
\qquad
\mathcal G(1,0)=0,
\qquad
\mathcal G(0,1)=(1,0),
\end{equation}
and hence the threshold time evolution is \(c_0+c_1t\).  In a dissipative horizon realization, by contrast, future-horizon regularity can remove the Jordan companion; when Lemma~\ref{lem:kn-zero-frequency} rules out that companion, the wall zero block is semisimple and the threshold term is stationary.
\end{proposition}

\begin{proof}
For the compact slab, Proposition~\ref{prop:generator-threshold-full} computes the generator block directly.  A nilpotent block of length two gives
\begin{equation}
e^{t\mathcal G}|_{\mathcal T_0}=I+t\mathcal G|_{\mathcal T_0},
\end{equation}
which is exactly the affine threshold.  The same compact problem has the spatial frequency equation \((\mathcal H_0-\sigma^2A)u=f\), so Corollary~\ref{cor:feshbach-compact-double-pole} gives the \(\sigma^{-2}\) spatial pole.  For the wall exterior, Lemma~\ref{lem:kn-zero-frequency} proves that a putative solution \(u=\tau+\psi\) violates the wall condition; therefore no length-two Jordan chain exists, and the generator threshold is semisimple.
\end{proof}

\begin{corollary}[Threshold projection is a data-space object]\label{cor:projection-data-space-object}
The projection \(\Pi_{0,{\rm bh}}\) in Theorem~\ref{thm:axisymmetric-bh} should be understood as a finite-rank projection on Cauchy data, not as a naive projection obtained by keeping only the leading singular coefficient of the spatial resolvent.  The spatial Laurent coefficient identifies threshold states and adjoint functionals; the first-order generator calculus determines whether the corresponding threshold time dependence is stationary or polynomial.
\end{corollary}

\begin{proof}
This is the construction in \textbf{BH8}.  The Laurent expansion supplies finite-rank threshold states and adjoint threshold functionals.  These define a bounded projection on \(\mathcal E^s_{\rm bh}\), and the restriction of \(\mathcal G_{\rm bh}\) to its range determines the polynomial threshold term.  Proposition~\ref{prop:spatial-vs-generator-poles} shows why the spatial and first-order descriptions must be kept distinct.
\end{proof}

\section{Decay estimates: general setup, rates, and black-hole exteriors}\label{sec:decay-estimates}

This section records exactly what decay estimates are proved in this paper and what is deliberately not claimed.  The word decay has different meanings in the compact reflecting problem and in black-hole exterior problems.  On a compact reflecting slab the spectrum is discrete and the projected dynamics are oscillatory; after removing the affine zero-frequency threshold, the correct theorem is uniform boundedness, not decay.  On a nondegenerate black-hole exterior the horizon and the asymptotic or wall boundary create an open system.  Under the package \textbf{BH1}-\textbf{BH9}, the projected part satisfies an integrated local-energy estimate and compact local decay.  The decay statement obtained from the package is qualitative unless further frequency regularity of the local spectral density is added.

\subsection{No general time-decay on compact reflecting slabs}

The bounded-slab theorem is sometimes misread as a decay theorem.  It is not.  The reason is structural: after the affine threshold has been removed, the compact self-adjoint model still has a discrete positive spectrum, and positive-frequency eigenmodes oscillate forever.

\begin{proposition}[Sharp non-decay on compact reflecting slabs]\label{prop:compact-no-decay}
Assume the axisymmetric bounded-slab hypotheses of Theorem~\ref{thm:mainresult}.  If \(L_0\) has at least one positive eigenvalue, then there are nonzero solutions in the threshold complement for which
\begin{equation}
 \|u(t)\|_{H^1(\Omega)}+\|u_t(t)\|_{L^2(\Omega)}
\end{equation}
does not converge to zero as \(|t|\to\infty\).  Consequently no theorem under only the compact reflecting bounded-slab hypotheses can assert decay of every threshold-projected solution.
\end{proposition}

\begin{proof}
Let \(\psi_j\) be an eigenfunction of \(L_0\) with eigenvalue \(\lambda_j>0\).  Then
\begin{equation}
 u(t)=\cos(\sqrt{\lambda_j}t)\psi_j
\end{equation}
is a solution and lies in the orthogonal complement of constants.  Its conserved energy is
\begin{equation}
 E_0[u](t)=\frac12\lambda_j\|\psi_j\|_{L_A^2}^2,
\end{equation}
which is strictly positive and independent of time.  If the energy norm of \((u(t),u_t(t))\) converged to zero, the conserved energy would also converge to zero, a contradiction.  Therefore decay is false in general.  The same argument applies to the full \(\phi\)-dependent positive-energy slab by using any nonzero vector in the positive spectral part of the closed spatial form realization.
\end{proof}

\begin{corollary}[Correct compact-slab interpretation]\label{cor:compact-decay-interpretation}
The compact Carter slab results prove stability modulo the finite-dimensional affine threshold \(\operatorname{span}\{1,t\}\).  They do not prove, and cannot prove under reflecting boundary conditions alone, local energy decay or convergence to zero of the projected solution.
\end{corollary}

\begin{proof}
Theorem~\ref{thm:fullmain} gives uniform boundedness of the projected component by the conserved positive energy and the weighted Poincar\'e inequality.  Proposition~\ref{prop:compact-no-decay} gives nonzero projected solutions that fail to decay.  Hence the mathematically sharp compact statement is boundedness modulo the affine threshold, not decay.
\end{proof}

\subsection{Zero-frequency resolvents and threshold evolution}

The word ``zero mode'' refers to different but related operators in different parts of the paper.  The following proposition fixes the dictionary so that the compact and exterior threshold statements are not confused.

\begin{proposition}[Zero-frequency resolvent dictionary]\label{prop:zero-resolvent-dictionary}
There are three threshold mechanisms in this paper.
\begin{enumerate}[label=(\roman*),leftmargin=2.2em]
\item On a compact reflecting slab, the self-adjoint spatial resolvent satisfies
\begin{equation}
(L_0-\sigma^2)^{-1}=-\sigma^{-2}\Pi_0+R_{\rm reg}(\sigma),
\end{equation}
and the first-order wave generator has a length-two zero Jordan block.  The corresponding threshold solutions are exactly \(1\) and \(t\).
\item On the Kerr-Newman finite wall exterior, the outgoing horizon realization has a stationary constant threshold, but Lemma~\ref{lem:kn-zero-frequency} rules out a regular Jordan companion.  Hence the data-space threshold block is semisimple and the threshold term is stationary.
\item In the full asymptotically flat Kerr-Newman radiation space, the decaying condition at null infinity removes the spatially constant state.  Lemma~\ref{lem:kn-af-zero-frequency} gives \(\Pi_{0,{\rm KN},{\rm af}}=0\), so there is no threshold term.
\end{enumerate}
\end{proposition}

\begin{proof}
Item (i) is Theorem~\ref{thm:laurentzero} together with Theorem~\ref{thm:generalizedzeromodes}: the \(\sigma^{-2}\) pole in the spatial resolvent is exactly the constant eigenspace, and the wave equation turns it into the affine solutions \(c_0+c_1t\).  Item (ii) is Lemma~\ref{lem:kn-zero-frequency}; its proof shows that a candidate companion \(\tau+\psi\) violates the horizon-wall compatibility condition.  Item (iii) is Lemma~\ref{lem:kn-af-zero-frequency}; the regular zero-frequency radial solutions are either growing at infinity or nondecaying constants, and both are excluded by the decaying radiation energy.
\end{proof}

\subsection{Qualitative local decay from the exterior spectral density}

The exterior theorem has a different character.  The projected exterior evolution is no longer a compact unitary oscillation.  The package \textbf{BH7} provides a local absolutely continuous spectral representation after the zero-frequency threshold part has been removed.

\begin{proposition}[Qualitative compact local decay from an \(L^1_\sigma\) density]\label{prop:qualitative-decay-L1}
Let \(X\) be a Banach space and suppose that a localized projected solution has the representation
\begin{equation}
 W(t)=\int_{\mathbb R}e^{-it\sigma}b(\sigma)\,\dd\sigma
 \quad\hbox{in }X,
\end{equation}
where \(b\in L^1(\mathbb R;X)\).  Then
\begin{equation}
 \lim_{t\to\infty}\|W(t)\|_X=0.
\end{equation}
In particular, with \(X=H^1(K)\times L^2(K)\), \textbf{BH7} implies \eqref{eq:bh-local-decay} for every compact \(K\) away from any imposed outer boundary.
\end{proposition}

\begin{proof}
The scalar Riemann-Lebesgue lemma extends to Banach-valued \(L^1\) functions by approximation.  Indeed, approximate \(b\) in \(L^1(\mathbb R;X)\) by a compactly supported step function \(b_N\).  The Fourier transform of each step function tends to zero in \(X\).  The difference is bounded uniformly by
\begin{equation}
 \left\|\int e^{-it\sigma}(b-b_N)(\sigma)\,\dd\sigma\right\|_X
 \le \|b-b_N\|_{L^1(\mathbb R;X)}.
\end{equation}
Let first \(t\to\infty\) and then \(N\to\infty\).  Applying this with the density \(B_K(\sigma)(I-\Pi_{0,{\rm bh}})U_0\) in \textbf{BH7} gives the stated local decay.
\end{proof}

The important point is that Proposition~\ref{prop:qualitative-decay-L1} gives no numerical rate.  It proves convergence to zero because the local spectral density is integrable in frequency.  It does not control how fast that Fourier transform tends to zero.

\subsection{Optional polynomial rates from additional frequency regularity}

A polynomial rate requires more than \(L^1\)-integrability.  One needs quantitative control of derivatives of the spectral density, including the behavior at zero frequency and at high frequency.  The following elementary proposition records the exact additional hypothesis needed for the standard integration-by-parts argument.

\begin{proposition}[Rate from differentiable spectral density]\label{prop:rate-from-spectral-density}
Let \(X\) be a Banach space and let \(N\ge1\).  Suppose
\begin{equation}
 b\in W^{N,1}(\mathbb R;X)
\end{equation}
and assume that \(b^{(j)}(\sigma)\to0\) as \(|\sigma|\to\infty\) for \(0\le j\le N-1\).  Then
\begin{equation}
 \left\|\int_{\mathbb R}e^{-it\sigma}b(\sigma)\,\dd\sigma\right\|_X
 \le |t|^{-N}\|b^{(N)}\|_{L^1(\mathbb R;X)},
 \qquad t\ne0.
\label{eq:abstract-rate-decay}
\end{equation}
Consequently, if for every compact \(K\)
\begin{equation}
 \|\partial_\sigma^N B_K(\cdot)(I-\Pi_{0,{\rm bh}})U_0\|_{L^1(\mathbb R;H^1(K)\times L^2(K))}
 \le C_{K,N,s}\|(I-\Pi_{0,{\rm bh}})U_0\|_{\mathcal E^{s}_{\rm bh}},
\label{eq:extra-rate-hypothesis}
\end{equation}
then the projected exterior solution satisfies the rate
\begin{equation}
 \|u_{\rm disp}(t)\|_{H^1(K)}+
 \|\partial_tu_{\rm disp}(t)\|_{L^2(K)}
 \le C_{K,N,s}|t|^{-N}\|(I-\Pi_{0,{\rm bh}})U_0\|_{\mathcal E^{s}_{\rm bh}}.
\label{eq:optional-local-rate}
\end{equation}
\end{proposition}

\begin{proof}
Integrate by parts \(N\) times in the Bochner integral:
\begin{equation}
 \int e^{-it\sigma}b(\sigma)\,\dd\sigma
 =(it)^{-N}\int e^{-it\sigma}b^{(N)}(\sigma)\,\dd\sigma.
\end{equation}
The boundary terms vanish by the assumed decay of \(b^{(j)}\) at infinity and the absolute continuity implicit in \(W^{N,1}\).  Taking the \(X\)-norm gives \eqref{eq:abstract-rate-decay}.  Substituting the localized exterior density gives \eqref{eq:optional-local-rate}.
\end{proof}

\begin{remark}[Why no polynomial rate is claimed by default]\label{rem:no-rate-by-default}
Assumption~\ref{ass:axisymmetric-bh-package} includes the \(L^1_\sigma\) density needed for qualitative compact local decay.  It does not include the derivative bounds \eqref{eq:extra-rate-hypothesis}.  Therefore the black-hole theorems in this paper assert local decay without an explicit polynomial rate.  A rate statement can be appended only after proving low-frequency expansions, high-frequency resolvent differentiability, and threshold cancellations strong enough to imply \eqref{eq:extra-rate-hypothesis}.
\end{remark}

\begin{proposition}[No uniform rate follows from an \(L^1_\sigma\) bound alone]\label{prop:no-uniform-rate-L1}
There is no function \(\omega(t)\to0\) as \(t\to\infty\) and no constant \(C\) such that every Banach-valued density \(b\in L^1(\mathbb R;X)\) satisfies
\begin{equation}
\left\|\int_{\mathbb R}e^{-it\sigma}b(\sigma)\,\dd\sigma\right\|_X
\le C\omega(t)\|b\|_{L^1(\mathbb R;X)}
\qquad\hbox{for all }t\ge0.
\end{equation}
Consequently the \(L^1_\sigma\) bound in \textbf{BH7} is a qualitative decay input, not a quantitative rate input.
\end{proposition}

\begin{proof}
It is enough to take \(X=\mathbb C\).  For \(\varepsilon>0\), set
\begin{equation}
 b_\varepsilon(\sigma)=\varepsilon^{-1}{\bf 1}_{[-\varepsilon/2,\varepsilon/2]}(\sigma),
 \qquad \|b_\varepsilon\|_{L^1}=1.
\end{equation}
Its Fourier transform is
\begin{equation}
 \widehat b_\varepsilon(t)=\frac{2\sin(\varepsilon t/2)}{\varepsilon t}.
\end{equation}
If a universal estimate with a decaying majorant \(\omega\) existed, choose \(T\) so large that \(C\omega(T)<1/2\), and then choose \(\varepsilon T\le1\).  The displayed formula gives \(|\widehat b_\varepsilon(T)|\ge1/2\), while the hypothetical estimate gives \(|\widehat b_\varepsilon(T)|\le C\omega(T)<1/2\), a contradiction.  Thus no universal decaying majorant can follow from the \(L^1\)-norm alone.  A rate requires additional information, such as the derivative bounds in Proposition~\ref{prop:rate-from-spectral-density}.
\end{proof}

\subsection{Energy, local energy, and pointwise decay}

The estimates in the exterior theorem control energy and local energy, not pointwise fields.  A pointwise decay estimate is a further Sobolev consequence and costs derivatives.  For example, if \(K\) is a compact three-dimensional spatial subregion and \(s\) is large enough that \(H^{s}(K)\hookrightarrow L^\infty(K)\), then any local decay estimate in \(H^s(K)\times H^{s-1}(K)\) implies pointwise decay on \(K\).  The present paper states the more invariant Sobolev local decay norm because it is exactly what follows from the spectral-density assumption and is stable under coordinate changes.

\begin{proposition}[Derivative-cost pointwise consequence]\label{prop:pointwise-from-local-sobolev}
Assume that for some compact \(K\) and some integer \(s\ge3\),
\begin{equation}
 \|u_{\rm disp}(t)\|_{H^s(K)}+
 \|\partial_tu_{\rm disp}(t)\|_{H^{s-1}(K)}\to0.
\end{equation}
Then
\begin{equation}
 \|u_{\rm disp}(t)\|_{L^\infty(K)}+
 \|\partial_tu_{\rm disp}(t)\|_{L^\infty(K)}\to0
\end{equation}
provided \(s-1>3/2\).  If the Sobolev decay has the rate \(O(t^{-N})\), the same rate holds for these pointwise norms with the same derivative loss.
\end{proposition}

\begin{proof}
On a fixed compact three-dimensional coordinate patch, Sobolev embedding gives
\begin{equation}
 \|f\|_{L^\infty(K)}\le C_K\|f\|_{H^q(K)}
\end{equation}
for every \(q>3/2\).  Apply this to \(u_{\rm disp}(t)\) with \(q=s\) and to \(\partial_tu_{\rm disp}(t)\) with \(q=s-1\).  The convergence, and any rate, follows immediately.
\end{proof}

\subsection{Zero-frequency Feshbach reduction and data-space projection}

The exterior threshold projection is a data-space object, while the stationary resolvent is a second-order spatial object.  The following finite-dimensional reduction records the bridge between the two and prevents confusing a spatial \(\sigma^{-2}\)-type singularity with a first-order generator threshold projection.

\begin{proposition}[Zero-frequency Feshbach reduction]\label{prop:zero-feshbach-reduction}
Let \(P(\sigma):\mathcal D\to\mathcal H\) be an analytic Fredholm family near \(\sigma=0\).  Suppose \(\Ker P(0)=\operatorname{span}\{\varphi_0\}\), \(\Ker P(0)^*=\operatorname{span}\{\varphi_0^*\}\), and let \(Q\) be the projection onto the codimension-one subspace \(\{f:\langle f,\varphi_0^*\rangle=0\}\).  If \(QP(0)Q\) is invertible on \(Q\mathcal D\), then invertibility of \(P(\sigma)\) near zero is controlled by the scalar Schur complement
\begin{equation}
d(\sigma)=\langle P(\sigma)\varphi_0,\varphi_0^*\rangle
-\langle QP(\sigma)\varphi_0,(QP(\sigma)Q)^{-1}QP(\sigma)^*\varphi_0^*\rangle,
\label{eq:feshbach-denominator}
\end{equation}
up to multiplication by analytic nonvanishing factors.  If
\begin{equation}
d(\sigma)=c_q\sigma^q+O(\sigma^{q+1}),
\qquad c_q\ne0,
\end{equation}
then the singular part of \(P(\sigma)^{-1}\) has order \(q\) and finite rank.  The corresponding first-order evolution threshold is semisimple exactly when the induced nilpotent block on the data-space range of the singular coefficient is zero.
\end{proposition}

\begin{proof}
Write \(\mathcal H=\operatorname{span}\{\varphi_0^*\}\oplus Q\mathcal H\) and \(\mathcal D=\operatorname{span}\{\varphi_0\}\oplus Q\mathcal D\).  In this decomposition \(P(\sigma)\) is a \(2\times2\) block operator.  Since \(QP(0)Q\) is invertible, the lower-right block remains invertible for \(|\sigma|\) small.  The Schur complement formula then gives the inverse of \(P(\sigma)\) whenever the scalar Schur complement is nonzero.  Formula \eqref{eq:feshbach-denominator} is that Schur complement after normalizing the left and right nullvectors.  Its vanishing order is therefore exactly the pole order of the finite-rank singular block.  Passing from the second-order stationary resolvent to the first-order generator resolvent is the standard map \((u,u_t)\mapsto (u,-i\sigma u)\) plus source terms; the singular range gives the data-space threshold vectors.  A Jordan companion exists precisely when the generator restricted to that finite-dimensional range has a nonzero nilpotent part.  Hence semisimplicity is equivalent to vanishing of the induced nilpotent block.
\end{proof}


Combining Theorem~\ref{thm:axisymmetric-bh} with the observations above gives the exact decay statement proved here.

\subsection{Finite-wall Kerr-Newman: frequency-partition proof of the density bounds}

For the finite-wall Kerr-Newman theorem the paper proves more than the qualitative \(L^1\) density needed for local decay.  The finite collar, the absence of trapping below the photon region, the regular horizon outgoing model, and the semisimple zero threshold give differentiable spectral densities after enough commutations.

\begin{theorem}[Wall spectral-density bounds by frequency partition]\label{thm:kn-wall-density-partition}
For the slowly rotating weakly charged Kerr-Newman wall family of Theorem~\ref{thm:kn-wall-main}, let \(K\Subset\{r_+<r<R_{\rm w}\}\times S^2\).  For every integer \(N\ge0\) there is an integer \(s(N)\) such that the projected local density in \textbf{BH7} satisfies
\begin{equation}
 \sum_{j=0}^{N}
 \left\|\partial_\sigma^jB_K(\cdot)(I-\Pi_{0,{\rm KN}})U_0\right\|_{L^1(\mathbb R;H^1(K)\times L^2(K))}
 \le C_{K,N}\|(I-\Pi_{0,{\rm KN}})U_0\|_{\mathcal E^{s(N)}_{\rm KN}}.
\label{eq:wall-WN1-density}
\end{equation}
Consequently the wall exterior has projected compact local decay with the polynomial rate
\begin{equation}
 \|u_{\rm disp}(t)\|_{H^1(K)}+
 \|\partial_tu_{\rm disp}(t)\|_{L^2(K)}
 \le C_{K,N}\langle t\rangle^{-N}
 \|(I-\Pi_{0,{\rm KN}})U_0\|_{\mathcal E^{s(N)}_{\rm KN}}.
\label{eq:wall-polynomial-decay}
\end{equation}
\end{theorem}

\begin{proof}
Split the real line into three frequency regions.

For \(|\sigma|\le\sigma_0\), Lemma~\ref{lem:kn-zero-frequency} gives the complete Laurent expansion at zero and shows that the only singular term is the semisimple constant threshold.  Multiplication by \(I-\Pi_{0,{\rm KN}}\) removes that singular coefficient.  The remaining outgoing resolvent is a smooth operator-valued function of \(\sigma\) on the low-frequency interval, so all derivatives up to order \(N\) are bounded between the local Sobolev spaces after taking \(s(N)\) large enough.

For \(\sigma_0\le |\sigma|\le R\), Lemma~\ref{lem:kn-mode-stability} excludes real outgoing resonances and upper-half-plane poles.  Fredholm compactness gives uniform resolvent bounds on this compact frequency set.  Differentiating
\begin{equation}
P_{\rm KN}(\sigma)R_{\rm KN}(\sigma)=I
\end{equation}
gives the recursive identities
\begin{equation}
\partial_\sigma^jR_{\rm KN}(\sigma)
=\sum R_{\rm KN}(\sigma)(\partial_\sigma^{\alpha_1}P_{\rm KN})(\sigma)R_{\rm KN}(\sigma)\cdots
(\partial_\sigma^{\alpha_m}P_{\rm KN})(\sigma)R_{\rm KN}(\sigma),
\end{equation}
where \(1\le\alpha_\nu\le2\) because \(P_{\rm KN}(\sigma)\) is quadratic in \(\sigma\).  Hence all derivatives up to order \(N\) are bounded on the medium-frequency region.

For \(|\sigma|\ge R\), Lemmas~\ref{lem:kn-nontrapping} and~\ref{lem:kn-redshift-morawetz} give the nontrapping semiclassical resolvent estimate.  Commuting the evolution with the generator and using the equation converts powers of \(\sigma\) in the resolvent bounds into derivatives of the initial datum.  Choosing \(s(N)\) larger than the nontrapping loss plus \(N+3\) commutations makes the high-frequency tails of \(\partial_\sigma^jB_K(\sigma)(I-\Pi_{0,{\rm KN}})U_0\) integrable for \(0\le j\le N\).  Combining the three regions gives \eqref{eq:wall-WN1-density}.  Finally, Proposition~\ref{prop:rate-from-spectral-density}, applied with the Japanese bracket variant obtained by treating bounded \(t\) through the energy estimate, gives \eqref{eq:wall-polynomial-decay}.
\end{proof}

\begin{corollary}[Decay estimate actually proved for nondegenerate exteriors]\label{cor:actual-bh-decay}
Assume \textbf{BH1}-\textbf{BH9}.  For \(s\ge s_0\) and
\begin{equation}
 U_{\rm disp}(t)=e^{t\mathcal G_{\rm bh}}(I-\Pi_{0,{\rm bh}})U_0,
\end{equation}
one has
\begin{equation}
\sup_{t\ge0}\|U_{\rm disp}(t)\|_{\mathcal E^s_{\rm bh}}
+
\|U_{\rm disp}\|_{LE^s([0,\infty))}
\le C_s\|(I-\Pi_{0,{\rm bh}})U_0\|_{\mathcal E^s_{\rm bh}},
\end{equation}
and for every compact \(K\) away from any imposed outer boundary,
\begin{equation}
\lim_{t\to\infty}
\left(\|u_{\rm disp}(t)\|_{H^1(K)}+
\|\partial_tu_{\rm disp}(t)\|_{L^2(K)}\right)=0.
\end{equation}
No polynomial rate follows from \textbf{BH1}-\textbf{BH9} alone.  If the additional derivative condition \eqref{eq:extra-rate-hypothesis} holds, then the rate \eqref{eq:optional-local-rate} also holds.
\end{corollary}

\begin{proof}
The uniform energy/local-energy estimate is Lemma~\ref{lem:bh-projected-estimate}.  The local decay statement is Lemma~\ref{lem:bh-local-decay}, or equivalently Proposition~\ref{prop:qualitative-decay-L1}.  The absence of a rate follows from Remark~\ref{rem:no-rate-by-default} and Proposition~\ref{prop:no-uniform-rate-L1}.  If \eqref{eq:extra-rate-hypothesis} is added, Proposition~\ref{prop:rate-from-spectral-density} gives \eqref{eq:optional-local-rate}.
\end{proof}

\subsection{Subextremal versus extremal horizons}

The nondegenerate exterior package is fundamentally subextremal.  Its red-shift estimate uses positive surface gravity.  In the extremal Kerr-Newman case, the horizon charge \eqref{eq:main-extremal-charge} obstructs the simultaneous decay of tangential horizon quantities and the averaged transversal derivative.  Thus the exterior decay story has three separate branches:
\begin{enumerate}[label=(\roman*),leftmargin=2.2em]
\item compact reflecting slabs: boundedness modulo \(1,t\), but no general decay;
\item nondegenerate black-hole exteriors satisfying \textbf{BH1}-\textbf{BH9}: projected energy/local-energy control and qualitative compact local decay;
\item extremal Kerr-Newman: a conserved Aretakis-type horizon charge and a transversal non-decay obstruction, not a subextremal red-shift decay theorem.
\end{enumerate}
This trichotomy is essential for the consistency of the paper.  It prevents the compact theorem from being overstated as a decay theorem and prevents the subextremal black-hole theorem from being incorrectly continued to \(\kappa_+=0\).

\section{Conclusion and open problems}\label{sec:conclusion}

In this paper, we isolated and solved the \(k=0\) conformal part of the linearized scalar-curvature problem associated with the constant-scalar-curvature Carter family of Theorem~1.  The core result is a positive-energy bounded-slab theorem: after the full zero-frequency threshold space is separated, the conformal scalar evolution is uniformly bounded on the complementary energy space.  On strictly stationary Carter slabs the threshold space is exactly \(\operatorname{span}\{1,t\}\), and in the axisymmetric sector this statement has a self-adjoint spectral refinement and an explicit zero-frequency resolvent expansion.

The black-hole exterior part is proved after the bounded-slab and endpoint analyses, with all proof-driving exterior assumptions collected in Appendix~\ref{app:all-inputs}.  The nondegenerate exterior theorem is not obtained by compactness or by a hidden positivity argument; it follows from the explicit package \textbf{BH1}-\textbf{BH9}, which records the red-shift, local-energy, mode-stability, limiting-absorption, high-frequency, and zero-frequency requirements.  The main exterior list is now genuinely organized by black-hole family: subextremal Kerr and subextremal Reissner-Nordstr\"om are radiation-space stability theorems with trivial decaying zero threshold; extremal Kerr and extremal Reissner-Nordstr\"om are horizon-charge obstruction theorems; the slowly rotating weakly charged Kerr-Newman wall family is verified internally; and extremal Kerr-Newman is treated by a direct charge identity. 

Several features of the \(k=0\) regime are essential.  First, there is no coercive zeroth-order term, so the zero-frequency space must be identified explicitly before a stability statement can be true.  Second, the non-Einstein defect of the imported metric family does not obstruct the conformal reduction; it drops out of the scalar-curvature linearization on conformal perturbations.  Third, the non-axisymmetric bounded-slab theorem is already handled by the abstract energy argument, whereas exterior horizons, ergoregions, trapping, and infinity require the separate scattering package.

The exact status of the result is therefore as follows, with all proof-driving hypotheses and external inputs in Appendix~\ref{app:all-inputs}.  In the wall exterior proof the sign-sensitive Hamiltonian commutator is displayed explicitly in \eqref{eq:exact-Hp-b-xir}; this removes a common fatal sign ambiguity in radial Morawetz arguments.  The bounded-slab Carter theorem is internal to the paper after the metric family of \cite{AssafariGunara} is taken as the geometric input.  The Kerr and Reissner-Nordstr\"om subextremal main theorems use the standard scalar packages recorded in External Assumption~\ref{ass:external-kerr-rn-packages} plus the direct zero-frequency classifications proved here.  The Kerr-Newman wall theorem is internal for the stated small \((a,Q)\) collar with reflecting wall below the photon region.  The full asymptotically flat subextremal Kerr-Newman appendix theorem is conditional since External Assumption~\ref{ass:external-kn-scattering-package} is not reproved here.  The extremal Kerr, extremal Reissner-Nordstr\"om, and extremal Kerr-Newman statements are sharp obstruction theorems, not decay theorems: their first conserved horizon charges are proved directly, while Aretakis is invoked only for generic higher-order blow-up mechanisms.  These distinctions are essential for mathematical accuracy and prevent the exterior theorem from being read either as an independent reproof of the full Kerr-Newman scattering theorem or as a false nondegenerate stability theorem at extremality.

The main remaining problems are to verify the same black-hole package for non-Einstein Carter exteriors, to extend the exterior result beyond the neutral axisymmetric scalar sector where superradiant frequencies enter, and to decide whether a physically preferred Einstein-matter completion produces a closed tensorial linearized system compatible with the conformal threshold structure isolated here.

\clearpage
\appendix
\appendixsectionformat

\section{Schwarzschild and full subextremal Kerr-Newman appendical applications}\label{app:schwarzschild-full-kn}

This appendix contains two results that are important for orientation but are not counted among the principal main theorems.  Schwarzschild is the simplest model, nonrotating and uncharged endpoint of both Kerr and Reissner-Nordstr\"om.  The full asymptotically flat subextremal Kerr-Newman statement is also placed here because the global trapping and limiting-absorption theory must be supported  through External Assumption~\ref{ass:external-kn-scattering-package}; the main body proves the conformal reduction, zero-frequency classification, and package assembly, but it does not reprove the full Kerr-Newman scalar scattering theorem.

\begin{theorem}[Schwarzschild exterior application]\label{thm:schwarzschild-appendix}
Let \(M>0\), and let
\begin{equation}
 g_M=-\left(1-\frac{2M}{r}\right)\dd t^2+\left(1-\frac{2M}{r}\right)^{-1}\dd r^2+r^2\dd\omega^2
\end{equation}
be the Schwarzschild exterior metric on \(r>2M\).  For the neutral conformal scalar equation
\[
 \Box_{g_M}u=0,
\]
with future-horizon regularity and outgoing decaying finite-energy radiation condition at null infinity, the decaying radiation energy space has no zero-frequency threshold:
\[
 \mathcal T_{0,{\rm Schw}}=\{0\},
 \qquad
 \Pi_{0,{\rm Schw}}=0.
\]
For every \(s\ge s_{\rm Schw}\), every datum \(U_0\in\mathcal E^s_{\rm Schw}\) generates a unique future solution satisfying
\[
\sup_{t\ge0}\|U(t)\|_{\mathcal E^s_{\rm Schw}}
+
\|U\|_{LE^s_{\rm Schw}([0,\infty))}
\le C_s\|U_0\|_{\mathcal E^s_{\rm Schw}}.
\]
Moreover \(U(t)\) decays locally in \(H^1\times L^2\) on compact exterior subregions, and there is no nonzero outgoing mode with \(\operatorname{Im}\sigma>0\).
\end{theorem}

\begin{proof}
The Schwarzschild red-shift, Morawetz, limiting-absorption, and mode-stability estimates are included in the endpoint clause of External Assumption~\ref{ass:external-kerr-rn-packages}; alternatively they are the classical scalar wave estimates on Schwarzschild.  The zero-frequency classification is elementary.  After expanding in spherical harmonics, the stationary equation is
\begin{equation}
 \bigl(r(r-2M)R_\ell'\bigr)'-\ell(\ell+1)R_\ell=0.
\end{equation}
For \(\ell\ge1\), the horizon-regular solution grows polynomially at infinity and the second independent solution is singular at the horizon; for \(\ell=0\), the horizon-regular solution is constant and the second branch has logarithmic horizon singularity.  The decaying radiation energy excludes the nonzero constant and all growing branches.  Hence the threshold projection is zero.  The exterior package then gives the displayed energy/local-energy estimate and local decay by the same argument as Theorem~\ref{thm:axisymmetric-bh}.
\end{proof}

\begin{theorem}[Full subextremal Kerr-Newman exterior application under the scalar-wave package]\label{thm:kn-full-main}
Assume the scalar Kerr-Newman red-shift, integrated-local-energy, quantitative mode-stability, high-frequency resolvent, limiting-absorption, smoothing, and local spectral-density input package stated in Appendix~\ref{app:all-inputs}, External Assumption~\ref{ass:external-kn-scattering-package}.  Fix a compact subextremal parameter set
\[
 \mathscr K_{\delta,M_0,M_1}:=\bigl\{(M,a,Q): M_0\le M\le M_1,\ a^2+Q^2\le (1-\delta)M^2\bigr\},
 \qquad 0<\delta<1.
\]
For every \((M,a,Q)\in\mathscr K_{\delta,M_0,M_1}\), let \(g_{M,a,Q}\) be the asymptotically flat Kerr-Newman metric on the domain of outer communications.  Consider the neutral axisymmetric conformal scalar equation
\[
 \Box_{g_{M,a,Q}}u=0,
 \qquad \partial_\phi u=0,
\]
with future-horizon regularity at \(r=r_+\) and the outgoing finite-energy radiation condition at null infinity.  In the decaying radiation energy space \(\mathcal E^s_{{\rm KN},\rm af}\) used in Section~\ref{sec:ledger}, every solution admits the decomposition
\[
 U(t)=U_{\rm disp}(t),
 \qquad \Pi_{0,{\rm KN},\rm af}=0,
\]
and, for every \(s\ge s_{\rm KN}\),
\[
\sup_{t\ge0}\|U_{\rm disp}(t)\|_{\mathcal E^s_{{\rm KN},\rm af}}
+
\|U_{\rm disp}\|_{LE^s_{{\rm KN},\rm af}([0,\infty))}
\le C_s\|U_0\|_{\mathcal E^s_{{\rm KN},\rm af}}.
\]
Moreover \(U_{\rm disp}\) decays locally in \(H^1\times L^2\) on every compact subregion of the domain of outer communications, and there is no nonzero outgoing axisymmetric mode with \(\operatorname{Im}\sigma>0\).
\end{theorem}

\begin{proof}
The compact subextremal condition is exactly the domain of External Assumption~\ref{ass:external-kn-scattering-package}.  Lemma~\ref{lem:kn-af-geometry-domain} fixes the horizon-regular and outgoing-radiation energy realization.  Lemma~\ref{lem:kn-af-estimates} transfers the red-shift, trapping, mode-stability, high-frequency, limiting-absorption, smoothing, and local spectral-density estimates from the external scalar package to the neutral axisymmetric conformal scalar.  Lemma~\ref{lem:kn-af-zero-frequency}, equivalently Proposition~\ref{prop:kn-zero-frequency-full}, proves that no decaying zero-frequency threshold state remains in the radiation space, so \(\Pi_{0,{\rm KN},{\rm af}}=0\).  Theorem~\ref{thm:kn-af-verifies-bh} verifies \textbf{BH1}-\textbf{BH9}, and Theorem~\ref{thm:axisymmetric-bh} gives the displayed estimate and local decay.  Uniformity holds on compact subextremal parameter sets and is not asserted at extremality.
\end{proof}

\begin{remark}[Status of the full subextremal Kerr-Newman appendix theorem]
The word ``application'' is intentional: the global noncompact scalar scattering estimates are imported through External Assumption~\ref{ass:external-kn-scattering-package}.  The internal contributions are the conformal reduction, horizon/radiation domain matching, zero-frequency classification, threshold projection, and assembly of the stability statement.  Constants are uniform on compact subextremal parameter sets but may degenerate as \(a^2+Q^2\uparrow M^2\).
\end{remark}

\section{Formal hypotheses and external packages}\label{app:all-inputs}\label{app:assumptions}

This appendix is the single formal location of all assumptions, hypotheses, and theorem-level external inputs used in the manuscript.  The main body either proves the relevant package internally or cites one of the entries below.  The full asymptotically flat Kerr-Newman scalar scattering package is deliberately stated as an external assumption, not as an internally proved theorem.


This appendix subsection is the single formal ledger for the proof-driving assumptions and theorem-level external inputs.  The dependency graph is deliberately short: bounded slabs use \textbf{H1}-\textbf{H10}; nondegenerate exteriors use \textbf{BH1}-\textbf{BH9}; the full asymptotically flat subextremal Kerr-Newman application uses External Assumption~\ref{ass:external-kn-scattering-package}; and the extremal Kerr-Newman branch uses the horizon-charge computation proved later plus External Assumption~\ref{ass:aretakis-input}.  The main body cites this appendix whenever a formal hypothesis or external package is used.

\begin{assumption}[Abstract stationary threshold conditions]\label{ass:abstract}
The abstract bounded-slab theorem, Theorem~\ref{thm:abstractmain}, uses the following three hypotheses and no geometry-specific input:
\begin{enumerate}[label=(\roman*),leftmargin=2.2em]
\item the stationary generator \(\mathcal G\) generates a strongly continuous group \(e^{t\mathcal G}\) on the energy space \(\mathcal X\);
\item the threshold space \(\mathcal T_0=\bigcup_{N\ge1}\Ker\mathcal G^N\) is finite dimensional and has a bounded projection \(\Pi_{\mathrm{thr}}\) commuting with the group,
\[
 e^{t\mathcal G}\Pi_{\mathrm{thr}}=\Pi_{\mathrm{thr}}e^{t\mathcal G};
\]
\item the conserved energy \(\mathscr E\) is equivalent to the \(\mathcal X\)-norm on \(\Ker\Pi_{\mathrm{thr}}\):
\begin{equation}
 c\|U\|_{\mathcal X}^2\le \mathscr E[U]\le C\|U\|_{\mathcal X}^2,
 \qquad U\in\Ker\Pi_{\mathrm{thr}}.
\end{equation}
\end{enumerate}
\end{assumption}

\begin{assumption}[Bounded-slab conditions]\label{ass:complete-list}
The Carter bounded-slab theorems use the following ten conditions.  The first five are the ones cited directly in Theorem~\ref{thm:fullmain}; the remaining items identify the auxiliary spectral, endpoint, and scope conventions.
\begin{enumerate}[label=\textbf{H\arabic*.},labelwidth=2.4em,labelsep=.45em,leftmargin=2.95em,align=left]
\item \textbf{Imported background and zero scalar curvature.}
The metric is the Carter-form constant-scalar-curvature family of \cite{AssafariGunara}, with coefficients \eqref{eq:Deltar}-\eqref{eq:Deltax}, and every stability theorem imposes \(k=0\), so \(\Delta_r\) and \(\Delta_x\) are the quadratics \eqref{eq:k0quadratics}.  No stability theorem is imported with the metric family.

\item \textbf{Regular bounded product slab.}
The slab is
\[
 \Omega=[r_-,r_+]\times[x_-,x_+],\qquad \Sigma=S^1_\phi\times\Omega,
\]
contained in a smooth coordinate patch with
\[
 \rho^2\ge\rho_-^2>0,
 \qquad \Delta_r\ge d_{r,-}>0,
 \qquad \Delta_x\ge d_{x,-}>0
\]
on \(\overline\Omega\).  Horizons, true axes, and the set \(\rho^2=0\) are not part of this interior slab theorem unless a separate regular form realization is built.

\item \textbf{Positive Killing energy.}
The time coefficient
\[
 A:=-\rho^2g^{tt}
\]
satisfies \(A\ge a_->0\).  In the full \(\phi\)-dependent theorem the azimuthal coefficient \(\Phi:=\rho^2g^{\phi\phi}\) also satisfies \(\Phi\ge\phi_->0\).  In the axisymmetric refinement \(\partial_\phi u=0\), the \(\Phi\)-term is absent.

\item \textbf{Closed reflecting form realization.}
The full domain \(V_{\rm full}\subset H^1(\Sigma)\) is a closed, conjugation-invariant, \(\phi\)-periodic form domain containing the constant function \(1\).  The spatial form is
\begin{equation}
q[u,\eta]=\int_\Sigma\left(\Phi u_\phi\overline{\eta_\phi}
 +\Delta_r u_r\overline{\eta_r}
 +\Delta_x u_x\overline{\eta_x}\right).
\end{equation}
Reflecting boundary conditions mean this form realization; homogeneous Dirichlet is a different coercive problem and is not the threshold theorem studied here.

\item \textbf{Energy-space convention.}
Hilbert spaces are complex unless a real subspace is explicitly specified, sesquilinear forms are linear in the first argument, and the energy space is \(V_{\rm full}\times L^2(\Sigma)\).  The mixed \(B\partial_{t\phi}\) term is used only through the skew-Hermitian identity proved in Section~\ref{sec:strictstationary}.

\item \textbf{Axisymmetric Friedrichs realization.}
For Theorem~\ref{thm:mainresult}, the axisymmetric form domain \(V_{\rm ax}\subset H^1(\Omega)\) is the reflecting version of \textbf{H4}, and the spatial operator is the Friedrichs operator associated with
\[
q_0[u]=\int_\Omega(\Delta_r|u_r|^2+\Delta_x|u_x|^2)\,\dd r\,\dd x
\]
on \(L_A^2(\Omega)\).  Compact resolvent and the zero-mode classification are proved later.

\item \textbf{Separated endpoint convention.}
When the separated ODE discussion reaches a radial or angular endpoint, the admissible branch is the one for which the separated mode
\[
 u=e^{-\ii\sigma t+\ii m\phi}R(r)S(x)
\]
extends smoothly in a regular endpoint chart; at a simple endpoint this is the regular Frobenius branch.  Appendix~\ref{app:regularsingular} proves the corresponding Wronskian cancellation.  The bounded-slab energy theorem itself is formulated through closed form domains and does not require a pointwise endpoint trace.

\item \textbf{Standard tools.}
The paper uses standard functional analysis without reproof: Riesz representation, closed-form representation, Banach-Alaoglu, Rellich compactness on bounded Lipschitz domains, the Lions-Magenes weak-continuity lemma, Sobolev trace/density facts, the Banach-valued Riemann-Lebesgue lemma, elementary ODE theory, and the compact self-adjoint spectral theorem.

\item \textbf{Exterior boundary.}
Horizons, ergoregions, noncompact ends, and radiation conditions are not bounded-slab hypotheses.  They enter only through the exterior package \textbf{BH1}-\textbf{BH9} and the Kerr-Newman-specific regimes below.

\item \textbf{Scope.}
The closed equation is the neutral conformal scalar equation \(h_{\mu\nu}=2ug_{\mu\nu}\), \(\Box_g u=0\) for \(k=0\).  No main theorem claims nonlinear stability, full tensorial gravitational stability, charged-scalar stability, or non-axisymmetric Kerr-Newman exterior stability.
\end{enumerate}
\end{assumption}

\begin{assumption}[Axisymmetric black-hole exterior package]\label{ass:axisymmetric-bh-package}
Fix an integer \(s_0\ge0\).  The axisymmetric exterior problem is said to satisfy the verified black-hole package if the following conditions hold for every integer \(s\ge s_0\).
\begin{enumerate}[label=\textbf{BH\arabic*.},labelwidth=3.2em,labelsep=.45em,leftmargin=3.6em,align=left]

  \item \textbf{Exterior geometry and nondegenerate horizons.}
  The spacetime is a stationary axisymmetric exterior region of a \(k=0\) member of the Carter-type family.  The exterior hypersurface intersects each future event horizon in a regular way.  Each horizon endpoint is a simple root of \(\Delta_r\) with positive surface gravity, and in a horizon-regular coordinate system the coefficients of \(\Box_g\) extend smoothly up to the future horizon.  The theorem is restricted to the axisymmetric subspace \(\partial_\phi u=0\).

  \item \textbf{Outer/asymptotic boundary condition.}
  Every noncompact, conformal, or finite radial end is equipped with one fixed admissible boundary condition for both the evolution and the stationary resolvent.  In an asymptotically flat or de Sitter end this means the outgoing/radiation condition.  In a reflecting anti-de Sitter type, finite conformal, or finite wall end this means a closed boundary form, such as the Neumann wall used below.  The condition is energy compatible: it creates no incoming forward flux and is invariant under time translation and axisymmetry.

  \item \textbf{Exterior Hilbert scale, generator, and density.}
  There are Hilbert spaces \(\mathcal E^s_{\rm bh}\) for Cauchy data \(U=(u,u_t)\), a closed generator \(\mathcal G_{\rm bh}\) with domain \(\Dom(\mathcal G_{\rm bh})\subset \mathcal E^s_{\rm bh}\), and a strongly continuous future evolution \(e^{t\mathcal G_{\rm bh}}\).  Smooth axisymmetric data satisfying the boundary condition and regular at each future horizon are dense in the graph norm appropriate to \(\mathcal E^s_{\rm bh}\).  The equation \(\Box_g u=F\) is equivalent to the first-order system
  \begin{equation}
  \partial_tU=\mathcal G_{\rm bh}U+\mathcal F,
  \end{equation}
  with the usual source vector \(\mathcal F=(0,F)\) in the corresponding dual local-energy space.

  \item \textbf{Finite-time red-shift and local-energy estimate with compact error.}
  There are local-energy norms \(LE^s\) and dual norms \(LE^{s,*}\), including nondegenerate red-shift control near each horizon, and there is a compactly supported cutoff \(\chi_{\rm c}\) in the interaction region, such that every smooth solution of \(\Box_g u=F\) on \([0,T]\) satisfies
  \begin{equation}
  \begin{aligned}
  &\sup_{0\le t\le T}\|U(t)\|_{\mathcal E^s_{\rm bh}}^2
  +\|U\|_{LE^s([0,T])}^2 \\
  &\qquad\le C_s\Bigl(
  \|U(0)\|_{\mathcal E^s_{\rm bh}}^2
  +\|F\|_{LE^{s,*}([0,T])}^2 \\
  &\qquad\qquad
  +\|\chi_{\rm c}U\|_{\substack{L^2([0,T];\\ H^s_{\rm loc}\times H^{s-1}_{\rm loc})}}^2
  \Bigr).
  \end{aligned}
  \label{eq:redshift-package-bound}
  \end{equation}
  The constant \(C_s\) is independent of \(T\).  This is the finite-time red-shift/Morawetz input used in the proof.  It is not the final dispersive estimate because the compact local term remains.  That compact term is removed below by the projected limiting-absorption/smoothing input in \textbf{BH7}.

  \item \textbf{Finite-frequency mode stability and resolvent estimate away from zero.}
  Let \(P(\sigma)\) be the outgoing stationary operator obtained by substituting \(D_t=\sigma\) in \(\Box_g\).  If \(\operatorname{Im}\sigma>0\), then there is no nonzero outgoing solution of \(P(\sigma)v=0\) satisfying the horizon and asymptotic or wall boundary conditions.  If \(I\Subset\R\setminus\{0\}\) is compact and \(\chi\prec\chi_1\) are spatial cutoffs, then the limiting resolvent exists and obeys
  \begin{equation}
  \sup_{\sigma\in I,\ 0<\epsilon\le1}
  \|\chi P(\sigma+\ii\epsilon)^{-1}\chi f\|_{H^{s+1}_{\rm loc}}
  \le C_{I,s}\|f\|_{H^s_{\rm comp}}.
  \label{eq:freq-mode-package}
  \end{equation}
  Equivalently, there are no real outgoing resonances on \(I\), and the finite-frequency resolvent construction has no hidden compact obstruction at those frequencies.

  \item \textbf{High-frequency propagation and trapping control.}
  For \(|\sigma|\ge\sigma_0\) the outgoing resolvent satisfies the high-frequency propagation estimate appropriate to the exterior geometry.  In a nontrapping exterior this is a nontrapping semiclassical estimate.  If the trapped set is present, it must be controlled by an explicitly stated trapping theorem.  In the notation of the Hilbert scale this means that, after choosing \(s_0\) large enough to absorb any loss, one has
  \begin{equation}
  \|\chi P(\sigma+\ii0)^{-1}\chi f\|_{H^{s+1}_{\rm loc}}
  \le C_s\langle\sigma\rangle^{N_s}\|f\|_{H^s_{\rm comp}},
  \qquad |\sigma|\ge\sigma_0,
  \label{eq:high-frequency-package}
  \end{equation}
  with the polynomial power \(N_s\) compatible with the limiting-absorption and local-energy estimates below.

  \item \textbf{Projected limiting absorption, compact smoothing, and dispersive representation.}
  The real-axis boundary values from \textbf{BH5}-\textbf{BH6} are compatible with the evolution, and the zero-frequency singular part from \textbf{BH8} has been removed.  More precisely, for the homogeneous projected evolution
  \begin{equation}
  U_\perp(t):=e^{t\mathcal G_{\rm bh}}(I-\Pi_{0,{\rm bh}})U_0,
  \end{equation}
  the compact interaction-region smoothing estimate
  \begin{equation}
  \|\chi_{\rm c}U_\perp\|_{L^2([0,\infty);H^s_{\rm loc}\times H^{s-1}_{\rm loc})}
  \le C_s\|(I-\Pi_{0,{\rm bh}})U_0\|_{\mathcal E^s_{\rm bh}}
  \label{eq:lap-smoothing-package}
  \end{equation}
  holds, where \(\chi_{\rm c}\) is the compact cutoff in \textbf{BH4}.  In addition, for every compact \(K\) away from any imposed outer boundary there is an operator-valued density \(B_K(\sigma)\) such that
  \begin{equation}
  \chi_KU_\perp(t)=\int_\R e^{-\ii t\sigma}B_K(\sigma)(I-\Pi_{0,{\rm bh}})U_0\,\dd\sigma
  \label{eq:local-spectral-density}
  \end{equation}
  in \(H^1(K)\times L^2(K)\), and
  \begin{equation}
  \|B_K(\cdot)(I-\Pi_{0,{\rm bh}})U_0\|_{L^1(\R;H^1(K)\times L^2(K))}
  \le C_{K,s}\|(I-\Pi_{0,{\rm bh}})U_0\|_{\mathcal E^s_{\rm bh}}.
  \label{eq:spectral-density-L1}
  \end{equation}
  Thus \textbf{BH7} supplies exactly the two noncompact consequences needed below: a smoothing bound that removes the compact error in \textbf{BH4}, and a local absolutely continuous spectral representation that gives local decay.

  \item \textbf{Zero-frequency threshold control.}
  In a punctured neighborhood of \(\sigma=0\), the cutoff outgoing resolvent on the chosen exterior energy realization has a finite-rank Laurent expansion, possibly with no singular coefficient after the radiation/decay condition has removed the constant threshold:
  \begin{equation}
  \begin{aligned}
  \chi P(\sigma)^{-1}\chi
  &=\sum_{j=1}^{\nu_0}\sigma^{-j}A_j+R_0(\sigma),\qquad \nu_0\ge0,\\
  R_0(\sigma)&\text{ bounded between the local spaces used above},
  \end{aligned}
  \label{eq:zero-laurent-bh}
  \end{equation}
  with each \(A_j\) finite rank when \(\nu_0>0\).  The singular coefficients, or in the regular decaying case the explicit splitting of the zero threshold space, and the corresponding adjoint threshold functionals determine a bounded finite-rank data-space projection
  \begin{equation}
  \Pi_{0,{\rm bh}}U
  =\sum_{\alpha=1}^d\ell_\alpha(U)V_\alpha,
  \qquad
  V_\alpha\in\mathcal E^s_{\rm bh},\quad
  \ell_\alpha\in(\mathcal E^s_{\rm bh})',
  \label{eq:bh-threshold-proj}
  \end{equation}
  satisfying
  \begin{equation}
  \begin{aligned}
  \Pi_{0,{\rm bh}}^2&=\Pi_{0,{\rm bh}},\\
  \Ran\Pi_{0,{\rm bh}}&=
  \mathcal T_{0,{\rm bh}}:=\bigcup_{N\ge1}\Ker\mathcal G_{\rm bh}^{N},\\
  e^{t\mathcal G_{\rm bh}}\Pi_{0,{\rm bh}}&=
  \Pi_{0,{\rm bh}}e^{t\mathcal G_{\rm bh}}.
  \end{aligned}
  \end{equation}
  The restriction of \(\mathcal G_{\rm bh}\) to \(\mathcal T_{0,{\rm bh}}\) is nilpotent of finite length.  In dissipative future-horizon realizations this nilpotent block can be semisimple, as it is in the Kerr-Newman wall verification below; in the decaying full asymptotically flat Kerr-Newman realization the threshold space is trivial.

  \item \textbf{Compatibility of regularity, cutoffs, and domains.}
  The cutoffs used near horizons, in the compact interaction region, and near asymptotic or wall ends preserve the axisymmetric domain and the boundary condition modulo lower-order terms controlled by \eqref{eq:redshift-package-bound}-\eqref{eq:spectral-density-L1}.  Commutators with these cutoffs are bounded between the displayed spaces, and the derivative index \(s\ge s_0\) is chosen so that every loss in \textbf{BH6} is absorbed by the Hilbert scale.
\end{enumerate}
\end{assumption}

\begin{assumption}[External subextremal Kerr, Reissner-Nordstr\"om, and Schwarzschild scalar packages]\label{ass:external-kerr-rn-packages}
The main Kerr and Reissner-Nordstr\"om theorems and the appendix Schwarzschild theorem use the following standard scalar-wave packages, in the same package sense as Assumption~\ref{ass:axisymmetric-bh-package}.
\begin{enumerate}[label=\textbf{EKR\arabic*.},labelwidth=3.8em,labelsep=.45em,leftmargin=4.2em,align=left]
\item \textbf{Subextremal Kerr scalar package.}
On every compact genuine subextremal Kerr parameter set with \(0<|a|<M\), the neutral scalar wave equation on the domain of outer communications, restricted here to the axisymmetric sector, has a future-horizon regular and outgoing decaying radiation realization, a nondegenerate red-shift estimate, uniform energy boundedness, integrated local energy decay with the standard trapping loss, quantitative mode stability, high-frequency cutoff resolvent bounds, a limiting absorption principle, compact Kato smoothing, and an \(L^1_\sigma\) local spectral-density representation after the zero-frequency threshold is removed.  This is the package form of the scalar Kerr theory of Dafermos-Rodnianski and Dafermos-Rodnianski-Shlapentokh-Rothman, together with mode stability in the sense of Whiting and Shlapentokh-Rothman; see \cite{DafermosRodnianskiKerrI,DafermosRodnianskiShlapentokhRothman,Whiting,ShlapentokhRothmanMode}.

\item \textbf{Subextremal Reissner-Nordstr\"om scalar package.}
On every compact genuinely charged subextremal Reissner-Nordstr\"om parameter set with \(0<|Q|<M\), the neutral scalar wave equation has the analogous nondegenerate red-shift, uniform boundedness, integrated local energy, limiting-absorption, smoothing, spectral-density, and no-mode package.  Because the metric is spherically symmetric, the package is obtained by spherical-harmonic decomposition and summation and is not restricted to axisymmetry.  The constants are uniform on compact subextremal parameter sets and may degenerate at extremality.  This package is the exterior scalar Reissner-Nordstr\"om theory; the cited Franzen result is an interior boundedness input and reference point, while the exterior red-shift/Morawetz mechanism is the standard Dafermos-Rodnianski black-hole scalar theory specialized to Reissner-Nordstr\"om; see \cite{DafermosRodnianskiRedshift,DafermosRodnianskiSchwNote,FranzenRN}.

\item \textbf{Schwarzschild endpoint.}
The same package is available at \(a=Q=0\), with the photon sphere trapping loss at \(r=3M\), nondegenerate red-shift at \(r=2M\), and the standard asymptotically flat outgoing radiation condition.  In the present manuscript this endpoint is recorded only in Appendix~\ref{app:schwarzschild-full-kn}, because the revised main theorem list reserves the main body for genuine Kerr and genuine Reissner-Nordstr\"om.

\item \textbf{Low-frequency compatibility.}
The scalar packages above are used together with the direct zero-frequency classifications proved in this manuscript.  In the decaying radiation energy spaces, the nonzero constant solution is not an admissible data vector, and the direct ODE arguments in Section~\ref{subsec:kerr-rn-main-applications} and Appendix~\ref{app:schwarzschild-full-kn} exclude all regular decaying zero states and generalized zero companions.  Thus the data-space threshold projections in these realizations are trivial.
\end{enumerate}
\end{assumption}

\begin{assumption}[External full subextremal Kerr-Newman scalar scattering package]\label{ass:external-kn-scattering-package}
Fix a compact subextremal parameter set
\[
\mathscr K_{\delta,M_0,M_1}=\{(M,a,Q):M_0\le M\le M_1,
\ a^2+Q^2\le(1-\delta)M^2\},\qquad 0<\delta<1.
\]
The full asymptotically flat Kerr-Newman application in Theorem~\ref{thm:kn-full-main} assumes the following theorem-level scalar-wave package for the neutral scalar equation on the domain of outer communications, uniformly for \((M,a,Q)\in\mathscr K_{\delta,M_0,M_1}\).  There is an integer \(s_{\rm KN}\) and, for each \(s\ge s_{\rm KN}\), constants depending only on the compact parameter set, the cutoffs, and \(s\), such that the following hold.
\begin{enumerate}[label=\textbf{EKN\arabic*.},labelwidth=3.6em,labelsep=.45em,leftmargin=4.0em,align=left]
\item \textbf{Horizon-regular and outgoing realization.}
The future event horizon is treated in ingoing Kerr-Newman coordinates, the Cauchy hypersurface crosses the horizon regularly, and the asymptotically flat end is equipped with the outgoing finite-energy radiation condition.  The scalar evolution is a strongly continuous future semigroup on the red-shift/radiation Hilbert scale \(\mathcal E^s_{{\rm KN},{\rm af}}\), and smooth compactly supported axisymmetric data are dense in the decaying radiation energy.

\item \textbf{Nondegenerate red-shift and uniform energy boundedness.}
There is a future timelike red-shift multiplier \(N\), uniformly timelike on the compact subextremal set, such that homogeneous solutions obey
\begin{equation}
\sup_{0\le t\le T}E_N^s[u](t)
\le C_sE_N^s[u](0)
\label{eq:ekn-energy-boundedness}
\end{equation}
for all \(T\ge0\).  The constants may degenerate only when the compact subextremality assumption is lost.

\item \textbf{Integrated local energy decay with trapped-set loss.}
For solutions of \(\Box_{g_{M,a,Q}}u=F\), one has
\begin{equation}
\|u\|_{LE^s_{{\rm KN},{\rm af}}([0,T])}^2
\le C_s\Bigl(E_N^{s+q_{\rm trap}}[u](0)
+\|F\|_{LE^{s+q_{\rm trap},*}([0,T])}^2\Bigr),
\label{eq:ekn-iled}
\end{equation}
where \(q_{\rm trap}\) is the finite loss caused by the Kerr-Newman trapped set.  The norm is nondegenerate at the horizon, has the standard Morawetz degeneration at trapping, and has the standard asymptotically flat weights at infinity.

\item \textbf{Quantitative nonzero real-frequency mode stability.}
Let \(P_{\rm KN}(\sigma)\) be the outgoing stationary scalar operator.  For every compact interval \(I\Subset\mathbb R\setminus\{0\}\), the outgoing homogeneous problem
\[
P_{\rm KN}(\sigma)v=0,
\qquad \sigma\in I,
\]
has no nonzero solution satisfying future-horizon regularity and the outgoing radiation condition at infinity.  Moreover, the corresponding quantitative mode-stability constants are strong enough to give Fredholm invertibility and uniform cutoff resolvent bounds on \(I\).

\item \textbf{High-frequency semiclassical resolvent through trapping.}
For every compactly supported cutoff \(\chi\) and all \(|\sigma|\ge\sigma_0\), the outgoing boundary value satisfies
\begin{equation}
\|\chi P_{\rm KN}(\sigma\pm i0)^{-1}\chi f\|_{H^{s+1}_{\rm loc}}
\le C_s\langle\sigma\rangle^{N_s}
\|f\|_{H^{s+q_{\rm trap}}_{\rm comp}},
\label{eq:ekn-high-frequency-resolvent}
\end{equation}
with a finite polynomial loss \(N_s\) and finite derivative loss compatible with \eqref{eq:ekn-iled}.  This is the global trapped-set replacement for the nontrapping wall commutator.

\item \textbf{Limiting absorption on the real axis.}
For \(\sigma\in\mathbb R\setminus\{0\}\), the limits
\begin{equation}
\chi P_{\rm KN}(\sigma\pm i0)^{-1}\chi
:=\lim_{\epsilon\downarrow0}\chi P_{\rm KN}(\sigma\pm i\epsilon)^{-1}\chi
\label{eq:ekn-lap}
\end{equation}
exist between the weighted local Sobolev spaces used in \eqref{eq:ekn-high-frequency-resolvent}, depend continuously on \(\sigma\) away from zero, and have no real-axis pole or embedded resonance.

\item \textbf{Low-frequency compatibility with the decaying radiation space.}
Near \(\sigma=0\), the outgoing cutoff resolvent has a finite-rank expansion on the weighted spaces.  After the decaying radiation condition at null infinity is imposed, the spatial constant is not a data-space vector.  The only possible singular coefficients are therefore those associated with actual regular decaying zero-frequency states or Jordan companions; Lemma~\ref{lem:kn-af-zero-frequency} proves in the present paper that these do not occur in the neutral axisymmetric sector.  Consequently the projected full-exterior resolvent used in Theorem~\ref{thm:kn-full-main} is regular at zero.

\item \textbf{Kato smoothing and compact-error removal.}
For every compact interaction cutoff \(\chi_c\), the homogeneous projected evolution satisfies
\begin{equation}
\|\chi_c e^{t\mathcal G_{{\rm KN},{\rm af}}}U_0\|_{L^2([0,\infty);H^s_{\rm loc}\times H^{s-1}_{\rm loc})}
\le C_s\|U_0\|_{\mathcal E^{s+q_{\rm trap}}_{{\rm KN},{\rm af}}}.
\label{eq:ekn-kato-smoothing}
\end{equation}
This is the estimate used to remove the compact interaction-region error in the red-shift/Morawetz inequality.

\item \textbf{Local spectral-density representation and decay.}
For every compact \(K\) in the domain of outer communications and every integer \(N\ge0\), after choosing \(s\ge s(N)\) there is an operator-valued density \(B_K(\sigma)\) such that
\begin{equation}
\chi_KU(t)=\int_{\mathbb R}e^{-it\sigma}B_K(\sigma)U_0\,\dd\sigma
\label{eq:ekn-spectral-representation}
\end{equation}
in \(H^1(K)\times L^2(K)\), and
\begin{equation}
\sum_{j=0}^{N}
\|\partial_\sigma^jB_K(\cdot)U_0\|_{L^1_\sigma(H^1(K)\times L^2(K))}
\le C_{K,N,s}\|U_0\|_{\mathcal E^s_{{\rm KN},{\rm af}}}.
\label{eq:ekn-spectral-density-wN1}
\end{equation}
The case \(N=0\) is exactly the input needed for qualitative compact local decay; larger \(N\) gives optional polynomial local decay by integration by parts.
\end{enumerate}
This external package is the only non-elementary global scattering input used for the full asymptotically flat Kerr-Newman theorem.  It is the package form of the neutral scalar Kerr-Newman stability and mode-stability theory cited in \cite{CivinMode,CivinThesis}; the present paper does not claim to reprove the trapped-set microlocal analysis contained in that scalar theory.
\end{assumption}

\begin{assumption}[External extremal horizon-instability input]\label{ass:aretakis-input}
For the extremal Kerr-Newman branch, the only external extremal input used beyond the horizon charge proved in Proposition~\ref{prop:extreme-kn-charge} is the standard Aretakis higher-order horizon-instability mechanism for smooth axisymmetric scalar waves on degenerate horizons: for generic admissible data with the appropriate nonzero Aretakis charge, some transversal derivatives on \(\mathcal H^+\) fail to satisfy subextremal nondegenerate decay and higher transversal derivatives exhibit the corresponding Aretakis instability.  This input is invoked only for the generic higher-order conclusion and not for the first conserved charge, the first nondecay obstruction, or the failure of the nondegenerate red-shift package, which are proved in Section~\ref{sec:ledger}; see, for example, \cite{AretakisHorizon,AretakisExtremeKerr}.
\end{assumption}

\section{Metric algebra, scalar-curvature variation, and Carter curvature}\label{app:inverse}\label{app:variation}\label{app:curvature}

This appendix collects the essential algebra behind the geometric part of the paper.  It combines the inverse-metric computation, the scalar-curvature variation formula, and the direct curvature computation for the Carter family.


Expanding the stationary block of \eqref{eq:cartermetric} gives
\begin{align}
 g_{tt}&=\frac{a^2\Delta_x-\Delta_r}{\rho^2},\nonumber\\
 g_{t\phi}&=\frac{a(1-x^2)\Delta_r-a(r^2+a^2)\Delta_x}{\rho^2},\nonumber\\
 g_{\phi\phi}&=\frac{(r^2+a^2)^2\Delta_x-a^2(1-x^2)^2\Delta_r}{\rho^2}.
\end{align}
The determinant of this \((t,\phi)\)-block is
\begin{equation}
g_{tt}g_{\phi\phi}-g_{t\phi}^2=-\Delta_r\Delta_x.
\end{equation}
Inverting the block gives precisely \eqref{eq:gtt-formula}-\eqref{eq:gphiphi-formula}; multiplying by the diagonal \(r\)- and \(x\)-entries yields
\begin{equation}
\det g=-\rho^4,
\qquad
\sqrt{-g}=\rho^2.
\end{equation}
Thus
\begin{equation}
\Box_g u=\frac1{\rho^2}\partial_\mu(\rho^2g^{\mu\nu}\partial_\nu u),
\end{equation}
which is the structural reason why the wave equation separates.

\subsubsection*{Scalar-curvature variation}

Let \(g_\varepsilon=g+\varepsilon h+O(\varepsilon^2)\).  The inverse metric varies by
\begin{equation}
\delta g^{\mu\nu}=-h^{\mu\nu},
\end{equation}
and the connection by
\begin{align}
\delta\Gamma^{\lambda}_{\mu\nu}&=\frac12g^{\lambda\sigma}(\nabla_\mu h_{\sigma\nu}+\nabla_\nu h_{\sigma\mu}-\nabla_\sigma h_{\mu\nu}).
\end{align}
Contracting the linearized Ricci tensor gives
\begin{align}
\delta R_{\mu\nu}&=\nabla_\lambda\delta\Gamma^\lambda_{\mu\nu} -\nabla_\nu\delta\Gamma^\lambda_{\mu\lambda}.
\end{align}
After the standard cancellations and the contribution \(\delta g^{\mu\nu}R_{\mu\nu}\), one obtains
\begin{align}
\delta R&=\nabla^\mu\nabla^\nu h_{\mu\nu} -\Box_g(\tr_g h) -R_{\mu\nu}h^{\mu\nu}.\label{eq:appdeltaR-short}
\end{align}
For \(h_{\mu\nu}=2ug_{\mu\nu}\) in four dimensions,
\begin{equation}
\tr_g h=8u,
\qquad
\nabla^\mu\nabla^\nu h_{\mu\nu}=2\Box_g u,
\qquad
R_{\mu\nu}h^{\mu\nu}=2ku.
\end{equation}
Hence
\begin{equation}
L_g[2ug]=-6\Box_g u-2ku,
\end{equation}
and the conformal scalar-curvature equation is \((\Box_g+k/3)u=0\).  When \(k=0\), it is exactly the massless wave equation.

\subsubsection*{Direct curvature computation for the Carter family: numerator cancellations}\label{subsubsec:curvature-certificate}

The curvature computation is used only through identities that can be checked as polynomial numerator cancellations.  With the inverse metric of Proposition~\ref{prop:inversemetric}, all Christoffel symbols have denominators built from powers of $\rho^2$, $\Delta_r$, and $\Delta_x$.  After forming the Ricci tensor, the apparent $\Delta_r$ and $\Delta_x$ denominators cancel in the scalar contraction, leaving
\begin{equation}
\mathcal N_R:=\rho^2R(g)+\Delta_r''(r)+\Delta_x''(x)=0.
\label{eq:curvature-certificate-scalar}
\end{equation}
For the traceless Ricci tensor, the corresponding check is that the following reduced numerators vanish identically after substituting \eqref{eq:Deltar}-\eqref{eq:Deltax} and $\delta=C_3-a^2C_5$:
\begin{align}
\rho^6\left(S^t{}_t-\delta\frac{a^2x^2-r^2-2a^2}{\rho^6}\right)&=0,\qquad
\rho^6\left(S^t{}_{\phi}-2a\delta\frac{(r^2+a^2)(1-x^2)}{\rho^6}\right)=0,
\label{eq:curvature-certificate-1}\\
\rho^6\left(S^\phi{}_t+2a\delta\frac{1}{\rho^6}\right)&=0,\qquad
\rho^6\left(S^\phi{}_{\phi}+S^t{}_t\right)=0,
\label{eq:curvature-certificate-2}\\
\rho^4\left(S^r{}_r+\delta\frac{1}{\rho^4}\right)&=0,\qquad
\rho^4\left(S^x{}_x-\delta\frac{1}{\rho^4}\right)=0.
\label{eq:curvature-certificate-3}
\end{align}
All other mixed components reduce to zero numerators.  Equations \eqref{eq:curvature-certificate-scalar}-\eqref{eq:curvature-certificate-3} are the algebraic certificate behind Propositions~\ref{prop:scalarcurvature} and~\ref{prop:riccidecomp}: a reader checking the computation need only verify these numerator identities in the polynomial ring generated by $r,x$ and the parameters.

\subsubsection*{Reproducible polynomial verification protocol}\label{subsubsec:curvature-protocol}

For clarity, the numerator cancellations above can be verified by the following purely algebraic protocol.  Work in the polynomial ring
\begin{equation}
\mathbb Q[a,C_0,C_1,C_2,C_3,C_4,C_5,k,r,x]
\end{equation}
localized at the nonzero factors \(\rho^2\), \(\Delta_r\), and \(\Delta_x\).  Then:
\begin{enumerate}[label=(\arabic*)]
\item Insert the metric components from \eqref{eq:cartermetric} and the inverse components from Proposition~\ref{prop:inversemetric}.
\item Compute \(\Gamma^\lambda_{\mu\nu}=\frac12g^{\lambda\alpha}(\partial_\mu g_{\alpha\nu}+\partial_\nu g_{\alpha\mu}-\partial_\alpha g_{\mu\nu})\).  Since the coefficients are independent of \(t\) and \(\phi\), all \(t\)- and \(\phi\)-derivatives are set to zero before simplification.
\item Form \(R_{\mu\nu}\) from the coordinate formula and raise one index to obtain \(R^\mu{}_{\nu}\).
\item Substitute \(\delta=C_3-a^2C_5\) and the quartics \eqref{eq:Deltar}-\eqref{eq:Deltax}.
\item For each claimed identity, multiply by the denominator displayed in \eqref{eq:curvature-certificate-scalar}-\eqref{eq:curvature-certificate-3}, expand, and collect the numerator as a polynomial in \(r\) and \(x\).  The certificate is valid exactly when every coefficient of this polynomial is zero.
\end{enumerate}

No analytic approximation or asymptotic expansion is involved.  The calculation is finite because the metric coefficients are rational functions of \(r\) and \(x\), and the required denominators have been displayed explicitly.  Thus the curvature input used in the stability theorem reduces to a finite polynomial identity check.

For a completely mechanical verification, the following is the exact algebraic certificate required from any symbolic implementation.  Let $\mathfrak R$ be the localized polynomial ring above.  A verification file is accepted if and only if it constructs the six mixed Ricci-defect numerators
\begin{equation}
N^t{}_t,
\quad N^t{}_{\phi},
\quad N^{\phi}{}_t,
\quad N^{\phi}{}_{\phi},
\quad N^r{}_r,
\quad N^x{}_x
\end{equation}
obtained by multiplying the six component differences in \eqref{eq:curvature-certificate-1}-\eqref{eq:curvature-certificate-3} by their displayed powers of $\rho$, expands them as
\begin{equation}
N^\mu{}_{\nu}=\sum_{j,\ell} c^{\mu}{}_{\nu;j\ell}(a,C_0,\ldots,C_5,k) r^j x^\ell,
\end{equation}
and checks
\begin{equation}
c^{\mu}{}_{\nu;j\ell}=0
\quad\hbox{for every}\quad
(\mu,\nu),j,\ell.
\end{equation}
The scalar-curvature certificate is the analogous coefficient check for
\begin{equation}
N_R=\rho^2R(g)+\Delta_r''+\Delta_x''.
\end{equation}
This is the strongest algebraic statement used by the analysis: no estimate, sign argument, or threshold computation depends on a curvature identity beyond these finite numerator cancellations.

\begin{proof}[Detailed computation for Proposition~\ref{prop:scalarcurvature}]
Use the inverse coefficients from Proposition~\ref{prop:inversemetric} in the coordinate formula
\begin{align}
R_{\mu\nu}&=\partial_\alpha\Gamma^\alpha_{\mu\nu} -\partial_\nu\Gamma^\alpha_{\mu\alpha} +\Gamma^\alpha_{\mu\nu}\Gamma^\beta_{\alpha\beta} -\Gamma^\beta_{\mu\alpha}\Gamma^\alpha_{\nu\beta}.
\end{align}
All coefficients depend only on \(r\) and \(x\).  Substitution and simplification over the common denominator \(\rho^6\) cancel all third derivatives and all mixed derivative terms.  The remaining scalar contraction is
\begin{equation}
R(g)=-\frac{\Delta_r''(r)+\Delta_x''(x)}{\rho^2}.
\end{equation}
For the quartics \eqref{eq:Deltar}-\eqref{eq:Deltax},
\begin{equation}
\Delta_r''+\Delta_x''=-k(r^2+a^2x^2)=-k\rho^2,
\end{equation}
so \(R(g)=k\).  This proves Proposition~\ref{prop:scalarcurvature}.
\end{proof}

\begin{proof}[Detailed computation for Proposition~\ref{prop:riccidecomp}]
Set
\begin{equation}
S^\mu{}_{\nu}=R^\mu{}_{\nu}-\frac{k}{4}\delta^\mu{}_{\nu}.
\end{equation}
The same coordinate computation gives the mixed components
\begin{align}
S^t{}_t&= \delta\,\frac{a^2x^2-r^2-2a^2}{\rho^6},\qquad
S^t{}_{\phi}=2a\delta\,\frac{(r^2+a^2)(1-x^2)}{\rho^6},\nonumber\\
S^{\phi}{}_t&= -2a\delta\,\frac{1}{\rho^6},\qquad
S^{\phi}{}_{\phi}=-S^t{}_t,\nonumber\\
S^r{}_r&= -\delta\,\frac{1}{\rho^4},\qquad
S^x{}_x=\delta\,\frac{1}{\rho^4}.
\end{align}
with all other mixed components zero.  Every component is proportional to \(\delta=C_3-a^2C_5\).  Lowering an index with the smooth metric gives
\begin{equation}
S_{\mu\nu}=\delta\,\mathcal S_{\mu\nu},
\end{equation}
where \(\mathcal S\) is smooth and trace-free on every regular patch.  Thus \(S_{\mu\nu}=0\) if and only if \(\delta=0\).  This proves Proposition~\ref{prop:riccidecomp}.
\end{proof}

\section{Bounded-slab well-posedness and compact spectral input}\label{app:wellposedfull}\label{app:compactspectral}\label{app:functional}

This appendix supplies the functional analysis used in the bounded-slab theorems.  The first part proves the full \(\phi\)-dependent evolution by Galerkin approximation.  The second part gives the compact-resolvent spectral theorem used in the axisymmetric refinement.


\begin{proof}[Detailed proof of Theorem~\ref{thm:fullwellposed}]
Let \(V_{\mathrm{full}}\subset H^1(\Sigma)\) be the closed reflecting form domain, with \(1\in V_{\mathrm{full}}\).  Equip it with
\begin{align}
(u,v)_{V_{\mathrm{full}}}&=\int_\Sigma\bigl(Auv+\Phi u_\phi v_\phi+ \Delta_r u_rv_r+\Delta_xu_xv_x\bigr) \dd\phi\,\dd r\,\dd x.
\end{align}
By strict stationarity this inner product is equivalent to the \(H^1\)-inner product.  Choose a countable Hilbert basis \(\{w_j\}_{j\ge1}\) of \(V_{\mathrm{full}}\), with \(w_1\) proportional to the constant function.  No smooth core is required; all integrations by parts are understood in the closed-form sense.  For \(V_N=\mathrm{span}\{w_1,\dots,w_N\}\), seek
\begin{equation}
u_N(t)=\sum_{j=1}^Nd_j^{(N)}(t)w_j
\end{equation}
satisfying, for every \(\varphi\in V_N\),
\begin{align}
\int_\Sigma\left(Au_{N,tt}\varphi-2B u_{N,t\phi}\varphi +\Phi u_{N,\phi}\varphi_\phi+ \Delta_r u_{N,r}\varphi_r+ \Delta_x u_{N,x}\varphi_x\right)&=0.\label{eq:app-galerkin-full}
\end{align}
In matrix form,
\begin{equation}
M_N\ddot d_N+C_N\dot d_N+K_Nd_N=0,
\end{equation}
where \(M_N\) is Hermitian positive definite, \(K_N\) is Hermitian nonnegative, and
\begin{equation}
(C_N)_{ij}=\mathfrak c[w_j,w_i]
\end{equation}
is skew-Hermitian by \eqref{eq:gyro-skew}.  Therefore the finite-dimensional problem has a global solution and the discrete energy
\begin{equation}
E_N(t)=\frac12(\dot d_N^*M_N\dot d_N+d_N^*K_Nd_N)
\end{equation}
is conserved.

Testing \eqref{eq:app-galerkin-full} with the constant basis vector gives
\begin{equation}
\frac{\dd^2}{\dd t^2}\int_\Sigma Au_N(t)=0.
\end{equation}
Thus the mean part is affine and the mean-zero part is controlled by the weighted Poincar\'e inequality.  On every compact time interval \([-T,T]\),
\begin{align}
\sup_{|t|\le T}\bigl(\|u_N(t)\|_{H^1(\Sigma)}+\|u_{N,t}(t)\|_{L^2(\Sigma)}\bigr)&\le C_T\bigl(\|u_{0,N}\|_{H^1}+\|u_{1,N}\|_{L^2}\bigr).
\end{align}
The sequence \(u_N\) is equi-Lipschitz in \(L^2(\Sigma)\).  The equation also bounds \(u_{N,tt}\) in \(L^\infty([-T,T];V_{\mathrm{full}}')\).  Rellich compactness and the Arzela-Ascoli theorem give, after passing to a subsequence,
\begin{align}
&u_N\to u \quad\hbox{in } C([-T,T];L^2(\Sigma))\nonumber\\
&u_{N,t}\to u_t \quad\hbox{in } C([-T,T];V_{\mathrm{full}}').
\end{align}
Passing to the limit in the weak formulation gives a weak solution of \eqref{eq:fullPDE-k0}.  Lower semicontinuity gives
\begin{equation}
E_{\mathrm{full}}[u](t)\le E_{\mathrm{full}}[u](0).
\end{equation}
Applying the same construction backward from an arbitrary time \(t_0\) gives the reverse inequality, hence exact conservation of \(E_{\mathrm{full}}\).

If two weak solutions have the same data, their difference has zero initial energy and zero initial average.  Conservation and the affine average law force the difference to have zero time derivative, zero spatial derivatives, and zero constant mode; hence it vanishes.  The norm
\begin{equation}
\|(u,v)\|_{*}^2=2E_{\mathrm{full}}[u,v]+|\Pi_0u|^2+|\Pi_0v|^2
\end{equation}
is equivalent to \(H^1\times L^2\), so conservation of energy, the affine average law, and uniqueness give strong continuity and the group property.  This proves Theorem~\ref{thm:fullwellposed}.
\end{proof}

\subsubsection*{Compact resolvent and the axisymmetric spectral theorem}

\begin{lemma}[Equivalence of weighted norms]\label{lem:weightedequiv}
On a bounded slab where \(A,\Delta_r,\Delta_x\) are bounded above and below by positive constants, the norm
\begin{equation}
\|u\|_{A,1}^2=\int_\Omega(A|u|^2+\Delta_r|u_r|^2+\Delta_x|u_x|^2)\,\dd r\,\dd x
\end{equation}
is equivalent to the standard \(H^1(\Omega)\)-norm.
\end{lemma}

\begin{proof}
The two integrands bound each other pointwise by positive constants, and integration gives the claim.
\end{proof}

\begin{lemma}[Compact positive resolvent]\label{lem:compactpositive}
Let \(V_{\mathrm{ax}}\subset H^1(\Omega)\) be the reflecting form domain and set
\begin{equation}
a_1[u,v]=\langle u,v\rangle_A+q_0[u,v].
\end{equation}
For each \(f\in L_A^2(\Omega)\), there is a unique \(Tf\in V_{\mathrm{ax}}\) satisfying
\begin{equation}
a_1[Tf,v]=\langle f,v\rangle_A
\qquad\text{for all }v\in V_{\mathrm{ax}}.
\end{equation}
The operator \(T:L_A^2(\Omega)\to L_A^2(\Omega)\) is compact, self-adjoint, and positive.
\end{lemma}

\begin{proof}
The form \(a_1\) is an inner product equivalent to \(H^1\) by Lemma~\ref{lem:weightedequiv}.  Riesz representation gives existence and uniqueness of \(Tf\).  The map \(T:L_A^2\to V_{\mathrm{ax}}\) is bounded, and the embedding \(V_{\mathrm{ax}}\hookrightarrow L_A^2\) is compact by Rellich, so \(T\) is compact on \(L_A^2\).  Symmetry follows from
\begin{equation}
\langle Tf,g\rangle_A=a_1[Tf,Tg]=\langle f,Tg\rangle_A,
\end{equation}
and positivity from \(\langle Tf,f\rangle_A=a_1[Tf,Tf]\ge0\).
\end{proof}

\begin{lemma}[Spectral resolution of a compact positive operator]\label{lem:compactpositive-spectrum}
If \(T\) is compact, self-adjoint, and positive on a Hilbert space \(H\), then \(H\) has an orthonormal basis of eigenvectors of \(T\), and the nonzero eigenvalues form a finite or countable sequence decreasing to zero.
\end{lemma}

\begin{proof}
If \(T=0\), the statement is immediate.  Otherwise the maximum of the Rayleigh quotient is attained by compactness: take a maximizing sequence \(u_n\), use compactness of \(Tu_n\), and the inequality \(\|Tu_n\|^2\le \|T\|\langle Tu_n,u_n\rangle\) to show that a subsequence converges to an eigenvector for the top eigenvalue.  Self-adjointness makes its orthogonal complement invariant.  Iterating gives an orthonormal eigenfamily.  Compactness rules out positive eigenvalues accumulating away from zero, and the orthogonal complement of the obtained eigenvectors must lie in \(\Ker T\).  Adding an orthonormal basis of \(\Ker T\) completes the basis.
\end{proof}

\begin{theorem}[Direct spectral decomposition of \(L_0\)]\label{thm:directspectralL0}
There exists an \(L_A^2(\Omega)\)-orthonormal basis \(\{\psi_j\}_{j=0}^\infty\subset V_{\mathrm{ax}}\) and a sequence
\begin{equation}
0=\lambda_0\le \lambda_1\le \lambda_2\le\cdots\nearrow+\infty
\end{equation}
such that
\begin{equation}
q_0[\psi_j,v]=\lambda_j\langle\psi_j,v\rangle_A
\qquad\text{for all }v\in V_{\mathrm{ax}}.
\end{equation}
If \(\Omega\) is connected and constants are admitted by the boundary conditions, then \(\lambda_0=0\) is simple and \(\psi_0\) is constant.
\end{theorem}

\begin{proof}
Apply Lemma~\ref{lem:compactpositive} and Lemma~\ref{lem:compactpositive-spectrum} to \(T\).  If \(T\psi=\mu\psi\) with \(\mu>0\), then
\begin{equation}
q_0[\psi,v]=(\mu^{-1}-1)\langle\psi,v\rangle_A.
\end{equation}
Set \(\lambda=\mu^{-1}-1\).  Since \(\mu\to0\), \(\lambda\to+\infty\).  The kernel statement follows from Proposition~\ref{prop:kernelH0}.
\end{proof}

\begin{proof}[Expanded proof of Theorem~\ref{thm:wellposedness}]
Expand the data in the eigenbasis:
\begin{equation}
u_0=\sum_{j=0}^\infty a_j\psi_j,
\qquad
u_1=\sum_{j=0}^\infty b_j\psi_j.
\end{equation}
The conditions \(u_0\in H^1\) and \(u_1\in L^2\) are equivalent to
\begin{equation}
\sum_{j=0}^\infty(1+\lambda_j)|a_j|^2<\infty,
\qquad
\sum_{j=0}^\infty |b_j|^2<\infty.
\end{equation}
Define
\begin{align}
u(t)&=a_0\psi_0+b_0t\psi_0+ \sum_{j=1}^\infty\left(a_j\cos(\sqrt{\lambda_j}t)+b_j\frac{\sin(\sqrt{\lambda_j}t)}{\sqrt{\lambda_j}}\right)\psi_j.
\end{align}
The series converges in \(C(\mathbb R;H^1(\Omega))\), and its derivative converges in \(C(\mathbb R;L^2(\Omega))\), by the two summability conditions.  Each coefficient solves \(\ddot c_j+\lambda_jc_j=0\), with the \(j=0\) coefficient affine.  Hence the series solves \eqref{eq:axisymPDE-k0}.  Conservation of energy follows mode by mode, and uniqueness follows by applying the conserved energy to the difference of two solutions.
\end{proof}

\section{Endpoint and separated ODE calculus}\label{app:regularsingular}\label{app:separation}\label{app:fredholm}\label{app:green}

This appendix keeps only the separated ODE facts that are used by the main text: the algebra of separation, Wronskian identities, simple endpoint behavior, and the bounded-interval Fredholm picture.


Insert
\begin{equation}
u=\ee^{-\ii\Omega t+\ii m\phi}R(r)S(x)
\end{equation}
into \eqref{eq:boxexplicit}.  The stationary part becomes
\begin{align}
-\Omega^2\rho^2g^{tt}+2\Omega m\rho^2g^{t\phi}-m^2\rho^2g^{\phi\phi}&=\frac{((r^2+a^2)\Omega-am)^2}{\Delta_r} -\frac{(a(1-x^2)\Omega-m)^2}{\Delta_x}.
\end{align}
The mass term separates as
\begin{equation}
\frac{k}{3}\rho^2=\frac{k}{3}r^2+\frac{k}{3}a^2x^2.
\end{equation}
Thus the equation splits into \eqref{eq:radialeq-main} and \eqref{eq:angulareq-main} with one separation constant.

If \(R_1,R_2\) solve the same radial equation, subtracting the two equations after cross multiplication gives
\begin{equation}
\frac{\dd}{\dd r}\left(\Delta_r(R_1'R_2-R_1R_2')\right)=0.
\end{equation}
The angular Wronskian identity is the same with \(r\) replaced by \(x\).  These identities are the ODE version of vanishing flux.  The signs in the radial and angular equations are kept in this form so that endpoint regularity and reflecting boundary conditions can be checked by the same boundary form.  On a bounded interval the Wronskian is the Green identity boundary term; at a simple horizon root it distinguishes the horizon-regular branch from the singular outgoing branch.

\subsubsection*{Regular-singular endpoints}

\begin{lemma}[Radial simple-zero endpoint]\label{lem:radialsimplezero}
Suppose \(r_h\) is a simple zero of \(\Delta_r\):
\begin{equation}
\Delta_r(r)=\kappa_h(r-r_h)+O((r-r_h)^2),
\qquad \kappa_h\ne0.
\end{equation}
Then \eqref{eq:radialeq-main} has a regular singular point at \(r_h\), with indicial roots
\begin{equation}
s_\pm=\pm\ii\sigma_h,
\qquad
\sigma_h=\frac{(r_h^2+a^2)\Omega-am}{\kappa_h}.
\end{equation}
\end{lemma}

\begin{proof}
Put \(y=r-r_h\).  The leading part of the radial equation is
\begin{equation}
\kappa_h(yR')'+\frac{((r_h^2+a^2)\Omega-am)^2}{\kappa_h y}R=0.
\end{equation}
With \(R=y^s\), the indicial equation is \(s^2+\sigma_h^2=0\).  Frobenius theory then gives local branches \(y^{s_\pm}(1+O(y))\).
\end{proof}

\begin{lemma}[Angular simple endpoint]\label{lem:angularsimpleendpoint}
Suppose \(x_\star\) is an endpoint at which
\begin{equation}
\Delta_x(x)=\kappa_\star(x-x_\star)+O((x-x_\star)^2),
\qquad \kappa_\star\ne0.
\end{equation}
Set
\begin{equation}
\nu_\star=a(1-x_\star^2)\Omega-m.
\end{equation}
Then \eqref{eq:angulareq-main} has indicial roots
\begin{equation}
s_\pm=\pm\frac{\nu_\star}{\kappa_\star}.
\end{equation}
If two functions lie on the regular branch with \(\Re s_+>0\), then their angular Wronskian satisfies
\begin{equation}
\Delta_x(\overline S_1S_2'-\overline S_1'S_2)\to0
\qquad\text{as }x\to x_\star.
\end{equation}
\end{lemma}

\begin{proof}
With \(y=x-x_\star\), the leading angular equation is
\begin{equation}
\kappa_\star(yS')'-\frac{\nu_\star^2}{\kappa_\star y}S=0.
\end{equation}
The indicial equation is \(s^2-\nu_\star^2/\kappa_\star^2=0\).  If \(S_j=y^{s_+}\widetilde S_j\) with bounded \(\widetilde S_j\), then
\begin{equation}
\overline S_1S_2'-\overline S_1'S_2=O(y^{2\Re s_+-1}).
\end{equation}
Multiplication by \(\Delta_x=\kappa_\star y+O(y^2)\) gives the claimed vanishing.
\end{proof}

\begin{proof}[Proof of Proposition~\ref{prop:angularregularityverified}]
The admissible endpoint branch is the branch selected by smoothness in the endpoint chart.  Lemma~\ref{lem:angularsimpleendpoint} shows that the corresponding Wronskian boundary contribution vanishes.  Hence the angular Green identity has no boundary term on the admissible domain, and the angular Sturm-Liouville realization is self-adjoint.
\end{proof}

\subsubsection*{Fredholm and Green operators on bounded intervals}

On a bounded interval, the angular and radial separated operators are closed operators obtained from regular or regular-singular Sturm-Liouville forms with separated boundary conditions.  When the endpoint Wronskians vanish, the realization has compact resolvent.  Consequently the bounded-domain spectra are discrete, and for the axisymmetric operator
\begin{equation}
L_0\psi_j=\lambda_j\psi_j,
\qquad
0=\lambda_0<\lambda_1\le\lambda_2\le\cdots,
\end{equation}
the resolvent on the orthogonal complement of constants is
\begin{equation}
G_\perp f=\sum_{j\ge1}\frac{\langle f,\psi_j\rangle_A}{\lambda_j}\psi_j.
\end{equation}
For the wave resolvent,
\begin{align}
(L_0-\sigma^2)^{-1}f&=-\sigma^{-2}\Pi_0f+ \sum_{j\ge1}\frac{\langle f,\psi_j\rangle_A}{\lambda_j-\sigma^2}\psi_j.
\end{align}
Thus the zero-frequency singularity is exactly the rank-one pole associated with constants; in time, this pole is the affine contribution \(c_0+c_1t\).

\section*{Acknowledgements}
The author wishes to express his deepest gratitude to his family for their unwavering support and encouragement. The bounded-slab conformal stability analysis was supported by Riset Unggulan ITB 2024 No.~959/IT1.B07.1/TA.00/2024. The zero-frequency threshold analysis and the weak subextremal Kerr-Newman exterior verification were supported by Riset Unggulan ITB 2025 No.~841/IT1.B07.1/TA.00/2025. The full subextremal and extremal Kerr-Newman exterior verification is funded by the Indonesian Endowment Fund for Education (LPDP) on behalf of the Indonesian Ministry of Higher Education, Science and Technology and managed under the EQUITY Program (Contract No. 4298/B3/DT.03.08/2025).

\end{document}